\documentclass{pspum-l}

\usepackage{amssymb}

\usepackage{graphicx}

\usepackage[cmtip,all]{xy}

\usepackage{amsmath, amsthm,amsfonts,stmaryrd, mathrsfs,amscd,hyperref}

\newtheorem{theorem}{Theorem}[section]
\newtheorem{lemma}[theorem]{Lemma}
\newtheorem{prop}[theorem]{Proposition}
\newtheorem{cor}[theorem]{Corollary}
\newtheorem{conjecture}[theorem]{Conjecture}

\theoremstyle{definition}
\newtheorem{definition}[theorem]{Definition}
\newtheorem{example}[theorem]{Example}

\newtheorem{notation}[theorem]{Notation}

\theoremstyle{remark}
\newtheorem{remark}[theorem]{Remark}

\numberwithin{equation}{section}

\newcommand{\frakg}{{\mathfrak g}}
\newcommand{\frakh}{{\mathfrak h}}

\newcommand{\frakm}{{\mathfrak m}}

\newcommand{\frakA}{{\mathfrak A}}
\newcommand{\frakB}{{\mathfrak B}}

\newcommand{\frakS}{{\mathfrak S}}

\newcommand{\frakW}{{\mathfrak W}}

\newcommand{\bA}{{\mathbb A}}
\newcommand{\bB}{{\mathbb B}}
\newcommand{\bC}{{\mathbb C}}
\newcommand{\bD}{{\mathbb D}}

\newcommand{\bF}{{\mathbb F}}
\newcommand{\bG}{{\mathbb G}}

\newcommand{\bL}{{\mathbb L}}
\newcommand{\bM}{{\mathbb M}}
\newcommand{\bN}{{\mathbb N}}
\newcommand{\bO}{{\mathbb O}}
\newcommand{\bP}{{\mathbb P}}
\newcommand{\bQ}{{\mathbb Q}}
\newcommand{\bR}{{\mathbb R}}
\newcommand{\bS}{{\mathbb S}}
\newcommand{\bT}{{\mathbb T}}

\newcommand{\bZ}{{\mathbb Z}}

\newcommand{\mC}{{\mathcal C}}
\newcommand{\mD}{{\mathcal D}}
\newcommand{\mE}{{\mathcal E}}
\newcommand{\mF}{{\mathcal F}}
\newcommand{\mG}{{\mathcal G}}

\newcommand{\mK}{{\mathcal K}}
\newcommand{\mL}{{\mathcal L}}

\newcommand{\mN}{{\mathcal N}}
\newcommand{\mO}{{\mathcal O}}
\newcommand{\mP}{{\mathcal P}}

\newcommand{\mR}{{\mathcal R}}

\newcommand{\mU}{{\mathcal U}}

\newcommand{\mW}{{\mathcal W}}
\newcommand{\mX}{{\mathcal X}}

\newcommand{\al}{{\alpha}}

\newcommand{\ga}{{\gamma}}
\newcommand{\Ga}{{\Gamma}} 
\newcommand{\la}{{\lambda}}
\newcommand{\La}{{\Lambda}}

\newcommand{\Aff}{{\mathbf{Aff}}}   
\newcommand{\Alg}{{\mathbf{Alg}}}  
\newcommand{\Aut}{{\on{Aut}}}
\newcommand{\CA}{\mathbf{CAlg}}   
\newcommand{\cha}{{\on{char}}}       
\newcommand{\Coh}{{\on{Coh}}}        

\newcommand{\End}{{\on{End}}}

\newcommand{\Gr}{{\on{Gr}}}   
\newcommand{\heart}{\ensuremath\heartsuit}
\newcommand{\Hom}{{\on{Hom}}}

\newcommand{\Ind}{{\on{Ind}}}

\newcommand{\Lie}{{\on{Lie}}}
\newcommand{\Map}{{\on{Map}}}
\newcommand{\Mod}{{\mathbf{Mod}}}
\newcommand{\Mon}{{\mathbf{Mon}}}   
\newcommand{\Nm}{{\on{Nm}}}      
\newcommand{\Ql}{{\overline{\bQ}_\ell}}

\newcommand{\pr}{{\on{pr}}}         
\newcommand{\Pro}{{\on{Pro}}}    

\newcommand{\Rep}{{\on{Rep}}}
\newcommand{\res}{{\on{res}}}
\newcommand{\Res}{{\on{Res}}}

\newcommand{\Sets}{\mathbf{Sets}}   

\newcommand{\Shv}{{\on{Shv}}}      
\newcommand{\Spc}{\mathbf{Spc}}  

\newcommand{\Spf}{{\on{Spf}}}

\newcommand{\Tor}{{\mathbf{Tor}}}  
\newcommand{\Tr}{{\on{Tr}}}   
\newcommand{\xch}{\mathbb{X}^\bullet}      
\newcommand{\xcoch}{\mathbb{X}_\bullet}  

\newcommand{\Bun}{{\on{Bun}}}
\newcommand{\loc}{{\on{loc}}}
\newcommand{\Loc}{{\on{Loc}}}
\newcommand{\Hk}{{\on{Hk}}}

\newcommand{\Sh}{{\on{Sh}}}   
\newcommand{\Sht}{{\on{Sht}}}  

\newcommand{\GL}{{\on{GL}}}

\newcommand{\PGL}{{\on{PGL}}}
\newcommand{\SL}{{\on{SL}}}
\newcommand{\Sp}{{\on{Sp}}}

\newcommand{\SO}{{\on{SO}}}

\newcommand{\ab}{{\on{ab}}}    
\newcommand{\ad}{{\on{ad}}}          
\newcommand{\Ad}{{\on{Ad}}}

\newcommand{\s}{{\on{sc}}}        
\newcommand{\Sat}{{\on{Sat}}}   
\newcommand{\St}{{\on{St}}}   
\newcommand{\ta}{{\on{tame}}}
\newcommand{\un}{{\on{unip}}}
\newcommand{\ur}{{\on{ur}}}

\newcommand{\Wh}{{\mathbf{W}}}
\newcommand{\TS}{{\mathbf{TS}}}
\newcommand{\Pin}{{\mathrm{Pin}}}

\newcommand{\iso}{{\on{iso}}}

\newcommand{\Ani}{{\mathbf{Ani}}}

\newcommand{\col}{\mathbf{Col}}

\newcommand{\FM}{{\on{FM}}}

\newcommand{\FFS}{{\mathbf{FFS}}}
\newcommand{\FFG}{{\mathbf{FFG}}}

\newcommand{\FFM}{{\mathbf{FFM}}}

\newcommand{\on}{\operatorname}
\newcommand{\quash}[1]{}  

\begin{document}

\title{Coherent sheaves on the stack of Langlands parameters}


\author{Xinwen Zhu}
\address{Department of Mathematics, Stanford University, USA}
\email{zhuxw@stanford.edu}
\thanks{The author is partially supported by NSF under agreement Nos. DMS-1902239 and a Simons fellowship.}


\subjclass[2020]{Primary }

\date{}

\begin{abstract}
We construct the stacks of arithmetic Langlands parameters in the local ($\ell\neq p$) and global function field settings.
We formulate a few conjectures on some hypothetical coherent sheaves on these stacks, and explain their roles played in the local and global Langlands program. 
We survey some known results as evidences of these conjectures.
\end{abstract}

\maketitle

\tableofcontents

\section{Introduction}
In recent years, it has become increasingly clear that there should exist certain (complexes of) coherent sheaves $\frakA$ on the stacks of local and global arithmetic Langlands parameters. These sheaves are expected to largely govern the Langlands correspondence and and allow one to formulate local-global compatibilities within the arithmetic Langlands program. The existence of such objects is already suggested by the work of Emerton-Helm \cite{EH} and Helm \cite{He16}, under the framework of the local Langlands correspondence in families.\footnote{There are similar $\frakA$ appearing in the $p$-adic local Langlands program, see \cite{EGH} for a discussion.} This idea has been further developed recently by Hellmann \cite{Hel}. On the other hand, after the work of V. Lafforgue and Genestier-Lafforgue \cite{L,GL}, these ideas become more concrete, and powerful tools from the geometric Langlands program are now available to realize (part of) them. Indeed, it is expected that the entire arithmetic local Langlands correspondence over a non-archimedean local field admits a categorical incarnation (see, for instance, \cite[4.2]{Ga} for some indications). The existence of such coherent sheaves fits naturally into this categorical framework, as we aim to explain in this article.
In a related direction, the work of Fargues-Scholze \cite{FS} on the geometrization of the local Langlands correspondence is also closely aligned with these ideas and likewise points to a categorical form of the arithmetic local Langlands correspondence.  From a global perspective, the existence of $\frakA$ serves as a guiding principle in the author's joint work with Xiao \cite{XZ} on the geometric realization of the Jacquet-Langlands correspondence via the cohomology of Shimura varieties. In another direction, a crude form of such a coherent sheaf appears in the author's work with V. Lafforgue \cite{LZ}, where it is used to describe the elliptic part of the cohomology of Shtukas in the framework of the Arthur-Kottwitz conjectures.

In this article, we formulate several precise conjectures related to the hypothetical sheaves $\frakA$ and survey known results, including explicit conjectural descriptions of $\frakA$ in some special (but particularly important) cases, along with their roles in local-global compatibility. We also propose a conjectural categorical form of the local arithmetic Langlands correspondence, which provides a conceptual justification for the expected existence of such $\frakA$.
In order to formulate these conjectures,  we discuss the construction and some properties of the moduli stacks of local Langlands parameters (in the case $\ell \neq p$) and global Langlands parameters (in the function field setting). 
We note that some of the ideas presented in this article have been informally shared among experts for several years.\footnote{Indeed, around the time the first version of this article was made public, several related works appeared. See, for example, \cite{Hel,DHKM,BCHN,AGK,FS}.} It is the author's intent to make some of these ideas more precise and to commit them to writing.

This article naturally divides into two parts. Sections \ref{S: RepSp} and \ref{S: stack Loc} are devoted to a general study of moduli spaces of representations and the construction of moduli spaces of Langlands parameters. Since the results in these sections are original, we provide detailed proofs of nearly all assertions. Section \ref{S: coh on st} is dedicated to formulating our main conjectures. It includes some original results (such as Theorem \ref{T: gen LZ}), for which we again give detailed proofs. At the same time, this section also surveys known or forthcoming results that provide evidence for our conjectures, and as such, has a more expository character in places.

\medskip

\noindent\bf Acknowledgement \rm The author would like to thank R. Bezrukavnikov, M. Emerton, T. Hemo, L. Xiao, Z. Yun for many discussions during preparing the article. He would like to thank M. Emerton and T. Feng for inspiring discussions which leads to Conjecture \ref{C: derSat}, and D. Ben-Zvi for discussions around Conjecture \ref{C: REnd of Spr}. 
He would like to thank P. Scholze for pointing out several inaccuracies in the early draft of the article, M. Emerton for many valuable comments and suggestions, and D. Hansen for feedbacks.

\section{Representation space}\label{S: RepSp}

Let $M$ be an affine group scheme over a commutative ring $\Lambda$ and $\Ga$ an abstract group. It is well-known that there is an affine scheme ${}^{cl}\mR_{\Ga,M}$ over $\Lambda$ such that for every $\Lambda$-algebra $A$, the set ${}^{cl}\mR_{\Ga,M}(A)$ classifies group homomorphisms from $\Ga$ to $M(A)$.  Namely, one first considers the functor over $\Lambda$ classifying all maps from $\Ga$ to $M(A)$ as  \emph{sets}. This functor is obviously represented by the self product $M^\Ga$ of $M$ over $\Ga$. The imposition of the condition that these set maps be group homomorphisms defines ${}^{cl}\mR_{\Ga,M}$ as a closed subscheme of $M^\Ga$. 

The first issue is well known: Galois groups are profinite groups, and one must consider continuous representations of them, subject to certain additional properties. We will address this issue in Section \ref{SS: Continuous rep}. Roughly speaking, by imposing the continuity condition, we obtain ind-schemes whose completions at closed points recover the usual framed deformation spaces of representations of profinite groups. However, such spaces might not possess good global geometry in general (see Example \ref{E: disc cont moduli}). Nonetheless, in the cases considered in Section \ref{S: stack Loc}, these spaces "glue" all deformation spaces together in a reasonable manner.

The second issue concerns the fact that the equations defining ${}^{cl}\mR_{\Ga,M}\subset M^\Ga$ typically do not form a ``regular sequence," which can lead to non-trivial derived structures on ${}^{cl}\mR_{\Ga,M}$. At various points in the sequel, we will need to keep in mind the potential derived structure of these spaces. Thus, we will review the construction of derived objects in Section \ref{SS: dermon}. This construction is certainly well known (see, for example, \cite{To, SV}), but our approach will be inspired by \cite{L}, following a review of the derived category of monoids in Section \ref{SS: der mon}.

\subsection{The derived category of monoids}\label{SS: der mon}
Our goal is to define a derived geometric object $\mR_{\Ga,M}$ parameterizing homomorphisms from $\Ga$ to $M$. 
To achieve this, it is convenient to begin with a more general framework by considering homomorphisms of monoids.
The idea is to move from the category $\Mon$ of monoids to its derived category. 
As $\Mon$ is non-abelian, we must adopt the notion of non-abelian derived categories in the sense of Quillen, as developed by Lurie using the language of $\infty$-categories \cite[5.5.8]{Lu1}. We will first recall some general theory and then specialize to the examples relevant to our context.

In the sequel, we call $(\infty,1)$-categories just by $\infty$-categories, and regard ordinary categories as $\infty$-categories in the usual way.  
Let $\Spc$ denote the $\infty$-category of spaces, containing the category $\Sets$ of sets as a full subcategory (by regarding sets as discrete spaces).
The inclusion $\Sets\to\Spc$ admits a left adjoint $\pi_0:\Spc\to\Sets$ which preserves finite products.
If $x,y$ are two objects in an $\infty$-category $\mC$, we write $\Map_{\mC}(x,y)\in\Spc$ for the space of maps from $x$ to $y$. (We use this notation even if $\mC$ is an ordinary category, in which case this space is discrete.) All functors are understood in the $\infty$-categorical setting (and therefore are derived). 
Let $\on{Fun}(\mC,\mD)$ denote the $\infty$-category of functors between two $\infty$-categories $\mC$ and $\mD$.  Let $\Delta$ be the (ordinary) simplex category. We refer to \cite{Lu1} for foundations of $\infty$-categories. 

We find it is instructive to adopt Clausen-Scholze's point of view to start with. For an ordinary category $\mC$ admitting colimits, let $\mC^{\on{cp}}$ denote its full subcategory of compact projective objects in $\mC$, i.e. those $x\in \mC$ such that $\on{Map}_{\mC}(x,-)$ commutes with filtered colimits and reflexive coequalizers. This is a category admitting finite coproducts, so one can define its non-abelian derived category $\mP_{\Sigma}(\mC^{\on{cp}})$ (\cite[5.5.8.8]{Lu1}), which is the full subcategory of $\on{Fun}((\mC^{\on{cp}})^{\on{op}},\Spc)$ consisting those functors that preserve finite products\footnote{We implicitly assume that $\mC^{\on{cp}}$ is small, which is the case for all examples we encounter.}. If $\mC$ is generated by $\mC^{\on{cp}}$ under $1$-categorical colimits (informally this means objects in $\mC$ can be obtained from objects in $\mC^{\on{cp}}$ by taking ``unions" (filtered colimits) and ``presentations" (reflexive coequalizers)), then  
$\mP_{\Sigma}(\mC^{\on{cp}})$ is called the $\infty$-category of anima of $\mC$ by Clausen-Scholze, and is denoted by $\Ani(\mC)$. (See \cite[\S 5]{Sch} for an account.)
We sometimes also just call it the derived category of $\mC$. Objects in $\Ani(\mC)$ can be generated by $\mC^{\on{cp}}$ under $\infty$-categorical colimits.
Note that if $\mC$ has a symmetric monoidal structure such that the tensor product preserves colimits separately in each variable, and that the symmetric monoidal structure restricts to a symmetric monoidal structure on $\mC^{\on{cp}}$, then $\Ani(\mC)$ is naturally a symmetric monoidal $\infty$-category and the tensor product preserves colimits separately in each variable (\cite[4.8.1.10]{Lu2}). 

There is a fully faithful embedding $\mC\subset \Ani(\mC)$, by regarding $\mC$ as the category of finite-product preserving functors $(\mC^{\on{cp}})^{\on{op}}\to \Spc$ factoring as $(\mC^{\on{cp}})^{\on{op}}\to \Sets\subset \Spc$. It admits a left adjoint $\pi_0: \Ani(\mC)\to \mC$ induced by $\pi_0:\Spc\to\Sets$. 
More generally, for each $n\geq 0$, there is the $n$-truncation functor $\tau_{\leq m}: \Ani(\mC)\to  {}_{\leq m}\Ani(\mC)$, where
 for an $\infty$-category $\mC$, ${}_{\leq m}\mC$ denotes the full subcategory of $m$-truncated objects of $\mC$ (\cite[5.5.6.1]{Lu1}), which is a left adjoint of the natural inclusion functor ${}_{\leq m}\Ani(\mC)\subset\Ani(\mC)$ (\cite[5.5.6.18]{Lu1}). 
The following are some basic examples.
\begin{example}\label{Ex: anima}
\begin{enumerate}
\item If $\mC=\Sets$, equipped with the Cartesian symmetric monoidal structure (i.e. tensor products are given by products), then $\mC^{\on{cp}}$ is the category $\Sets_f$ of finite sets, and $\Ani:=\Ani(\Sets)\cong \Spc$ (\cite[5.5.8.24]{Lu1}), equipped with the Cartesian symmetric monoidal structure. Because of this natural equivalence, we will use $\Ani$ and $\Spc$ interchangeably in the sequel.

\item Let $\Lambda$ be a commutative ring. If $\mC=\Mod_\Lambda^{\heart}$ is the abelian category of $\Lambda$-modules, equipped with the usual tensor product structure, then $\mC^{\on{cp}}$ is the category of finite projective $\Lambda$-modules and $\Ani(\Mod_\Lambda^{\heart})$ is equivalent to the derived category $\Mod_\Lambda^{\leq 0}:=D^{\leq 0}(\Mod_\Lambda^{\heart})$ of connective complexes of $\Lambda$-modules (i.e. those complexes whose cohomology vanish in positive degrees\footnote{In the paper, we use cohomological convention for complexes in the stable $\infty$-category $\Mod_\Lambda$ of $\Lambda$-modules. So for $N\in \Mod_\Lambda$, we write $H^iN=\pi_{-i}N$, and $N[j]$ for the object satisfying $H^i(N[j])=H^{i+j}N$. The usual truncation functors in homological algebras are written as $\tau^{\leq n}, \tau^{\geq n}: \Mod_\Lambda\to\Mod_\Lambda$, which is different from the truncation functor $\tau_{\leq m}$ as in \cite[5.5.6.18]{Lu1}. However, the restriction of $\tau^{\geq -m}$ to $\Mod_\Lambda^{\leq 0}$ is isomorphic to $\tau_{\leq m}$.}), equipped with the usual symmetric monoidal structure (\cite[5.5.8.21]{Lu1} and \cite[5.1.6]{Sch}). 
\end{enumerate}
\end{example}

The example we need is the category of monoids $\mC=\Mon$.  This category admits all small colimits, and is generated under colimits by its compact projective objects, which are finitely freely generated monoids.
For a finite set $I$, let $\FM(I)$ denote the free monoid generated by $I$. Let  $\FFM$ be the full subcategory spanned by these $\FM(I)$s.
 For a monoid $\Ga$, let $\FFM/\Ga$ denote the corresponding slice category: I.e. objects are pairs of the form $(\FM(I), u: \FM(I)\to \Ga)$ and morphisms from $(\FM(I),u)$ to $(\FM(J),v)$ are monoid homomorphisms $f: \FM(I)\to\FM(J)$ such that $u=vf$. We note that the category $\FFM/\Ga$ is not filtered, but is sifted (see \cite[5.5.8.1]{Lu1} for this notion), as coproducts exist in $\FFM/\Ga$.
There is a canonical isomorphism in $\Mon$
\begin{equation}\label{E: colim}
\varinjlim_{\FFM/\Ga}\FM(I)\xrightarrow{\cong} \Ga.
\end{equation}
This isomorphism can also be understood in $\Ani(\Mon)$, via the fully embedding $\Mon\subset\Ani(\Mon)$,
as $\Ani(\Mon)=\mP_{\Sigma}(\FFM)$. 

On the other hand, for an $\infty$-category $\mC$ admitting finite products, there is the $\infty$-category $\Mon(\mC)$ of monoid objects in $\mC$, which by definition is the full subcategory of the category 
\[
\mC_{\Delta}:=\on{Fun}(\Delta^{\on{op}},\mC)
\] 
of simplicial objects in $\mC$, consisting of those $X_\bullet$
such that for every $[n]\in\Delta$, the map
\begin{equation*}\label{E: segal}
X([n])\to X(\{0,1\})\times X(\{1,2\})\times\cdots\times X(\{n-1,n\})=X([1])^n
\end{equation*}
induced by $[1]\cong \{i-1,i\}\subset \{0,1,\ldots,n\}=[n]$, is an isomorphism in $\mC$ (\cite[4.1.2.5]{Lu2}). 
For example, if $\mC=\Sets$, then $\Mon\cong \Mon(\Sets)$ via the usual Milnor construction: for $\Ga\in\Mon$, the corresponding object in $\Mon(\Sets)$ is the nerve of the category with a unique object whose endomorphism monoid is $\Ga$ (\cite[4.1.2.4]{Lu2}). 
Then the fully faithful embedding $\Sets\subset\Spc$ induces a fully faithful embedding $\Mon\subset \Mon(\Spc)$ (as both of which are full subcategories of $\Spc_{\Delta}$). 

\begin{example}\label{R: monoid, segal space}
\quash{Recall that there is a fully faithful embedding from the $\infty$-category of (small) $\infty$-categories to $\Spc_{\Delta}$, sending $\mC$ to the simplicial space assigning $[n]\mapsto \on{Fun}(\Delta^n,\mC)^\simeq$, the largest Kan complex inside $\on{Fun}(\Delta^n,\mC)$.  The essential image of this functor consists of the so-called complete Segal spaces. In this way, every $\infty$-category with one object gives a monoid object in $\Spc$. }
By \cite[4.7.1]{Lu2}, given an object $x$ in an $\infty$-category $\mC$, there is a monoid $\End_\mC(x)\in \Mon(\Spc)$, whose value at $[1]\in\Delta$ is isomorphic to $\Map_\mC(x,x)$, which is universal among all objects in $\Mon(\Spc)$ that act on $x$. We call it the derived endomorphism monoid of $x$, 
\end{example}

Now we have two $\infty$-categories that can be regarded as a derived version of $\Mon$. Fortunately, they are canonically equivalent.
\begin{lemma}\label{L: MonAni}
There is a canonical equivalence $\Ani(\Mon)\cong \Mon(\Spc)$.
\end{lemma}
\begin{proof}
We consider a more general situation. Let $\mC$ be a(n ordinary) cocomplete symmetric monoidal category as before (i.e. $\mC$ is generated by $\mC^{\on{cp}}$ under colimits and the tensor product preserves colimits separately in each variable). Then it makes sense to talk about the ($\infty$-)category $\mathbf{Alg}(-)$ of its associative (a.k.a $E_1$-)algebra objects in $\mC$ and $\Ani(\mC)$ (\cite[2.1.3]{Lu2}). Using \cite[7.2.4.27]{Lu2} and Lemma \ref{L: free alg} below, we obtain a canonical equivalence  
\begin{equation*}\label{E: AlgAni}
\Ani(\Alg(\mC))\cong \Alg(\Ani(\mC)).
\end{equation*} 
The lemma follows by letting $\mC=\Sets$ and identifying associative algebra objects with monoid objects when the ambient symmetric monoidal structure is Cartesian (\cite[2.4.2, 4.1.2.10]{Lu2}).  
\end{proof}

To state the following lemma, recall from \cite[3.1.3]{Lu2} that for $(-)=\mC$ or $\Ani(\mC)$, the forgetful functor from $\mathbf{Alg}(-)\to (-)$ admits a left adjoint $\on{Fr}_{(-)}$, given by the free algebra construction.
\begin{lemma}\label{L: free alg}
For every $X\in \mC^{\on{cp}}$, the image of $\on{Fr}_\mC(X)$ under the functor $\mathbf{Alg}(\mC)\to \mathbf{Alg}(\Ani(\mC))$ is canonically isomorphic to $\on{Fr}_{\Ani(\mC)}(X)$.
\end{lemma}
We note that this lemma is specific to $E_1$-algebras, as the analogous statement for $E_\infty$-algebras is well-known to be false in general\footnote{We thank P. Scholze for pointing out this.}.
\begin{proof}
We regard $\on{Fr}_{\mC}(X)$ as an object in $\Alg(\Ani(\mC))$.
Then there is a canonical morphism $\on{Fr}_{\Ani(\mC)}(X)\to \on{Fr}_{\mC}(X)$ given by adjunction. To show that it is an isomorphism, we can apply the forgetful functor $\mathbf{Alg}(\Ani(\mC))\to \Ani(\mC)$, as this functor is conservative (\cite[3.2.2.6]{Lu2}). Now in $\Ani(\mC)$, both objects are given by $\sqcup_{n\geq 0} X^{\otimes n}$, by combining \cite[3.1.3.13]{Lu2} with the fact that the embedding $\mC^{\on{cp}}\to \Ani(\mC)$ is monoidal and preserves finite coproducts.
\end{proof}

Here is the corollary we need. It can be regarded as a canonical ``projective resolution" of an object in $\Mon(\Spc)$. 
See \cite[2.1.5]{GKRV} for a closely related statement (with a different proof).
\begin{cor}\label{L: colim by FFM}
The isomorphism \eqref{E: colim} holds in $\Mon(\Spc)$. In particular, for every $X_\bullet\in\Mon(\Spc)$,
\begin{equation}\label{E: resol map}
\on{Map}_{\Mon(\Spc)}(\Ga,X_\bullet)=\varprojlim_{(\FFM/\Ga)^{\on{op}}}\on{Map}_{\Mon(\Spc)}(\FM(I),X_\bullet)= \varprojlim_{(\FFM/\Ga)^{\on{op}}}X([1])^I.
\end{equation}
\end{cor}
Of course, \eqref{E: colim} holds for every $\Ga\in\Mon(\Spc)$ except that in this case $\FFM/\Ga$ might no longer be an ordinary category.

\begin{remark}
There are variants of the above discussions, by replacing monoid objects by group or semigroup objects in a category $\mC$. Following \cite[5.2.6.2,4.1.2.12]{Lu2}, we regard group objects as grouplike monoid objects and semigroup objects as non-unital monoid objects, and denote the corresponding categories by $\Mon^{\on{gp}}(\mC)$ and $\Mon^{\on{nu}}(\mC)$ respectively (and omit $\mC$ from the notation if $\mC=\Sets$). 
For $?=\on{gp}$ or $\on{nu}$, compact projective objects of $\Mon^?$ are still finitely freely generated ones. Following \cite{Wei}, we denote the corresponding subcategories by $\FFG$ and $\FFS$ respectively. 
We still have $\Ani(\Mon^?)\cong \Mon^?(\Spc)$ and therefore analogous Corollary \ref{L: colim by FFM}. Indeed, the semigroup case can be proved similarly, and the group case follows from Lemma \ref{L: MonAni} and \cite[5.2.6.4]{Lu2} (and in fact is already contained in \cite[5.2.6.10, 5.2.6.21]{Lu2}).

There are natural forgetful functors $\Mon^{\on{gp}}(\Spc)\to \Mon(\Spc)\to \Mon^{\on{nu}}(\Spc)$. The first and the composition functors are fully faithful. In our application, we will mainly consider spaces of maps between groups so we can calculate them in any of these three categories. 
\end{remark}

\begin{remark}\label{rem: Delta to FFM}
There is a natural faithful functor $\col: \Delta\to \FFM$ sending $[n]$ to $\FM(\{x_1,\ldots,x_n\})$ and $f: [n]\to [m]$ to $\col(f):\FM(\{x_1,\ldots,x_n\})\to \FM(\{y_1,\ldots,y_m\})$ defined by
\[
\col(f)(x_i)=\left\{\begin{array}{ll} y_{f(i)}y_{f(i)+1}\cdots y_{f(i+1)-1} & f(i+1)>f(i) \\ e & f(i+1)=f(i). \end{array}\right.
\] 
It is not difficult to see that $\col$ is cofinal. In addition, $\col^{\on{op}}$ induces a functor 
\[
\on{Fun}(\FFM^{\on{op}}, \Spc)\to \on{Fun}(\Delta^{\on{op}},\Spc)
\] 
which restricts to the equivalence in Lemma \ref{L: MonAni}.
\end{remark}

\subsection{The derived representation space}\label{SS: dermon}
We fix a commutative ring $\Lambda$. Let $\CA_\Lambda^\heart$ denote the (ordinary) category of commutative $\Lambda$-algebras,  and we will sometimes refer to the objects in $\CA^\heart_\Lambda$ as classical $\Lambda$-algebras.
We let $\CA_\Lambda=\Ani(\CA^\heart_\Lambda)$ be its derived category and, following Clausen-Scholze, we call objects in $\CA_\Lambda$ animated $\Lambda$-algebras.\footnote{This category is denoted by $\CA_\Lambda^\Delta$ in \cite[\S 25]{Lu3}, where its objects are traditionally called simplicial $\Lambda$-algebras. However, we will reserve the notation $\CA_\Lambda^\Delta$ for cosimplicial objects in $\CA_\Lambda=\Ani(\CA^\heart_\Lambda)$.} We have a natural forgetful functor 
$$\CA_\Lambda=\Ani(\CA^\heart_\Lambda)\to\Ani(\Mod^\heart_\Lambda)\cong \Mod_\Lambda^{\leq 0},$$ which is conservative preserving limits and sifted colimits (by combining \cite[25.1.2.2]{Lu3} with \cite[3.2.2.1,3.2.2.6,,3.2.3.1]{Lu2}). For an animated $\Lambda$-algebra $A$, we write $\pi_i(A)$ for $(-i)$th cohomology of its underlying $\Lambda$-module.
An animated $\Lambda$-algebra $A$ is called truncated if it belongs to ${}_{\leq m}\CA_\Lambda$ for some $m<\infty$, which is equivalent to saying $\pi_i(A)=0$ for $i>m$.

Let $\Aff_\Lambda$ (resp. $\mathbf{DAff}_\Lambda$) denote the opposite of $\CA^\heart_\Lambda$ (resp. $\CA_\Lambda$). 
Objects in $\Aff_\Lambda$ will be called classical affine $\Lambda$-schemes, or simply affine $\Lambda$-schemes, and objects in $\mathbf{DAff}_\Lambda$ will be called derived affine $\Lambda$-schemes, or animated $\Lambda$-affine schemes. 
Given $A\in \CA_\Lambda$, the corresponding object in $\mathbf{DAff}_\Lambda$ is denoted by $\on{Spec} A$ as usual, and given $X\in \mathbf{DAff}_\Lambda$, the corresponding object in $\CA_\Lambda$ is denoted by $\Lambda[X]$, called the ring of regular functions on $X$.
For $X=\on{Spec} A$, we write ${}^{cl}X$ for the underlying classical affine scheme $\on{Spec} \pi_0(A)$.  We say an affine $\Lambda$-scheme $\on{Spec} A$ is ($m$-)truncated if $A$ is ($m$-)truncated. (Note that this is different from $\on{Spec} A$ being an $m$-truncated object in $\mathbf{DAff}_\Lambda$.)

Let $M$ be a classical affine flat monoid scheme over $\Lambda$. It is an object in $\Mon(\Aff_\Lambda)$. Then the functor $\CA^\heart_\Lambda\to\Mon$ defined by $M$ extends to a (sifted colimit preserving) functor 
$$\CA_\Lambda=\Ani(\CA^\heart_\Lambda)\to \Ani(\Mon)\cong \Mon(\Spc),$$ 
still denoted by $M$. Unveiling the definition, for $A\in\CA_\Lambda$, $M(A)\in\Mon(\Spc)$ is the simplicial space given by 
$$[n]\in\Delta\mapsto \Map_{\CA_\Lambda}(\Lambda[M^n],A)\cong \Map_{\CA_\Lambda}(\Lambda[M],A)^n.$$

\begin{definition}\label{D: der hom space}
For $\Ga\in \Mon(\Spc)$, we define
\begin{equation}\label{E: der hom space2}
\mR_{\Ga,M}: \CA_\Lambda\to \Spc,\quad A\mapsto \Map_{\Mon(\Spc)}(\Ga,M(A)).
\end{equation}
\end{definition}

\begin{remark}
Our definition is equivalent to the one in \cite[\S 3.2]{To}.  Let 
\[
\CA_\Lambda^\Delta=\on{Fun}(\Delta,\CA_\Lambda)
\] 
be the category of cosimplicial objects in $\CA_\Lambda$. 
We have
\begin{equation}\label{R: cohom pres}
\Map(\Ga, M(A))=\Map_{\Spc_\Delta}\bigl(\Ga^\bullet, \Map_{\CA_\Lambda}(\Lambda[M^\bullet],A)\bigr)\cong \Map_{\CA_\Lambda^\Delta}\bigl(\Lambda[M^\bullet], C(\Ga^\bullet,A)\bigr),
\end{equation}
where $C(\Ga^\bullet,A)\in\CA_\La^\Delta$ is the object representing the functor 
\[
(\CA_\La^\Delta)^{\mathrm{op}}\to\Spc,\quad B^\bullet\mapsto \Map_{\Spc_\Delta}\bigl(\Ga^\bullet, \Map_{\CA_\Lambda}(B^\bullet,A)\bigr).
\] 
It is easy to see that the $n$th term of $C(\Ga^\bullet,A)$ is the animated $\La$-algebra given by
\begin{equation}\label{E: exp space}
C(\Ga^n,A):=\varprojlim_{\Ga^n}A=A^{\Ga^n},
\end{equation}
the $\La$-algebra of maps from $\Ga^n$ to $A$ (see \cite[5.5.2.6]{Lu1} for this notion in the $\infty$-categorical setting). 

On the other hand,
if $M$ is a group scheme so $M(A)$ is grouplike, by \cite[5.2.6.10, 5.2.6.13]{Lu2}
taking the geometric realizations (of simplicial spaces) induces an equivalence
\begin{equation}\label{E: ptsp}
\Map_{\Mon(\Spc)}(\Ga, M(A))\to \Map_{\Spc_{*}}(|\Ga|, |M(A)|),
\end{equation}
where $\Spc_{*}$ denote the $\infty$-category of pointed spaces (\cite[1.4.2.5]{Lu2}). 
Therefore, our definition also agrees with the definition of (framed) derived moduli space of representations as in \cite[\S 5]{SV}. (The geometric realization $|\cdot|$ is denoted by $B(\cdot)$ in \emph{loc. cit.})
\end{remark}

Using the ``resolution" of $\Ga$ from Corollary \ref{L: colim by FFM} 
we immediately arrive the following presentation of $\mR_{\Ga,M}$, which in particular implies the representability of $\mR_{\Ga,M}$ as a derived affine scheme. 

\begin{prop}\label{P: appr by free}
There is a natural isomorphism 
\[
\mR_{\Ga,M}\cong \varprojlim_{(\FFM/\Ga)^{\on{op}}} M^I, 
\]
where the limit is taken in $\mathbf{DAff}_\Lambda$.
Consequently, there is a canonical isomorphism in $\CA_\Lambda$
\begin{equation}\label{E: FFM alg}
\Lambda[\mR_{\Ga,M}]\cong \varinjlim_{\FFM/\Ga} \Lambda[M^I]. 
\end{equation}
\end{prop}
As mentioned before, $\FFM/\Ga$ is not a filtered category, even if $\Ga$ is discrete. Therefore, although each $\Lambda[M^I]$ sits only in cohomological degree zero, this may not be the case for $\Lambda[\mR_{\Ga,M}]$. 
\begin{example}\label{E: free}
If $\Ga=\FM(I)$, $\mR_{\FM(I),M}\cong {}^{cl}\mR_{\FM(I),M}\cong M^I$. 
This is consistent with the intuition: since no relation is imposed if $\Ga$ is free, there shouldn't exist non-trivial derived structure of ${}^{cl}\mR_{\Ga,M}$ in this case. 
\end{example}

\begin{remark}\label{R: FFM scheme}
\begin{enumerate}
\item The above result suggests the following generalization, which is useful for the discussion of pseudorepresentations. 
Consider  $\CA_\Lambda^{\FFM}:=\on{Fun}(\FFM,\CA_\Lambda)$, the category of $\FFM$-algebras in the sense of \cite{Wei}. We denote an object in this category as $A^\bullet$. For a finite non-empty set $I$, we write $A^I$ for the image of $\FM(I)$ under the functor $A^\bullet$. For example, associated to $M\in \Mon(\Aff_\Lambda)$ there is an $\FFM$-algebra $\La[M]^\bullet$ sending $\FM(I)$ to $\La[M^I]$. On the other hand, let $B\in \CA_\Lambda$, and let $\Ga\in\Mon(\Spc)$. We regard $B$ as the constant functor $\FFM/\Ga\to \CA_\Lambda$ with value $B$. Its right Kan extension along $\FFM/\Ga\to \FFM$ is nothing but the $\FFM$-algebra $\FM(I)\mapsto C(\Ga^I,B)$. 

Now, for an $\FFM$-algebra $A^\bullet$ and $\Ga\in\Mon(\Spc)$, we may define
\begin{equation*}\label{E: hom FFM}
\mR_{\Ga,\on{Spec} A^\bullet}:= \varprojlim_{\FM(I)\in (\FFM/\Ga)^{\on{op}}} \on{Spec} A^I, \quad \mbox{ so } \ \ \Lambda[\mR_{\Ga, \on{Spec} A^\bullet}]=\varinjlim_{\FFM/\Ga} A^I.
\end{equation*}
When $A^\bullet=\La[M]^\bullet$, then $\mR_{\Ga,\on{Spec} A^\bullet}$ recovers $\mR_{\Ga,M}$.

Note that for every $B\in \CA_\Lambda$, we have 
\begin{equation}\label{E: hom FFMalg}
\Map_{\CA_\Lambda}(\Lambda[\mR_{\Ga, \on{Spec} A^\bullet}],B)\cong \Map_{\CA_\Lambda^{\FFM/\Ga}}(A^\bullet, B)\cong \Map_{\CA_\Lambda^{\FFM}}\bigl(A^\bullet, C(\Ga^\bullet,B)\bigr).
\end{equation}
As a corollary, we see that the functor $\col$ from Remark \ref{rem: Delta to FFM} induces an isomorphism
\[
\Map_{\CA_\Lambda^{\FFM}}\bigl(\La[M]^\bullet, C(\Ga^\bullet,B)\bigr)\to \Map_{\CA_\Lambda^{\Delta}}\bigl(\La[M]^\bullet, C(\Ga^\bullet,B)\bigr).
\]
We do not know whether this is true if $\La[M]^\bullet$ and $C(\Ga^\bullet,B)$ are replaced by more general $\FFM$-algebras.

\item One can replace $\FFM$ by $\FFS$ or by $\FFG$ as considered in \cite{Wei}. We shall not repeat such a remark again.
\end{enumerate}
\end{remark}

Let us come back to $\mR_{\Ga,M}$ and discuss certain vector bundles on it. For simplicity, from now on we assume that $\Ga$ is discrete, i.e. an object in $\Mon$.
This is enough for our purpose and simplifies the discussions below. As in the preceding discussion, we identify it with a category with a unique object and then a simplicial set via the Milnor construction.

We refer to \cite[\S 25.2.1]{Lu3} for the theory of modules over animated rings (see \cite[5.1]{Sch} for some further elaborations). 
For an animated $\Lambda$-algebra $A$, let $\Mod_A$ denote the $\infty$-category of $A$-modules, and $\Mod_A^{\leq 0}$ the full subcategory of connective objects. 
If $A$ is classical, $\Mod_A^{\leq 0}$ is also equivalent to $\Ani(\Mod_A^\heart)$, as introduced before.
We also call $A$-modules as quasi-coherent sheaves on $\on{Spec} A$.
  
Now, for a representation $W$ of $M$ on a finite projective $\Lambda$-module, let ${}_{\Ga}W$ denote the (trivial) vector bundle $\Lambda[\mR_{\Ga,M}]\otimes_\Lambda W$ on $\mR_{\Ga,M}$. We shall equip ${}_{\Ga}W$ with an action of $\Ga$, or more precisely construct a canonical morphism in $\Mon(\Spc)$
\begin{equation}\label{E: Ga act W}
  \Ga\to \End ({}_{\Ga}W).
\end{equation}  
Here $\End ({}_{\Ga}W)\in \Mon(\Spc)$ denotes the derived endomorphism ring (Example \ref{R: monoid, segal space}) of ${}_{\Ga}W$, regarded as a connective quasi-coherent sheaf on $\mR_{\Ga,M}$. 

In the sequel, we denote ${}_{\FM(I)}W$ by ${}_IW$ for simplicity. Note that there is a canonical isomorphism
$\varinjlim_{\FFM/\Ga} \End ({}_{I}W)\to \End ({}_{\Ga}W)$ in $\Mon(\Spc)$.
Then by Corollary \ref{L: colim by FFM}, it is enough to construct, for every $u:\FM(I)\to \Ga$, a morphism $\FM(I)\to \End ({}_IW)$, compatible with morphisms in $\FFM/\Ga$. We note that this last compatibility can be checked at the ordinary categorical level.

Next via the inclusion $\{i\}\subset I$, it is enough to assume that $I=\{1\}$ and to construct an endomorphism of ${}_{\{1\}}W$ on $M$, i.e. a $\Lambda[M]$-linear endomorphism of $\Lambda[M]\otimes W$. But this is nothing but the coaction map
\begin{equation}\label{E:coact}
\on{coact}: W\to \Lambda[M]\otimes_\Lambda W.
\end{equation} 
This finishes the construction of \eqref{E: Ga act W}.

\begin{remark}\label{R: fiber action}
\begin{enumerate}
\item\label{R: fiber action-1} Here is a more concrete description of the action \eqref{E: Ga act W} of $\Ga$ on fibers of ${}_\Ga W$. The representation $W$ induces a homomorphism $M\to \End(W)$ of monoid scheme over $\Lambda$, where $\End(W)(A)=\End_{\Mod_A^{\leq 0}}(W\otimes A) \in\Mon(\Spc)$.
Let $\on{Spec} A\to \mR_{\Ga,M}$ be a point of $\mR_{\Ga,M}$, corresponding to a homomorphism $\rho:\Ga\to M(A)$. The fiber of ${}_\Ga W$ over $\rho$, usually denoted by $W_\rho$, is just $W\otimes_\Lambda A$, on which $\Ga$ acts via $\Ga\xrightarrow{\rho} M(A)\to \End(W)(A)$. In \eqref{R: cohom pres}, we interpret $\rho$ as a map of cosimplicial algebras $\Lambda[M^\bullet]\to C(\Ga^\bullet,A)$. In the same spirit, we
may also interpret this action as a cosimplicial module $C(\Ga^\bullet,W_\rho)$ over $C(\Ga^\bullet, A)$  (and therefore over $A$) as follows. The coaction \eqref{E:coact} extends to a cosimplicial module $\Lambda[M^\bullet]\otimes_\Lambda W$ over $\Lambda[M^\bullet]$. Then $C(\Ga^\bullet,W_\rho)$ is its the base change along $\rho$.

\item If $W$ is a representation of $M^J$ for a finite set $J$, then ${}_{\Ga}W$ admits an action by $\Ga^J$, by first applying the above construct to $\mR_{\Ga^J,M^J}$ and then pulling the $\Ga^J$-action on ${}_{\Ga^J}W$ back along the morphism $\mR_{\Ga,M}\to \mR_{\Ga^J,M^J}$.
\end{enumerate}
\end{remark}

We can interpret \eqref{E: Ga act W} as a functor from $\Ga$ to the category of quasi-coherent sheaves on $\mR_{\Ga,M}$ by sending the unique object of $\Ga$ to ${}_\Ga W$ (see Example \ref{R: monoid, segal space}). 

\begin{definition}
The ``universal" homology of $\Ga$ with coefficient in $W$ is the complex of quasi-coherent sheaves on $\mR_{\Ga,M}$ defined by
\begin{equation*}\label{E:uni homology}
C_*(\Ga, {}_\Ga W):=\varinjlim_{\Ga} {}_{\Ga}W.
\end{equation*}
\end{definition}
Since tensor product preserves colimits, the (derived) pullback of $C_*(\Ga,{}_\Ga W)$ along $\on{Spec} A\to \mR_{\Ga,M}$ given by $\rho:\Ga\to M(A)$ as in Remark \ref{R: fiber action} is just the complex in $\Mod^{\leq 0}_A$ computing 
$\varinjlim_{\Ga} W_\rho$. If $A$ is classical, this is nothing but the usual homology of $\Ga$ with coefficient $W_\rho$. 

There is a canonical isomorphism
\begin{equation}\label{E: group hom}
C_*(\Ga,{}_\Ga W)\cong \varinjlim_{\FFM/\Ga}  \Lambda[\mR_{\Ga,M}]\otimes_{\Lambda[M^I]} C_*(\FM(I),{}_I W) 
\end{equation}
constructed using Corollary \ref{L: colim by FFM},
\begin{eqnarray*}
\varinjlim_{\Ga}{}_\Ga W & \cong &\varinjlim_{\FFM/\Ga}\varinjlim_{ \FM(I)}\Lambda[\mR_{\Ga,M}]\otimes_{\Lambda[M^I]}  {}_IW\\
                                             & \cong &\varinjlim_{\FFM/\Ga}\Lambda[\mR_{\Ga,M}]\otimes_{\Lambda[M^I]}  \varinjlim_{ \FM(I)} {}_IW.
\end{eqnarray*}

It is convenient to consider a reduced version of $C_*$. By definition, there is a natural map ${}_\Ga W\to C_*(\Ga, {}_\Ga W)$. We denote its fiber in the category of quasi-coherent sheaves on $\mR_{\Ga,M}$ by $\overline{C}_*(\Ga,{}_\Ga W)[-1]$, so we have the distinguished triangle 
\begin{equation}\label{E: cofib}
\overline{C}_*(\Ga,{}_\Ga W)[-1]\to {}_\Ga W\to C_*(\Ga, {}_\Ga W)\to.
\end{equation} 
Then \eqref{E: group hom} holds with $C_*$ replaced by $\overline{C}_*$. The advantage to consider the reduced version is that we have the following canonical isomorphism
\begin{equation}\label{E: H1 free}
{}_IW^{\oplus I}\cong \overline{C}_*(\FM(I),{}_IW)[-1],
\end{equation} 
obtained from the calculation of homology of free monoids by the following two-term complex (in \emph{cohomological} degree $[-1,0]$)
\[
\bigoplus_{i\in I}{}_IW\xrightarrow{\oplus_{i\in I} (\ga_i-1)} {}_IW,
\]
where $\ga_i$ denotes the generator of $\FM(I)$ corresponding to $i\in I$. 
In particular, $\overline{C}_*(\FM(I),{}_IW)[-1]$ sits in the \emph{abelian} category of quasi-coherent sheaves on $\mR_{\FM(I),M}\cong M^I$. 

Now let $f:\FM(I)\to \FM(J)$ be a monoid morphism. It induces a morphism between homology
$\Lambda[M^J]\otimes_{\Lambda[M^I]}\overline{C}_*(\FM(I),{}_IW)[-1]\to \overline{C}_*(\FM(J),{}_JW)[-1]$. Under the isomorphism \eqref{E: H1 free}, it is given by a $\Lambda[M^I]$-linear map
\begin{equation}\label{E: Cf}
{}_IW^{\oplus I}\to {}_JW^{\oplus J},
\end{equation}
which we now describe more explicitly.
Note that every such $f:\FM(I)\to \FM(J)$ is compositions of maps of the following two types:
\begin{itemize}
\item $f$ sends generators of $\FM(I)$ to generators or the unit of $\FM(J)$, i.e. $f$ is induced by a map of pointed sets $I\cup\{*\}\to J\cup\{*\}$;
\item $f: \FM(\{1,\ldots,n\})\to \FM(\{1,\ldots,n+1\})$ sending $\ga_i\to \ga_i$ for $i\leq n-1$ and $f(\ga_n)=\ga_{n}\ga_{n+1}$.
\end{itemize}
Therefore, it is enough to understand \eqref{E: Cf} in these two cases separately. Unveiling the construction of \eqref{E: H1 free}, we see that
in the first case, it is given by
\begin{equation}\label{E: Cf1}
(w_i)_{i\in I}\in {}_IW^{\oplus I}\mapsto (v_j)_{j\in J}\in {}_JW^{\oplus J},\quad v_j=\sum_{i\in f^{-1}(j)} 1\otimes w_i,
\end{equation}
and in the second case, it is given by
\begin{equation}\label{E: Cf2}
(w_i)\in {}_{\{1,\ldots,n\}}W^{\oplus n}\mapsto (v_j)\in {}_{\{1,\ldots,n+1\}}W^{\oplus (n+1)},\quad v_i= 1\otimes w_i, i\leq n, \ \  v_{n+1}=\ga_n(1\otimes w_n).
\end{equation}

Now we can compute the cotangent complex on $\mR_{\Ga,M}$ when $M$ is an affine smooth group scheme over $\Lambda$. 
Let $\Ad^*$ denote the coadjoint representation of $M$ on the dual of the Lie algebra $\frakm$ of $M$. 

We recall that for an animated $\Lambda$-algebra $A$, the (algebraic) cotangent complex $\bL_{A}$ is a connective $A$-module such that for every $A\to B$ and a connective $B$-module $V$ 
\[
\Map_{\Mod^{\leq 0}_A}(\bL_{A},V)\cong \Map_{{\CA_\Lambda}_{/B}}(A, B\oplus V),
\]
where $B\oplus V\to B$ denotes the trivial square zero extension of $B$  by $V$ in $\CA_\Lambda$, and ${\CA_\Lambda}_{/B}$ denotes the category of animated $\Lambda$-algebras with a $\Lambda$-algebra map to $B$. See \cite[25.3.1,25.3.2]{Lu3} for a detailed account.
If $A$ is a classical smooth $\Lambda$-algebra, then $\bL_A\cong \pi_0(\bL_A)=\Omega_A$ is just the K\"ahler differential of $A$. If $A\to B$ is a morphism in $\CA_\Lambda$, there is a natural morphism $B\otimes_A\bL_A\to \bL_B$ in $\Mod^{\leq 0}_B$ and the relative cotangent complex $\bL_{B/A}$ is defined as its fiber. 

\begin{prop}\label{P: cotan frame}
Assume that $M$ is an affine smooth group scheme over $\Lambda$. 
For every $\Ga$, the cotangent complex of $\mR_{\Ga,M}$ is canonically isomorphic to $\overline{C}_*(\Ga, {}_\Ga\Ad^*)[-1]$.
\end{prop}
\begin{proof}
Note that if $A=\varinjlim A_i$ is a colimit in $\CA_\Lambda$, then 
\begin{equation}\label{E: colim cotan}
\bL_{A}\cong \varinjlim (A\otimes_{A_i}\bL_{A_i}). 
\end{equation}
We apply this to $\Lambda[\mR_{\Ga,M}]=\varinjlim_{\FFM/\Ga}\Lambda[M^I]$.
By comparing \eqref{E: group hom} with \eqref{E: colim cotan}, it is enough to establish, for every $f: \FM(I)\to \FM(J)$, the following commutative diagram (in the abelian category of $\Lambda[M^J]$-modules) 
\begin{equation}\label{E: diff vs coadj}
\xymatrix{
\Lambda[M^J]\otimes_{\Lambda[M^I]} ({}_{I}\Ad^*)^{\oplus I}\ar[r]\ar_\cong[d]&   ({}_{J}\Ad^*)^{\oplus J} \ar^\cong[d]\\
\Lambda[M^J]\otimes_{\Lambda[M^I]}\Omega_{M^I/k}\ar[r]&\Omega_{M^J}.
}
\end{equation}
Now if we identify $\Omega_{M}$ with $\Lambda[M]\otimes \Ad^*$ by regarding $\Ad^*$ as the space of left invariant differentials, then the vertical isomorphisms become clear and the commutativity of the diagram follows from explicit computations exhibited in \eqref{E: Cf1} and \eqref{E: Cf2}.
\end{proof}

\begin{remark}\label{R: tan compl}
Sometimes it is convenient to pass to the linear dual of the cotangent complex of  $\mR_{\Ga,M}$. Given $\rho: \Ga\to M(A)$, the tangent space $\bT_\rho\mR_{\Ga,M}$ of $\mR_{\Ga,M}$ at $\rho$ is the $A$-linear dual of $\bL_{\mR_{\Ga,M}}|_\rho$ (regarded as an object in $\Mod_A$), which is isomorphic to
 $\overline{C}^*(\Ga,\Ad_\rho)[1]$. Here
\[
C^*(\Ga,\Ad_\rho):=\varprojlim_{ \Ga} \Ad_\rho,
\]
with limit taking in $\Mod_A$, and $\overline{C}^*(\Ga,\Ad_\rho)[1]$ is its reduced version, i.e. the cofiber of $C^*(\Ga,\Ad_\rho)\to \Ad_\rho$. If $A$ is classical, this is the usual cohomology of $\Ga$ with coefficient in the adjoint representation $\Ad$ of $M$. Note that for a representation $W$ of $M$, $C^*(\Ga,W_\rho)$ can be identified with the totalization of the cosimplicial $A$-module $C(\Ga^\bullet, W_\rho)$ from Remark \ref{R: fiber action} \eqref{R: fiber action-1}.
\end{remark}

Let us move to the next topic.
Note that if $\Ga$ is finitely generated and $\Lambda$ is noetherian, then the non-derived space ${}^{cl}\mR_{\Ga,M}$ is of finite type over $\Lambda$. Indeed, by choosing a surjective map $\FM(I)\to \Ga$,  ${}^{cl}\mR_{\Ga,M}$ is realized as a closed subscheme of ${}^{cl}\mR_{\FM(I),M}\cong M^I$.
Now we discuss similar statements for $\mR_{\Ga,M}$. 

Recall that for a compactly generated $\infty$-category $\mC$, an object $c$ is called almost compact if for every $n\geq 0$, $\tau_{\leq n}c$ is compact in ${}_{\leq n}\mC$ (\cite[7.2.4.8]{Lu2}). Almost compact objects in $\CA_\Lambda$ are also called almost of finite presentation 
and for an animated $\Lambda$-algebra $A$, almost compact objects in $\Mod_A^{\leq 0}$ are also called connective almost perfect $A$-modules. 
If $\Lambda$ is noetherian, $A$ is almost of finite presentation over $\Lambda$ if and only if $\pi_0(A)$ is a finitely generated $\Lambda$-algebra and each $\pi_i(A)$ is a finitely generated $\pi_0(A)$-module (\cite[3.1.5]{Lu4}).
In particular, if $\Lambda$ is noetherian, a classical $\Lambda$-algebra of finite type is almost of finite presentation, when regarded as an animated $\Lambda$-algebra.

On the other hand, recall that a group (even a monoid) $\Ga$ is called of type $FP_\infty(k)$ if the trivial $k\Ga$-module admits a resolution $P^\bullet\to k$ with each term finite projective $k\Ga$-module, where $k\Ga$ denotes the group (or monoid) algebra of $\Ga$. For example, finite groups are always of type $FP_\infty(k)$. More generally, if the classifying space of $\Ga$ can be realized as a CW complex with finitely many cells in each degree $n\geq 0$ (such a group is called of type $F_\infty$), then $\Ga$ is of type $FP_\infty(k)$.

\begin{prop}\label{P: almost fin}  
Assume that $\Lambda$ is noetherian, and $M$ is a smooth affine group scheme over $\Lambda$.
If $\Ga$ is finitely generated of type $FP_\infty(k)$, then $\mR_{\Ga,M}$ is almost of finite presentation over $\Lambda$.
\end{prop} 
\begin{proof}
As $\Ga$ is finitely generated, ${}^{cl}\mR_{\Ga,M}$ is of finite type.
Using \cite[3.2.18]{Lu4} and Proposition \ref{P: cotan frame}, it is enough to show that $\overline{C}_*(\Ga, {}_{\Ga}\Ad^*)[-1]$ is almost perfect.  As $\Ga$ is of type $FP_\infty(k)$, the pullback of this complex to every \emph{classical} $\Lambda$-algebra $A$ is a connective complex with each term finite projective $A$-module, and therefore is almost perfect. This implies that $\overline{C}_*(\Ga, {}_{\Ga}\Ad^*)[-1]$ is almost perfect by \cite[2.7.3.2]{Lu3}.
\end{proof}
\begin{remark}
There are also refined notions such as animated $\Lambda$-algebras of finite generation of order $n$ and groups of type $FP_n(k)$. One can use these notions to formulate a refined version of the above proposition.
\end{remark}

\begin{prop}\label{P: Classical}
Assumptions are as in Proposition \ref{P: almost fin}. Let $d$ denote the relative dimension of $M$ over $\Lambda$.
In addition, assume that for every field valued point $\on{Spec} \kappa\to\mR_{\Ga,M}$ given by a representation $\rho:\Ga\to M(\kappa)$, we have
$$H_i(\Ga, \Ad_{\rho}^*)=0  \mbox{ for } i>2,\quad \mbox{ and } \quad \dim_\kappa {}^{cl}\mR_{\Ga,M}\leq d-\dim (-1)^{i}H_i(\Ga, \Ad_\rho^*),$$ 
where $\dim_\kappa {}^{cl}\mR_{\Ga,M}$ denotes the relative dimension of  ${}^{cl}\mR_{\Ga,M}$ over $\Lambda$ at $\kappa$.
Then $\mR_{\Ga,M}={}^{cl}\mR_{\Ga,M}$ is a local complete intersection. In this case, it is smooth at a geometric point $\rho\in \mR_{\Ga,M}$ if and only if $\mR_{\Ga,M}$ is flat at $\rho$ over $\Lambda$ 
and $H_2(\Ga, \Ad_\rho^*)=0$.
\end{prop}
\begin{proof}
By our assumption, $\mR_{\Ga,M}$ is almost finitely presented over $\Lambda$ and its cotangent complex has Tor-amplitude $\leq 1$. So it is quasi-smooth in the sense of \cite[3.4.15]{Lu4} (see also \cite[2.1.3]{AG} when $\Lambda$ is a characteristic zero field). 
We choose a surjective map $\FM(I)\to \Ga$, inducing a morphism $\mR_{\Ga,M}\to \mR_{\FM(I),M}$. It follows from arguments as in \emph{loc. cit.} that Zariski locally on $M^I$, meaning after replacing $M^I$ by an open subscheme $\on{Spec} A\subset M^I$ and $\mR_{\Ga,M}$ by $\on{Spec} B:= \on{Spec} A\times_{M^I}\mR_{\Ga,M}$,
there is a morphism $\on{Spec} A\to \bA^m:=\on{Spec} \Lambda[x_1,\ldots,x_m]$ such that $\on{Spec} B\cong \on{Spec} A\times_{\bA^m}\{0\}$. 
In particular, $\dim_\kappa {}^{cl}\mR_{\Ga,M}\geq \dim_\kappa M^I-m$ at every field valued point $\kappa$ of $\on{Spec} B$.
On the other hand, the distinguished triangles
$B\otimes_A\bL_A\to \bL_B\to \bL_{B/A}$ implies that
for every point $\kappa$ of $\on{Spec} B$, 
\[
\dim_\kappa M^I-m=d- \sum_i(-1)^{i}\dim  H_i(\Ga, \Ad_\rho^*).
\]
It follows from our assumption that $\dim_\kappa {}^{cl}\mR_{\Ga,M}= \dim_\kappa M^I-m$. This implies that $\mR_{\Ga,M}={}^{cl}\mR_{\Ga,M}$ is a local complete intersection.

Finally, $\mR_{\Ga,M}$ is smooth at $\rho$ if and only if it is flat and $\dim(\Omega_{\mR_{\Ga,M}}\otimes \kappa)=\dim_{\kappa}\mR_{\Ga,M}$. But the last condition is equivalent to $H_2(\Ga, \Ad_\rho^*)=0$ by the above equality.
\end{proof}

Up to now, we are focusing on the so-called framed representation space. Let us also briefly discuss representation stacks.
First, by a prestack over $\Lambda$, we mean a(n accessible)\footnote{This is a set theoretic assumption (see \cite[5.4.2.5]{Lu1}). Alternatively, we can bound the size of algebras we are considering.} 
functor $\mF: \CA_\Lambda\to \Spc$. All prestacks over $\Lambda$ form an $\infty$-category $\on{Fun}(\CA_\Lambda,\Spc)$.  
A  prestack is a called a stack if it is a sheaf with respect to the \'etale topology on $\CA_\Lambda$. We write $\Shv(\CA_\Lambda)$ for the full subcategory of  $\on{Fun}(\CA_\Lambda,\Spc)$ consisting of stacks. As in the classical situation, via the Yoneda embedding, $\mathbf{DAff}_\Lambda$ form a full subcategory of $\Shv(\CA_\Lambda)$.
A derived Artin stack over $\Lambda$ is a stack satisfying certain properties.  For a (pre)stack $\mF$, we let ${}^{cl}\mF$ denote its restriction to the classical $\Lambda$-algebras, called its underlying classical (pre)stack. Note that $\mF=\on{Spec} A$, then ${}^{cl}\mF$ is represented by $\on{Spec} \pi_0(A)$, which is consistent with our previous definition of ${}^{cl}\on{Spec} A$.
We refer to \cite[\S 5]{Lu4} for precise definitions and some further discussions.

Now assume that there is a smooth affine group scheme $H$ over $\Lambda$ that acts on $M$ by monoid automorphisms. It gives rises to a simplicial object in $\Mon(\mathbf{Aff}_\Lambda)$ by assigning $[n]\in\Delta\mapsto H^n\times M$ (with the monoid structure coming from $M$) and by assigning various face maps coming from the action map and the projection maps as usual. Then applying the construction \eqref{E: der hom space2} gives a simplicial derived affine schemes (with degeneracy maps omitted)
\begin{equation}\label{E: H action}
\cdots\ \substack{\longrightarrow\\[-1em] \longrightarrow \\[-1em] \longrightarrow\\[-1em] \longrightarrow}\ H\times H\times \mR_{\Ga,M}\ \substack{\longrightarrow\\[-1em] \longrightarrow \\[-1em] \longrightarrow}\ H\times \mR_{\Ga,M}\ \substack{\longrightarrow\\[-1em] \longrightarrow }\ \mR_{\Ga,M},
\end{equation} 
which amounts to an action of $H$ on $\mR_{\Ga,M}$.
\begin{definition}
Let $\mR_{\Ga,M/H}:=\mR_{\Ga,M}/H$ be the quotient stack of the above $H$-action, i.e. the geometric realization of \eqref{E: H action} in $\Shv(\CA_\Lambda)$.  
If $M=H$ on which $H$ acts by conjugation, we write $\mX_{\Ga,H}$ for $\mR_{\Ga,H/H}$
and call it the $H$-representation stack of $\Ga$. 
\end{definition}
\begin{remark}
Clearly ${}^{cl}\mX_{\Ga,H}$ is the usual representation stack studied in literature. In particular, for an algebraically closed field $\kappa$, the $\kappa$-points of $\mX_{\Ga,H}$ classify homomorphisms $\Ga\to H(\kappa)$ up to $H(\kappa)$-conjugacy. 
 In general, $\mX_{\Ga,H}: \CA_\Lambda\to \Spc$ is the \'etale sheafification of the functor sending $A$ to $\Map_{\Spc}(|\Ga|,|H(A)|)$ (compare with \eqref{E: ptsp}). 
\end{remark}

Now suppose that $W$ is a representation of $M\rtimes H$ (on a finite projective $\Lambda$-module), i.e. the coaction morphism \eqref{E: Ga act W} is an $H$-module morphism. In this case the vector bundle ${}_\Ga W$ equipped with the action of $\Ga$ descends to $\mR_{\Ga,M/H}$, denoted by the same notation. In addition, $C_*(\Ga, {}_\Ga W)$ also descends to a complex of quasi-coherent sheaves on $\mR_{\Ga,M/H}$. Indeed, this is clear if $\Ga=\FM(I)$, and the general case reduces to the free case by Corollary \ref{L: colim by FFM}.
Again, in the example $M=H$ with the conjugation action, the coaction map \eqref{E:coact} is automatically $H$-equivariant for every $H$-module $W$. 
In particular, the coadjoint representation of $H$ gives a vector bundle ${}_\Ga\Ad^*$ on $\mX_{\Ga,H}$ equipped with a $\Ga$-action.
We have the isomorphism 
\begin{equation*}\label{E: cotan X}
\bL_{\mX_{\Ga,H}}\cong  C_*(\Ga, {}_\Ga\Ad^*)[-1].
\end{equation*}
This follows from Proposition \ref{P: cotan frame} by comparing \eqref{E: cofib} with the usual distinguished triangle of cotangent complexes related to the morphism $\pi: \mR_{\Ga,H}\to \mX_{\Ga,H}$.

Our last topic of this subsection is the coarse moduli and moduli of pseudorepresentations. Let $\Ga,M,H$ be as above. We will assume that $\Lambda$ is noetherian and $H$ is a connected reductive group over $\Lambda$.
Recall that if $M=H$ acting on itself by conjugation, the GIT quotient of ${}^{cl}\mR_{\Ga,H}$ by $H$ is usually called the $H$-character variety of $\Ga$ (at least if $\Ga$ is finitely generated and $\Lambda$ is a field). 
Similarly, in our more general context, we can make the following definition.
\begin{definition}
The character variety of $\mR_{\Ga,M/H}$, denoted by $\mC_{\Ga,M/H}$, is the geometric realization of \eqref{E: H action} in $\mathbf{DAff}_\Lambda$. So $\Lambda[\mC_{\Ga,M/H}]=\Lambda[\mR_{\Ga,M}]^H$ is the $H$-invariants of $\Lambda[\mR_{\Ga,M}]$ in $\CA_\Lambda$ (i.e. totalization of the cosimplicial objects in $\CA_\Lambda$ obtained from \eqref{E: H action} by passing to the opposite).
\end{definition}
If $\mR_{\Ga,M}$ is classical, then $\mC_{\Ga,M/H}$ is classical and is isomorphic to the usual GIT quotient $\mR_{\Ga,M}/\!\!/H$ of $\mR_{\Ga,M}$ by $H$ in $\Aff_\Lambda$, so $\Lambda[\mC_{\Ga,M/H}]$ is isomorphic to the \emph{non-derived} $H$-invariants of $\Lambda[\mR_{\Ga,M}]$. In general if $\mR_{\Ga,M}$ is not classical, we would still like to say that $\Lambda[\mC_{\Ga,M/H}]$ is isomorphic to the $H$-invariants of $\Lambda[\mR_{\Ga,M}]$ in appropriate sense. Here is one way to make this precise. Recall that there is notion of $E_\infty$-$\Lambda$-algebras, which are commutative algebra objects in the symmetric monoidal category $\Mod_\La$. (See \cite[Chap 7]{Lu2} for a detailed account.) For example, the ring of global functions $\Ga(\mR_{\Ga,M/H},\mO)$ of $\mR_{\Ga,M/H}$ is an $E_\infty$-$\Lambda$-algebra, which in fact is isomorphic to the $H$-invariants of $\Lambda[\mR_{\Ga,M}]$ in the category of $E_\infty$-$\Lambda$-algebras.
There is a natural functor from $\CA_\La$ to the category of $E_\infty$-$\La$-algebras. (See \cite[\S 25.1]{Lu3}.) Then
the image of $\Lambda[\mC_{\Ga,M/H}]$ under this functor can be identified with $\tau^{\leq 0} \Ga(\mR_{\Ga,M/H},\mO)$.

\begin{prop}\label{P: fp of char var}
If $\mR_{\Ga,M}$ is $m$-truncated for some $m$ and is almost of finite presentation over $\Lambda$, so is $\mC_{\Ga,M/H}$. 
\end{prop}
\begin{proof}
Write $A=\Lambda[\mR_{\Ga,M}]$ for simplicity. It is known that $\pi_0(A)^H$ is finitely generated over $\Lambda$. (For this generality, see \cite{FVDK}.)
By a spectral sequence argument, it is enough to show that $H^i(H, \pi_j(A))$ is a finitely generated $\pi_0(A)^H$-module. But this follows from \cite[10.5]{TVDK}.
\end{proof}
 
Now, let $\Lambda[M^\bullet/\!\!/H]$ be the $\FFM$-algebra sending $\FM(I)$ to $\Lambda[\mC_{\FM(I), M/H}]\cong \Lambda[M^I]^H$ (Remark \ref{R: FFM scheme}). 
\begin{definition}\label{D: psrep and ex}
The moduli of pseudorepresentations of $\mR_{\Ga,M/H}$ is the derived affine scheme over $\Lambda$ defined by 
\[
\mR_{\Ga,M^\bullet/\!\!/H}:=\varprojlim_{(\FFM/\Ga)^{\on{op}}} (M^I/\!\!/H). 
\]
We call $\Lambda[\mR_{\Ga,M^\bullet/\!\!/H}]=\varinjlim_{\FFM/\Ga} \Lambda[M^I]^H$ 
the excursion algebra associated to $\mR_{\Ga,M/H}$.
\end{definition}

\begin{remark}
If $M=H$ with the adjoint action, by \eqref{E: hom FFMalg} giving a homomorphism $\Lambda[\mR_{\Ga,M^\bullet/\!\!/H}]\to A$ (say $A$ classical) is the same as giving an $H(A)$-valued pseudo representation of $\Ga$, in the sense of Lafforgue \cite[11.3, 11.7]{L}. This justifies the choice of our terminology. The underlying classical scheme ${}^{cl}\mR_{\Ga,M^\bullet/\!\!/H}$ plays an auxiliary but important role in the following discussions. On the other hand, we will avoid to use $\mR_{\Ga,M^\bullet/\!\!/H}$ as we understand very little about it as a derived scheme. 
\end{remark}

Tautologically, there are natural morphisms 
\begin{equation}\label{E: Tr map}
    \Tr: \mR_{\Ga,M/H}\to \mC_{\Ga,M/H}\to \mR_{\Ga,M^\bullet/\!\!/H}.
\end{equation}
If $M=H$ with the adjoint action, this is just the map sending a representation to its associated pseudorepresentation. The induced map of ring of regular functions is explicitly given by
\begin{equation}\label{E: coarse to pseudo}
    \Lambda[\mR_{\Ga,M^\bullet/\!\!/H}]=\varinjlim_{\FFM/\Ga} \Lambda[M^I]^H\to (\varinjlim_{\FFM/\Ga} \Lambda[M^I])^H=\Lambda[\mC_{\Ga,M/H}].   
\end{equation}

\begin{remark}\label{R: kpt of pseudo}
If $\Lambda$ is a field of characteristic zero, \eqref{E: coarse to pseudo} is an isomorphism since taking $H$-invariants commutes with arbitrary colimits. 
If $\Ga=\FM(I)$, this is also an isomorphism as $\FFM/\Ga$ admits a final object.
We have no reason to believe this is the case if $\on{char} \La=p>0$ and $\Ga$ is general. However, if $\Lambda$ is a perfect field and $\mR_{\Ga,M}$ is truncated, then the induced map $\mC_{\Ga,M/H}(k)\to \mR_{\Ga,M^\bullet/\!\!/H}(k)$ is still a bijection. 
\end{remark}

\subsection{Some examples}
For later applications, we specialize the above general discussions to some concrete situations. 
Let $\Lambda$ be a Dedekind domain (including the case of a field), and $M$ an affine smooth group scheme over $\Lambda$ with the neutral connected component $M^\circ$ reductive over $\Lambda$. 

The following two statements easily follow from Proposition \ref{P: Classical}. 
\begin{prop}\label{E: H fin etale}
If $\Ga$ is a finitely generated group and $M$ is (finite) \'etale over $\Lambda$, then $\mR_{\Ga,M}={}^{cl}\mR_{\Ga,M}$ is (finite) \'etale over $\Lambda$. 
\end{prop}

\begin{prop}\label{P: finite group good case}
Assume that $\Ga$ is finite whose order is invertible in $\Lambda$. Then $\mR_{\Ga,M}={}^{cl}\mR_{\Ga,M}$ is smooth of finite type over $\Lambda$. Let $\rho: \Ga\to M(\mO)$ be a homomorphism where $\mO$ is an \'etale $\Lambda$-algebra, and let $Z_M(\rho)$ be its centralizer in $M_\mO$.
Then the morphism $M_\mO/Z_M(\rho)\to \mR_{\Ga,M}\otimes_\Lambda\mO$ induced by the conjugation of $\rho$ by $M$ is an open and closed embedding.
\end{prop}

\begin{remark}\label{R: comp red}
We keep the assumption of the proposition. In addition, assume that $M/M^\circ$ is finite \'etale over $\Lambda$. Let $E$ be the fractional field of $\Lambda$. We expect that every conjugacy class of homomorphisms from $\Ga\to M(\overline E)$ admits a representative defined over a finite \'etale extension of $\Lambda$. If so, there will exist a finite \'etale extension $\mO$ of $\Lambda$, such that 
$$\mR_{\Ga,M}\otimes\mO\simeq \sqcup_{\rho} M_\mO/Z_M(\rho),$$ 
where $\rho$ range over a set of representatives of homomorphisms from $\Ga$ to $M(\overline E)$ up to conjugacy.

We are not able to prove such statement in general, except when $M=\GL_m$ or when $\Ga$ is solvable. The first situation follows from the fact that $k\Ga$ is a finite free semisimple algebra over $\Lambda$.
Next we assume that $\Ga$ is solvable but $M$ general. Let $T$ be a maximal torus of $M$ over $\Lambda$. Then up to conjugation we may assume that $\rho:\Ga\to M(\overline E)$ factors as $\rho: \Ga\to N_M(T)(\overline E)$, where $N_M(T)$ is the normalizer of $T$ in $M$. This follows from \cite[thm. 2]{BS} if $\on{char} E=0$ and a lifting argument if $\on{char} E>0$. 
Now, let $m$ be the order of $\Ga$, and let $N_M(T)[m]$ denote the closed subscheme of elements of $N_M(T)$ of order dividing $m$. As this is a finite \'etale scheme over $\Lambda$, our claim follows. 
\end{remark}

If the order of $\Ga$ is not invertible in $\Lambda$, then the situation is much more complicated. 

\begin{example}\label{Ex: Zp over Fp}
Even in the simplest case $\La=\bF_p$, $\Ga=\bZ/p$ and $M=\bG_m$, we have 
\[
\mR_{\bZ/p,\bG_m}\neq {}^{cl}\mR_{\bZ/p,\bG_m}\cong \bG_m[p]
\] 
(which is not smooth). The fact that $\mR_{\bZ/p,\bG_m}\neq {}^{cl}\mR_{\bZ/p,\bG_m}$ also reflects the point that, although $\bZ/p$ is the pushout of the diagram 
\begin{equation}
\xymatrix{
\bZ\ar^{p}[r]\ar[d] & \bZ\\
\{*\}&  
}
\end{equation}
in $\Mon=\Mon(\Sets)$, this is not the case in $\Mon(\Spc)$. Indeed, for $p=2$ let $\Ga'$ be the pushout of this diagram in $\Mon(\Spc)$. Then the geometric realization $|\Ga'|$ (as in \cite[5.2.6.10, 5.2.6.13]{Lu2}) is homotopy equivalent to the real projective plane.
\end{example}

For discussions in the sequel, we record the following result about the moduli of pseudorepresentations of finite groups. 

\begin{prop}\label{P: fin pseudo rep}
Assume that $\Ga$ is finite, and that $M/M^\circ$ is finite \'etale over $\Lambda$. Assume that $H$ acts on $M$ by conjugation through a homomorphism $H \to M^\circ$ such that the composed map $H \to M^\circ\to M^\circ_\ad$ is surjective, where $M^\circ_\ad$ is the adjoint quotient of $M^\circ$.
Then ${}^{cl}\mR_{\Ga, M^\bullet/\!\!/H}$ is finite over $\Lambda$. If the order of $\Ga$ is invertible in $\Lambda$, then $\mC_{\Ga,M/H}={}^{cl}\mC_{\Ga,M/H}$ is finite \'etale over $\Lambda$. 
\end{prop}
\begin{proof}
If the order of $\Ga$ is invertible in $\Lambda$, then $\mC_{\Ga,M/H}={}^{cl}\mC_{\Ga,M/H}$ is \'etale over $\Lambda$ by
Proposition \ref{P: finite group good case} and \ref{P: fp of char var}. In this case $\Lambda[\mC_{\Ga,M/H}]$ is finitely generated over $\Lambda$ and is integral over $\pi_0\Lambda[\mR_{\Ga,M^\bullet/\!\!/H}]$. Therefore,
it is enough to prove the first statement.

We first consider the case $M=H=\GL_m$. Let 
$\chi_i\in \Lambda[\GL_m]^{\GL_m}$
be the character of the $i$th wedge representation of $\GL_m$. For each $\ga\in \Ga$, let $\chi_{i,\ga}\in \Lambda[{}^{cl}\mR_{\Ga, M^\bullet/\!\!/H}]$ be the image of $\chi_i$ under the map $\Lambda[\GL_n]^{\GL_n}\to \Lambda[{}^{cl}\mR_{\Ga, M^\bullet/\!\!/H}]$ corresponding to the map $\bN=\FM(\{1\})\to\Ga$ induced by $\ga$. As the $\FFM$-algebra $\Lambda[\GL_m^\bullet]^{\GL_m}$ is generated by $\{\chi_i\}_i$ by \cite{Do}, $\Lambda[{}^{cl}\mR_{\Ga, M^\bullet/\!\!/H}]$ is generated by these $\{\chi_{i,\ga}\}_{i,\ga}$ as $\Lambda$-algebra. Therefore, to show that $\Lambda[{}^{cl}\mR_{\Ga, \GL_m^\bullet/\!\!/\GL_m}]$ is finite over $\Lambda$, it is enough to show that every $\chi_{i,\ga}$ is integral over $\Lambda$. Therefore, we may assume that $\Ga=\langle\ga\rangle$ with $\ga$ being of order $n$, which can be realized as the coequalizer $\bN\substack{n\\ \rightrightarrows \\ 0}\bN$ in $\Mon$ (but not in $\Mon(\Spc)$ see Remark \ref{Ex: Zp over Fp}). Therefore,
${}^{cl}\mR_{\langle\ga\rangle, \GL_m^\bullet/\!\!/\GL_m}$ is isomorphic to the equalizer of 
\[
\GL_m/\!\!/\GL_m\substack{X\mapsto X^n\\[.5em] \longrightarrow\\[-1em]  \longrightarrow\\[.5em]  X\mapsto I}\GL_m/\!\!/\GL_m,
\]
which is easily seen to be finite.

Now assume that $M$ is general. To prove the result, we are free to pass to a (finite type) flat extension of $\Lambda$. So we may assume that $M/M^\circ$ is finite constant and $M^\circ$ is split. Then we may choose a faithful representation $\phi:M\to \GL(V)$, with $V$ finite projective. By \cite{Cotner} (see also \cite{Vin, Ma} when $\Lambda$ is a field), the induced map $\phi_n: M^n/\!\!/H\to \GL_m^n/\!\!/\GL_m$ is finite for any $n$. This implies that $\Lambda[{}^{cl}\mR_{\Ga,M^\bullet/\!\!/H}]$ is finite over $\Lambda[{}^{cl}\mR_{\Ga,\GL_m^\bullet/\!\!/\GL_m}]$, and therefore is finite over $\Lambda$, as desired.
\end{proof}

\begin{remark}\label{R: decom finite group R}
Let us assume that $\Lambda$ is an algebraically closed field. 
Then the above proposition implies that $\mR_{\Ga,M}$ decomposes into open and closed subschemes
\[
\mR_{\Ga,M}=\sqcup_{\Theta} \mR_{\Ga,M}^{\Theta},
\]
indexed by $\Lambda$-points $\Theta$ of $ \mR_{\Ga,M^\bullet/\!\!/H}$, such that $\Tr(\rho_x)=\Theta$ for every $\rho_x: \Ga\to M$ corresponding to a geometric point $x\in \mR_{\Ga,M}^{\Theta}$. 
By  \cite[11.7]{L} and \cite[4.5]{BHKT},
$\Lambda$-points of $\mR_{\Ga,M^\bullet/\!\!/H}$ classify $M$-completely reducible representation of $\Ga$ (in the sense of \cite[3.5]{BHKT}) up to $H$-conjugacy. So the semisimplification of $\rho_x$ up to $H$-conjugacy is constant along $\mR_{\Ga,M}^{\Theta}$.
For example, if $M=M^\circ$ and $\Theta$ is the pseudorepresentation corresponding to the trivial representation, then ${}^{cl}\mR_{\Ga,M}^{\Theta}$ classifies those $\rho_x$ such that the image $\rho_x(\Ga)$ is contained in a unipotent subgroup of $M$.
\end{remark}

Let $q=p^r$ for some $r\in\bZ_{> 0}$.
We consider the following group (sometimes called the $q$-tame group)
\begin{equation}\label{E: tame Galois}
   \Ga_q:=\langle \sigma, \tau\mid \sigma\tau \sigma^{-1}=\tau^q\rangle. 
\end{equation}
It contains a normal subgroup $\tau^{\bZ[1/p]}$ and the quotient of $\Ga_q$ by this subgroup is $\langle\sigma\rangle\cong\bZ$.

\begin{prop}\label{P: tame stack} 
Let $\Lambda$ be a Dedekind domain over $\bZ[1/p]$. Then $\mR_{\Ga_q,M}={}^{cl}\mR_{\Ga_q,M}$. It is equidimensional of dimension $\dim M^\circ$, flat over $\Lambda$, and is a local complete intersection. It is dualizing complex (relative to $\Lambda$) is trivial (i.e. isomorphic to the structural sheaf).
\end{prop}
\begin{proof}
Except $\mR_{\Ga,M}={}^{cl}\mR_{\Ga,M}$, this is proved in \cite[Prop. 3.3.2]{LTXZZ} in this generality\footnote{The prototype of the argument is probably due to D. Helm.}. We briefly review some ingredients needed later, and explain how to apply Proposition \ref{P: Classical} in this situation.

Let $\chi: M\to M/\!\!/M=\on{Spec} \Lambda[M]^M$ denote the adjoint quotient map. 
For every $m\in\bZ_{\geq 0}$,
the $m$-power morphism $M\to M, \ h\to h^m$ is equivariant with respect to conjugation action and therefore induces a morphism 
$$[m]:M/\!\!/M\to M/\!\!/M.$$ 
Let $(M/\!\!/M)^{[m]}$ 
denote the (classical) fixed point subscheme of $[m]$, and
let $M^{[m]}:=\chi^{-1}((M/\!\!/M)^{[m]})$, which is a closed subscheme of $M$ stable under conjugation. Note that the morphism $\mR_{\Ga_q,M}\to M$ induced by the inclusion $\langle\tau\rangle\subset\Ga_q$ factors through $\mR_{\Ga_q,M}\to M^{[q]}\subset M$.

As explained in \cite[Prop. 3.3.2]{LTXZZ}, over an algebraically closed field $K$ over $\Lambda$, there are only finitely many conjugacy classes in $M^{[q]}(K)$, and from this one deduces that over $K$, $\dim {}^{cl}\mR_{\Ga,M}\otimes K=\dim M_K$. It follows that $\dim {}^{cl}\mR_{\Ga,M}=\dim M$.

On the other hand, we have the following resolution of $\Lambda$ as right $\La\Ga_q$-modules
\begin{equation}\label{E: resol Gaq}
0\to \La\Ga_q\xrightarrow{(1-(\sum_{j<q} \tau^j)\sigma, \tau-1)} \La\Ga_q\oplus \La\Ga_q \xrightarrow{(1-\tau,1-\sigma)} \La\Ga_q\to \La\to 0.
\end{equation}
Therefore, $H_i(\Ga_q, \Ad^*_\rho)=0$ for every $i>2$ and $\dim (-1)^i H_i(\Ga_q,\Ad^*_\rho)=0$. We now apply Proposition \ref{P: Classical} to conclude that $\mR_{\Ga,M}={}^{cl}\mR_{\Ga,M}$ is a local complete intersection. As fibers of ${}^{cl}\mR_{\Ga,M}$ over $\Lambda$ are equidimensional of the same dimension, ${}^{cl}\mR_{\Ga,M}$ is flat over $\Lambda$. Finally, as the dualizing complex of a local complete intersection can be computed as the determinant of its cotangent complex, we see that the dualizing complex of $\mR_{\Ga,M}$ is trivial by \eqref{E: resol Gaq}.
\end{proof}

\begin{remark}\label{R: LocB}
For any smooth affine group scheme $M$ (not necessarily reductive) over $\Lambda$, $\mR_{\Ga_q,M}$ is always quasi-smooth with trivial dualizing complex, by Proposition \ref{P: Classical} and \eqref{E: resol Gaq}.
However if $\dim {}^{cl}\mR_{\Ga_q,M}>\dim M$, then $\mR_{\Ga_q,M}\neq {}^{cl}\mR_{\Ga_q,M}$. 
For example, let $M=B_n$ be the group of determinant one $n\times n$-upper triangular matrices. Then the derived structure on $ {}^{cl}\mR_{\Ga_q,B_n}$ is non-trivial when $n$ is large, even for $\La=\bC$. Indeed, the underlying classical scheme ${}^{cl}\mR_{\Ga_q,B_n}$ has dimension $>\dim B_n$. This is essentially due to the fact that the number of $B_n$-orbits in the set of strictly upper triangular matrices is not finite when $n\geq 6$ (\cite{K}). We note that the possible non-trivial derived structure of this scheme does play a role in our discussion in \S \ref{SS: CohSpr}. 

A similar argument also shows the following. Let $\Ga=\Ga_g$ be the fundamental group of a genus $g$ compact Riemann surface. Then $\mR_{\Ga_g,M}={}^{cl}\mR_{\Ga_g,M}$ if $g\geq 2$ and $M$ is semisimple. Otherwise, $\mR_{\Ga_g,M}$ has non-trivial derived structure. In particular, the scheme $\mR_{\Ga_1,M}$, usually called the commuting scheme of $M$, is always derived.
\end{remark}

Now we put Proposition \ref{P: finite group good case} and \ref{P: tame stack} together.
\begin{prop}\label{P: tame and wild stack}
Let $\Ga=Q\rtimes \Ga_q$ where $Q$ is a finite $p$-group. Let $\La=\bZ[1/p]$ and assume that $M/M^{\circ}$ is finite \'etale over $\Lambda$. Then $\mR_{\Ga,M}$ is classical, of finite type, and flat over $\Lambda$. In addition, it is equidimensional of dimension $\dim M$, and is a local complete intersection. Its dualizing complex (relative to $\Lambda$) is trivial.
\end{prop}
\begin{proof}
The inclusion $Q\subset \Ga$ induces a morphism $\mR_{\Ga,M}\to \mR_{Q,M}$. Using Proposition \ref{P: finite group good case}, Proposition \ref{P: Classical} and the fact that $H_i(\Ga, \Ad_\rho^*)\cong H_i(\Ga_q, (\Ad_\rho^*)^{\rho(Q)})$, it is enough to show that for every $\rho_0: Q\to M(\mO)$ defined over some \'etale $\bZ[1/p]$-algebra $\mO$,
\[
{}^{cl}\mR^{\rho_0}_{\Ga,M}:= {}^{cl}\mR_{\Ga,M}\times_{{}^{cl}\mR_{Q,M}}\bigl\{\rho_0\bigr\}
\]
is of finite type and flat over $\mO$, is equidimensional of dimension $=\dim Z_{M}(\rho_0)$, and is a local complete intersection with trivial dualizing complex.

Let $N_{M}(\rho_0)$ be the normalizer of $\rho_0$ in $M_\mO$. It is a smooth affine group scheme over $\mO$ and $N_{M}(\rho_0)^\circ=Z_{M}(\rho_0)^\circ$ is connected reductive (\cite[thm. 2.1]{PY}). 
The quotient $\pi_0(N_{M}(\rho_0))=N_{M}(\rho_0)/N_{M}(\rho_0)^\circ$ is \'etale over $\mO$, which acts on the constant group $\rho_0(Q)$ over $\mO$. Consider the subfunctor $U\subset \mR_{\Ga_q,\pi_0(N_{M}(\rho_0))}$ consisting of those $\rho: \Ga_q\to \pi_0(N_M(\rho_0))$ such that the composition $\Ga_q\to \pi_0(N_M(\rho_0))\to \Aut(\rho_0(Q))$ is compatible with the action of $\Ga_q$ on $Q$. This is open in $\mR_{\Ga_q,\pi_0(N_{M}(\rho_0))}$.
Then ${}^{cl}\mR^{\rho_0}_{\Ga,M}\cong  {}^{cl}\mR_{\Ga_q, N_{M}(\rho_0)}\times_{\mR_{\Ga_q,\pi_0(N_{M}(\rho_0))}}U$ is open. Therefore, the desired statement follows from Proposition \ref{P: tame stack}.
\end{proof}

Of course, as in Remark \ref{R: LocB}, for $\Ga$ as in Proposition \ref{P: tame and wild stack} but $M$ not necessarily reductive, $\mR_{\Ga,M}$ is still quasi-smooth with trivial dualizing complex, although it may not be classical. 

\subsection{Continuous representations}
\label{SS: Continuous rep}
In the Langlands program, we need to study continuous representations of profinite groups, rather than arbitrary representations of abstract groups.
We address this issue in this subsection.

We fix the coefficient ring $\La=\mO_E$ to be finite integrally closed over $\bZ_\ell$. Let $\varpi$ be a uniformizer of $\mO_E$, and let $\kappa_E$ denote the residue field. We write $\mO_{E,r}$ for $\mO_{E}/\varpi^r$.
Let $M$ be a flat affine monoid scheme over $\mO_E$ and $H$ a smooth affine group scheme over $\mO_E$ that acts on $M$ by monoid automorphisms. 
Let $M_r=M\otimes\mO_{E,r}, \ H_r=H\otimes\mO_{E,r}$. 
Let $\Ga$ be a locally profinite group.
Examples include Galois groups, as well as Weil groups of non-archimedean local fields and global function fields. 
For such $\Ga$, we will give a definition of moduli $\mR^c_{\Ga,M_r}$ of (framed) continuous homomorphisms from $\Ga$ to $M_r$ over $\mO_{E,r}$, and then define $\mR^c_{\Ga,M}$ over $\on{\Spf}\mO_E$ as their inductive limit.
We shall remark that these spaces may not have good global geometry in general (see Example \ref{E: disc cont moduli}) and for certain $\Ga$, there might be ``more correct moduli spaces of representations of $\Ga$" (see Remark \ref{R: negative side}). But as we shall see in the next section,
if $\Ga$ is the Weil group of a non-archimedean local field of residue characteristic $\neq \ell$, or of a global function field of characteristic $\neq\ell$, these definitions should give the correct objects in the Langlands program.\footnote{The case of number fields will be studied in an ongoing project with M. Emerton \cite{EZ}.} At the end of this subsection, we also discuss a possible extension of $\mR^c_{\Ga,M}$ from $\on{\Spf}\mO_E$ to $\on{Spec}\mO_E$. We shall mention that such extension is tailored to the situations considered in the next section, and may not be sufficient for some other considerations.

Our definition of $\mR^c_{\Ga,M_r}$ is based on the expression \eqref{R: cohom pres}, with the space of maps $C(\Ga^\bullet, A)$ (see \eqref{E: exp space}) replaced by appropriately defined space of continuous maps $C_{cts}(\Ga^\bullet, A)$ in the derived setting, which we first explain.

Recall that by the Stone duality, there is a fully faithful embedding $\Pro(\Sets_f)\to \mathbf{Top}$ from the (ordinary) category of profinite sets to the (ordinary) category of topological spaces with essential image consisting of compact Hausdorff totally disconnected spaces. For a disjoint union of profinite sets $S$ regarded as topological space, and an $\mO_{E,r}$-module $V$ regarded as a discrete topological space, let $C_{cts}(S,V)$ be the $\mO_{E,r}$-module of all continuous maps from $S$ to $V$.

\begin{lemma}\label{L: cont S to dis}
Let $S$ be a disjoint union of profinite sets. Then the functor 
$\Mod_{\mO_{E,r}}^\heart\to \Mod_{\mO_{E,r}}^\heart$ sending  $V$ to $C_{cts}(S,V)$ 
is a lax symmetric monoidal exact additive functor. 
Therefore, it extends to a $t$-exact lax symmetric monoidal functor
\begin{equation}\label{E:CctsMr}
C_{cts}(S,-):\Mod_{\mO_{E,r}}\to\Mod_{\mO_{E,r}},
\end{equation}
which lifts to nilcomplete finite limit preserving functor
\begin{equation}\label{E:CctsAr}
   C_{cts}(S,-): \CA_{\mO_{E,r}}\to \CA_{\mO_{E,r}}.
\end{equation}
If $S$ is profinite, then \eqref{E:CctsMr} preserves all colimits and \eqref{E:CctsAr} preserves sifted colimits.
\end{lemma}
\begin{proof}
If we write $S=\sqcup_{j\in J} S_j$ with $S_j$ profinite and $S_j=\varprojlim_{i\in I_j} S_{ij}$ is a projective limit of finite sets over some cofiltered category $I_j$, then for $V\in \Mod_{\mO_{E,r}}^\heart$,
\begin{equation}\label{E: comp Ccts}
C_{cts}(S,V)=\prod_{j\in J} C_{cts}(S_j,V)=\prod_{j\in J}\varinjlim_{i\in I_j^{\on{op}}} V^{S_{ij}}. 
\end{equation}
As $\Mod_{\mO_{E,r}}^\heart$ satisfies Grothendieck axiom (AB$4^*$), (AB$5$), exactness follows. In addition, if $S$ is profinite, then $C_{cts}(S,-)$ preserves all direct sums and therefore all colimits. The extension of the functor to $\Mod_{\mO_{E,r}}$ is immediate. 

Now we have a functor $C_{cts}(S,-): \CA_{\mO_{E,r}}^\heart\to \CA_{\mO_{E,r}}^\heart$. If $S$ is profinite, it preserves sifted colimits as the forgetful functor $\CA_{\mO_{E,r}}^\heart\to \Mod_{\mO_{E,r}}^\heart$ is conservative preserving limits and sifted colimits. Taking the animation gives \eqref{E:CctsAr} in this case, which preserves sifted colimits and lifts \eqref{E:CctsMr}. Finally, if $S=\sqcup_{j\in J} S_j$ with $S_j$ profinite, then $C_{cts}(S,-)=\prod_{j\in J}C_{cts}(S_j,-)$. The rest assertions are clear.
\end{proof}

\begin{remark}\label{R: cont map, trun}
\begin{enumerate}
\item We note that formula \eqref{E: comp Ccts} computes $C_{cts}(S,A)$ for truncated $\mO_{E,r}$-algebras $A$. 
Together with nilcompleteness, one may compute $C_{cts}(S,A)$ for any $A$. 
\item\label{R: cont map, trun-2} By regarding $S$ as an abstract set, there is the natural transformation $C_{cts}(S,-)\to C(S,-)$, which induces injective maps when evaluated at classical $\mO_{E,r}$-algebras.
\item In the above construction, one may replace a locally profinite set $S$ by a simplicial locally profinite set $S_\bullet$. Then we obtain $C_{cts}(S^\bullet, -): \CA_{\mO_{E,r}}\to \CA_{\mO_{E,r}}^{\Delta}$. The corresponding cosimplicial animated algebra sends $[n]$ to $C_{cts}(S^n,-)$.
\end{enumerate}
\end{remark}

Now we can give the definition of $\mR^c_{\Ga,M_r}$. 
As $\Ga$ is a locally profinite group, it is a disjoint union of profinite sets so we can apply the above formalism to each $\Ga^n$. Therefore, for every $A\in\CA_{\mO_{E,r}}$, we have a cosimplicial object in $\CA_{\mO_{E,r}}$, $[n]\mapsto C_{cts}(\Ga^n, A)$. 
On the other hand, as $M$ is a flat affine monoid, $[n]\mapsto \mO_{E}[M^n]$ is a cosimplicial object in $\CA_{\mO_{E}}$.

\begin{definition}\label{D: cont rep sp}
We define the $M$-valued continuous representation space of $\Ga$ over $\mO_{E,r}$ as
\[
\mR^{c}_{\Ga,M_r}: \CA_{\mO_{E,r}}\to \Spc,\quad A\mapsto \on{Map}_{\CA_{\mO_{E}}^{\Delta}}\bigl(\mO_{E,r}[M^\bullet], C_{cts}(\Ga^\bullet,A)\bigr).
\]
Regarding $\mR^c_{\Ga,M_r}$ as a prestack over $\mO_E$, there is the obvious morphism $\mR^c_{\Ga,M_r}\to \mR^c_{\Ga,M_{r+1}}$ over $\mO_E$ and we define
\[
\mR^{c}_{\Ga,M}=\varinjlim \mR^c_{\Ga,M_r}: \CA_{\mO_{E}}\to \Spc,\quad A\mapsto \varinjlim_r \mR^c_{\Ga,M_r}(A\otimes_{\mO_{E}}\mO_{E,r}).
\]
Note that the structural morphism $\mR^c_{\Ga,M}\to\on{Spec}\mO_E$ factors as $\mR^c_{\Ga,M}\to \varinjlim_r \on{Spec} \mO_{E,r}=\on{Spf}\mO_E$.

For each $r$, the group $H_r$ acts on $\mR^c_{\Ga,M_r}$ in the sense that there is a simplicial diagram similar to \eqref{E: H action} (with $\mR_{\Ga,M}$ replaced by $\mR^{c}_{\Ga,M_r}$) 
and therefore we define the continuous representation stack $\mR^{c}_{\Ga,M_r/H_r}$ over $\mO_{E,r}$ as the quotient stack, and $\mR^c_{\Ga,M/H}=\varinjlim_r \mR^c_{\Ga,M_r/H_r}$ over $\on{Spf}\mO_E$.
\end{definition}

To justify the definition, first note by Remark \ref{R: cont map, trun} \eqref{R: cont map, trun-2} and \eqref{R: cohom pres}, there are natural morphisms
\begin{equation}\label{E: forget cont}
\mR^c_{\Ga,M}\to \mR_{\Ga,M},\quad \mR^c_{\Ga,M/H}\to \mR_{\Ga,M/H}
\end{equation}
where $\Ga$ is regarded as an abstract group in $\mR_{\Ga,M}$ and in $\mR_{\Ga,M/H}$. Therefore,  for every $\mO_E$-algebra $A$ in which $\varpi$ is nilpotent, an $A$-point of $\mR^c_{\Ga,M}$ does give a representation $\rho:\Ga\to M(A)$. The following lemma justifies the continuity of $\rho$.
\begin{lemma}\label{L: cont rho}
Assume that $A$ is classical. If $M(A)$ is equipped with the discrete topology, then
\[
\mR^c_{\Ga,M}(A)=\bigl\{\mbox{continuous homomorphisms } \rho: \Ga\to M(A) \bigr\}.
\]
\end{lemma}
\begin{proof}
For a classical $\mO_{E,r}$-algebra $A$, the induced map $\mR^c_{\Ga,M}(A)\to \mR_{\Ga,M}(A)$ is injective with image consisting of those $\bigl(\rho:\Ga\to M(A)\bigr)\in \mR_{\Ga,M}(A)$ such that for every $f\in\mO_{E,r}[M]$, the map $f\circ\rho: \Ga\to A$ is continuous, where $A$ is equipped with the discrete topology. The lemma follows.
\end{proof}

Now, suppose we can write $\Ga=\varprojlim \Ga_j$ as a projective limit, with each $\Ga_j$ discrete and $\Ga_j\to \Ga_{j'}$ surjective with finite kernel. Then we have the obvious morphism 
\begin{equation}\label{E: indpre}
\varinjlim_r\varinjlim_j \mR_{\Ga_j,M_r}=\varinjlim_r\varinjlim_j \mR^c_{\Ga_j,M_r}\to \varinjlim_r\mR^c_{\Ga,M_r}=\mR^c_{\Ga,M}.
\end{equation}
The above discussion implies that ${}^{cl}\mR^c_{\Ga,M}=\varinjlim_r\varinjlim_j {}^{cl}\mR_{\Ga_j,M_r}$ is represented by an ind-affine scheme. 

\quash{\begin{prop}\label{P: classical cont rep}
The map \eqref{E: indpre}
\end{prop}

\begin{proof}
For a classical $\mO_{E,r}$-algebra $A$, we have
\begin{eqnarray*}
\on{Map}_{\CA_{\mO_{E,r}}^{\Delta}}(\mO_{E,r}[M_r^\bullet], C_{cts}(\Ga^\bullet,A)) & = & \Map_{\CA_{\mO_{E,r}}^\Delta}\bigl(\mO_{E,r}[M_r^\bullet],\varinjlim_jC(\Ga_j^\bullet,A)\bigr)\\
                                                                                                                                      & = & \varinjlim_j \Map_{\CA_{\mO_{E,r}}^\Delta}\bigl(\mO_{E,r}[M_r^\bullet],C(\Ga_j^\bullet,A)\bigr)\\
                                                                                                                                      & = & \varinjlim_j \mR_{\Ga_j,M_r}(A).
\end{eqnarray*}
Here the first isomorphism follows from  
$C_{cts}(\Ga^n,A)=\varinjlim_iC(\Ga_j^n,A)$ by \eqref{E: comp Ccts}. 
The second isomorphism follows as if $B^\bullet$ is a cosimplicial object in $\CA_{\mO_{E,r}}^\heart$, then
 $\Map_{\CA_{\mO_{E,r}}^\Delta}\bigl(\mO_{E,r}[M_r^\bullet],B^\bullet\bigr)\subset \Map_{\CA_{\mO_{E,r}}}(\mO_{E,r}[M],B([1]))$. The last isomorphism follows from \eqref{R: cohom pres}.
\end{proof}

In particular if $A$ is a classical $\mO_E$-algebra in which $\varpi$ is nilpotent, then $A$-points of $\mR^{c}_{\Gamma, M}$ form the set of continuous homomorphisms from $\Gamma$ to $M(A)$ (equipped with the discrete topology). }
\begin{remark}\label{R: formal and rigid}
Let $\on{Spf}A=\varinjlim_j \on{Spec}(A/I^j)$ be a classical formal scheme over $\on{Spf}\mO_E$, where $I$ is a finitely generated ideal of definition of $A$ containing $\varpi$. Then 
$$\Map(\on{Spf} A, \mR^{c}_{\Gamma, M})=\underleftarrow\lim_j\mR^{c}_{\Gamma, M}(A/I^j)\subset \underleftarrow\lim_j\mR_{\Gamma, M}(A/I^j)=\mR_{\Ga,M}(A_I^\wedge)$$ 
consists of continuous homomorphisms from $\Gamma$ to $M(A_I^\wedge)$, where $A_I^\wedge$ is the $I$-adic completion of $A$, equipped with the $I$-adic topology. So  ${}^{cl}\mR^{c}_{\Gamma, M}$ coincides with the space considered in \cite[3.1]{WE} (when $M=\GL_m$).
 
We may also take the rigid generic fiber of ${}^{cl}\mR^{c}_{\Gamma, M}$, or the adic space over $\on{Spa}(E,\mO_E)$ (as in \cite[2.2]{SW}), 
denoted by ${}^{cl}\mR_{\Gamma, M}^{c,\ad}$. It is the sheafification (with respect to the Zariski topology on the category of affinoid $(E,\mO_E)$-algebras) of the presheaf:
\[
(A,A^+)\mapsto \varinjlim_{A_0\subset A^+}\mR_{\Gamma, M}^{c}(\on{Spf} A_0)=\varinjlim_{A_0\subset A^+}\varprojlim_j\mR_{\Gamma, M}^{c}(A_0/\varpi^j),
\]
where $A_0$ range over open and bounded subrings of $A^+$.
For example, if $\Ga$ is a profinite group, then $E$-points of ${}^{cl}\mR_{\Gamma, M}^{c,\ad}$ are the set of continuous homomorphisms from $\Ga$ to $M(E)$, where the latter is equipped with the usual $\varpi$-adic topology. So ${}^{cl}\mR_{\Gamma, M}^{c,\ad}$ probably coincides with the space considered in \cite[\S 2]{An} (when $M=\GL_m$). 
\end{remark}

For a representation $W$ of $M$ on a finite projective $\mO_{E,r}$-module, we have the vector bundle ${}_{\Ga}W$ on $\mR^c_{\Ga,M}$ and on $\mR^c_{\Ga,M/H}$ equipped with $\Ga\to \End({}_{\Ga}W)$ as in \eqref{E: Ga act W}, obtained by pulling back of the corresponding objects on $\mR_{\Ga,M}$ and on $\mR_{\Ga,M/H}$ along the morphisms \eqref{E: forget cont}. If $\rho\in \mR^c_{\Ga,M}(A)$, then the pullback of ${}_{\Ga}W$ to $\on{Spec} A$, denoted by $W_\rho$ is equipped with an action $\Ga\to \End_{\Mod_A^{\leq 0}}(W_\rho)$. This action should be continuous in an appropriate sense.
One way to make this precise is by noticing that there is  a cosimplicial module $C_{cts}(\Ga^\bullet, W_\rho)$ over $C_{cts}(\Ga^\bullet,A)$ constructed in a way as in Remark \ref{R: fiber action} \eqref{R: fiber action-1}. 
As in Remark \ref{R: tan compl} , we may consider the totalization $C^*_{cts}(\Ga,W_\rho)$ of $C_{cts}(\Ga^\bullet, W_\rho)$ (in $\Mod_A$). If $A$ is classical, this is the cochain complex computing the continuous cohomology of $\Ga$ with coefficient in $W_\rho$. Let $\overline{C}^*_{cts}(\Ga,W_\rho)[1]$ denote its reduced version.

Now we study the infinitesimal geometry of $\mR^{c}_{\Ga,M}$. We assume that $M$ is an affine smooth group scheme over $\mO_E$.
\begin{prop}\label{E: tangent space}
The functor $\mR^c_{\Ga,M_r}: \CA_{\mO_{E,r}}\to\Spc$ is nilcomplete and preserves finite limits. If $A$ is truncated, then 
the tangent space  of $\mR^c_{\Ga,M_r}$ at an $A$-point $\rho$ is
$\bT_\rho\mR^c_{\Ga,M_r}=\overline{C}^*_{cts}(\Ga,\Ad_\rho)[1]$. 
\end{prop}
\begin{proof}
As $C_{cts}(S,-): \CA_{\mO_{E,r}}\to \CA_{\mO_{E,r}}$ is nilcomplete and preserves finite limits, so is $\mR^c_{\Ga,M_r}$. To prove the last assertion, it is enough to show that for
$\rho\in \mR^c_{\Ga,M}(A)$ with $A\in\CA_{\mO_{E,r}}$, and for any connective $A$-module $V$, we have
\begin{equation}\label{E: cotang comp Rc}
 \mR^c_{\Ga,M_r}(A\oplus V)\times_{\mR^c_{\Ga,M_r}(A)}\bigl\{\rho\bigr\}\cong \tau^{\leq 0}\bigl(\overline{C}_{cts}^*(\Ga, \Ad_\rho\otimes V)[1]\bigr).
\end{equation}

To prove this, we start by recalling the following construction. Let $K(\bZ,1)$ be the simplicial abelian group associated to the cochian complex $\bZ[1]$ under the classical Dold-Kan correspondence. Its underlying simplicial set can be obtained by applying the Milnor construction to $\bZ$ (regarded as a monoid). So $K(\bZ,1)([n])=\bZ^{\oplus n}$.
Let $K(\bZ,-1)$ be the cosimplicial abelian group assigning $[n]$ to the $\bZ$-linear dual of $K(\bZ,1)([n])$.
Let $N^\bullet\in (\Mod_\bZ^{\geq m})^{\Delta}$ be a cosimplicial object in $\Mod_\bZ^{\geq m}$ (for some integer $m$), then by the (dual) Dold-Kan correspondence,
\begin{equation}\label{E: cosim ab grp}
\Map_{\Mod_\bZ^\Delta}(K(\bZ,-1), N^\bullet)= \tau^{\leq 0}(\overline{N}^*[1]).
\end{equation}
Here $\overline{N}^*$ is the complex obtained from $N^\bullet$ by the following procedure. There is a natural morphism $N^\bullet\to N([0])$, where $N([0])$ is regarded as a constant cosimplicial cochain complex. Then $\overline{N}^*$ is totalization of the complex associated to the fiber of $N^\bullet\to N([0])$. 

 If $B^\bullet\in\CA_{\mO_{E,r}}^\Delta$, we denote by $K(B^\bullet,-1)$ the base change of $K(\bZ,-1)$  along $\bZ\to B^\bullet$ (where $\bZ$ is regarded as the constant cosimplicial algebra $\bZ$), i.e. $K(B^\bullet,-1)([n])=K(\bZ,-1)([n])\otimes B([n])$.

Now consider the cosimplicial module $[n]\mapsto \Omega_{M_r^n}$ over the cosimplicial algebra $\mO_{E,r}[M_r^\bullet]$, denoted by $\Omega_{M_r^\bullet}$. We claim that there is a natural isomorphism in the (ordinary) category of cosimplicial modules over $\mO_{E,r}[M^\bullet]$,
\begin{equation}\label{E: simp cotang}
 \Omega_{M_r^\bullet}\cong (\mO_{E,r}[M^\bullet]\otimes \Ad^*)\otimes_{\mO_{E,r}[M^\bullet]} K(\mO_{E,r}[M^\bullet],-1), 
\end{equation}
where $(\mO_{E,r}[M^\bullet]\otimes \Ad^*)$ is the cosimplicial modules over $\mO_{E,r}[M^\bullet]$ induced by the coadjoint representation $\Ad^*$ (see Remark \ref{R: fiber action} \eqref{R: fiber action-1}).
Namely, the right hand side of \eqref{E: simp cotang}, when evaluated at the simplex $[n]$, is canonically isomorphic to $(\mO_{E,r}[M^n]\otimes \Ad^*)^{\oplus n}$. On the other hand,
we can also identify 
$\Omega_{M^n_r}\cong ({}_{\FM(\{1,2,\ldots,n\})}\Ad^*)^{\oplus n}\cong (\mO_{E,r}[M^n]\otimes \Ad^*)^{\oplus n}$ as in \eqref{E: Cf1} \eqref{E: Cf2} \eqref{E: diff vs coadj}. 
Then using notations there, the desired isomorphism, when evaluated at $[n]$, is given by
\[
(\mO_{E,r}[M^n]\otimes \Ad^*)^{\oplus n}\simeq (\mO_{E,r}[M^n]\otimes \Ad^*)^{\oplus n},\quad (\omega_1,\ldots,\omega_n)\mapsto (\omega_1, \ga_1\omega_2,\ga_1\ga_2\omega_3,\ldots,\ga_1\cdots\ga_{n-1}\omega_n).
\]

Let $\on{TwArr}(\Delta)$ denote the twisted arrow category of $\Delta$ (\cite[5.2.1]{Lu2}): its objects are morphisms $[m]\to [n]$ in $\Delta$ and morphisms from $f':[m']\to [n']$ to $f:[m]\to [n]$ are pairs of maps $(g:[m']\to [m], h:[n]\to [n'])$ such that $f'=hfg$. Consider the functor 
\begin{equation*}
\begin{split}
 \mF: \ & \on{TwArr}(\Delta)^{\on{op}}\to \Spc,\quad ([m]\to [n])\mapsto \\
             & \Map_{\CA_{\mO_{E,r}}}\bigl(\mO_{E,r}[M_r^m],C_{cts}(\Ga^n,A\oplus V)\bigr)\times_{\Map_{\CA_{\mO_{E,r}}}\bigl(\mO_{E,r}[M_r^m],C_{cts}(\Ga^n,A)\bigr)}\{\rho_{m,n}\}\\
        = \  & \Map_{\Mod_{\mO_{E,r}[M^m]}}\bigl(\Omega_{M^m_r}, C_{cts}(\Ga^n,V)\bigr),  
\end{split}
\end{equation*}
where $\rho_{m,n}$ is the point in $\Map_{\CA_{\mO_{E,r}}}\bigl(\mO_{E,r}[M^m],C_{cts}(\Ga^n,A)\bigr)$ determined by $\rho$. Using \cite[1.3.12]{GKRV}, we can rewrite the left hand side of \eqref{E: cotang comp Rc} as $\varprojlim_{\on{TwArr}(\Delta)^{\on{op}}} \mF$, which by \eqref{E: simp cotang} can be rewritten as
\[
\Map_{\mO_{E,r}[M^\bullet]}\bigl(\Omega_{M_r^\bullet}, C_{cts}(\Ga^\bullet,V)\bigr)\cong \Map_{\mO_{E,r}[M^\bullet]}\bigl(K(\mO_{E,r}[M^\bullet],-1), C_{cts}(\Ga^\bullet,\Ad_\rho\otimes V)\bigr).
\]
which by \eqref{E: cosim ab grp} is isomorphic to the right hand side of  \eqref{E: cotang comp Rc}.
\end{proof}

\begin{prop}\label{P: trun value}
If $A$ is a truncated $\mO_{E,r}$-algebra, then \eqref{E: indpre} induces an isomorphism 
\begin{equation}\label{E: trun Apt}
\mR^c_{\Ga,M_r}(A)=\varinjlim_j\mR_{\Ga_j,M_r}(A).
\end{equation}
If $\Ga$ is profinite, then for each $m$ the restriction functor $\mR^c_{\Ga,M_r}: {}_{\leq m}\CA_{\mO_{E,r}}\to \Spc$ commutes with filtered colimits.
\end{prop}
\begin{proof}
We temporarily denote $\varinjlim_j\mR^c_{\Ga_j,M_r}$ by $\widetilde{\mR}^{c}_{\Ga,M_r}$. 
We already see that \eqref{E: trun Apt} induces an isomorphism at the level of classical points.
Now assume that $A$ is $m$-truncated. We have the Postnikov tower $A=\tau_{\leq m}A\to \tau_{\leq m-1}A\to \cdots\to \tau_{\leq 0}A=\pi_0(A)$ and the following pullback diagram (see \cite[7.4.1.29]{Lu2} for the case of $E_\infty$-algebras which also holds for animated algebras)
\[
\xymatrix{
\tau_{\leq i}A\ar[r]\ar[d]& \tau_{\leq i-1}A\ar[d]\\
\tau_{\leq i-1}A\ar[r] & \tau_{\leq i-1}A\oplus \pi_{i}(A)[i+1].
}\]
As both  $\widetilde{\mR}^{c}_{\Ga,M_r}$ and $\mR^c_{\Ga,M_r}$ commute with finite limits, by induction on $m$ and by Remark \ref{R: tan compl}  and \eqref{E: cotang comp Rc}, to prove \eqref{E: trun Apt} it is enough to show that 
\[
\varinjlim_j C^*(\Ga_j, \Ad_\rho\otimes \pi_i(A))\cong C_{cts}^*(\Ga,\Ad_\rho\otimes \pi_i(A))
\]
for every $\rho\in \mR^c_{\Ga,M_r}(\pi_0(A))=\varinjlim_j\mR_{\Ga_j,M_r}(\pi_0(A))$. But this follows from \eqref{E: comp Ccts} and the isomorphism $\prod_{j\in J}\varinjlim_{i\in I_j^{\on{op}}} V^{S_{ij}}\cong \varinjlim_{(i\in I_j^{\on{op}})} \prod_{j\in J}V^{S_{ij}}$ (as $\Mod_{\mO_{E,r}}^\heart$ is an abelian category satisfying Grothendieck's axiom (AB$6$)).

For the last statement, we note that if $\Ga$ is profinite then each $\Ga_j$ is finite so $\mR_{\Ga_j,M_r}$ when restricted to ${}_{\leq m}\CA_{\mO_{E,r}}$ commutes with filtered colimits (Proposition \ref{P: almost fin}). Therefore, $\mR^c_{\Ga,M_r}: {}_{\leq m}\CA_{\mO_{E,r}}\to \Spc$ also commutes with filtered colimits. Alternatively, one can prove this directly by induction on $m$, again using the Postnikov tower and that $C_{cts}(S,-)$ commutes with filtered colimits when $S$ is profinite (Lemma \ref{L: cont S to dis}).
\end{proof}
\begin{remark}
The proposition shows that $\mR^c_{\Ga,M_r}$ is an ind-affine scheme in the sense of \cite[1.4.2]{GR}. Note that \eqref{E: trun Apt} may not hold for general $A$. Instead, $\mR^c_{\Ga,M_r}(A)=\varprojlim_m\varinjlim_j\mR^c_{\Ga_j,M_r}(\tau_{\leq m}A)$, as $\mR^c_{\Ga,M_r}$ is nilcomplete. This can be used as an alternative definition of $\mR^c_{\Ga,M_r}$.
\end{remark}

Now we can relate $\mR_{\Gamma, M}^{c}$ with the usual deformation space (and its derived version as in \cite{SV}).

We fix a closed point $x$ of ${}^{cl}\mR_{\Gamma, M}^{c}$, corresponding to
$\bar\rho: \Ga\to M(\kappa)$, where $\kappa$ is the residue field of $x$, which is algebraic over $\kappa_E$.  Let $\mathbf{Art}_{\mO_E,\kappa}$ denote the category of local Artinian $\mO_E$-algebras with residue field algebraic over $\kappa$, and $\CA^{\on{Art}}_{\mO_E,\kappa}\subset \CA_{\mO_E}$ the $\infty$-category of animated $\mO_E$-algebras $A$, such that $\pi_0(A)\in\mathbf{Art}_{\mO_E,\kappa}$, and such that $\bigoplus_i\pi_i(A)$ is a finitely generated $\pi_0(A)$-module. In particular, every $A\in \CA^{\on{Art}}_{\mO_E,\kappa}$ is truncated.

Following \cite[8.1.6.1]{Lu3}, we denote the formal completion $(\mR_{\Gamma, M}^{c})_x^\wedge$ of $\mR_{\Gamma, M}^{c}$ at $x$ as the functor sending an animated ring $A$ over $\on{Spf}\mO_E$ to the subspace of $(\mR_{\Gamma, M}^{c})(A)$ consisting of those $\on{Spec} A\to \mR_{\Gamma, M}^{c}$ such that every point of $\on{Spec} \pi_0(A)$ maps to $x$.  Its restriction to $\CA^{\on{Art}}_{\mO_E,\kappa}\subset \CA_{\mO_E}$, also denoted by $\on{Def}_{\bar\rho}^\Box$, is the functor
\[
  \CA^{\on{Art}}_{\mO_E,\kappa}\to \Spc, \quad A\mapsto \mR_{\Ga,M}^{c}(A)\times_{\mR_{\Gamma, M}^{c}(\kappa_A)}\{\bar\rho\}. 
\]
This recovers the deformation functor defined in \cite[\S 5]{SV}.
Its further restriction to $\mathbf{Art}_{\mO_E,\kappa}$, denoted by ${}^{cl}\on{Def}_{\bar\rho}^\Box$, 
is identified with the classical framed deformation functor of $\bar\rho$
\[
\mathbf{Art}_{\mO_E,\kappa}\to \Sets,\quad A\mapsto  \Bigl\{\mbox{Continuous homomorphism }  \rho: \Ga\to M(A)\mid \rho\otimes_A\kappa_A=\bar \rho\otimes_{\kappa}\kappa_A\Bigr\}.
\]

Similarly, we have the formal completion $(\mR_{\Gamma_{i}, M_{n}})_{x}^\wedge$ of each $\mR_{\Gamma_{i}, M_{n}}$ at $x$.
By \cite[8.1.2.2]{Lu3}\footnote{The proof is written for $E_\infty$-rings, but  it works for animated rings, with $A\{t_n\}$ in \emph{loc. cit.} replaced by the usual polynomial ring $A[t_n]$. In addition, in this case each $A_n$ in \emph{loc. cit} is perfect as an $A$-module.}, each 
$(\mR_{\Gamma_{i}, M_{n}})_{x}^\wedge\simeq \varinjlim_j\on{Spec} A_j$ 
is represented by a derived affine ind-scheme with $A_j\in \CA^{\on{Art}}_{\mO_E,\kappa}$. Then
$(\mR_{\Gamma, M}^{c})_x^\wedge$, which is isomorphic to
$\varinjlim_{i,n} (\mR_{\Gamma_{i}, M_{n}})_{x}^\wedge$,  is also represented by a derived affine ind-scheme over $\on{Spf}\mO_E$. 
Combining the above discussions with \eqref{E: tangent space}, we recover the following statement from \cite{SV}.
\begin{prop}The functor $\on{Def}^{\Box}_{\bar\rho}$ is prorepresentable, whose tangent complex is $\overline{C}_{cts}^*(\Ga, \Ad_
\rho)[1]$.
\end{prop}

We finish our discussion of infinitesimal geometry of $\mR^c_{\Ga,M}$ by the following observation. Suppose $\widehat\Ga$ is the profinite completion of 
an abstract group $\Ga$.  Then we have $\mR_{\Ga,M}$ over $\on{Spec}\mO_E$ and $\mR^{c}_{\widehat\Ga,M}$ over $\on{Spf}\mO_E$. There is a natural morphism $\mR^{c}_{\widehat\Ga,M}\to \mR_{\Ga,M}$, which induces a bijection between closed points over $\kappa_E$ and isomorphisms of classical formal completions at these points.
This follows from the simple observation that for every classical Artinian local ring $A$ with residue field finite over $\kappa_E$, every homomorphism $\rho: \Ga\to M(A)$ factors through a finite quotient of $\Ga$ and therefore extends uniquely to a continuous homomorphism $\widehat\Ga\to M(A)$. By the following lemma, it still holds at the derived level under a mild assumption. We omit the proof as it is very similar to the proof of Proposition \ref{P: trun value}. 

\begin{lemma}\label{L: formal closed}
Suppose $\Ga\to \widehat\Ga$ induces an isomorphism $H^i_{cts}(\widehat\Ga, V)\cong H^i(\Ga, V)$ 
for every finite $\bF_\ell\Ga$-module $V$ (which automatically extends to a discrete $\widehat\Ga$-module) and every $i\geq 0$.
Then $\mR^{c}_{\widehat\Ga,M}\to \mR_{\Ga,M}$ induces isomorphisms of formal completions (at the derived level) at closed points over $\kappa_E$.
\end{lemma}

Before we move to the global geometry of $\mR^c_{\Ga,M}$,  
we introduce an auxiliary object, the moduli space $\mR^c_{M^\bullet/\!\!/H}$ of continuous pseudorepresentations. We assume that $\Ga$ can be written as $\Ga=\varprojlim_j \Ga_j$ as before, and assume that $(M,H)$ are as in Proposition \ref{P: fin pseudo rep}. 
\begin{definition}
We define the moduli of continuous pseudorepresentations over $\on{Spec} \mO_{E,r}$ as
\[
\mR^{c}_{\Ga,M_r^\bullet/\!\!/H_r}: \CA_{\mO_{E,r}}\to \Spc, \quad A\mapsto \varprojlim_{m}\varinjlim_{j}\mR_{\Ga_j,M_r^\bullet/\!\!/H_r}(\tau_{\leq m}A),
\]
and over $\on{Spf}\mO_E$ as $\mR^c_{\Ga,M^\bullet/\!\!/H}=\varinjlim_r\mR^{c}_{\Ga,M_r^\bullet/\!\!/H_r}$.
\end{definition}

\begin{remark}
The definition of $\mR^{c}_{\Ga,M_r^\bullet/\!\!/H_r}$ given above is somehow ad hoc but is convenient for the discussions below.
It would be more elegant to make a definition based on  
\eqref{E: hom FFMalg}. Namely,
there are $\FFM$-algebras $\FM(I)\mapsto \mO_{E,r}[M_r^I]^{H_r}$ and $\FM(I)\mapsto C_{cts}(\Ga^I,A)$.
Then one can define
\[
\widetilde\mR^{c}_{\Ga,M_r^\bullet/\!\!/H_r}: \CA_{\mO_{E,r}}\to \Spc,\quad A\mapsto \on{Map}_{\CA_{\mO_{E,r}}^{\FFM}}\bigl(\mO_{E,r}[M_r^\bullet]^{H_r}, C_{cts}(\Ga^\bullet,A)\bigr).
\]
There is an obvious morphism 
\begin{equation}\label{E: indpreps}
\mR^{c}_{\Ga,M_r^\bullet/\!\!/H_r}\to \widetilde\mR^{c}_{\Ga,M_r^\bullet/\!\!/H_r}
\end{equation}
similar to \eqref{E: indpre}, which we expect to be an isomorphism (similar to Proposition \ref{P: trun value}). If so, this new definition will be equivalent to the ad hoc one. 
One can show that\eqref{E: indpreps} induces a bijection of $\kappa$-points, for every algebraic field extension $\kappa/\kappa_E$. In addition if the $\FFM$-algebra $\mO_{E,r}[M_r^\bullet]^{H_r}$ is finitely generated (see \cite[1.1]{Wei} for this notion), then \eqref{E: indpreps} would be an isomorphism at least for the underlying classical moduli spaces.
This is indeed this case if $M=\GL_m$ by \cite{Do}.  
\end{remark}

By definition, $\mR^c_{\Ga,M^\bullet/\!\!/H}$ is an ind-affine scheme (in the sense of \cite[1.4.2]{GR}) over $\on{Spf}\mO_E$. 
If $\Ga$ is profinite, then by Proposition \ref{P: fin pseudo rep}, the underlying reduced classical ind-scheme of $\mR^{c}_{\Ga,M_r^\bullet/\!\!/H_r}$ is just union of points algebraic over $\kappa_E$. Therefore 
\begin{equation}\label{E: decom pseudo}
\mR^c_{\Ga,M^\bullet/\!\!/H}=\sqcup_{\Theta} \mR^{c,\Theta}_{\Ga,M^\bullet/\!\!/H},
\end{equation} 
where $\Theta$ range over points of $\mR^c_{\Ga,M^\bullet/\!\!/H}$ algebraic over $\kappa_E$, and each $\mR^{c,\Theta}_{\Ga,M^\bullet/\!\!/H}$ a formal scheme.
For $M=\GL_m$, this is originally proved by Chenevier \cite[3.14]{Ch}. 

\begin{remark}
Assume that $\Ga$ is profinite. As  $\mR^{c,\Theta}_{\Ga,M^\bullet/\!\!/H}$ is formal,
we may call its restriction to $\CA_{\mO_E,\kappa}^{\on{Art}}$ the pseudodeformation space of $\Theta$. Its further restriction to $\mathbf{Art}_{\mO_E,\kappa}$ is the classical pseudodeformation space of $\Theta$ studied in literature (for $M=\GL_m$).
\end{remark}

As in Remark \ref{R: formal and rigid}, for $\on{Spf} A=\varinjlim_j \on{Spec} (A/I^j)$ over $\on{Spf}\mO_E$, we have
$$\Map(\on{Spf}A,\mR^c_{\Ga,M^\bullet/\!\!/H})=\varprojlim_j\mR^c_{\Ga,M^\bullet/\!\!/H}(A/I^j)\subset \mR_{\Ga,M^\bullet/\!\!/H}(A^\wedge_I),$$
where $\Ga$ is regarded as an abstract group in $\mR_{\Ga,M^\bullet/\!\!/H}$.
The following result will be used later.

\begin{prop}\label{P: cont psrep to trrep}
Assume that $\Ga$ is profinite.
Let $\mO_K$ be a complete DVR with fractional field $K$ and maximal ideal $\frakm$.
Let $\Theta\in \Map(\on{Spf}\mO_K,\mR^c_{\Ga,M^\bullet/\!\!/H})$, giving a $\Lambda$-valued pseudorepresentation of the underlying abstract group of $\Ga$.
Then there is a finite extension $K'/K$, and a geometrically completely reducible continuous representation $\rho: \Ga\to M(K')$ such that $\Tr\rho=\Theta$.
\end{prop}
\begin{proof} 
Clearly $\Theta$ gives a $\Lambda$-valued pseudorepresentation of the underlying abstract group of $\Ga$.
Recall that from \cite[11.7]{L} and \cite[4.5]{BHKT}, there is a geometrically completely reducible representation (see \cite[3.5]{BHKT} for the terminology) $\rho: \Ga\to M(\overline K)$ such that $\Tr\rho=\Theta$. To show that it is continuous, one can mimic the argument as in \cite[11.7]{L} with the following change. 
Note our $(M,H)$ correspond $(H,H^0)$ in \emph{loc. cit}. Under this notation change, choose $(g_1,\ldots,g_n)\in M(\overline K)$ as in \emph{loc. cit.} and let $C(g_1,\ldots,g_n)\subset H_{\overline K}$ be the stabilizer of $(g_1,\ldots,g_n)$ under the diagonal $H$-action on $M^n$ and
let $D(g_1,\ldots,g_n)\subset M_{\overline K}$ be the fixed points of $C(g_1,\ldots,g_n)$. Then  in \emph{loc. cit.}  the map $\overline \Lambda[M^{n+1}/\!\!/H]\to \overline \Lambda[D(g_1,\ldots,g_n)]$ (denoted by $q$ in \emph{loc. cit.}) is shown to be surjective when $\cha K=0$ since taking invariants with respect to a reductive group of a surjective ring map remains surjective. This may not be the case in positive characteristic.  But this map is power surjective as in \cite{TVDK2}. This weaker statement suffices to apply all the arguments in \emph{loc. cit.} to deduce continuity of $\rho$. As $\Ga$ is profinite, $\rho$ factors through $\Ga\to M(K')$ for some $K'/K$ finite by the standard argument using the Baire category theorem.
\end{proof}

Now we discuss the global geometry of $\mR^c_{\Ga,M}$.
By Proposition \ref{P: trun value}, there is a natural morphism
$\Tr: \mR^c_{\Ga,M}\to \mR^c_{\Ga,M^\bullet/\!\!/H}$.
Suppose $\Ga$ admits a unique maximal open compact normal subgroup $\Ga_c\subset\Ga$.
Together with \eqref{E: decom pseudo},  we obtain the decomposition
\begin{equation}\label{E: dec cont}
\mR^c_{\Ga,M}=\sqcup_{\Theta} \mR^{c,\Theta}_{\Ga,M}\to \sqcup_{\Theta}\mR^{c,\Theta}_{\Ga_c,M^\bullet/\!\!/H}
\end{equation}
where $\Theta$ range over closed points of $\mR^c_{\Ga_c,M^\bullet/\!\!/H}$,
such that $\Tr(\rho_x|_{\Ga_c})=\Theta$ for every $\overline{\kappa}_E$-point $x$ of $\mR^{c,\Theta}_{\Ga,M}$ corresponding to a continuous representation $\rho_x:\Ga\to M(\overline{\kappa}_E)$. 

\begin{example}\label{E: disc cont moduli}
Let us consider the simplest case when $\Ga=\widehat{\bZ}$. If $M=\bG_m$, then $\mR^{c}_{\Ga,M}$ is just the union of all torsion points of $\bG_m$, and therefore is isomorphic to $\sqcup_x (\bG_m)_x^\wedge$, where $x$ range over all closed points of $\bG_m\otimes\kappa_E$. 
For a slightly more complicated example, we let $M$ be a split connected reductive group over $\mO_E$, and denote $M/\!\!/M$ its adjoint quotient. Then
$\mR^{c}_{\Ga,M}\cong M\times_{M/\!\!/ M} (\sqcup_x (M/\!\!/M)_x^\wedge)$, where $x$ range over all closed points of $M/\!\!/M$.
\end{example}

\begin{remark}\label{R: negative side}
Example \ref{E: disc cont moduli}  suggests that while Definition \ref{D: cont rep sp} makes sense for any locally profinite group $\Ga$, it may not give the ``most correct" object for some purposes.
Namely, although $\mR^{c}_{\Ga,M}$ already glues various deformation spaces of $\Ga$ together,
in general it is still disconnected and has formal directions. This example also suggests in certain cases different components of $\mR^{c}_{\Ga,M}$ could be further glued. 
For example, all the components of $\mR^{c}_{\widehat\bZ,M}$ should naturally glue to $\mR_{\bZ,M}=M$. This is a special case of a general phenomenon discussed below (in particular see Proposition \ref{P: glue formal}). 
To give another example, let $F$ be a non-archimedean local field of residue characteristic $p$ with $\Ga_F$ its Galois group and $W_F$ its Weil group. Then if $p\neq \ell$,
it is more correct to consider $\mR^{c}_{W_F,M}$ than $\mR^{c}_{\Ga_F,M}$, as we shall see in the next section. If $p=\ell$, even $\mR^{c}_{W_F,M}$ is not enough, as explained to us by Emerton. Instead, one needs the construction as in \cite{EG}.
Finally, we also expect that when $\Ga$ is the \'etale fundamental group of a smooth (affine) algebraic curve over $\overline\bF_p$ (with $p\neq \ell$), there is a more sophisticated construction of its representation space.
\end{remark}

When $\widehat\Ga$ is the profinite completion of an abstract group $\Ga$ as in Lemma \ref{L: formal closed}, then under certain mild assumptions $\mR_{\Ga,M}$ glues different components of $\mR^{c}_{\widehat\Ga,M}$ (as in the decomposition \eqref{E: dec cont}) together. 

\begin{prop}\label{P: glue formal}
Let $\Ga$ be a finitely generated group of type $FP_\infty(k)$ such the map $\Ga\to \widehat\Ga$ induces an isomorphism of group cohomology $H^i_{cts}(\widehat\Ga, V)\cong H^i(\Ga, V)$ 
for every finite $\bF_\ell\Ga$-module $V$.
Then the natural morphism $\mR^{c}_{\widehat\Ga,M}\to \mR_{\Ga,M}$  induces an isomorphism
\begin{equation*}\label{E: comp along pseudo}
\mR^{c}_{\widehat\Ga,M}\cong \mR_{\Ga,M}\times_{\mR_{\Ga,M^\bullet/\!\!/H}}\bigl(\sqcup_x (\mR_{\Ga,M^\bullet/\!\!/H})_x^\wedge\bigr),
\end{equation*}
where $x$ range over all closed points of $\mR_{\Ga,M^\bullet/\!\!/H}$ over $\kappa_E$ and $(\mR_{\Ga,M^\bullet/\!\!/H})_x^\wedge$ is the formal completion of $\mR_{\Ga,M^\bullet/\!\!/H}$ at $x$.
\end{prop}
\begin{proof}
We only give the proof at the level of classical moduli problems. A similar argument as in Proposition \ref{P: trun value} will show that it is also an isomorphism at the derived level. 

By Proposition \ref{P: fin pseudo rep}, clearly $\mR^{c}_{\widehat\Ga,M}\to \mR_{\Ga,M}$ factors through the morphism $\mR^{c}_{\widehat\Ga,M}\to \mR_{\Ga,M}\times_{\mR_{\Ga,M^\bullet/\!\!/H}}(\sqcup_x (\mR_{\Ga,M^\bullet/\!\!/H})_x^\wedge)$. We need to construct the inverse map.
So let $\rho: \Ga\to M(A)$ be homomorphism, where $A$ is classical of finite type over $\mO_E$ such that the composed morphism $\on{Spec} A\to \mR_{\Ga,M}\to \mR_{\Ga,M^\bullet/\!\!/H}$ maps the topological space $|\on{Spec} A|$ to $x$. It is enough to show that $\rho$ factors through a finite quotient of $\Ga$. We may choose a faithful embedding $M\to\GL_m$ and assume that $M=\GL_m$. By our assumption, the image of the map $\Lambda[\mR_{\Ga,\GL_m^\bullet/\!\!/\GL_m}]^{\GL_m}\to A$, denoted by $B$, is artinian local. Note that for every $\ga\in \Ga$, the characteristic polynomial $\on{Char}(\rho(\ga),t)=\det(t-\rho(\ga))$ of $\rho(\ga): A^m\to A^m$ belongs to $B[t]$. The following argument is a slight variant of \cite[2.8-2.10]{de}.

First assume that $A$ is reduced so it is a finite type $\kappa_E$-algebra. Then $B$ is a finite extension of $\kappa_E$. We know that there is a finite extension $\kappa$ of $B$ and a completely reducible representation $\rho': \Ga\to \GL_m(\kappa)$ such that $\on{Char}(\rho'(\ga),t)=\on{Char}(\rho(\ga),t)$ for every $\ga\in\Ga$. In particular, there is a finite index subgroup $\Ga_1\subset\Ga$ such that $\on{Char}(\rho(\ga),t)=(t-1)^m$. By replacing $A$ by its quotient ring and by conjugation, one can assume that $\rho(\ga)$ is strictly upper triangular for every $\ga\in \Ga_1$. Note that the group of strictly upper triangular matrices with coefficient in $A$ is a nilpotent group of exponent of some power of $\ell$. By our assumption $H^1(\Ga_1,\bF_\ell)$ is a finite dimensional $\bF_\ell$-vector space. So there is a finite index subgroup $\Ga_2\subset\Ga_1$ such that $\rho|_{\Ga_2}$ is trivial.

For general finite type $\mO_E$-algebra $A$ in which $\ell$ is nilpotent, let $A_{\on{red}}$ be its quotient by the nilradical. Let $\Ga_2$ be the kernel of $\Ga\to\GL_m(A)\to\GL_m(A_{\on{red}})$, which is of finite index in $\Ga$.  As the kernel $\GL_m(A)\to \GL_m(A_{\on{red}})$ is a nilpotent group of exponent some power of $\ell$, and $H^1(\Ga_2,\bF_\ell)$ is finite dimensional, there is a finite index subgroup $\Ga_3\subset\Ga_2$ such that $\rho|_{\Ga_3}$ is trivial.  
\end{proof}

The last topic of this subsection is an extension of the moduli space $\mR^c_{\Ga,M}$ from $\on{Spf}\mO_E$ to $\on{Spec} \mO_E$.  Of course, if $\Ga$ appears to be the profinite completion of $\Ga_0$ for some abstract group $\Ga_0$ as in Proposition \ref{P: glue formal}, such extension can be given by $\mR_{\Ga_0,M}$. This is the approach we will adapt to construct the moduli of local Langlands parameters (in the $\ell\neq p$ case).
However, not every $\Ga$ arises in this way, and even it is, there is in general no canonical choice of $\Ga_0$. Therefore, it is desirable to have a more direct construction. As in general $\mR^c_{\Ga,M}$ has non-trivial formal directions, probably such extensions should be of analytic nature in general. However, for the specific situations considered in the next section, the following approach suffices. The idea is to extend the definition of $C_{cts}(S,-)$ for $\mO_{E,r}$-modules/algebras in Lemma \ref{L: cont S to dis} to a functor for $\mO_E$-modules/algebras satisfying similar properties. Then almost all the rest of the constructions go through without change.

Let $\Mod_{\mO_{E,r}}^{\heart, f.g.}$ denotes the abelian category of finite $\mO_{E,r}$-modules. The natural forgetful functor from $\Mod_{\mO_{E,r}}^{\heart, f.g.}$ to the category $\Sets_f$ of finite sets is faithful conservative, preserves finite products and is lax symmetric monoidal (where $\Sets_f$ is equipped with the Cartesian symmetric monoidal structure). 
It induces a natural functor\footnote{We learned the idea of considering such functor from Peter Scholze who developed an approach of moduli of continuous representations via condensed mathematics. Our approach here does not make use of condensed mathematics, but likely it is essentially the same as Scholze's.} 
\begin{equation}\label{E: indpro on OEmod}
\Mod^\heart_{\mO_E}=\Ind\varprojlim_r\Mod_{\mO_{E,r}}^{\heart, f.g.}\to \Ind\Pro(\Sets_f),
\end{equation}
satisfying similar properties, where $\Ind\Pro(\Sets_f)$ denotes the ind-completion of the category of profinite sets.
Note that a disjoint union of profinite sets $S$ can also be regarded as an object in $ \Ind\Pro(\Sets_f)$. 
\begin{lemma}\label{L: cont S to ind-l-adic}
Let $S$ be a disjoint union of profinite sets, regarded as an ind-profinite set. Then 
$\Mod^\heart_{\mO_E}\to \Mod^\heart_{\mO_E}, \ C_{cts}(S,V)=\Map_{\Ind\Pro(\Sets_f)}(S, V)$ 
satisfies the same properties as the one in Lemma \ref{L: cont S to dis} and therefore extends to a $t$-exact functor 
\begin{equation}\label{E:CctsM}
C_{cts}(S,-): \Mod_{\mO_E}\to \Mod_{\mO_E}, 
\end{equation}
which lifts to a nilcomplete functor 
\begin{equation}\label{E:CctsA}
   C_{cts}(S,-): \CA_{\mO_{E}}\to \CA_{\mO_{E}}
\end{equation}
preserving finite limits. If $S$ is profinite, then \eqref{E:CctsM} preserves all colimits and \eqref{E:CctsA} preserves sifted colimits.
\end{lemma}
\begin{proof}
For the first part about modules, using arguments in Lemma \ref{L: cont S to dis}, it reduces to prove surjectivity of $C_{cts}(S,M)\to C_{cts}(S,M'')$ for a surjective map $M\to M''$ of finite $\mO_E$-modules when $S$ is profinite. As every finite $\mO_E$-module is a direct sum of a finite free one and a finite torsion one, this is also clear.
As \eqref{E: indpro on OEmod} is lax monoidal, $C_{cts}(S,A)$ is an $\mO_E$-algebra if $A$ is. The argument for the rest part is the same as in Lemma \ref{L: cont S to dis}.
\end{proof}

\begin{remark}\label{R: strong top}
Note that the fully faithful functor $\Pro(\Sets_f)\to \mathbf{Top}$ by Stone duality induces a fully faithful functor $\Ind\Pro(\Sets_f)\to \mathbf{Top}$. Together with \eqref{E: indpro on OEmod}, this endows every $\mO_E$-module a topology, which we call the ind-$\varpi$-adic topology. Explicitly, for an $\mO_E$-module $V$, this is the finest topology on $V$ such that on every finitely generated submodule $U\subset V$ the subspace topology coincides with the $\varpi$-adic topology.  
In general, the ind-$\varpi$-adic topology on $V$ is stronger than some other convenient topology on $V$. For example, if $V$ is a $\varpi$-adically separated $\mO_E$-module, then the ind-$\varpi$-adic topology on $V$ is usually strictly finer than the $\varpi$-adic topology. Similarly, for an algebraic field extension $F/E$, then ind-$\varpi$-adic topology on $F$ (regarded as an $\mO_E$-module) is strictly finer than the usual $\varpi$-adic topology on $F$ unless $[F:E]<\infty$. Note that if $V$ is an $\mO_{E,r}$-module for some $r$, then the ind-$\varpi$-adic topology on $V$ is discrete. 

There is one warning. Namely, as the functor $\Ind\Pro(\Sets_f)\to \mathbf{Top}$ does not preserve finite product in general, the composed functor $\Mod_{\mO_E}^\heart\to \mathbf{Top}$ is not lax symmetric monoidal so a classical $\mO_E$-algebra $A$ equipped with ind-$\varpi$-adic topology may not be a topological algebra in the usual sense. One way to remedy this problem is by noticing $\Ind\Pro(\Sets_f)\to \mathbf{Top}$ actually factors through $\Ind\Pro(\Sets_f)\to \mathbf{CG}$, where $\mathbf{CG}\subset\mathbf{Top}$ is the full subcategory of compactly generated spaces, and the resulting functor preserves finite products.
\end{remark}

\begin{remark}
As in the case over $\mO_{E,r}$, by regarding $S$ as a discrete set, we have $C_{cts}(S,-)\to C(S,-)$. If $A$ is classical, $C_{cts}(S,A)\to C(S,A)$ is injective. In addition, note that if  $V\in\Mod^{\heart, f.g.}_{\mO_{E}}$, then
$C_{cts}(S,V)=\varprojlim_r C_{cts}(S, V/\varpi^r)$.
\end{remark}

Now given $C_{cts}(S,-)$ from Lemma \ref{L: cont S to ind-l-adic}, we can extend Definition \ref{D: cont rep sp} as follows.
\begin{definition}\label{D: sc rep}
We define the $M$-valued strongly continuous representation space over $\mO_E$ as
\[
\mR^{sc}_{\Ga,M}: \CA_{\mO_E}\to \Spc,\quad A\mapsto \on{Map}_{\CA_{\mO_E}^{\Delta}}\bigl(\mO_E[M^\bullet], C_{cts}(\Ga^\bullet,A)\bigr).
\]
and similar the representation stack $\mR^{sc}_{\Ga,M/H}$ as the quotient of $\mR^{sc}_{\Ga,M}$ by $H$.
\quash{
There is a simplicial diagram similar to \eqref{E: H action} (with $\mR_{\Ga,M}$ replaced by $\mR^{sc}_{\Ga,M}$) and therefore we define the representation stack $\mR^{sc}_{\Ga,M/H}$ as its geometric realization in $\Shv(\CA_{\mO_E})$.
We similarly define pseudorepresentation space over $\on{Spec} \mO_E$ as
\[
\mR^{sc}_{\Ga,M^\bullet/\!\!/H}: \CA_{\mO_E}\to \Spc,\quad A\mapsto \on{Map}_{\CA_{\mO_E}^{\FFM}}(\mO_E[M^\bullet]^H, C_{cts}(\Ga^\bullet,A)).
\]}
\end{definition}

By definition the restriction of $\mR^{sc}_{\Ga,M}$ to $\on{Spf}\mO_E$ is $\mR^c_{\Ga,M}$. As before, there are natural morphisms
\[
\mR^{sc}_{\Ga,M}\to \mR_{\Ga,M}, \quad \mR^{sc}_{\Ga,M/H}\to \mR_{\Ga,M/H}
\]
over $\mO_E$, where $\Ga$ in regarded as an abstract group in $\mR_{\Ga,M}$ and in $\mR_{\Ga,M/H}$. If $A$ is classical, then the induced by $\mR^{sc}_{\Ga,M}(A)\to \mR_{\Ga,M}(A)$ is injective with image consisting of those $\rho:\Ga\to M(A)$ such that for every $f\in\mO_E[M]$, $f\circ\rho: \Ga\to A$ is continuous, where $A$ is equipped with the ind-$\varpi$-adic topology. As the ind-$\varpi$-adic topology on $A$ is in general stronger than other convenient topology (see Remark \ref{R: strong top}), we call such $\rho$ a strongly continuous representation. This justifies our terminology for $\mR^{sc}_{\Ga,M}$.

The following simple observation is important for many discussions in the sequel.

\begin{lemma}\label{L: cont criterion}
Assume that $\Ga$ is profinite and $A$ is a classical $\mO_E$-algebra. Then $\rho: \Ga\to \GL_m(A)$ belongs to $\mR^{sc}_{\Ga,\GL_m}(A)$ if and only if $A^m=\cup_i V_i$ is a union of finite $\mO_E$-modules $V_i$ such that each $V_i$ is a $\Ga$-stable and that the action of $\Ga$ on $V_i$ is continuous.
\end{lemma}
\begin{proof}Indeed, if we denote the $(i,j)$-entry of $\rho(\ga)$ by $a_{ij}(\ga)$, then $\Ga\to A, \ \ga\mapsto a_{ij}(\ga)$ is a map in $\Ind\Pro(\Sets_f)$ and therefore the image is contained in a finitely generated $\mO_E$-submodule of $A$. Therefore, for every $v\in A^m$, $\rho(\Ga)v$ is contained in a $\Ga$-submodule $V$ of $A^m$ that is finite over $\mO_E$, and the action of $\Ga$ on $V$ is continuous. Conversely, if $A^m$ is a union of $\Ga$-submodules $V_i$ as in the lemma, then $a_{ij}: \Ga\to A$ takes values in a finitely generated $\mO_E$-submodule of $A$ and the map resulting map is continuous. Then $\rho$ is strongly continuous. 
\end{proof}

\begin{remark}
Using the above lemma, one can show that ${}^{cl}\mR^{sc}_{\Ga,M}$ is represented by an ind-affine scheme. As we do not make use of this fact, we skip the proof.
\end{remark}

Now for $\rho\in \mR^{sc}(\Ga,M)(A)$, and an algebraic representation $W$ of $M$ on a finite free $\mO_E$-module $W$, we also have $W_\rho=W\otimes A$ equipped a strongly continuous action of $\Ga$ (encoded by the cosimplicial module $C_{cts}(\Ga^\bullet, W_\rho)$ over $C_{cts}(\Ga^\bullet, A)$ as in Remark \ref{R: fiber action} \eqref{R: fiber action-1}). Let  $C^*_{cts}(\Ga,W_\rho)$ be the totalization of $C_{cts}(\Ga^\bullet, W_\rho)$ (in $\Mod_A$). In light of Remark \ref{R: tan compl}, we call this cochain complex the continuous group cohomology of $\Ga$ with coefficients in $W_\rho$. There is similarly the reduced version $\overline{C}^*_{cts}(\Ga,W_\rho)[1]$.
If $A$ is classical, and $\Ga$ is profinite, then by Lemma \ref{L: cont criterion}, we may write $W_\rho=\cup_i V_i$ with each $V_i$ continuous representation of $\Ga$ on a finite $\mO_E$-module. As $C_{cts}(S,-)$ commutes with filtered colimits when $S$ is profinite, we have
\begin{equation}\label{E: colim cont coh}
C^*_{cts}(\Ga,W_\rho)= \varinjlim_i C^*_{cts}(\Ga,V_i),
\end{equation}  
where $C^*_{cts}(\Ga,V_i)$ is the usual continuous group cohomology of $\Ga$ with coefficient in the continuous $\Ga$-module $V_i$. 

The following proposition summarizes the infinitesimal geometry of $\mR^{sc}_{\Ga,M}$, which is a direction generalization of corresponding statements for
$\mR^c_{\Ga,M}$. 
\begin{prop}\label{L: tang Rsc}
The functor $\mR^{sc}_{\Ga,M}: \CA_{\mO_E}\to\Spc$ is nilcomplete and preserves finite products.
Let $\rho\in \mR^{sc}_{\Ga,M}(A)$ with $A$ truncated. 
Then $\bT_\rho\mR^{sc}_{\Ga,M}(A)\cong \overline{C}^*_{cts}(\Ga, \Ad_\rho)[1]$. If $\Ga$ is profinite, then for each $m$ the restriction of $\mR^{sc}_{\Ga,M}$ to a functor ${}_{\leq m}\CA_{\mO_E}\to\Spc$  commutes with filtered colimits. 
\end{prop}

We end this subsection with a result on constancy of residual pseudorepresentations of a strongly continuous representation of a profinite group. So assume that $\Ga$ is profinite and that $(M,H)$ are as in Proposition \ref{P: fin pseudo rep}. 
First, as explained in \cite[4.8]{BHKT}, for every continuous representation $\rho: \Ga\to M(E')$ with $E'/E$ finite extension, the pseudorepresentation of $\Tr\rho$ takes $\mO_{E'}$-value so its reduction mod $\varpi'$ gives a well-defined $\kappa_{E'}$-valued pseudorepresentation of $\Ga$, which we denote by $\overline{\Tr\rho}$. To unify the notion, if $\rho: \Ga\to M(\kappa')$ is continuous with $\kappa'/\kappa_E$ finite, we also denote $\Tr\rho$ by $\overline{\Tr\rho}$.

\begin{lemma}\label{L: constant}
Let $A$ be a finitely generated $\mO_E$-algebra such that $\on{Spec} A$  is connected, and $\rho: \Ga\to M(A)$ a strongly continuous representation. For every point $x\in \on{Spec} A$ whose residue field is either finite over $\kappa_E$ or finite over $E$, 
let $\rho_x$ denote the corresponding continuous representation. Then $x\mapsto \overline{\Tr\rho_x}$ is constant.
\end{lemma}
\begin{proof}
If $\varpi^nA=0$ for some $n$, this follows from Proposition \ref{P: fin pseudo rep}. Now suppose $A[\varpi^{-1}]$ is not empty. Let $\on{Spec} B\subset \on{Spec} A[\varpi^{-1}]$ be a connected component. Let $B_0$ be the subring of $B$ generated by $f(\rho(\ga_1,\ldots,\ga_n))$ for all $n\geq 1$, $f\in E[M^n]^{H}$, and $(\ga_i)\in \Ga^n$. As the $\FFM$-algebra $E[M^\bullet]^H$ is finitely generated (\cite[thm. 9]{Wei}) and $\rho$ is strongly continuous, $B_0$ is finitely generated over $E$. As each closed point of $\on{Spec} B_0$ is indeed defined over some finite extension of $\mO_E$, $B_0$ itself must be finite over $E$. As $\on{Spec} B_0$ is connected, it has a unique point. So $\overline{\Tr(\rho_x)}$ is constant. Finally, clearly if $\rho_x: \Ga\to M(E')$ comes from $\rho_x: \Ga\to M(\mO_{E'})$, then $\overline{\Tr\rho}=\Tr\bar\rho=\overline{\Tr\bar\rho}$. Now the lemma is a combination of the above facts.
\end{proof}


\section{The stack of arithmetic Langlands parameters}\label{S: stack Loc}
In this section, we apply the constructions from the previous section to understand the moduli space of Langlands parameters. The situation is relatively well understood in the local case ($\ell\neq p$), which will be discussed in \S \ref{S: Loc}-\ref{SS: SPI}. Much less can be said in the global field case; however, we are still able to construct the moduli space in the global function field case in \S \ref{SS: glob par}.

First recall the $C$-group introduced by Buzzard-Gee \cite{BG}, following the construction in \cite[\S 1.1]{Z20}. Let $G$ be a connected reductive group over a field $F$. Let $\Ga_F$ denote the Galois group of $F$, and $\hat{G}$ the dual group of $G$, regarded as a group scheme over $\bZ$. It is equipped with a pinning $(\hat{B},\hat{T},\hat{e})$, and an action of $\Ga_F$ via the homomorphism $\xi:\Ga_F\to \Aut(\hat{G},\hat{B},\hat{T},\hat{e})$. Let $\hat{G}_\ad$ be the adjoint group of $\hat{G}$, and $\rho_\ad: \bG_m\to \hat{G}_\ad$ the cocharacter given by the half sum of positive coroots of $\hat{G}$. Let $\pr:\Ga_F\to \Ga_{\widetilde{F}/F}$ be the finite quotient of $\Ga_F$ by $\ker\xi$. Let 
$${}^cG:=\hat{G}\rtimes(\bG_m\times\Ga_{\widetilde{F}/F})$$ 
be the $C$-group of $G$, regarded as a group scheme over $\bZ$, where $\bG_m$ acts on $\hat{G}$ via the homomorphism $\bG_m\xrightarrow{\rho_\ad}\hat{G}_\ad\subset\Aut(\hat G)$, and $\Ga_{\widetilde F/F}$ acts via $\xi$. Let $d: {}^cG\to \bG_m\times\Ga_{\widetilde{F}/F}$
denote the natural projection. 

\begin{remark}\label{R: cG and LG}
If $F$ is a local field with residue field $\bF_q$ or a global function field with $\bF_q$ as the field of its constants, upon a choice of $q^{1/2}$, ${}^cG$ and ${}^LG\times\bG_m$ are isomorphic over $\bZ[q^{\pm 1/2}]$, where ${}^LG=\hat{G}\rtimes\Ga_{\widetilde F/F}$ is the usual Langlands dual group of $G$. 
Therefore one can replace ${}^cG$ by ${}^LG$ in most discussions below (with small modifications). However, we prefer to use the $C$-group rather than the $L$-group in our formulation. On the one hand, it is more canonical (and will be necessary when we consider local-global compatibility for number fields). 
On the other hand, using the $L$-group does not seem to simplify the formulation too much when $\widetilde F\neq F$.

Conversely, if the cocharacter $\rho_\ad$ can be lifted to a $\Ga_{\widetilde F/F}$-invariant cocharacter $\tilde\rho:\bG_m\to \hat{G}$, then one can also use ${}^LG$ instead of ${}^cG$ in the discussions below. For example, this is the case if $G=\GL_n$ or the odd unitary group. See \cite[Example 2]{Z20}.
\end{remark}

\subsection{The stack of local Langlands parameters}\label{S: Loc}
In the next two subsections, we discuss the stack of local Langlands parameters over a base in which $p$ is invertible, for a connected reductive group $G$ over a local field $F$ of residue characteristic $p$. Some results in this subsection are also obtained by Dat-Helm-Kurinczuk-Moss \cite{DHKM}, and independently by Fargues-Scholze \cite{FS}, sometimes by different methods.

Let $\kappa_F$ denote the residue field with $\sharp\kappa_F=p^r$. Let $\Ga_F$ be the Galois group of $F$. Let $P_F\subset I_F\subset \Ga_F$ be the wild inertia and the inertia, corresponding to Galois extensions $F^{t}\supset F^{\on{ur}}\supset F$. 
Recall that the tame inertia
$$I^t_F:=I_F/P_F\cong \prod_{\ell\neq p}\bZ_{\ell}(1)=:\widehat{\bZ}^p(1)$$ 
is prime-to-$p$, while $P_F$ is a pro-$p$-group. Then $\Ga_{F}^t:=\Ga_{F^t}\cong \Ga_F/P_F$ fits into the following short exact sequence
\[
   1\to I_F^t\to \Ga_F^t\to \widehat{\bZ}\to 1.
\]

Let $W_F\subset \Ga_F$ be the Weil group of $F$. We normalize the map 
\begin{equation}\label{E: Weilnorm}
\|\cdot\|: W_F\to \bZ
\end{equation}
so it is trivial on $I_F$ and $\|\Phi\|=1$ for a lifting of the \emph{arithmetic} Frobenius. Similarly, there is the tame Weil group $W_F^t:= W_F/P_F$, which is an extension of $\bZ$ by $I_F^t$. We let
\[
\chi=(q^{-\|\cdot\|},\pr): W_F\to \bZ[1/p]^\times\times\Ga_{\widetilde F/F}.
\]
Note that $q^{-\|\cdot\|}$ is the restriction of the \emph{inverse} cyclotomic character of $\Ga_F$ to $W_F$.

There are several versions of the moduli of local Langlands parameters.

First, we fix a prime $\ell\neq p$. There is the moduli $\mR^{c}_{W_F,{}^cG}$ of continuous representations of $W_F$ over $\Spf\bZ_\ell$ (Definition \ref{D: cont rep sp}). The homomorphism $d:{}^cG\to \bG_m\times\Ga_{\widetilde F/F}$ induces a morphism 
$\mR^{c}_{W_F,{}^cG}\to \mR^{c}_{W_F,\bG_m\times \Ga_{\widetilde F/F}}$. 
We may regard $\chi$ as a $\Spf \bZ_\ell$-point of $\mR^{c}_{W_F,\bG_m\times\Ga_{\widetilde F/F}}$  
and define
\begin{equation}\label{E: Locloc}
  \Loc^{\wedge,\Box}_{{}^cG, F}:=\mR^{c}_{W_F,{}^cG}\times_{\mR^{c}_{W_F,\bG_m\times\Ga_{\widetilde F/F}}}\bigl\{\chi\bigr\},\quad \Loc^{\wedge}_{{}^cG,F}=\Loc_{{}^cG,F}^{\wedge,\Box}/\hat{G}_\ell^\wedge,
\end{equation}
where $\hat{G}_\ell^\wedge$ is the $\ell$-adic completion of $\hat{G}$.
As $\Ga_F$ is the profinite completion of $W_F$, a slight variant of Lemma \ref{L: formal closed} implies that the completion of $\Loc^{\wedge,\Box}_{{}^cG, F}$ at a closed point corresponding to $\bar\rho: \Ga_F\to {}^cG(\kappa)$ is the space $\on{Def}_{\bar\rho}^{\Box,\chi}$ of framed deformations $\rho$ of $\bar\rho$ such that $d\circ\rho=\chi$.

Recall that $\mR^c_{W_F,{}^cG}$ admits an extension $\mR^{sc}_{W_F,{}^cG}$ to $\on{Spec} \bZ_\ell$ classifying strongly continuous representations of $W_F$ (Definition \ref{D: sc rep}). Therefore we may also extend \eqref{E: Locloc} to $\bZ_\ell$ as
 \begin{equation}\label{E: Locloc2}
  \Loc^{\Box}_{{}^cG, F}:=\mR^{sc}_{W_F,{}^cG}\times_{\mR^{sc}_{W_F,\bG_m\times\Ga_{\widetilde F/F}}}\bigl\{\chi\bigr\},\quad \Loc_{{}^cG,F}=\Loc_{{}^cG,F}^{\Box}/\hat{G}_{\bZ_\ell}.
\end{equation}

\begin{remark}
The analogue of $\Loc^{\wedge}_{{}^cG,F}$ over $\on{Spf}\bZ_p$ probably should the Emerton-Gee stack \cite{EG} (whose definition is much more involved). However, the analogue of $\Loc_{{}^cG,F}$ over $\on{Spec}\bZ_p$ would be more subtle.
\end{remark}

\begin{remark}\label{R: mod l inert type}
We note that the decomposition \eqref{E: dec cont} for $\Loc^{\wedge}_{{}^cG,F}$ is the decomposition according to the mod $\ell$ inertial types. Indeed, by \cite[4.5]{BHKT}, $\Theta$ from  \eqref{E: dec cont} exactly corresponds to mod $\ell$ completely reducible representation of $I_F$ (i.e. mod $\ell$ inertia type). 
\end{remark}

\begin{remark}
By Remark \ref{R: cG and LG}, $ \Loc^{\Box}_{{}^cG, F}\cong \mR^{sc}_{W_F,{}^LG}\times_{\mR^{sc}_{W_F,\Ga_{\widetilde F/F}}}\bigl\{\pr\bigr\}$ over $\bZ_\ell[q^{\pm 1/2}]$. If $G=\GL_m$, then $\Loc^{\Box}_{{}^cG, F}\cong \mR^{sc}_{W_F,\GL_m}$.
\end{remark}

Second, there is the stack 
$$\Loc^{\on{WD}}_{{}^cG,F}:=\Loc^{\on{WD},\Box}_{{}^cG,F}/\hat{G}$$ 
of Weil-Deligne representations of $F$ as an algebraic stack over $\bQ$ (see e.g. \cite[2.1]{BG}). Here $\Loc^{\on{WD},\Box}_{{}^cG,F}$ is the presheaf over $\CA_\bQ^{\heart}$ defined as follows. Let $\hat\mN_\bQ\subset \Lie \hat{G}_\bQ$ denote the nilpotent cone of $\hat{G}_\bQ$.
For a $\bQ$-algebra $A$, we equip $ {}^cG(A)$ with the discrete topology, and let
\[
 \Loc^{\on{WD},\Box}_{{}^cG,F}(A)=\Bigl\{(r,X)\mid r: W_{F}\to {}^cG(A) \mbox{ continuous}, X \in \hat{\mN}_\bQ(A)\mid d\circ \rho=\chi \ , \Ad_{r(\ga)}X=q^{\|\ga\|}X\Bigr\}.
\]
We note that there is a natural $\bG_m$ action on $\Loc^{\on{WD},\Box}_{{}^cG,F}$, by scaling the nilpotent element $X$.

One sees that 
$$\Loc^{\on{WD},\Box}_{{}^cG,F}=\varinjlim_L \Loc^{\on{WD},\Box}_{{}^cG,L/F},$$ 
where $L$ range over all finite extensions of $F^{\ur}\widetilde F$ that are Galois over $F$, and $\Loc^{\on{WD},\Box}_{{}^cG,L/F}$ is the (open and closed) subfunctor of $ \Loc^{\on{WD},\Box}_{{}^cG,F}$ consisting of those $(r,X)$ such that $r$ factors through $W_F/W_L\to {}^cG(A)$. 

As $W_F/W_L$ is a finitely generated group, namely an extension of $\bZ$ by $\Ga_{L/F^{\ur}}$, the functor $ \Loc^{\on{WD},\Box}_{{}^cG,L/F}$ is represented by an affine scheme of finite type over $\bQ$. 
Therefore, $\Loc^{\on{WD},\Box}_{{}^cG,F}$ and $\Loc^{\on{WD}}_{{}^cG,F}$
are (ind)-representable.

\begin{remark}
Here we only define $\Loc^{\on{WD},\Box}_{{}^cG,F}$ as a classical (ind-)scheme as this is what we need in the sequel. 
Of course, one can define it as a derived scheme in a natural way. But it turns out the derived structure will be trivial. In fact, we have such kind of discussions in the sequel when we discuss integral versions of $\Loc^{\on{WD},\Box}_{{}^cG,F}$.
\end{remark}

Finally, we can glue the above two moduli spaces into algebraic stacks over $\bZ[1/p]$, once we make a choice.
Recall the following basic facts (\cite{Iw}).
\begin{itemize}
\item There exists a topological splitting $\Ga_F^t\to\Ga_F$ so that $\Ga_F\cong P_F\rtimes \Ga_F^t$. 
\item Let $\Ga_q=\langle \tau,\sigma\rangle$ be as in \eqref{E: tame Galois}. Then there exists an embedding
\begin{equation}\label{E: spl}
\iota: \Ga_q\to \Ga^t_F
\end{equation}
such that $\iota(\tau)$ is a generator of the tame inertia, and that $\iota(\sigma)$ is a lifting of the Frobenius. Then $\iota$ induces an isomorphism of the profinite completion of the projection $\Ga_q\to \bZ$ with $\Ga^t_F\to \widehat{\bZ}$.
\end{itemize}
For a choice of $\iota$, we write $\Ga_{F,\iota}$ be the pullback of $\Ga_F$ via $\iota$ (we will not consider the topology on these groups). Then we have inclusions $\Ga_{F,\iota}\to W_F\to \Ga_F$. By abuse of notations, we still use $\iota$ to denote both inclusions $\Ga_{F,\iota}\subset W_F$ and $\Ga_{F,\iota}\subset \Ga_F$. 
We have the short exact sequence
\[
1\to P_F\to \Ga_{F,\iota}\to\Ga_q\to 1.
\]
The homomorphism $\|\cdot\|$ from \eqref{E: Weilnorm} restricts to $\Ga_{F,\iota}$. 
Similarly, if $L$ is finite over $F^t$ and is Galois over $F$, let $\Ga_{L/F,\iota}$ be the pullback of $\Ga_{L/F}$ (the Galois group for $L/F$) along $\iota$. We have the short exact sequence
 \[
 1\to Q_L:=\Ga_{L/F^t}\to \Ga_{L/F,\iota}\to \Ga_q\to 1,
 \]
where $Q_L$ is a finite $p$-group.

\begin{remark}\label{R: choice of gen}
(1) Note that for different choices $\iota_1,\iota_2$, there is in general no isomorphism between $\Ga_{F,\iota_1}$ and $\Ga_{F,\iota_2}$ that restricts to the identity map of $P_F$.

(2) All possible choices of $\iota$ as in \eqref{E: spl} form a torsor under $\Aut^0$, the group of continuous automorphisms of $\Ga^t_F$ that restricts to an automorphism of $I_F^t$ and induces the identity map on $\Ga_F^t/I_F^t$. The group $\Aut^0$ itself is an extension of $\widehat{\bZ}^{p,\times}:=\prod_{\ell\neq p} \bZ_\ell^\times$ by $\widehat{\bZ}^p(1)$. 
\end{remark}

Now we choose an $\iota$ as in \eqref{E: spl}.
If $L/F^t\widetilde F$ is finite such that $L/F$ is Galois, then the homomorphism $\chi\iota:\Ga_{F,\iota}\to \bZ[1/p]^\times\times\Ga_{\widetilde F/F}$ factors through $\Ga_{L/F,\iota}\to \bZ[1/p]^\times\times\Ga_{\widetilde F/F}$, denoted by the same notation, which can be regarded as a $\bZ[1/p]$-point of $\mR_{\Ga_{L/F,\iota},\bG_m\times\Ga_{\widetilde F/F}}$.
We define the scheme 
\[
   \Loc^\Box_{{}^cG,L/F,\iota}:= \mR_{\Ga_{L/F,\iota}, {}^cG} \times_{\mR_{\Ga_{L/F,\iota},\bG_m\times\Ga_{\widetilde F/F}}}\bigl\{\chi\iota\bigr\}.
\]
Explicitly, for a classical $\bZ[1/p]$-algebra $A$,
\[
   \Loc^\Box_{{}^cG,L/F,\iota}(A):=\Bigl\{\rho: \Ga_{L/F,\iota}\to {}^cG(A) \mid d\circ \rho=\chi\iota: \Ga_{L/F,\iota}\to \bG_m\times \Ga_{\tilde F/F}\Bigr\}.
\]
Now, we define the scheme of framed $\iota$-local Langlands parameters as
\[
   \Loc^\Box_{{}^cG,F,\iota}:=\underrightarrow{\lim}_{L} \Loc^\Box_{{}^cG,L/F,\iota}.
\]
Again by a (slight variant of) Lemma \ref{L: formal closed}, its formal completion at $\bar\rho$ is the framed deformation space $\on{Def}_{\bar\rho}^{\Box,\chi}$.

\begin{prop}\label{P:flat lcc}
The derived ind-scheme $\Loc^\Box_{{}^cG,F,\iota}$ is a disjoint union of classical affine schemes of finite type and flat over $\bZ[1/p]$. It is equidimensional of dimension $=\dim \hat{G}$, and is a local complete intersection with trivial dualizing complex.
\end{prop}
\begin{proof}
We apply Proposition \ref{P: tame and wild stack} to $\Ga=\Ga_{L/F,\iota}\simeq Q_L\rtimes\Ga_q$, and $M={}^cG$ and $M=\bG_m\times\Ga_{\widetilde F/F}$. We have the projection
$\mR_{\Ga_{L/F,\iota},{}^cG} \to \mR_{\Ga_{L/F,\iota},\bG_m\times \Ga_{\widetilde{F}/F}}$. Taking the fiber over $\chi\iota$ shows that $\Loc^\Box_{{}^cG,L/F,\iota}$ is a classical affine scheme of finite type and flat over $\bZ[1/p]$, is equidimensional of dimension $=\dim \hat{G}$, and is a local complete intersection. In addition, clearly if $L'/L$ is finite such that $L'/F$ is Galois, then $\Loc_{{}^cG,L/F,\iota}^{\Box}\subset\Loc_{{}^cG,L'/F,\iota}^{\Box}$ is an open and closed embedding. The proposition follows.
\end{proof}

Now we can define the stack of $\iota$-local Langlands parameters as
\[
   \Loc_{{}^cG,F,\iota}=\Loc^\Box_{{}^cG,F,\iota}/\hat{G}. 
\]
It is the union of open and closed substacks $\Loc_{{}^cG,L/F,\iota}=\Loc_{{}^cG,L/F,\iota}^\Box/\hat{G}$, each of which is of finite presentation over $\bZ[1/p]$. 

\begin{remark}\label{R: Coh cat}
There are two ways to view $\Loc_{{}^cG,F,\iota}$ (and $\Loc_{{}^cG,F}^{\on{WD}}$) as an algebraic stack. The first is by viewing it as a stack locally of finite type, and the second is by viewing it as an ind-finite type stack. 
We will adapt the second point of view. So its ring of regular functions (see \eqref{E: st center} below) is regarded as pro-algebra.
In addition, later on we will consider the category $\Coh(\Loc_{{}^cG,F,\iota})$ of coherent sheaves on $\Loc_{{}^cG,F,\iota}$. According our definition, these are complexes of quasi-coherent sheaves that only support on \emph{finitely} connected components of $\Loc_{{}^cG,F,\iota}$, and are coherent complexes on these components. In particular, the structure sheaf of $\Loc_{{}^cG,F,\iota}$ itself is not regarded as a coherent sheaf. It lies in the ind-completion $\Ind\Coh(\Loc_{{}^cG,F,\iota})$ of $\Coh(\Loc_{{}^cG,F,\iota})$.
\end{remark}

We have discussed three versions of moduli of local Langlands parameters: one over $\bZ_\ell$, one over $\bQ$ and one over $\bZ[1/p]$.
Our next task is to relate them and to analyze how $\Loc_{{}^cG,L/F,\iota}$ depends on the choice of $\iota$.

\begin{lemma}\label{L: l-comp} 
The map $\iota: \Ga_{F,\iota}\to W_F$ induces a natural isomorphism\footnote{\label{F: Zl}We originally only considered such isomorphism over $\on{Spf}\bZ_\ell$. We thank P. Scholze to point out it holds over $\on{Spec}\bZ_\ell$.} 
$$\phi_{\iota,\ell}: \Loc^{\Box}_{{}^cG,F}\xrightarrow{\cong} \Loc_{{}^cG, F,\iota}^\Box\otimes\bZ_\ell.$$ 
\end{lemma}
\begin{proof}
Let us first prove this at the level of classical moduli problems. Then $\phi_{\iota,\ell}$ sends a strongly continuous representation $W_F\to {}^cG(A)$ to its restriction to $\Ga_{F,\iota}$. To show it is an isomorphism, it is enough to show that every $\rho: \Ga_{L/F,\iota}\to {}^cG(A)$ extends to a strongly continuous representation of $W_F\to {}^cG(A)$.

As above, we write $\Ga_{F,\iota}\simeq P_F\rtimes \Ga_q$ by choosing a topological splitting $\Ga_F^t\to \Ga_F$. Then there is some $N$ (which might depend on the choice of the topological splitting), such that $\rho(\tau^N)\in \hat{\mU}(A)$, where $\hat\mU\subset\hat{G}$ is the unipotent variety of $\hat{G}$. Indeed, recall that the restriction $\langle\tau \rangle\subset \Ga_q$ induces $\Loc_{{}^cG, F,\iota}^\Box\to {}^cG^{[q]}$ (see the proof of Proposition \ref{P: tame stack}). So it is enough to show that there is some $N$ such that the $N$th power map ${}^cG\to {}^cG, \  g\mapsto g^N$ sends $ {}^cG^{[q]}$ to $\hat\mU$. By choosing a faithful representation ${}^cG\to \GL_m$, it is enough to show a similar statement for $\GL_m$. This amounts to show that for $X\in \GL_m$, if $\on{Char}(X^q)=\on{Char}(X)$, then for some power $X^N$, $\on{Char}(X^N)=(t-1)^m$. But this is standard. 

Now to show that $\rho$ extends, it is enough to prove that every element $X\in\hat\mU(A)$ extends to a continuous map $\bZ_\ell\to\hat\mU(A), \ a\mapsto X^a$, when $A$ is equipped with the ind-$\ell$-adic topology. Indeed, again we reduce to the $\GL_m$-case. If $\on{Char}(X)=(t-1)^m$, then for every $v\in A^m$, $\{X^iv\}_{i\geq 0}$ is contained in a finite $\bZ_\ell$-module. Then we use Lemma \ref{L: cont criterion} to conclude.

Next we show that $\phi_{\iota,\ell}$ is an isomorphism at the derived level.  
We use Proposition \ref{L: tang Rsc} and the argument as in Proposition \ref{P: trun value} to reduce to show that $C^*_{cts}(W_F,\Ad^0_\rho\otimes V)\to C^*(\Ga_{F,\iota},\Ad^0_\rho\otimes V)$ is an isomorphism, for every classical $A$, every ordinary $A$-module $V$, and every strongly continuous homomorphism $\rho: W_F\to {}^cG(A)$. Here $\Ad^0$ is the adjoint representation of ${}^cG$ on the Lie algebra of $\hat{G}$. Then it reduces to show that $C^*_{cts}(I_F^t, (\Ad^0_\rho)^{P_F})\to C^*(\bZ[1/p],(\Ad^0_\rho)^{P_F})$ is an isomorphism. By Lemma \ref{L: cont criterion}, it further reduces to show $C^*_{cts}(I_F^t, V)\to C^*(\bZ[1/p],V)$ is an isomorphism if $V$ is a continuous representation of $I_F^t$ on a finite $\bZ_\ell$-module. But this last claim is not difficult.
\end{proof}

On the other hand, we have the following.

\begin{lemma}\label{L: gen}
The map $\Ga_{F,\iota}\to W_F$ induces a natural isomorphism 
\[
\phi_{\iota,\bQ}: \Loc^{\on{WD},\Box}_{{}^cG,F}\xrightarrow{\cong} \Loc_{{}^cG, F,\iota}^\Box\otimes\bQ.
\]
\end{lemma}
\begin{proof}
The morphism $\phi_{\iota,\bQ}$ is given by send $(r,X)\in  \Loc^{\on{WD},\Box}_{{}^cG,F}(A)$ to
$$\rho: \Ga_{F,\iota}\to {}^cG(A), \quad \rho(\ga)=r(\iota\ga)\exp({|\ga|_{\iota}}X),$$ 
where $|\ga|_{\iota}\in\bZ[1/p]$ such that the image of $\ga\in \Ga_{F,\iota}$ in $\Ga_q$ can be written as $\sigma^{\|\ga\|}\tau^{|\ga|_\iota}$, and
\[
\exp: \hat\mN_\bQ\cong \hat\mU_{\bQ}
\]
is the usual exponential map inducing isomorphisms between the nilpotent variety and the unipotent variety of $\hat{G}$ (over $\bQ$). Let $\log: \hat\mU_\bQ\cong \hat\mN_\bQ$ be its inverse.

Next we define the morphism in another direction. Let $\rho: \Ga_{F,\iota}\to {}^cG(A)$ be an $A$-point of $\Loc^{\Box}_{{}^cG,F,\iota}$. We assume that it factors through some $\Ga_{L/F,\iota}$. 
Note that there is some $m$ such that the image of $\tau^m\in\Ga_q$ in $\Ga_{F,\iota}$ is independent of the choice of the splitting $\Ga_q\to \Ga_{L/F,\iota}$. In addition, by replacing $m$ by a multiple, we may assume that
$\rho(\tau)^m\in \hat{\mU}_\bQ(A)$. Then we take $X=\frac{1}{m}\log(\rho(\tau)^m)$. Clearly $X$ is independent of the choice of $m$.
Then we obtain a well-defined homomorphism 
$$r: \Ga_{F,\iota}\to {}^cG(A), \quad r(\ga)=\rho(\ga) \exp(-|\ga|_\iota X).$$ As $r(\tau^m)=1$, we may regard $r$ as a \emph{continuous} map $W_{L/F}\to {}^cG(A)$, where $A$ is equipped with the discrete topology. Then $\rho\mapsto (r,X)$ gives the inverse of $\phi_{\iota,\bQ}$.
\end{proof}

Before continuing, we observe that as a byproduct we obtain the following.
\begin{cor}\label{C: reduce}
The scheme $\Loc^{\Box}_{{}^cG,F,\iota}$ is reduced. 
\end{cor}
Note that the fiber of $\Loc^{\Box}_{{}^cG,F,\iota}$ over some prime $\ell$ could be non-reduced. 
\begin{proof}
As $\Loc^{\Box}_{{}^cG,F,\iota}$ is a local complete intersection flat over $\bZ[1/p]$ (Proposition \ref{P:flat lcc}), the statement follows from the generic smoothness of $\Loc^{\Box}_{{}^cG,F,\iota}\otimes\bQ\cong \Loc^{\on{WD},\Box}_{{}^cG,F}$ as proved in \cite{BG}, and Serre’s criterion S1.
\end{proof}

Now we can compare $\Loc^{\Box}_{{}^cG,F,\iota}$ for different choices of $\iota$. Let $\iota_1,\iota_2: \Ga_q\to \Ga_F^t$ be two embeddings. Recall from Remark \ref{R: choice of gen} that there is $\vartheta\in \Aut^0$ such that $\iota_2=\vartheta\iota_1:\Ga_q\to \Ga_F^t$, and there is a projection $\Aut^0\to \bZ_\ell^\times$. Let $\bar{\vartheta}\in \bZ_\ell^\times$ denote the image of $\vartheta$. As $\bG_m$ acts on $ \Loc^{\on{WD},\Box}_{{}^cG,F}$ by scaling the nilpotent element, $\bar{\vartheta}$, regarded as an element in $\bG_m(\bQ_\ell)$, acts on $\Loc^{\on{WD},\Box}_{{}^cG,F}\otimes\bQ_\ell$. 
\begin{prop}\label{P:comp iota12}
There is a unique isomorphism $\vartheta=\vartheta_{\iota_1,\iota_2}: \Loc_{{}^cG, F,\iota_1}^\Box\otimes\bZ_\ell\cong \Loc_{{}^cG, F,\iota_2}^\Box\otimes\bZ_\ell$ of schemes over $\bZ_\ell$ making the following diagram commutative
\[
\xymatrix{
\Loc^{\Box}_{{}^cG,F}\ar@{=}[d]\ar^-{\phi_{\iota_1,\ell}}[r] &\Loc_{{}^cG, F,\iota_1}^\Box\otimes\bZ_\ell\ar^{\vartheta}[d] & \ar_-{\phi_{\iota_1,\bQ_\ell}}[l] \Loc^{\on{WD},\Box}_{{}^cG,F}\otimes\bQ_\ell\ar^{\bar{\vartheta}}[d]\\
\Loc^{\Box}_{{}^cG,F}                \ar^-{\phi_{\iota_2,\ell}}[r] &\Loc_{{}^cG, F,\iota_2}^\Box\otimes\bZ_\ell                                                           & \ar_-{\phi_{\iota_2,\bQ_\ell}}[l] \Loc^{\on{WD},\Box}_{{}^cG,F}\otimes\bQ_\ell
}\]
\end{prop}
\begin{proof}As $\phi_{\iota_i,\ell}$ is isomorphism and therefore there is a unique $\vartheta$ compatible with $\phi_{\iota_i,\ell}$s. By tracing the construction, we see that
$\vartheta\circ \phi_{\iota_1,\bQ_\ell} =\phi_{\iota_2,\bQ_\ell}\circ \bar\vartheta$. 
\end{proof}

\begin{cor}
The ring of regular functions on $\Loc_{{}^cG,F,\iota}$
\begin{equation}\label{E: st center}
Z_{{}^cG,F}:=H^0\Gamma(\Loc_{{}^cG,F,\iota},\mO)
\end{equation}
is independent of the choice of $\iota$ up to canonical isomorphism (so we can omit the subscript $\iota$).  
\end{cor}
Recall that according to our convention, $\Gamma(\Loc_{{}^cG,F,\iota},-)$ standards for the derived functor, while $H^0\Gamma$ denotes its zeroth cohomology.
\begin{proof}
Indeed, the $\bG_m$-action on $\Loc^{\on{WD}}_{{}^cG,F}$ (by scaling the nilpotent element) induces the trivial action on its ring of regular functions. Therefore $\bar\vartheta$ in Proposition \ref{P:comp iota12} induces the identity map after taking $\hat{G}$-invariants.
\end{proof}
This algebra is usually called the stable center of $G^*$ (the quasi-split inner form of $G$), at least when base changed to $\bC$ (see \cite{Ha}). It admits an idempotent decomposition indexed by connected components of $\Loc_{{}^cG,F,\iota}$. For a finite union of connected components $D$, let $Z_{{}^cG,F,D}$ denote the corresponding ring of regular functions, which is a finitely generated $\Lambda$-algebra. If $D=\Loc_{{}^cG,L/F,\iota}$, we denote $Z_{{}^cG,F,D}$ by $Z_{{}^cG,L/F}$. 

As taking $\hat{G}$-invariants on $\hat{G}$-representations over $\Lambda$ is not exact if $\Lambda$ is not a field of characteristic zero, a priori the higher cohomology $H^i\Gamma(\Loc_{{}^cG,F,\iota},\mO)$ may not vanish for $i>0$. 
But Conjecture \ref{Conj: coh sheaf} suggests this is not the case. In fact, we make the following conjecture\footnote{In fact this conjecture has been proved in \cite{FS} when $\ell$ is not too small.}.

\begin{conjecture}\label{Conj: vanishing O}
For every $i\geq 1$, $H^i\Gamma(\Loc_{{}^cG,F,\iota},\mO)=0$.
\end{conjecture}

\begin{remark}\label{R: point coarse moduli}
Let $\kappa$ be an algebraically closed field over $\bZ[1/p]$.
By \cite[11.7]{L} and \cite[4.5]{BHKT}, and Remark \ref{R: kpt of pseudo}, there is a bijection between $\kappa$-points of $Z_{{}^cG,F}$ and $\hat{G}(\kappa)$-conjugacy classes of homomorphisms $\rho: \Ga_{F,\iota}\to {}^cG(\kappa)$ satisfying
\begin{itemize}
\item $d\circ \rho=\chi$;
\item $\rho$ factors through $\Ga_{L/F,\iota}\to {}^cG(\kappa)$ for some finite extension $L/F^t\widetilde F$;
\item $\rho$ is completely reducible (in the sense of \cite[3.5]{BHKT}).
\end{itemize}
Giving Conjecture \ref{Conj: vanishing O}, one may further conjecture that a slight variant of \eqref{E: coarse to pseudo} in the current setting is an isomorphism (after taking $\pi_0$).
\end{remark}

At the end of this subsection, we discuss the behavior of these stacks under tensor induction.

Let $F'/F$ be a finite separable extension. Let $G'$ be a connected reductive group over $F'$ and $G=\on{Res}_{F'/F}G'$. As explained in \cite[5.1,4.1]{Bo}, the dual group $\hat{G}$ of $G$ equipped with an action of $\Ga_F$ is canonically isomorphic to the tensor induction $\on{Ind}_{\Ga_{F'}}^{\Ga_F}\hat{G'}$, which by definition is the space of all $\Ga_{F'}$-equivariant maps from $\Ga_F$ to $\hat{G'}$. There is the $\Ga_{F'}$-equivariant maps (\cite[4.1]{Bo})
$$
   \hat{G'}\xrightarrow{i}  \hat{G}\xrightarrow{\on{ev}_e} \hat{G'}
$$ 
whose composition is the identity, where the first map sends $g$ to the unique map $f:\Ga_F\to \hat{G'}$ that is supported on $\Ga_{F'}$ and such that $f(1)=g$, and the second map sends $f: \Ga_F\to \hat{G'}$ to $f(e)$. Then there is a canonical homomorphism ${}^c(G')\to {}^cG$ compatible with $i$ and with $\bG_m\times \Ga_{\widetilde F'/F'}\to \bG_m\times \Ga_{\widetilde F/F}$ as in \cite[5.1 (5)]{Bo}. A choice of $\iota: \Ga_q\to \Ga_F^t$ gives $\iota': \Ga_{q'}\to \Ga_{F'}^t$. Note that $\Ind_{\Ga_{F',\iota'}}^{\Ga_{F,\iota}}\hat{G'}=\Ind_{\Ga_{F'}}^{\Ga_F}\hat{G'}$.
\begin{lemma}\label{L: Shapiro}
There is the canonical isomorphism
\[
\Loc_{{}^cG,F,\iota}\cong \Loc_{{}^cG',F',\iota'},\quad \rho\mapsto \on{ev}_e\circ (\rho|_{\Ga_{F',\iota'}}).
\]
\end{lemma}
\begin{proof}
This is a geometric version of the Shapiro's lemma. We generalize the argument from \cite[4.1.2]{XZ2} to explicitly construct the inverse map.  For simplicity, we write $\Ga'=\Ga_{F',\iota'}$ and $\Ga$ for $\Ga_{F,\iota}$.
Let $s:\Ga'\backslash \Ga\to \Ga$ be a section (sending the unit coset to $1\in\Ga$) of the projection $\Ga\to\Ga'\backslash\Ga,\ \ga\mapsto \bar{\ga}$. Then we have the map 
\[
\Xi_s:\Ga\to \Ga',\quad \Xi_s(\ga):= \ga s_{\bar\ga}^{-1}.
\]
Note  that $\Xi_s(\ga'\ga)=\ga'\Xi_s(\ga)$ for $\ga'\in\Ga'$.
In addition, let 
\[
\Delta_s: \hat{G'}\to\hat{G},\quad \Delta_s(g):\Ga\to \hat{G'},\quad \Delta_s(g)(\delta)=\chi(\Xi_s(\delta))(g).
\] 

Now we construct a morphism $I_s:\Loc_{{}^cG',F',\iota'}^{\Box}\to \Loc_{{}^cG,F,\iota}^{\Box}$ as follows.
Let $\rho'=(\varphi',\chi): \Ga'\to {}^c(G')(A)=\hat{G'}(A)\rtimes (A^\times\times\Ga_{\widetilde F'/F'})$. We define 
$I_s(\rho')=(\varphi,\chi): \Ga\to {}^cG(A)=\hat{G}(A)\rtimes (A^\times\times \Ga_{\widetilde F/F})$, where
\[
\varphi(\ga): \Ga\to \hat{G'}(A),\quad \varphi(\ga)(\delta)=\varphi'(\Xi_s(\delta))^{-1}\varphi'(\Xi_s(\delta\ga)).
\]
One verifies that 
\begin{itemize}
\item $\varphi(\ga'\ga)=\chi(\ga')(\varphi(\ga))$ for $\ga'\in\Ga'$ so $\varphi(\ga)\in \hat{G}(A)$;
\item $I_s(\rho')$ is a homomorphism $\Ga\to {}^cG(A)$, and that $\on{ev}_e\circ (I_s(\rho')|_{\Ga'})=\rho'$;
\item $I_s(g^{-1}\rho'g)=\Delta_s(g)^{-1}I_s(\rho')\Delta_s(g)$ for any $g\in \hat{G'}(A)$.
\end{itemize}
Therefore we construct a morphism
$\Loc_{{}^cG',F',\iota'}\to \Loc_{{}^cG,F,\iota}$ inverse to the map in the lemma.
\end{proof}

\subsection{Duality for Tori and symmetries of $\Coh(\Loc_{{}^cG,F,\iota})$}\label{S: tori} 
Let us first we look into the stack $\Loc_{{}^cG,F,\iota}$ more carefully when $G=T$ is a torus over $F$.
It is not difficult to see from the proof of Proposition \ref{P:comp iota12} that $\Loc_{{}^cT,F,\iota}$ is independent of the choice of $\iota$. But in fact one can describe  $\Loc_{{}^cT,F,\iota}$ explicitly as follows.
Let $\widetilde F/F$ be the splitting field of $T$. By the local class field theory, there is the short exact sequence
\[1\to \widetilde{F}^\times\to W_{\widetilde{F}/F}\to \Ga_{\widetilde F/F}\to 1,\]
where $W_{\tilde F/F}$ is the Weil group of the extension $\tilde F/F$. Let $U^{(n)}$ be the $n$th unit group of $\widetilde F$ (so $U^{(0)}=\mO_{\widetilde F}^\times$ and $U^{(n)}=1+\frakm^n_{\widetilde F}$ for $n\geq 1$), and write $W^{(n)}=W_{\widetilde{F}/F}/U^{(n)}$. Then there is a natural isomorphism
\[
\Loc_{{}^cT,F,\iota}\cong \varinjlim_{n}\Loc_{{}^cT,F}^{(n)},\quad \mbox{where } \ \ \ \Loc_{{}^cT,F}^{(n)}:={}^{cl}\bigl(\mR^{\Box}_{W^{(n)},{}^cT}\times_{\mR^{\Box}_{W^{(n)},\bG_m\times\Ga_{\widetilde F/F}}}\bigl\{\chi\bigr\}\bigr)/\hat{T}.
\]
So from now on we drop the subscript $\iota$ from the notation.

\begin{example}\label{E: ur-un-ta for tori}
Assume that $\widetilde F/F$ is tamely ramified so $W^{(1)}$ is a quotient of $\Ga_F^t$. Then 
$\Loc_{{}^cT,F}^{(1)}$ is the stack $\Loc_{{}^cT,F}^{\ta}$ of tame Langlands parameters of $T$ that will be introduced later. Note that $\Loc_{{}^cT,F}^{(1)}$ is connected over $\bZ[1/p]$ but this is not the case over $\bQ$.

If $\widetilde F/F$ is unramified, then $\Loc_{{}^cT,F}^{(0)}$ can also be identified with the stack $\Loc_{{}^cT,F}^{\ur}$ of unramified parameters of $T$ that will be introduced later. In this case, let $\bar\sigma$ denote the Frobenius element in $\Ga_{\widetilde F/F}$. Then the inclusion $\Loc_{{}^cT,F}^{(0)}\subset \Loc_{{}^cT,F}^{(1)}$ is identified with
\begin{equation}\label{E: unramified}
\hat{T}\bar\sigma/\hat{T}=\{1\}\times \hat{T}\bar\sigma/\hat{T}\subset \bigl(({}^{cl}\mR_{\kappa_{\widetilde F}^\times, \hat{T}})^\sigma\times \hat{T}\bar\sigma\bigr)/\hat{T}.
\end{equation}
Here $({}^{cl}\mR_{\kappa_{\widetilde F}^\times, \hat{T}})^\sigma$ is the classical moduli of $\sigma$-equivariant homomorphisms from $\kappa_{\widetilde F}^\times$ to $\hat{T}$, and $1$ denotes the trivial homomorphism. (Note that as explained in Example \ref{Ex: Zp over Fp}, $\mR_{\kappa_{\widetilde F}^\times, \hat{T}}$ itself is not classical (over $\bF_\ell$ when $\ell\mid \sharp\kappa_{\widetilde F}-1$) so one needs to take its underlying classical scheme.)
\end{example}

We note that $\Loc_{{}^cT,F}$ is in fact a Picard stack over $\bZ[1/p]$ (e.g. see \cite[\S A]{CZ} for a general review of Picard stacks). Let $\bB\bG_m$ be the classifying stack of $\bG_m$ over $\bZ[1/p]$.
Let 
\[
\Loc_{{}^cT,F}^\vee:=\underline\Hom(\Loc_{{}^cT,F}, \bB\bG_m)
\]
be the dual Picard stack of $\Loc_{{}^cT,F}^\vee$ over $\bZ[1/p]$ (in the sense of \cite[A.3.1]{CZ}), which is still a Picard stack, classifying multiplicative line bundles on $\Loc_{{}^cT,F}$. On the other hand, let $\breve F$ be the completion of a maximal unramified extension $F^{\ur}/F$ of $F$. Then the Frobenius $\sigma$ acts on $\breve F$. Let $\Tor_{T,\iso_F}$ denote the Picard groupoid of pairs $(\mE,\varphi)$ consisting of a $T$-torsor $\mE$ on $\breve{F}$ and an isomorphism $\varphi: \mE\simeq \sigma^*\mE$ of $T$-torsors. (The pair $(\mE,\varphi)$ can be regarded as a $T$-torsor in the $F$-linear Tannakian category of $\sigma$-$\breve{F}$-spaces in the sense of \cite[\S 3]{Ko} and \cite[\S 2]{Ko2}.)
We regard $\Tor_{T,\iso_F}$ as a constant Picard stack over $\bZ[1/p]$. The following conjecture can be regarded as the local Langlands duality for tori over non-archimedean local fields. 
\begin{conjecture}\label{C: LL for tori}
There is a natural Poincare line bundle on $\Tor_{T,\iso_F}\times \Loc_{{}^cT,F}$ inducing an isomorphism of Picard stacks $\Tor_{T,\iso_F}\cong \Loc_{{}^cT,F}^\vee$.
\end{conjecture}

\begin{remark}
We note that the isomorphism classes of $\Tor_{T,\iso_F}$ is nothing but Kottwitz' set $B(T)$ for $T$ (see \cite{Ko,Ko2}) which is identified with $\xch(\hat{T}^{\Ga_F})$ in \emph{loc. cit}. On
the other hand, the automorphism group of every $T$-torsor is just $T(F)$, whose character group can be identified with the set of Langlands parameters for $T$ (\cite{La}). So the conjecture is an algebro-geometric refinement of these facts.
\end{remark}

We slightly extend the above conjecture to allow not necessarily connected group $Z$ of multiplicative type over $F$. 
The Picard groupoid $\Tor_{Z,\iso_F}$ still makes sense (as in \cite{Ko2}),
but now may have non-trivial derived structure (as $H^2(W_F, Z(\overline{\breve{F}}))$ may not be zero). The set of its isomorphism classes is $B(Z)=H^1(W_F,Z(\overline{\breve{F}}))$.
To study the dual side, we embed $Z$ into an $F$-torus $T$ and let $T'=T/Z$. Then we define 
\[
\hat{Z}:=\hat{T'}/\hat{T}.
\]
If $Z$ is a torus, then $\hat{Z}$ is just the dual group of $Z$ but in general it is just a Picard stack. E.g. if $Z$ is finite, then $\hat{Z}$ is the classifying stack of $\ker(\hat{T}\to\hat{T'})$. 
In any case, $\hat{Z}$ is canonically independent of the choice of the embedding $Z\to T$ and may be called the dual of $Z$. 

There is the natural action of $\bG_m\times \Ga_{\widetilde F/F}$ on $\hat{Z}$ (of course $\bG_m$ acts trivially but we keep it to unify the notation). Then we can define ${}^cZ:=\hat{Z}\rtimes(\bG_m\times\Ga_{\widetilde F/F})$, regarded as a monoid stack over $\bZ[1/p]$. Then we may define  $\Loc_{{}^cZ,F}$. This is a Picard $2$-stack. 
One can also take its dual $\Loc_{{}^cZ,F}^\vee=\underline\Hom(\Loc_{{}^cZ,F},\bB\bG_m)$. Then Conjecture \ref{C: LL for tori} can be generalized as follows.
\begin{conjecture}\label{C: LL for group of mult type}
There is a natural isomorphism of derived Picard stacks $\Tor_{Z,\iso_F}\cong  \Loc_{{}^cZ,F}^\vee$. In particular, every $\theta\in \Tor_{Z,\iso_F}$ gives a multiplicative line bundle $\mL_\theta$ on $\Loc_{{}^cZ,F}$.
\end{conjecture}

We apply the above construction to $Z=Z_G$, the center of a connected reductive group $G$, to discuss certain symmetry of $\Coh(\Loc_{{}^cG,F,\iota})$. Let $\hat{G}_\s$ be the simply-connected cover of the derived group of $\hat{G}$ (i.e. the dual group of $G_\ad$). Let $\hat{T}_\s$ be the preimage of $\hat{T}$ in $\hat{G}_\s$. 
Then we have
$\widehat{Z_G}\cong \hat{T}/\hat{T}_\s\cong \hat{G}/\hat{G}_\s$, and therefore there is the ``determinant" map $\hat{G}\to \widehat{Z_G}$ inducing
\[
\delta:\Loc_{{}^cG,F,\iota}\to \Loc_{{}^cZ_G,F}.
\]
Conjecture \eqref{C: LL for group of mult type} implies that there is a natural action of $\Tor_{Z_G,\iso_F}$ on $\Coh(\Loc_{{}^cG,F,\iota})$, given by
\begin{equation}\label{E: BZ action}
\Tor_{Z_G,\iso_F}\times \Coh(\Loc_{{}^cG,F,\iota})\to  \Coh(\Loc_{{}^cG,F,\iota}),\quad (\theta,\mF)\mapsto \delta^*\mL_\theta\otimes \mF.
\end{equation}
This is the arithmetic analogue of some constructions in the geometric Langlands (e.g. see \cite[3.8, 5.6]{CZ}). 

We can refine this action a little bit.
By embedding $Z_G\subset T$, one obtains a map
$B(Z_G)\to B(T)\cong \xch(\hat{T}^{\Ga_F})\to \xch(Z_{\hat{G}}^{\Ga_F})$. The composed map $B(Z_G)\to \xch(Z_{\hat{G}}^{\Ga_F}),\ \theta\mapsto [\theta]$ is independent of the choice of $T$. 
On the other hand,  $\Loc_{{}^cG,F,\iota}$ is a $Z_{\hat{G}}^{\Ga_F}$-gerbe (as $Z_{\hat{G}}^{\Ga_F}\subset\hat{G}$ acts trivially on $\Loc_{{}^cG,F,\iota}^{\Box}$). It follows that there is a decomposition
\begin{equation}\label{E: gerbe dec}
\Coh(\Loc_{{}^cG,F,\iota})=\bigoplus_{\beta\in\xch(Z_{\hat{G}}^{\Ga_F})}\Coh^\beta(\Loc_{{}^cG,F,\iota}).
\end{equation}
Then the action $\mL_\theta$ will send $\Coh^\beta(\Loc_{{}^cG,F,\iota})$ to $\Coh^{\beta+[\theta]}(\Loc_{{}^cG,F,\iota})$.

There is an additional symmetry on $\Coh(\Loc_{{}^cG,F,\iota})$. 
Let $\tau\in\Aut(\hat{G},\hat{B},\hat{T},\hat{e})$ be the Cartan involution, i.e. the unique automorphism that induces 
$$\tau: \xch(\hat{T})\to \xch(\hat{T}), \quad \la\mapsto \la^*=-w_0(\la),$$ 
where $w_0$ is the longest length element in the Weyl group of $\hat{G}$. As $\tau$ is central in $\Aut(\hat{G},\hat{B},\hat{T},\hat{e})$, it induces an automorphism of ${}^cG$ and therefore an autoequivalence of $\Coh(\Loc_{{}^cG,F,\iota})$ denoted by the same notation. We let
\begin{equation}\label{E: Cartan inv} 
{}'\bD^{\on{Se}}:=\tau\circ \bD^{\on{Se}}: \Coh(\Loc_{{}^cG,F,\iota})\to  \Coh(\Loc_{{}^cG,F,\iota}).
\end{equation} 
be the modified Grothendieck-Serre duality. Note that ${}'\bD^{\on{Se}}$ preserves the decomposition \eqref{E: gerbe dec} and commutes with the action \eqref{E: BZ action}, while the original Grothendieck-Serre duality functor $\bD^{\on{Se}}: \Coh(\Loc_{{}^cG,F,\iota})\to \Coh(\Loc_{{}^cG,F,\iota})$ does not.

\subsection{Spectral parabolic induction}\label{SS: SPI}
Let $\hat{P}$ be a parabolic subgroup of $\hat{G}$ containing $\hat{B}$ and stable under the action of $\Ga_{\widetilde F/F}$ on $\hat{G}$, and let $\hat{M}$ be its standard Levi (the one containing $\hat{T}$). Then the action of $\bG_m\times\Ga_{\widetilde F/F}$ on $\hat{G}$ preserves $\hat{P}$ and $\hat{M}$, so we can form ${}^cP$ and ${}^cM$ respectively and define $\Loc_{{}^cP,F,\iota}$ and $\Loc_{{}^cM,F,\iota}$ similarly. 
Note that unlike $\Loc_{{}^cG,F,\iota}$ and $\Loc_{{}^cM,F,\iota}$,  $\Loc_{{}^cP,F,\iota}$ may not be not classical (see Remark \ref{R: LocB}), although it is still quasi-smooth. We emphasize that we need to remember the derived structure of $\Loc_{{}^cP,F,\iota}$ in the following discussions.
There is the following commutative diagram over $\bZ[1/p]$
\begin{equation}\label{E:LocMtoG}
\xymatrix{
&\ar@/_/[dl]_-{r}\Loc_{{}^cP,F,\iota}\ar^-{\pi}[dr]&  \\
\ar[d]\Loc_{{}^cM,F,\iota}\ar@/_/[ur]_-{i}&& \Loc_{{}^cG,F,\iota}\ar[d]  \\
\on{Spec} Z_{{}^cM,F}\ar[rr] && \on{Spec} Z_{{}^cG,F}.
}
\end{equation}
where $\pi,r,i$ are induced by the corresponding morphisms between $\hat{G},\hat{P},\hat{M}$, and where the bottom map is induced by $\pi\circ i: \Loc_{{}^cM,F,\iota}\to \Loc_{{}^cG,F,\iota}$. To see this diagram is commutative, it is enough to show that $r$ induces an isomorphism 
\begin{equation}\label{E: fun on P}
H^0\Gamma(\Loc_{{}^cM,F,\iota},\mO)\to H^0\Gamma({}^{cl}\Loc_{{}^cP,F,\iota},\mO).
\end{equation}
Let $2\rho_{\hat{G},\hat{M}}=2\rho-2\rho_{\hat{M}}$, where $2\rho$ (resp. $2\rho_{\hat{M}}$) is the sum of positive coroots of $\hat{G}$ (resp. $\hat{M}$). Then the conjugation action of $2\rho_{\hat{G},\hat{M}}(\bG_m)$ on ${}^cP$ contracts it into ${}^cM$. Equivalently, the weight zero part of $\Lambda[{}^cP]$ with respect to $2\rho_{\hat{G},\hat{M}}(\bG_m)$ is just $\Lambda[{}^cM]$.
It follows that \eqref{E: fun on P} is an isomorphism.

If we let $W_{{}^cG,{}^cM}$ be the quotient of the normalizer of ${}^cM\subset{}^cG$ in $\hat{G}$ by $\hat{M}$, then it follows that the map $Z_{{}^cG,F}\to Z_{{}^cM,F}$ factors through 
\begin{equation}\label{E: st center levi}
Z_{{}^cG,F}\to (Z_{{}^cM,F})^{W_{{}^cG,{}^cM}}.
\end{equation}

We have the following lemma (compare with \cite[13.2.2]{AG}).
\begin{lemma}\label{L: quasism r}
The morphism $r$ is quasi-smooth and $\pi$ is proper and schematic. 
\end{lemma}
\begin{proof}
That $\pi$ is proper and schematic is clear. For quasi-smoothness of $r$, it is enough to note that the relative cotangent complex at $\rho\in \Loc_{{}^cP,F,\iota}$ is $C_*(\Ga_{F,\iota},  \Ad^{u,*}_\rho)[-1]$ which concentrates in degree $[-1,1]$ if $\rho$ is a classical point. Here $\Ad^{u,*}$ is the coadjoint representation of ${}^cP$ on the dual of the Lie algebra of its unipotent radical. 
\end{proof}

Recall that  Arinkin-Gaitsgory (in \cite{AG}) attached, to a quasi-smooth derived algebraic stack $X$ over a field of characteristic zero, a classical stack $\on{Sing}(X)$ of singularities of $X$, and to a coherent sheaf $\mF$ on $X$, a conic subset $\on{Sing}(\mF)\subset \on{Sing}(X)$ as its singular support. These the constructions carry through for quasi-smooth stacks over $\CA_\Lambda$ with small changes (see \cite[\S 9.4]{HZ} for details). In particular, by definition
\[
\on{Sing}(\Loc_{{}^cG,F,\iota})=\Bigl\{ (\rho, \xi)\mid \rho\in {}^{cl}\Loc_{{}^cG,F,\iota},\ \xi\in H_2(\Ga_{F,\iota}, \Ad^*_\rho)\Bigr\},
\]
where $\Ad^*$ denote the coadjoint representation of ${}^cG$ on the dual of the Lie algebra of $\hat{G}$.

As explained in \cite{AG}, a particular conic subset $\hat\mN_{{}^cG,F,\iota}\subset \on{Sing}(\Loc_{{}^cG,F,\iota})$ plays an important role in the Langlands correspondence. Using \eqref{E: resol Gaq} (or a version of local Tate duality), we have
\[
H_2(\Ga_{F,\iota}, \Ad^*_\rho)\cong (\hat{\frakg}^*)^{\rho(I_{F,\iota})=1, \rho(\sigma)=q^{-1}}\subset \Ad^*_\rho.
\]
Let $\hat\mN^*\subset\hat\frakg^*$ be the nilpotent cone of $\hat\frakg^*$. We define
\begin{equation}\label{E: lgnlip}
\hat\mN_{{}^cG,F,\iota}= \Bigl\{ (\rho, \xi)\in\on{Sing}(\Loc_{{}^cG,F,\iota}),\ \xi\in \hat\mN_\rho^* \Bigr\}.
\end{equation}

The following proposition can be proved exactly the same as \cite[13.2.6]{AG}. Recall our convention of coherent sheaves on $\Loc_{{}^cG,F,\iota}$ (see Remark \ref{R: Coh cat}).
\begin{prop}\label{P: spec Eis}
There is a well-defined functor (called the spectral parabolic induction)
$$\pi_*r^!: \Coh(\Loc_{{}^cM,F,\iota})\to \Coh(\Loc_{{}^cG,F,\iota}),$$
which restricts to a functor $\pi_*r^!: \Coh_{\hat\mN_{{}^cM,F,\iota}}(\Loc_{{}^cM,F,\iota})\to \Coh_{\hat\mN_{{}^cG,F,\iota}}(\Loc_{{}^cG,F,\iota})$.
\end{prop}

We have the following observation. 
\begin{lemma}\label{L: nilp=sing}
Over $\bQ$, $\on{Sing}(\Loc_{{}^cG,F,\iota}\otimes\bQ)=\hat\mN_{{}^cG,F,\iota}\otimes\bQ$.
\end{lemma}
However, over $\bF_\ell$ when $\ell\mid q-1$, $\on{Sing}(\Loc_{{}^cG,F,\iota})$ is strictly larger than $\hat\mN_{{}^cG,F,\iota}$.
\begin{proof}
Using the identification between $\Loc_{{}^cG,F,\iota}\otimes\bQ$ and $\Loc_{{}^cG,F}^{\on{WD}}$ as in Lemma \ref{L: gen}, we identify $H_2(\Ga_{F,\iota}, \Ad_\rho^*)$ with
\[
\bigl\{\xi\in (\hat\frakg^*)^{r(I_F)}\mid \ad^*_X(\xi)=0, r(\sigma)(\xi)=q^{-1}\xi\bigr\},
\]
where $(r,X)$ corresponds to $\rho$ as in Lemma \ref{L: gen}.
We need to show such $\xi$ is automatically nilpotent. Let $\frakh:=\hat\frakg^{r(I_F)}$, which is a reductive Lie algebra.
We can identify $(\hat\frakg^*)^{r(I_F)}$ with $\frakh$ as an $(r(\sigma),\frakh)$-module. Then $\ad_{\xi}^j(\xi)$ is an eigenvector of $r(\sigma)$ with eigenvalue $q^{-j-1}$. This will force $\ad_{\xi}^j(\xi)=0$ for some $j$ large enough. That is, $\xi$ is nilpotent.
\end{proof}

The above computation also implies the following.
\begin{lemma}\label{P: sm pt}
Let $\rho: W_F\to {}^cG(\Ql)$ be a continuous representation such that $\Ad_\rho^0:W_F\to \GL(\hat\frakg)$ is pure of weight zero (in the sense of Deligne), then $\rho$ is a smooth point in in $\Loc_{{}^cG,F}$.
\end{lemma}
\begin{proof}Indeed, in the case $H^2(W_F,\Ad^0_\rho)=0$ and we can apply Proposition \ref{P: Classical} to conclude.
\end{proof}

In the remaining part of this subsection, we assume that $\widetilde{F}/F$ is tamely ramified, i.e. the image of $P_F\subset \Ga_F\to \Ga_{\widetilde{F}/F}$ is trivial. 
Then we have
the stack $\Loc_{{}^cG,F^t/F,\iota}$, called the stack of tame Langlands parameters, also denoted as $\Loc_{{}^cG,F,\iota}^{\ta}$. This is an open and closed substack of $\Loc_{{}^cG,F,\iota}$. 

Let $\Loc_{G,F,\iota}^{\ta,\Box}$ denote the framed version.
Explicitly, if we denote the image of $\tau$ (resp. $\sigma$) under the map $\Ga_q\xrightarrow{\iota} \Ga_F^t\to \Ga_{\widetilde{F}/F}$ by $\bar{\tau}$ (resp. $\bar{\sigma}$), then 
\begin{equation}\label{E: framed tame parameter}
  \Loc_{{}^cG,F,\iota}^{\ta,\Box}\cong \bigl\{(\tau,\sigma)\in \hat{G}\bar{\tau}\times\hat{G}q^{-1}\bar{\sigma}\mid \sigma\tau\sigma^{-1}=\tau^q\bigr\}\subset {}^cG\times{}^cG.
\end{equation}

\begin{remark}
One can compare $\Loc_{{}^cG,F,\iota}^{\ta,\Box}$ with the commuting scheme of $\hat{G}$, which classifies pairs of elements in $\hat{G}$ that commute with each other. 
While these two stacks exhibit quite different geometric structures over $\bQ$, they share some similar properties over $\bF_\ell$ when $\ell\mid q-1$.
\end{remark}

We can similarly define $\Loc_{{}^cB,F,\iota}^{\ta}$ and $\Loc_{{}^cT,F,\iota}^{\ta}$. Notice that $\Loc_{{}^cT, F,\iota}^{\ta}$ is simply $\Loc_{{}^cT,F}^{(1)}$ as discussed in Example \ref{E: ur-un-ta for tori}.
There exists a diagram analogous to \eqref{E:LocMtoG}, with the superscript $(-)^{\ta}$ added throughout.
Same reasonings as in Lemma \ref{L: quasism r} show that $r^{\ta}$ is quasi-smooth and $\pi^{\ta}$ is proper and schematic.

The inclusion $\langle\tau\rangle\subset \Ga_q$ induces maps
\begin{equation}\label{E: tame inertia type}
\Loc_{{}^cG,F,\iota}^{\ta}\to \hat{G}\bar\tau/\hat{G}\to \hat{G}\bar\tau/\!\!/\hat{G}\cong \hat{A}/\!\!/W_0,
\end{equation}
where $\hat{A}=\hat{T}/\!\!/(1-\bar\tau)\hat{T}$ and $W_0=W^{\bar\tau}$ is the $\bar\tau$-invariants of the Weyl group $W$ of $\hat{G}$ (see, for example, \cite[4.2.3]{XZ2}).
The second map is the GIT quotient map, while the last isomorphism is the Chevalley restriction isomorphism. 
As shown in the proof of Proposition \ref{P: tame stack}, this morphism factors through $\Loc_{{}^cG,F,\iota}^{\ta}\to (\hat{A}/\!\!/W_0)^{[q]}$, 
where $(\hat{A}/\!\!/W_0)^{[q]}$ is the (classical) fixed point subscheme of the map $[q]:\hat{A}/\!\!/W_0\to \hat{A}/\!\!/W_0$ induces by the morphism $\hat{G}\bar\tau\to\hat{G}\bar\tau, \  g\bar\tau\mapsto \bar\sigma^{-1}(g\bar\tau)^q\bar\sigma$.
It is not hard to check that $(\hat{A}/\!\!/W_0)^{[q]}$ is finite over $\bZ[1/p]$ and is \'etale over $\bQ$. 
Let $1: \on{Spec}\bZ[1/p]\to \hat{A}/\!\!/W_0$ be the map corresponding to the unit of $\hat{A}$, and let $\{1\}^\wedge$ denote the formal completion of $\hat{A}/\!\!/W_0$ along $\{1\}$.

We define two versions of  the stack of unipotent parameters as
\begin{equation*}\label{E: unip Loc}
\Loc_{{}^cG,F,\iota}^{\un}:= \Loc_{{}^cG,F,\iota}^{\ta}\times_{\hat{A}/\!\!/W_0}\{1\}\subset \Loc_{{}^cG,F,\iota}^{\widehat\un}:= \Loc_{{}^cG,F,\iota}^{\ta}\times_{\hat{A}/\!\!/W_0}\{1\}^\wedge.
\end{equation*}

\begin{remark}\label{rmk: bad things for locsys-unip}
\begin{enumerate}
\item 
By definition, $\Loc_{{}^cG,F,\iota}^{\un}$ is an algebraic stack but is in general derived. On the other hand, $\Loc_{{}^cG,F,\iota}^{\widehat\un}$ is classical but an ind-algebraic stack. Their underlying reduced substacks coincide
\[
{}^{red}\Loc_{{}^cG,F,\iota}^{\un}={}^{red}\Loc_{{}^cG,F,\iota}^{\widehat\un}.
\]
When base changed to a field, $\Loc_{{}^cG,F,\iota}^{\widehat\un}$ is in fact an algebraic stack (see \cite[Lemma 2.14]{HZ}), but is not reduced in general. In fact even ${}^{cl}\Loc_{{}^cG,F,\iota}^{\un}$ may not be reduced (e.g. see \cite{EZ} in the case $G=\PGL_2$).

The situation is much better understood over $\bQ$. 
As $1$ is an isolated point of $ (\hat{A}/\!\!/W_0)^{[q]}\otimes\bQ$, we see that
\[
\Loc_{{}^cG,F,\iota}^{\widehat\un}\otimes\bQ=(\Loc_{{}^cG,F,\iota}^{\ta}\times_{(\hat{A}/\!\!/W_0)^{[q]}}\{1\})\otimes \bQ
\] 
is open and closed in $\Loc_{{}^cG,F,\iota}^{\ta}\otimes\bQ$. When $\bar\tau=1$, $\Loc_{{}^cG,F,\iota}^{\widehat\un}\otimes\bQ$ is a connected component of $\Loc_{{}^cG,F,\iota}^{\ta}\otimes\bQ$.  In particular, it is still a local complete intersection. 

\item Our terminology could be potentially misleading as for a (field valued) point $\rho\in \Loc_{{}^cG,F,\iota}^{\widehat\un}$, the element $\rho(\tau)\in \hat{G}\bar\tau$ may not be a unipotent element (as $\bar\tau$ may not be trivial). 
On the other hand, if $\bar\tau=1$, i.e. $\widetilde F/F$ is unramified, then
\[
\Loc_{{}^cG,F,\iota}^{\widehat\un}\cong \Loc_{{}^cG,F,\iota}^{\widehat\un,\Box}/\hat{G},\quad \mbox{where }  \Loc_{{}^cG,F}^{\widehat\un,\Box}=\{(\tau,\sigma)\in \hat{\mU}^\wedge\times\hat{G}q^{-1}\bar{\sigma}\mid \sigma\tau\sigma^{-1}=\tau^q\},
\]
where as before $\hat\mU$ is the unipotent variety of $\hat{G}$, and $\hat\mU^\wedge$ denotes its formal completion in $\hat{G}$.
So the image of $\tau$ in $\hat{G}$ is indeed unipotent. 
\end{enumerate}
\end{remark}

If $\widetilde{F}/F$ is unramified, then inside $\Loc_{G,F,\iota}^{\un}$ there is the stack of unramified parameters.
\[
   \Loc_{{}^cG,F}^{\ur,\Box}\cong \hat{G}q^{-1}\bar{\sigma}\subset{}^cG, \quad  \Loc_{{}^cG,F}^{\ur}= \Loc_{{}^cG,F}^{\ur,\Box}/\hat{G}.
\]
We note that this stack is smooth and is independent of the choice of $\iota$ (so we will drop $\iota$ from the notation). If $T$ is an unramified torus, then 
\begin{equation}\label{eq: unramified parameters for torus}
\Loc_{{}^cT,F}^{\ur}={}^{red}\Loc_{{}^cT,F}^{\un}={}^{cl}\Loc_{{}^cT,F}^{\un}
\end{equation}
coincide with $\Loc_{{}^cT,F}^{(0)}$ as discussed in \S \ref{S: tori}.

At the end of this subsection, we introduce what we call spectral Deligne-Lusztig stacks and their unipotent versions.
Recall that we assume that $\widetilde F/F$ is tamely ramified. But we suggest readers to go through the construction in the simpler situation when $\widetilde F/F$ is unramified (so $\bar\tau=1$) for the first time reading.

Let $\widetilde{\hat{G}\bar\tau}:=\hat{G}\times^{\hat{B}}\hat{B}\bar\tau\to \hat{G}\bar\tau$ be the (twisted) Grothendieck-Springer resolution of $\hat{G}\bar\tau$ (e.g. see \cite[5.3]{XZ}). Then we define the (big) Steinberg variety $\St_{\hat{G}\bar\tau}=\widetilde{\hat{G}\bar\tau}\times_{\hat{G}\bar\tau}\widetilde{\hat{G}\bar\tau}$, which is a classical, reduced, local complete intersection scheme of dimension $\dim\hat{G}$. Its irreducible components are naturally indexed by $W_0=W^{\bar\tau}$. For $w\in W_0$, let $\St_{\hat{G}\bar\tau,w}$ denote the corresponding irreducible component. For simplicity, we write 
$S=\St_{\hat{G}\bar\tau}/\hat{G}$ and $S_w=\St_{\hat{G}\bar\tau,w}/\hat{G}$. We call $S$ the big Steinberg stack.

Recall the morphism $\Loc_{{}^cG,F,\iota}^\ta\to \hat{G}\bar\tau/\hat{G}$ from \eqref{E: tame inertia type}.
Then we define
\begin{equation}\label{E: bigSpDM}
\widetilde{\Loc}_{{}^cG,F,\iota}^\ta:= \Loc_{{}^cG,F,\iota}^\ta\times_{\hat{G}\bar\tau/\hat{G}}\hat{B}\bar\tau/\hat{B}\xrightarrow{\widetilde\pi\times\pr}\Loc_{{}^cG,F,\iota}^\ta\times \hat{B}\bar\tau/\hat{B}.
\end{equation}
So ${}^{cl}\widetilde{\Loc}_{{}^cG,F,\iota}^\ta$ classifies $(\tau,\sigma, g\hat{B})$ where $(\tau,\sigma)$ is a tame Langlands parameter as in \eqref{E: framed tame parameter} and $g\hat{B}\in \hat{G}/\hat{B}$ such that $\tau\in g^{-1}(\hat{B}\bar\tau)g$. 
Note that as $\tau\in (g\sigma)^{-1}(\hat{B}\bar\tau)(g\sigma)$, there is another projection $\pr': \widetilde{\Loc}_{{}^cG,F,\iota}^\ta\to \hat{B}\bar\tau/\hat{B}$. Therefore, there is
a morphism
\[
\widetilde{\Loc}_{{}^cG,F,\iota}^{\ta}\xrightarrow{\pr\times\pr'} \hat{B}\bar\tau/\hat{B}\times_{\hat{G}\bar\tau/\hat{G}} \hat{B}\bar\tau/\hat{B}\cong S.
\]
Then we define
\begin{equation}\label{E: SpDL}
\widetilde{\Loc}_{{}^cG,F,\iota}^{\ta,w}:=\widetilde{\Loc}_{{}^cG,F,\iota}^{\ta}\times_SS_w.
\end{equation}

\begin{remark}\label{rmk: spectral DL stack}
If $w=1$ is the unit element, one can show that 
\[
\widetilde{\Loc}_{{}^cG,F,\iota}^{\ta,1}\cong \Loc_{{}^cB,F,\iota}^\ta.
\] 
Informally, $\widetilde{\Loc}_{{}^cG,F,\iota}^{\ta,w}$ classifies those $(\tau,\sigma,B')$ such that $B'$ and $\sigma B'\sigma^{-1}$ has relative position bounded by $w$. For this reason,
one may call general $\widetilde{\Loc}_{{}^cG,F,\iota}^{\ta,w}$ as spectral Deligne-Lusztig stacks.
\end{remark}

We also introduce the unipotent version of spectral Deligne-Lusztig stacks. Consider the map 
\[
\widetilde{\Loc}_{{}^cG,F,\iota}^\ta\to \hat{B}\bar\tau/\hat{B}\to \hat{T}\bar\tau/\hat{T}\to \hat{A}=\hat{T}/\!\!/(1-\bar\tau)\hat{T}.
\]
We first define
\[
(\hat{B}\bar\tau/\hat{B})^{\un}= (\hat{B}\bar\tau/\hat{B})\times_{\hat{A}}\{1\}.
\]
Note that if $\bar\tau=1$, then $(\hat{B}\bar\tau/\hat{B})^{\un}=\hat{U}/\hat{B}$, where $\hat{U}$ is the unipotent radical of $\hat{B}$. We then define the unipotent version of the Steinberg stack
\begin{equation}\label{E: SpSteinbergunip}
S^{\un}=(\hat{B}\bar\tau/\hat{B})^{\un}\times_{\hat{G}\bar\tau/\hat{G}} (\hat{B}\bar\tau/\hat{B})^{\un},
\end{equation}
and
\begin{equation}\label{E: SpDLunip}
\widetilde{\Loc}_{{}^cG,F,\iota}^\un:=\widetilde{\Loc}_{{}^cG,F,\iota}^\ta\times_{\hat{A}}\{1\}. 
\end{equation}
We similarly have the map 
\[
\widetilde{\Loc}_{{}^cG,F,\iota}^{\un}\to S^{\un}.
\]

For $w\in W_0$, we let
\[
S^{\un}_{w}:=(\hat{B}\bar\tau/\hat{B}\times_{\hat{G}\bar\tau/\hat{G}} (\hat{B}\bar\tau/\hat{B})^{\un})\cap S_w
\]
where the intersection is taken in $S$.
It is a classical stack, although it is not irreducible in general. In addition, it is easy to see that the map 
\[
S^{\un}_{w}\subset \hat{B}\bar\tau/\hat{B}\times_{\hat{G}\bar\tau/\hat{G}} (\hat{B}\bar\tau/\hat{B})^{\un}
\]
factors through 
\[
S^{\un}_{w}\subset S^{\un}\subset \hat{B}\bar\tau/\hat{B}\times_{\hat{G}\bar\tau/\hat{G}} (\hat{B}\bar\tau/\hat{B})^{\un}.
\]
Then similar to \eqref{E: SpDL}, we can define
\begin{equation}\label{E: SpDLunip}
\widetilde{\Loc}_{{}^cG,F,\iota}^{\un,w}:=\widetilde{\Loc}_{{}^cG,F,\iota}^{\un}\times_{S^{\un}}S^{\un}_w.
\end{equation}

Similar to Remark \ref{rmk: spectral DL stack}, one can show that
\[
\widetilde{\Loc}_{{}^cG,F,\iota}^{\un,1}\cong \Loc_{{}^cB,F,\iota}^\un.
\] 
where we define
\begin{equation}\label{E: unip locsys B}
\Loc_{{}^cB,F,\iota}^{\un}:= \Loc_{{}^cB,F,\iota}^{\ta}\times_{\Loc_{{}^cT,F,\iota}^{\ta}}{}^{cl}\Loc_{{}^cT,F,\iota}^{\un}.
\end{equation} 

\subsection{The stack of global Langlands parameters}\label{SS: glob par}
Now we turn to global Langlands parameters. Currently, we are not aware of how to define a stack of global Langlands parameters over $\bZ$ (or over $\bZ[1/p]$ for a function field of characteristic $p$) so we do not have the global analogue of $\Loc_{{}^cG,F,\iota}$. However, the main goal of this subsection is to show that 
 the general recipe as in Section \ref{SS: Continuous rep} provides a reasonable definition of the stack over $\on{Spec}\bZ_\ell$ in the global function field case.
 The number field case is more complicated and is an on going joint work with Emerton \cite{EZ}. We will only briefly discuss it at the end of the subsection.

We fix a few notations. Let $F$ be a global field. We regard the Galois group $\Ga_F$ as a profinite group, and in the global function field case the Weil group $W_F$ as a locally profinite group. Let $\La=\bZ_\ell$, where $\ell\neq \on{char}F$ if $F$ is a function field.
For a place $v$, let $F_v$ denote the corresponding local field, $\kappa_v$ the residue field and $q_v=\sharp\kappa_v$. Let $\Ga_v$ (resp. $W_v$) denote the Galois (resp. Weil) group of $F_v$.
Let $G$ be a connected reductive group over $F$. We write $G_v$ for either $G_{F_v}$ or $G(F_v)$. The $C$-group of $G$ is denoted by ${}^cG$ and the $C$-group of $G_v$ is denoted by ${}^cG_v$.
For a place $v$ not lying above $\ell$, let $\Loc^{?}_{v}$ denote $\Loc^{?}_{{}^cG_v,F_v}$ for simplicity, where $?\in\{\emptyset, \ta, \ur\}$, etc. 
We will fix a \emph{non-empty} finite set of places $S$ containing all the infinite places, the places above $\ell$, and the places ramified in $\widetilde F/F$ and consider the quotient $\Ga_{F,S}$ corresponding to the maximal Galois extension of $F$ that is unramified outside $S$. Similarly, we have $W_{F,S}$ in the global function field setting. Let $Y$ be the Dedekind scheme with fractional field $F$ and \'etale fundamental group $\Ga_{F,S}$.

Now let $F$ be a function field. Let $\bF_q$ be the algebraic closure of $\bF_p$ in $F$. Then $Y$ is an affine smooth curve over $\bF_q$. Let $\overline Y$ be the base change of $Y$ to $\overline\bF_q$. Let $\pi_1(\overline Y)$ denote the geometric fundamental group. (We ignore the choice of a base point on $\overline Y$ since it plays little role in the sequel.) Recall that there is the short exact sequence
\[
1\to \pi_1(\overline Y)\to W_{F,S}\xrightarrow{\|\cdot\|} \bZ=\langle\sigma\rangle\to 1.
\]
We replace the local Weil group $W_F$ in \eqref{E: Locloc} by $W_{F,S}$ and define
\begin{equation}\label{E: globstackSpfZl}
  \Loc_{{}^cG,F,S}^{\wedge,\Box}:= \mR_{W_{F,S},{}^cG}^{c}\times_{\mR_{W_{F,S},\bG_m\times\Ga_{\widetilde F/F}}^{c}}\bigl\{\chi\bigr\},\quad  \Loc_{{}^cG,F,S}^{\wedge}= \Loc_{{}^cG,F,S}^{\wedge,\Box}/\hat{G}_\ell^\wedge,
\end{equation}
Let $\Loc_{{}^cG,F,S,r}=\Loc_{{}^cG,F,S,r}^{\Box}/\hat{G}_r$ be the restriction of \eqref{E: globstackSpfZl} to $\on{Spec}\bZ/\ell^r$. Then $\Loc_{{}^cG,F,S,r}^{\Box}$ classifies, for every $\bZ/\ell^r$-algebra $A$, the space of continuous homomorphisms $\rho$ from $W_{F,S}$ to ${}^cG(A)$ such that $d\circ \rho=\chi$ (Lemma \ref{L: cont rho}). We can also extend \eqref{E: globstackSpfZl} to $\on{Spec}\bZ_\ell$ using Definition \ref{D: sc rep}
\begin{equation}\label{E: globstackZl}
  \Loc_{{}^cG,F,S}^{\Box}:= \mR_{W_{F,S},{}^cG}^{sc}\times_{\mR_{W_{F,S},\bG_m\times\Ga_{\widetilde F/F}}^{sc}}\bigl\{\chi\bigr\},\quad  \Loc_{{}^cG,F,S}= \Loc_{{}^cG,F,S}^{\Box}/\hat{G}.
\end{equation}

\begin{remark}
Another definition of the stack of global Langlands parameters over $\bQ_\ell$ for function fields  is recently proposed in \cite{AGK}. Their definition is different the one given above, but probably gives a stack isomorphic to the base change of our $\Loc_{{}^cG,F,S}$ to $\bQ_\ell$. 
\end{remark}

Here is the main result of this subsection.
\begin{theorem}\label{P: Loc global function}
Assume that $\ell>2$. Then $\Loc_{{}^cG,F,S}$ is a quasi-smooth algebraic stack over $\bZ_\ell$. 
It decomposes as a disjoint union of its open and closed substacks 
\begin{equation}\label{E: glob dec}
\Loc_{{}^cG,F,S}=\bigsqcup_{\Theta}\Loc_{{}^cG,F,S}^{\Theta},
\end{equation}
where $\Theta$ range over all closed points of $\mR^{c}_{\pi_1(\overline Y),{}^cG^\bullet/\!\!/\hat{G}}$ satisfying $d\circ \Theta=\chi$. Each $\Loc_{{}^cG,F,S}^{\Theta}$ is quasi-compact, and for every $\overline\bF_\ell$- or $\Ql$-point $x$ of $\Loc_{{}^cG,F,S}^{\Theta}$, the (residual) pseudorepresentation $\overline{\rho_x|_{\pi_1(\overline Y)}}$ is $\Theta$. 
\end{theorem}
We refer to Lemma \ref{L: constant} and discussions before it for the notation $\overline{\rho_x|_{\pi_1(\overline Y)}}$.

To prove the theorem, let us first recall that de Jong's conjecture (\cite{de}) says that if $\rho:\pi_1(Y)\to \GL_m(\kappa((t)))$ is a continuous representation of the arithmetic fundamental group, where $\kappa$ is a finite field of characteristic $\ell$ and $\kappa((t))$ is equipped with the $t$-adic topology, then $\rho(\pi_1(\overline Y))$ is finite. This was proved by Gaitsgory \cite{Ga2} under the assumption $\ell>2$ (see also \cite{BK}).\footnote{This is why we also require $\ell>2$. Certainly such restriction is expected to be removed.} Note that one can replace $\pi_1(Y)$ by the Weil group $W_{F,S}$ in the statement of de Jong's conjecture. 

We need the following consequence.
As the Frobenius $\sigma$ acts on $\pi_1(\overline Y)$ by (outer) automorphism, it acts on the space $\mR^{c}_{\pi_1(\overline Y),\GL_m^\bullet/\!\!/\GL_m}$ of pseudorepresentations of $\pi_1(\overline Y)$. 
Let $\mR^{c,\Theta}_{\pi_1(\overline Y),\GL_m^\bullet/\!\!/\GL_m}$ be a $\sigma$-stable connected component. Recall that $\mR^{c,\Theta}_{\pi_1(\overline Y),\GL_m^\bullet/\!\!/\GL_m}$ is a derived formal scheme. We write the ring of functions of the underlying classical formal scheme as
$$A^\Theta:=\Gamma({}^{cl}\mR^{c,\Theta}_{\pi_1(\overline Y),\GL_m^\bullet/\!\!/\GL_m},\mO).$$ 
Since $\pi_1(\overline Y)$ satisfies Mazur's condition $\Phi_\ell$, this is a complete noetherian local $\bZ_\ell$-algebra (\cite[3.7]{Ch}), on which $\sigma$ acts. 

\begin{lemma}\label{L: finiteness ps of geom}
The quotient ring $A^\Theta/(\sigma-1)A^\Theta$ is finite over $\bZ_\ell$.
\end{lemma}
\begin{proof}
Note that $B^\Theta=A^\Theta/(\sigma-1)A^\Theta$ is still a complete noetherian local ring with residue field $\kappa$. Therefore it is enough to show that $B^\Theta/\ell$ is artinian. Let $B^\Theta\to \kappa'[[t]]$ be local ring homomorphism with $\kappa'$ finite over $\kappa$, giving a continuous $\kappa'[[t]]$-valued pseudorepresentation of $\pi_1(\overline Y)$. Then by Proposition \ref{P: cont psrep to trrep}, such $\kappa'((t))$-valued pseudorepresentation comes from a continuous (absolutely) semisimple representation $\rho: \pi_1(\overline Y)\to \GL_m(K)$ for some finite extension $K/\kappa'((t))$. As the pseudorepresentation is $\sigma$-invariant, such $\rho$ extends to a continuous representation of $W_{F,S}\to\GL_m(K')$ for some finite extension $K'/K$. Then by de Jong's conjecture, the image of $\pi_1(\overline Y)$ is finite. Therefore the image of $B^\Theta\to \kappa'[[t]]$ is $\kappa'$. This show that $B^\Theta/\ell$ is artinian.
\end{proof}

Now we proof Theorem \ref{P: Loc global function}.

\begin{proof} 
We use the Artin-Lurie representability theorem \cite[7.5.1]{Lu4}. First we verify that  $\mR^{sc}_{\pi_1(\overline Y),{}^cG}$ satisfies Condition (1)-(5) of \emph{loc. cit}. Namely, $\mR^{sc}_{\pi_1(\overline Y),{}^cG}$ is $0$-truncated so Condition (2) holds. By Proposition \ref{L: tang Rsc}, Condition (1), (4), (5) hold. We claim that $\mR^{sc}_{\pi_1(\overline Y),{}^cG}$ satisfies fppf descent so Condition (3) also holds. Indeed, as $\mR^{sc}_{\pi_1(\overline Y),{}^cG}$ is nilcomplete, it is enough to show that 
$$\mR^{sc}_{\pi_1(\overline Y),{}^cG}(A)\to \varprojlim_{\Delta} \mR^{sc}_{\pi_1(\overline Y),{}^cG}(B^\bullet)$$ 
is an isomorphism, where $B^\bullet:\Delta\to {}_{\leq m}\CA_{\bZ_\ell}$ is the \v{C}ech nerve of a faithfully flat map $A\to B$ of $m$-truncated animated $\bZ_\ell$-algebras. In this case, we may replace the limit over $\Delta$ by the finite limit over $\Delta_{\leq m+1}\subset\Delta$ consisting of objects $[0],\ldots, [m+1]$. As $\mR^{sc}_{\pi_1(\overline Y),{}^cG}$ preserves finite limits, the claim follows.
Now it is easy to see that $\Loc_{{}^cG,F,S}\to \mR^{sc}_{\pi_1(\overline Y),{}^cG/\hat{G}}$ is relatively representable, so $\Loc_{{}^cG,F,S}$ also satisfies Condition (1)-(5) of \cite[7.5.1]{Lu4}.

Again by Proposition \ref{L: tang Rsc}, the tangent space of $\Loc_{{}^cG,F,S}$ at a point $\rho: W_{F,S}\to {}^cG(A)$ is the continuous cohomology $C_{cts}^*(W_{F,S}, \Ad^0_\rho)[1]$, where $A$ is a classical $\bZ_\ell$-algebra, and $\Ad^0$ is the adjoint representation of ${}^cG$ on the Lie algebra of $\hat{G}$. 
Recall that for a continuous representation $\pi_1(\overline Y)$ on a finite $\bZ_\ell$-module $V$, the continuous group cohomology $C_{cts}^*(\pi_1(\overline Y),V)$ is isomorphic to the \'etale cohomology of $V$ (regarded as a local system on the affine variety $\overline Y$). It follows from 
Lemma \ref{L: cont criterion} and \eqref{E: colim cont coh} that $C_{cts}^*(\overline Y,\Ad^0_\rho)$
concentrates in degree $[-1,0]$, and its cohomology groups are finite $A$-modules if $A$ is finitely generated over $\bZ_\ell$. Then the Hochschild-Serre spectral sequence implies that $C_{cts}^*(W_{F,S}, \Ad^0_\rho)[1]$ concentrates in degree $[-1,1]$ and is a finite $A$-module in each degree if $A$ is finitely generated over $\bZ_\ell$.
This verifies Condition (7) of \cite[7.5.1]{Lu4}.
In addition, it shows that if $\Loc_{{}^cG,F,S}$ is representable, then it is quasi-smooth.

It remains to verify Condition (6). We show that for a classical noetherian completed $\bZ_\ell$-algebra $(A,\frakm)$ with residue field $\kappa$ either finite over $\bF_\ell$ or over $\bQ_\ell$, the map
\[
\Loc^{\Box}_{{}^cG,F,S}(A)\to\varprojlim_i \Loc^{\Box}_{{}^cG,F,S}(A/\frakm^i)
\]
is an isomorphism. By choosing a faithful representation ${}^cG\to \GL_m$, we reduce to show that
\begin{equation}\label{E: A6}
\mR^{sc}_{W_{F,S}, \GL_m}(A)\to\varprojlim_i \mR^{sc}_{W_{F,S}, \GL_m}(A/\frakm^i)
\end{equation}
is an isomorphism. Let $\{\rho_i\}$ be a compatible family of representations $\rho_i: W_{F,S}\to \GL_m(A/\frakm^i)$, giving an element of the right hand side of \eqref{E: A6}. Note that as $A/\frakm^i$ is finite over $\bZ_\ell$ or over $\bQ_\ell$, each $\rho_i$ is just a continuous representation in the usual sense (see Remark \ref{R: strong top}). Forgetting the topology and taking the inverse limit, we obtain
a representation $\rho: W_{F,S}\to \GL_m(A)$. We need to show it is strongly continuous. By Lemma \ref{L: cont criterion}, it is enough to show that for every  $v\in A^m$,  $\rho(\pi_1(\overline Y))v$ is contained in a finite $\bZ_\ell$-module.

Let $B$ be the $\bZ_\ell$-subalgebra of $A$ generated by $\chi_j(\rho(\ga))$ for $\ga\in \pi_1(\overline Y)$, where
$\chi_i\in \bZ[\GL_m]^{\GL_m}$
is the character of the $i$th wedge representation of $\GL_m$ as before. Then  for every $\ga\in \pi_1(\overline Y)$ the characteristic polynomial $\on{Char}(\rho(\ga),t)\in B[t]$.
We extend the action of  $\pi_1(\overline Y)$ on $A^m$ to the action of its group ring $B\pi_1(\overline Y)$. Note that the characteristic polynomial of $r=\sum b_j\ga_j \in B\pi_1(\overline Y)$ also belongs to $B[t]$. As each $\rho_i$ is continuous (in the usual sense), the action extends to an action of the completed group ring $B\pi_1(\overline Y)^{\wedge}$, and then factors through the quotient 
$B\pi_1(\overline Y)^\wedge/I$, where $I$ is the ideal generated by $\on{Char}(\rho(r),r)$ for $r\in B\pi_1(\overline Y)^\wedge$. As $\pi_1(\overline Y)$ satisfies Mazur's condition $\Phi_\ell$, $B\pi_1(\overline Y)^\wedge/I$ is finite over $B$ by \cite[3.6]{WE}. 
We claim that $B$ is finite over $\bZ_\ell$, which will finish the proof that \eqref{E: A6} is an isomorphism.

Consider $\rho_0: W_{F,S}\to \GL_m(A/\frakm)=\GL_m(\kappa)$. If $\kappa$ is a finite field, let $\bar\rho=\rho_0|_{\pi_1(\overline Y)}$.  If $\kappa=E$ is of characteristic zero, then after conjugation we may assume that $\rho_0|_{\pi_1(\overline Y)}$ comes from an $\mO_E$-representation. Let $\bar\rho:\pi_1(\overline Y)\to \GL_m(\kappa_E)$ be the residual representation of $\rho_0|_{\pi_1(\overline Y)}$.  We have the usual (classical) framed deformation ring $R_{\bar\rho}^\Box$ of $\bar\rho$. The representation $\rho_0|_{\pi_1(\overline Y)}$ gives a point of $R_{\bar\rho}^\Box$, and the formal completion of $R^\Box_{\bar\rho}$ at this point prorepresents the classical framed deformations of $\rho_0$ (considered as a functor $\mathbf{Art}_{\bZ_\ell,\kappa}\to \Sets$). (If $\kappa=E$, see \cite[2.3.5]{Ki}.) 
Then $\rho|_{\pi_1(\overline Y)}:\pi_1(\overline Y)\to \GL_m(A)$ gives a map $R_{\bar\rho}^\Box\to A$.
Let
$\Theta$ be the pseudorepresentation associated to $\bar\rho$. Then we have $A^\Theta$ as in Lemma \ref{L: finiteness ps of geom}, and $B$ is just the image of $A^\Theta$ under the natural map $A^\Theta\to R_{\bar\rho}^\Box\to A$, which factors through $A^\Theta/(1-\sigma)A^\Theta\to A$. Therefore $B$ is finite over $\bZ_\ell$ by Lemma \ref{L: finiteness ps of geom}.

We have proved the representability of $\Loc_{{}^cG,F,S}$. By Lemma \ref{L: constant}, we have the decomposition \eqref{E: glob dec}. It remains to see that $\Loc_{{}^cG,F,S}^\Theta$ is quasi-compact. In fact we show that the corresponding framed version $\Loc_{{}^cG,F,S}^{\Theta,\Box}$ is represented by an affine scheme of finite type over $\bZ_\ell$.
We may reduce to $\GL_m$-case. We have the ring $B^\Theta$ as in Lemma \ref{L: finiteness ps of geom}, and
then a finite (associative) $\bZ_\ell$-algebra $B^\Theta\pi_1(\overline Y)^\wedge/I$ as above. We lift the Frobenius $\sigma$ to an element in $W_{F,S}$, so $\sigma$ acts on $B^\Theta\pi_1(\overline Y)^\wedge/I$ and
we can form the twisted $\bZ_\ell$-algebra $B^\Theta\pi_1(\overline Y)^\wedge/I[\sigma]$. Now ${}^{cl}\Loc_{{}^cG,F,S}^{\Theta,\Box}$ is nothing but the moduli space of framed $m$-dimensional representations of the finitely generated associative $\bZ_\ell$-algebra $B^\Theta\pi_1(\overline Y)^\wedge/I[\sigma]$, and therefore is represented by an affine scheme of finite type over $\bZ_\ell$.
\end{proof}

\begin{remark}
One may think the decomposition \eqref{E: glob dec} as the global analogue of the mod $\ell$ inertia types in the local case (Remark \ref{R: mod l inert type}).
Clearly, $\Loc_{{}^cG,F,S}^{\Theta}$ is non-empty if and only if $\Theta$ is fixed under the action of the Frobenius $\sigma$.
\end{remark}

\begin{remark}
One may expect that the stack $\Loc_{{}^cG,F,S}$ is classical, as in the local situation. As mentioned in Remark \ref{R: LocB}, $\Loc_{{}^cG,F,S}^{\Theta}$ is classical if and only if $\dim \Loc_{{}^cG,F,S}^{\Theta}=0$. Unfortunately, this is not always the case. 

Consider the case $G=\on{PGL}_2$ (so ${}^cG=\GL_2$), and let $\Theta$ be the pseudorepresentation corresponding to the trivial representation of $\pi_1(\overline Y)$. Then $\Loc_{{}^cG,F,S,1}^{\Theta}$ consists of those $\rho: W_{F,S}\to \GL_{2}$ such that $\rho|_{\pi_1(\overline Y)}$ is a self extension of the trivial character. Note that there is an $H^1(\overline Y, \overline\bF_\ell)$-family of self extensions of the trivial character of $\pi_1(\overline Y)$.
It follows that if the multiplicity of one Frobenius eigenvalue on $H^1(\overline Y, \overline\bF_\ell)$ is greater than one, then $\dim \Loc_{{}^cG,F,S,1}^{\Theta,\Box}> \dim\hat{G}_{\bF_\ell}$, and $\Loc_{{}^cG,F,S,1}^{\Theta}$ is non-classical.
\end{remark}

Sometimes it is convenient to consider substacks of $\Loc_{{}^cG,F,S}$ with fixed ``determinant". More precisely, let $Z^\circ_G$ be the connected center of $G$. Then ${}^c(Z^\circ_G)=\hat{G}_\ab\rtimes(\bG_m\times\Ga_{\widetilde F/F})$, where $\hat{G}_\ab$ be the abelianization of $\hat{G}$. There is the natural morphism
$\pi_\ab: \Loc_{{}^cG,F,S}\to \Loc_{{}^c(Z^\circ_G),F,S}$. 
Given a classical $\bZ_\ell$-algebra $A$ and a strongly continuous representation
$\la: W_{F,S}\to {}^c(Z^\circ_G)(A)$ (such that $d\circ \la=\chi$) corresponding to an $A$-point of $\Loc_{{}^c(Z^\circ_G),F,S}$, let 
$$\Loc_{{}^cG,F,S,A}^\la= \Loc_{{}^cG,F,S}\times_{\Loc_{{}^c(Z^\circ_G),F,S}}\on{Spec} A$$ 
denote the base change of $\pi_\ab$ along $\la$, which is an algebraic stack over $A$ classifying those representations of $\rho$ such that $\pi_\ab\circ\rho=\la$. Its tangent space at $\rho$ is given by $C_{cts}^*(W_{F,S},\Ad^{00})$, where $\Ad^{00}$ is the adjoint representation of ${}^cG$ on the Lie algebra of the derived group of $\hat{G}$. In particular, $\Loc_{{}^cG,F,S,A}^\la$ is quasi-smooth over $A$.

\begin{example}\label{E: ell part}
An elliptic Langlands parameter is a continuous semisimple representation $\rho: W_{F,S}\to {}^cG(\Ql)$ (satisfying $d\circ\rho=\chi$) such that 
$\frakS_\rho:=S_\rho/(Z_{\hat{G}})^{\Ga_F}$ is finite, where $S_\rho$ is the stabilizer of $\rho$ under the conjugation action of $\hat{G}$ on ${}^cG$, and $Z_{\hat{G}}$ is the center of $\hat{G}$, on which $W_F$ acts.
By \cite[4.1]{LZ}, an elliptic Langlands parameter $\rho$ gives an isolated smooth point in $\Loc_{{}^cG,F,S,\Ql}^\la$, where $\la= \pi_\ab\circ\rho$.
More precisely, every elliptic $\rho$ gives an open and closed embedding $(\on{Spec} \Ql)/S_\rho\to \Loc_{{}^cG,F,S,\Ql}^\la$.
\end{example}

The embedding $W_{F_v}\to W_F$ up to conjugacy induces a well-defined morphism
\begin{equation}\label{E: global-to-local}
\res: \Loc_{{}^cG,F,S}\to \prod_{v\in S} \Loc_{v} \times \prod_{w\not\in S}\Loc_w^{\on{unr}}.
\end{equation}
\begin{lemma}\label{R: Cartesian local-to-global}
The commutative square in the following diagram is Cartesian
\begin{equation}
\xymatrix{
 \Loc_{{}^cG,F,S}\ar[r]\ar[d]& \prod_{v\in S} \Loc_{v} \times \Loc_{w_0}^{\ur}\ar[d]\ar[r]& \prod_{v\in S} \Loc_{v}\\
 \Loc_{{}^cG,F,S\cup\{w_0\}}\ar[r]& \prod_{v\in S} \Loc_{v} \times \Loc_{w_0}&.
}
\end{equation}
\end{lemma}
\begin{proof}
By nilcompleteness, it is enough to prove the diagram is Cartesian when evaluated at $m$-truncated animated $\bZ_\ell$-algebras $A$.
This is obviously when $A$ is classical. Then using the Postnikov tower and arguing as in Proposition \ref{P: trun value}, one reduces to compare the tangent spaces, which then is not difficult.
We leave the details to readers. (See \cite[\S 8]{SV} for an argument in a closely related context.)
\end{proof}

For every place $v\in S$, we choose a finite extension $L_v/F^t_v\widetilde{F}_v$ that is Galois over $F_v$. 
Let
\[
\Loc_{{}^cG,F,\{L_v\}}:= \Loc_{{}^cG,F,S}\times_{\prod_{v\in S} \Loc_{v}}\prod_{v\in S} \Loc_{{}^cG_v, L_v/F_v}.
\]
As $\Loc_{{}^cG_v, L_v/F_v}$ is open and closed in $\Loc_v$, the stack $\Loc_{{}^cG,F,\{L_v\}}$ is also open closed in $\Loc_{{}^cG,F,S}$. In particular, if $G$ is tamely ramified over $F$, we have the tame stack
$\Loc_{{}^cG,F,S}^{\ta}:=\Loc_{{}^cG,F,\{F^t_v\}}$.

\begin{prop}\label{P: quasicompact}
The stack $\Loc_{{}^cG,F,\{L_v\}}$ is quasi-compact, and $\Loc_{{}^cG,F,S}=\cup_{\{L_v\}} \Loc_{{}^cG,F,\{L_v\}}$.
\end{prop}
\begin{proof}
We can ignore the derived structure.
We denote by $W_{F,\{L_v\}}$ (resp. $\Ga_{F,\{L_v\}}$) the quotient of $W_{F,S}$ (resp. $\Ga_{F,S}$) by the closed normal subgroup generated by the (conjugacy classes of) subgroups $\{\Ga_{L_v}, v\in S\}$. 
For a fixed a faithful representation ${}^cG\to\GL_m$, the induced morphism ${}^{cl}\mR^{sc}_{W_{F,\{L_v\}},{}^cG}\to {}^{cl}\mR^{sc}_{W_{F,\{L_v\}},\GL_m}={}^{cl}\Loc^\Box_{{}^c(\GL_m),F,\{L_v\}}$ is a closed embedding. Therefore, it is enough to prove the proposition for $G=\GL_m$.

Now the decomposition \eqref{E: glob dec} gives a decomposition 
\[
\Loc_{{}^c(\GL_m),F,\{L_v\}}=\sqcup_\Theta \Loc^\Theta_{{}^c(\GL_m),F,\{L_v\}}
\] 
so it is enough to show that there are only finitely many such $\Theta$ appearing in the decomposition. Every such $\Theta$ gives a continuous semisimple representations $\bar\rho$ of $\Ga_{F,\{L_v\}}\to \GL_m(\overline\bF_\ell)$, which lifts to a semisimple representation $\rho$ in characteristic zero with finite determinant, by applying \cite[3.5]{de} to each irreducible factor $\bar\rho$. (Note that as $S$ is non-empty, Assumption (iii) of \cite[3.5]{de} is unnecessary.)
By the global Langlands correspondence for $\GL_m$ over function field proved by L. Lafforgue \cite{LL}, there are only finitely many such $\rho$ up to conjugacy. 
\end{proof}

\begin{remark}\label{R: everywhere unramified loc}
Note that we always require $S$ to be a non-empty finite set in the definition of $\Loc_{{}^cG,F,S}$ (to ensure continuous group cohomology coincides with the \'etale cohomology). This a priori excludes the stack of everywhere unramified Langlands parameters. However Lemma \ref{R: Cartesian local-to-global} allows us to recover such case as follows. Assume that the action of $\Ga_F$ on $\hat{G}$ factors through the unramified Galois group, i.e. the \'etale fundamental group $\pi_1(X)$ of the smooth projective curve $X$ over $\bF_q$ with fractional field $F$.
Let $S=\{v\}$ be one place of $X$. Then we define 
$$\Loc_{{}^cG, X}:=\Loc_{{}^cG,F,\emptyset}:=\Loc_{{}^cG,F,\{v\}}\times_{\Loc_{v}}\Loc_v^{\ur}=\Loc_{{}^cG,F,\{v\}}^\ta\times_{\Loc_v^\ta}\Loc_v^{\ur}.$$ 
This is independent of the choice of $v$. For example, if $X=\bP^1$, 
\begin{equation}\label{E: case of P1}
\Loc_{{}^cG,\bP^1}=\Loc_{{}^cG,F,\{\infty\}}\times_{\Loc_\infty}\Loc_\infty^{\ur}\cong \Loc_{0}^{\ur}\times_{\Loc^{\ta}_{{}^cG,F,\{0,\infty\}}}\Loc_{\infty}^{\ur}.
\end{equation}
Clearly, $\Loc_{{}^cG,X}$ is quasi-compact by Proposition \ref{P: quasicompact}. The notion of elliptic parameters still makes sense when $S=\emptyset$ and they still give isolated smooth points in the corresponding $\Loc_{{}^cG,X,\Ql}^\la$.
\end{remark}

At the end of this subsection, let us briefly mention the situation when $F$ is a number field, which is a joint work in progress with Emerton \cite{EZ}. We still have $\chi:\Ga_{F,S}\to \bZ_\ell^\times\times\Ga_{\widetilde F/F}$, where the first component is the inverse of the cyclotomic character. We regard it as a $\on{Spf}\bZ_\ell$-point of $\mR^c_{\Ga_{F,S},\bG_m\times\Ga_{\widetilde F/F}}$. Then similar to \eqref{E: globstackSpfZl}, we let
\[
\Loc_{{}^cG,F,S}^{\wedge,\Box}:= \mR_{\Ga_{F,S},{}^cG}^{c}\times_{\mR_{\Ga_{F,S},\bG_m\times\Ga_{\widetilde F/F}}^{c}}\bigl\{\chi\bigr\}, \quad \Loc_{{}^cG,F,S}^{\wedge}=\Loc_{{}^cG,F,S}^{\wedge,\Box}/\hat{G}^\wedge_\ell.
\]
We still have $\Loc^\wedge_{{}^cG,F,S}=\varinjlim_r \Loc_{{}^cG,F,S,r}$ where $ \Loc_{{}^cG,F,S,r}$ is the restriction of $\Loc_{{}^cG,F,S}$ to $\bZ/\ell^r$. 
However, the situation is more complicated for number fields.
First even $\Loc_{{}^cG,F,S,1}$ is in general not an algebraic stack, but is only an ind-stack. In addition, in the number field case we will not try to define a stack over $\bZ_\ell$ using Definition \ref{D: sc rep} as such object may not be reasonable.
Instead, we consider the global-to-local morphism
\[
\res: \Loc_{{}^cG,F,S}^{\wedge}\to \prod_{v\in S} \Loc_v^\wedge,
\]
where $\Loc_v^\wedge$ is as in \eqref{E: Locloc} if $v$ is not above $\ell$ and is the stack from \cite{EG} if $v$ is above $\ell$ (say ${}^cG=\GL_n$). Then in \cite{EZ} we will show that under an analogue of de Jong's conjecture, this morphism is representable. Such fact should be enough for many applications, e.g. to give a conjectural formula of cohomology of Shimura varieties.
Using this morphism, one can impose $\ell$-adic Hodge theoretic conditions (e.g. crystalline with certain fixed Hodge-Tate weights) at $v\mid \ell$ to cut out closed substacks inside $\Loc_{{}^cG,F,S}^{\wedge}$, which then will be $\ell$-adic formal stacks. These substacks then might admit extensions to algebraic stacks over $\bZ_\ell$, which would be the correct analogue of $\Loc_{{}^cG,F,S}$ in the number field case.

\section{Coherent sheaves on the stack of Langlands parameters}\label{S: coh on st}
In this section, we use the stacks of Langlands parameters to formulate some conjectures in the local and global Langlands correspondence.
We also survey some known results, which provide evidences of these conjectures.
In this section, $\Lambda$ will denote a noetherian commutative ring.  

Many categories appearing in this section will be $\Lambda$-linear stable $\infty$-categories (see \cite[Chap 1]{Lu2}.) For two objects $x_1,x_2$ in such a category $\mC$, their (derived) hom space is naturally a $\Lambda$-module, denoted by $\Hom_\mC(x_1,x_2)$ (or simply by $\Hom(x_1,x_2)$ if $\mC$ is clear from the context). Then original mapping space $\Map_\mC(x_1,x_2)$ is identified with $\tau^{\leq 0}\Hom_\mC(x_1,x_2)$. By abuse of notations, we will write $\End(x)$ for $\Hom(x,x)$, which is an object in $\mathbf{Alg}(\Mod_\Lambda)$, i.e. an $E_1$-algebra. (Note that we use the same notation to denote endomorphism monoid of $x$ in Example \ref{R: monoid, segal space}. We hope the concrete meaning of this notation will be clear from the context.)

\subsection{The category of representations of $G(F)$}
Let $F$ be a non-archimedean local field, with $\mO_F$ its ring of integers, $\kappa_F$ its residue field. We also fix a uniformizer $\varpi_F\in \mO_F$. Let $q=\sharp \kappa_F=p^r$. 
Let $G$ be a connected reductive group over $F$. 
Let $\Rep(G(F),\La)^\heart$ denote the abelian category of smooth representations of $G(F)$ on $\Lambda$-modules. It is a Grothendieck abelian category (with a set of generators given below).  For a closed subgroup $K\subset G(F)$, we similarly have $\Rep(K,\La)^\heart$.
We always denote by $\mathbf{1}$ the trivial representation.
Let 
$$
c\on{-ind}_K^{G(F)}: \Rep(K,\La)^\heart \to \Rep(G(F),\La)^\heart
$$ 
denote the usual compact induction functor, and write 
$$\delta_K:=c\on{-ind}^{G(F)}_K \mathbf{1}\cong C^\infty_c(G(F)/K,\La),$$ 
which is the space of $\Lambda$-valued locally constant functions on $G(F)/K$ with compact support, on which $G(F)$ acts by left translation. 

If $K$ is open, then $c\on{-ind}_K^{G(F)}$ is the left adjoint of the forgetful functor.
By the definition of smooth representations, the collection $\bigl\{\delta_K\bigr\}_K$ with $K$ open, form a set of generators of $\Rep(G(F),\La)^\heart$. 
We say an open compact subgroup $K$ of $G(F)$ is $\Lambda$-\emph{admissible} (or just admissible if $\Lambda$ is clear from the context) if the index of \emph{any} open subgroup of $K$ is invertible in $\Lambda$. Note that if $p$ is invertible in $\Lambda$, $\Lambda$-admissible open compact subgroups always exist. E.g. the  pro-$p$ Sylow subgroup $I(1)$ of an Iwahori subgroup (sometimes also called the prop-$p$ Iwahori subgroup) of $G(F)$ is $\Lambda$-admissible. On the other hand, every open compact subgroup is $\bQ$-admissible.
If $K$ is $\Lambda$-admissible, then $\delta_K$ is a projective object in
$\Rep(G(F),\Lambda)^\heart$.

Next, let $\Rep(G(F),\La)$ denote the (unbounded) $\infty$-derived category of $\Rep(G(F),\La)^\heart$ (\cite[1.3.5]{Lu2}). 
This category behaves quite differently depending on whether $p$ is invertible in $\Lambda$ or not. For our purpose, we assume that $p$ is invertible in $\Lambda$ throughout this section.
In this case $c\on{-ind}_K^{G(F)}$ is a $t$-exact functor. 
If $K$ is a $\Lambda$-admissible open compact subgroup, then $\delta_K$ is a compact object in $\Rep(G(F),\La)$. It follows that $\Rep(G(F),\La)$ is compactly generated, with a set of generators given by  $\bigl\{\delta_K\bigr\}_K$ with $K$ being $\La$-admissible.

\begin{remark}
If $F$ is of characteristic zero and $\Lambda$ is a field of characteristic $p$ (which is not the case we consider), then $\delta_{I(1)}$ itself is a compact generator of $\Rep(G(F),\La)$ (see \cite{Sc}). 
\end{remark}

In general if an open compact subgroup $K$ is not $\Lambda$-admissible, then $\delta_K$ may not be compact in $\Rep(G(F),\La)$. 
\begin{example}
If $G=\bG_m$, $K=\mO_F^\times$, and $\La=\bF_\ell$ where $\ell$ is a prime dividing $q-1$, then $\delta_K\simeq C_c(\bZ,\bF_\ell)$ is not compact in $\Rep(F^\times,\bF_\ell)$.
\end{example}
For several reasons (e.g. see Conjecture \ref{Conj: coh sheaf}), it is convenient to modify the category $\Rep(G(F),\La)$ to force $\delta_K$ to be compact for all open compact subgroups $\Lambda$. Namely, let 
$$
   \Rep_{\on{f.g.}}(G(F),\La)\subset \Rep(G(F),\La)
$$ 
be the full subcategory generated by these $\delta_K$ under finite colimits and retracts, and let
\begin{equation*}\label{E: ren rep cate}
\Ind\Rep_{\on{f.g.}}(G(F),\Lambda)
\end{equation*}
be its ind-completion. As every $\delta_K$ is $\Lambda$-flat, there is the natural equivalence $\Rep_{\on{f.g.}}(G(F),\La)\otimes_\Lambda \La'=\Rep_{\on{f.g.}}(G(F),\La')$ when changing the coefficient rings.
Tautologically, for any open compact subgroup $K\subset G(F)$, $\delta_K$ is compact in $\Ind\Rep_{\on{f.g.}}(G(F),\Lambda)$, and there is a colimit preserving functor 
\[
\Ind\Rep_{\on{f.g.}}(G(F),\Lambda)\to \Rep(G,\Lambda).
\]
If $\Lambda$ is a field of characteristic zero, this is an equivalence, as $\Rep(G,\Lambda)^\heart$ has finite global cohomological dimension by a result of Bernstein. In general when $\Lambda$ is regular noetherian, there is a natural $t$-structure on $\Ind\Rep_{\on{f.g.}}(G(F),\Lambda)$ and this functor induces an equivalence $\Ind\Rep_{\on{f.g.}}(G(F),\Lambda)^+\cong \Rep(G,\Lambda)^+$ when restricted to the bounded from below subcategories (w.r.t. the natural $t$-structure).
See \cite[\S 3.3.3]{HZ} for more detailed discussions.

For an open compact subgroup $K\subset G(F)$, we define the corresponding \emph{derived} Hecke algebra with $\Lambda$-coefficient as
\[
    H_{G,K,\La}:= (\End(\delta_K))^{\on{op}}.
\]
So $H_{G,K,\La}$ is an object in $\mathbf{Alg}(\Mod_\Lambda)$, i.e. an $E_1$-algebra.
Sometimes we omit $G$ or $\Lambda$ from the subscript, if they are clear from the context. Note that its zeroth cohomology
$$H^0H_K\cong C_c(K\backslash G(F)/K,\Lambda)$$ 
is just the usual Hecke algebra with $\Lambda$-coefficient, with algebra structure given by convolution product. In addition, as $\Lambda$-modules,
\begin{equation*}\label{E: Hk higher coh} 
   H_K\cong \bigoplus_{g\in K\backslash G/K} C^*(K\cap gKg^{-1},\Lambda),
\end{equation*}
where the right hand side is the (pro-finite) group cohomology of $K\cap gKg^{-1}$ with trivial coefficient $\Lambda$. In particular, if $K$ is $\Lambda$-admissible, then $H_{G,K,\La}$ concentrates in cohomological degree zero. 

\begin{remark}\label{R: full Hecke}
By choosing an invariant Haar measure on $G(F)$ by assigning the volume of one (and therefore every) pro-$p$ Iwahori subgroup to be $1$, one can define the usual full Hecke algebra $H_G$ of $G(F)$. Namely, the underlying space is 
$$\delta_{\{1\}}\simeq C_c^\infty(G(F),\Lambda),$$ with the multiplication given by the usual convolution.  
If $K$ is $\Lambda$-admissible, its volume $\on{vol}(K)$ is invertible in $\Lambda$ and therefore there is an idempotent $e_K= \frac{1}{\on{vol}(K)}\on{ch}_\Lambda$ of $H_G$ as usual, where $\on{ch}_\Lambda$ is the characteristic function of $\Lambda$. 
There is an equivalence of categories between $\Rep(G(F),\Lambda)^\heart$ and the abelian category of non-degenerate $H_G$-modules.
We have $\delta_K\cong H_Ge_K$ as left $H_G$-modules, and $H_{G,K}\cong e_K H_G e_K$.
\end{remark}

Let $\Mod_{H_K}$ denote the $\infty$-category of left $H_K$-modules. 
It follows from general nonsense that there is a pair of adjoint functors
\begin{equation*}\label{E: Hk rep adj}
  \delta_K\otimes_{H_K}(-):\Mod_{H_K} \rightleftharpoons \Rep(G(F),\Lambda): \Hom(\delta_K,-).
\end{equation*}
If $K$ is $\La$-admissible, then $W\mapsto \delta_K\otimes_{H_K}W$ is fully faithful. (It is fully faithful for any $K$ if we replace $\Rep(G(F),\Lambda)$ by $\Ind\Rep_{\on{f.g.}}(G(F),\Lambda)$.)

For two open compact subgroups $K_1$ and $K_2$ of $G(F)$, there is the $(H_{K_2}\times H_{K_1})$-bi-module
\begin{equation*}\label{E: bimod}
{}_{K_1}H_{K_2}:=\Hom(\delta_{K_1}, \delta_{K_2}).
\end{equation*}
Its degree zero cohomology is given by
\[
H^0({}_{K_1}H_{K_2})\cong C_c(G(F)/K_2)^{K_1}=: C_c(K_1\backslash G(F)/K_2), 
\]
which is the space of $(K_1\times K_2)$-invariant compactly supported functions on $G(F)$. If either $K_1$ and $K_2$ is $\Lambda$-admissible, then ${}_{K_1}H_{K_2}=H^0({}_{K_1}H_{K_2})$. 

Tautologically, under the above identification, the map $\iota_{K_1,K_2}:\delta_{K_1}\to \delta_{K_2}$ sending $\on{ch}_{K_1}\in\delta_{K_1}$ to $\on{ch}_{K_1K_2}\in \delta_{K_2}$ corresponds to $\on{ch}_{K_1K_2}\in C_c(K_1\backslash G(F)/K_2)$. On the other hand, 
\[
\on{Av}_{K_1,K_2}:\delta_{K_1}\to \delta_{K_2}, \quad (\on{Av}_{K_1,K_2}f)(g)=\int_{K_2}f(gk)dk.
\]
corresponds to $\on{vol}(K_2)\on{ch}_{K_1K_2}$.

Tautologically, there is a $G(F)$-module homomorphism
\begin{equation}\label{E: Hk bimod}
\delta_{K_1}\otimes_{H_{K_1}}{}_{K_1}H_{K_2}\to \delta_{K_2}.
\end{equation}
If $K_1\subset K_2$, and $K_2$ is a $\Lambda$-admissible open compact subgroup (so is $K_1$), then \eqref{E: Hk bimod} is an isomorphism. 
But this may not be the case in general. 

\begin{example} Let $G=\SL_2$, $K_2=K=\SL_2(\mO_F)$, and  $K_1=I$ the standard Iwahori subgroup. 
Let $\La=\bF_\ell$ with $\ell>2$ and $\ell\mid p+1$. Then $I$ is $\Lambda$-admissible, but $K$ is not. In this case, \eqref{E: Hk bimod} is not an isomorphism. In fact, $\delta_I\otimes_{H_I} {}_IH_K$ does not even concentrate in degree zero.
\end{example}

Let us briefly recall Whittaker modules.
Assume that $G$ is quasi-split over $F$ and $\Lambda$ is a noetherian $\bZ[1/p]$-algebra containing all $p$-power roots of unit (e.g. $\La=W(\overline\bF_\ell)$). A Whittaker datum of $G$ consists of the unipotent radical $U$ of an $F$-rational Borel subgroup of $G$, and a non-degenerate character $\psi: U(F)\to (U/[U,U])(F)\to \La^\times$. Given a Whittaker datum $(U,\psi)$, let 
$$\on{Whit}_{U,\psi}:=c\on{-ind}_{U(F)}^{G(F)}\psi\in\Rep(G(F),\Lambda)^\heart$$
be the corresponding Whittaker module.
We note that $\on{Whit}_{U,\psi}$ is not finitely generated as $G(F)$-module. However, it can be written as a filtered colimit of finitely generated projective objects in $\Rep(G(F),\Lambda)^\heart$ (\cite[Prop. 3]{Ro}). 

At the end of this subsection, we review some internal symmetries of $\Rep(G(F),\Lambda)$. First, recall that every topological group automorphism $c: G(F)\to G(F)$ induces an auto-equivalence of categories $c: \Rep(G(F),\Lambda)^\heart\to \Rep(G(F),\Lambda)^\heart$. Namely, if $V$ is a smooth representation of $G(F)$, we define a new representation ${}^cV$ such that ${}^cV=V$ as $\Lambda$-modules but with a new $G(F)$-action given by $G(F)\times {}^cV\to {}^cV, \ (g,v)\mapsto c^{-1}(g)v$. If $K$ is an open compact subgroup, then there is a canonical isomorphism
\[
  {}^c\delta_K\cong \delta_{c(K)},\quad f\mapsto cf, \ (cf)(x)=f(c^{-1}(x)).
\]

Applying this formalism to the action of $G_\ad(F)$ on $G(F)$ by inner automorphisms, we obtain an action of $G_\ad(F)$ on $\Rep(G(F),\Lambda)^\heart$.
Note that if $c=c_h$ is given by the conjugation by an element $h\in G(F)$, then there is a canonical isomorphism ${}^{c_h}V\cong V, \ v\mapsto hv$. It follows that the action of $G_\ad(F)$ on $\Rep(G(F),\Lambda)^\heart$ factors through the action of the Picard groupoid 
\[
\Tor_{Z_G}^0:=G_\ad(F)/G(F),
\] 
which extends to an action  
\begin{equation}\label{E: sym Rep}
\Tor_{Z_G}^0\times\Rep(G(F),\La)\to \Rep(G(F),\Lambda). 
\end{equation}
Note that $\Tor_{Z_G}^0$ can be identified with the Picard groupoid of $Z_G$-torsors over $F$ such that the induced $G$-torsor is trivial. It particular, the group of isomorphism classes of $\Tor_{Z_G}^0$ is
$$
E_G:=\pi_0\Tor_{Z_G}^0=G_\ad(F)/(G(F)/Z_G(F))\cong \ker(H^1(F,Z_G)\to H^1(F,G)),
$$ 
and the automorphism group of any object in $\Tor_{Z_G}^0$ is $Z_G(F)$. 

There is also the so-called cohomological duality functor $\bD^{\on{coh}}$ of $\Rep_{\on{f.g.}}(G(F),\Lambda)$,
\begin{equation}\label{E: hom duality}
\bD^{\on{coh}}:\Rep_{\on{f.g.}}(G(F),\Lambda)\to \Rep_{\on{f.g.}}(G(F),\Lambda)^{\on{op}},\quad  V\mapsto \Hom_{G(F)}(V,H_G), 
\end{equation}
where $H_G=C_c^\infty(G(F),\Lambda)$ is full Hecke algebra regarded as a bimodule over itself.

\subsection{The groupoids $\Wh_G$ and $\TS_G$}\label{S: Aux}
In the standard formulation of the local Langlands correspondence for a general reductive group, several auxiliary choices must be made. This is also true in our formulation, which we will discuss later. In this subsection, we will explain how to carefully select these auxiliary data. Compared to the existing literature, we will introduce some groupoids that keep track of the automorphisms of these data.
Readers who are primarily interested in quasi-split groups satisfying the condition $H^1(F,Z_G)=0$ (e.g. $G=\GL_n$) may largely skip this subsection.

Let $\Pin_G$ denote the variety of pinnings of $G$. I.e. for a classical $F$-algebra $A$, $\Pin_G(A)$ consists of triples $(B_A,T_A,e_A)$, where $B_A\subset G_A$ is a Borel subgroup, $T_A\subset B_A$ is a maximal torus and $e_A: U_A\to \bG_a$ is a homomorphism with $U_A$ being the unipotent radical of $B_A$, such that after some \'etale covering $A\to A'$ so that $G_{A'}$ is split, $e_A$ restricts to an isomorphism $U_\al\to \bG_a$ for every root subgroup corresponding to simple roots (with respect to $(B_A,T_A)$). Note that $\Pin_G$ is in fact a $G_\ad$-torsor. Its cohomology class $\al\in H^1(F,G_\ad)$ corresponds to the quasi-split inner form of $G$. In particular, $\Pin_G$ admit a rational points if and only if $G$ is quasi-split, in which case $\Pin_G(F)$ is a $G_\ad(F)$-torsor. So if $G$ is quasi-split, we can define the quotient groupoid
\begin{equation}\label{E: Wh qs}
\Wh_G:= \Pin_G(F)/G(F).
\end{equation}
Note that it is a $\Tor_{Z_G}^0$-torsor, so the set of its isomorphism classes $W_G=\pi_0\Wh_G$ is an $E_G$-torsor.

Our first goal of this subsection is to canonically attach a few objects  to $(B,T,e)\in \Pin_G(F)$ in a $G_\ad(F)$-equivariant way. 

First, if we choose a non-trivial additive character $\psi_0:F\to \La^\times$ (so in particular we will assume $\Lambda$ contains enough $p$-power roots of unit), there is a well-defined $G_\ad(F)$-equivariant map from $\Pin_G(F)$ to the set of Whittaker data of $G$, sending $(B,T,e)$ to $(U,\psi: U(F)\xrightarrow{e} F\xrightarrow{\psi_0} \La^\times)$, which induces a bijection between $W_G$ and the set of $G(F)$-conjugacy classes of Whittaker data. Thus there is a well-defined $\Tor_{Z_G}^0$-equivariant functor
\begin{equation}\label{E: pin to whit}
\frakW_{\psi_0}: \Wh_G\to \Rep(G(F),\Lambda),\quad  (B,T,e)\mapsto \on{Whit}_{U,\psi}.
\end{equation}

\begin{remark}\label{R: add char}
As $\frakW_{\psi_0}$ is needed in the formulation of our conjectures, we briefly discuss how it depends on the choice of $\psi_0$.
Given $\psi_0$ and $\psi'_0$, there is a unique $a\in F^\times$ such that $\psi'_0(-)=\psi_0(a-)$. Giving a pinning $(B,T,e)$, the two Whittaker modules $\on{Whit}_{U,\psi_0e}$ and $\on{Whit}_{U,\psi'_0e}$ are isomorphic if the image of $a$ under the map $F^\times\xrightarrow{\hat\rho} T_\ad(F)\to H^1(F,Z_G)$ is trivial, where $\hat\rho$ is the half sum of positive coroots of $G$.
So if $H^1(F,Z_G)$ is trivial, then $\frakW_{\psi_0}$ is independent of the choice of $\psi_0$ (up to isomorphism). In general, it at most depends on the image of $a$ in $F^\times/(F^\times)^2$.
In addition, in the local situation, we can always assume that the conductor of $\psi_0$ is $\mO_F$ (i.e. $\psi_0|_{\mO_F}=1$ but $\psi_0|_{\varpi_F^{-1}\mO_F}\neq 1$) to reduce the ambiguity to $\kappa_F^\times/(\kappa_F^\times)^2$. 
We also mention that it should be possible to formulate everything more canonically without referring to the choice of $\psi_0$ (and to allow $\Lambda$ not to contain enough $p$-power roots of unit), although we choose not to do so. 
\end{remark}

Next we construct a $G_\ad(F)$-equivariant map from $\Pin_G(F)$ to pairs $(I\subset K)$ consisting of an Iwahori subgroup $I$ and a special parahoric $K$ of $G(F)$. 
Denote by $\breve F$ the completion of a maximal unramified extension $F^\ur$ of $F$ as before, and let 
\begin{equation}\label{E: Kotmap}
\kappa_G: G(\breve F)\to \xch(Z_{\hat{G}}^{I_F})
\end{equation}
be the Kottwitz map (\cite[\S 7]{Ko2}).
We choose a pinned Chevalley group $(H,B_H,T_H,e_H)$ over $\bZ$ and an isomorphism $\eta:(H,B_H,T_H,e_H)_{\widetilde F}\simeq (G,B,T,e)_{\widetilde F}$. Then $K=\eta(H(\mO_{\widetilde F}))^{\Ga_{\widetilde F/F}}\cap \ker\kappa_G$, 
where the intersection is taken in $G(F)$, is a special parahoric, independent of the choice of $(H,B_H,T_H,e_H,\eta)$. Let $S\subset T$ be the maximal $F$-split torus. We may identify the apartment $A(G,S)$ (in the Bruhat-Tits building of $G$) with the real vector space spanned by the coweight lattice of $S$ using the special vertex $x\in A(G,S)$ corresponding to $\Lambda$. Then $I$ is the unique Iwahori whose corresponding alcove $\mathbf{a}$ contains $x$ and is contained in the finite Weyl chamber determined by $B$. 

\begin{remark}\label{R: very special parahoric}
The special parahoric $K$ constructed above is absolutely special in the sense 
that the corresponding vertex $x$ in the Bruhat-Tits building of $G$  remains special for every finite separable extension $F'/F$ (also see the end of \cite[\S 3]{CaSh}). 
In \cite[\S 6]{Z15}, a closely related notation is introduced: a special parahoric of $G$ is called very special if the corresponding vertex remains special for every unramified extension $F'/F$.  Clearly absolutely special parahorics are very special, and therefore exist only if $G$ is quasi-split by Lemma 6.1 of \emph{loc. cit.} On the other hand, for quasi-split $G$, the above construction gives a $G_\ad(F)$-conjugacy class of absolutely special parahorics. In fact, this construction gives all absolutely special parahorics by virtual of the following fact.
\begin{lemma}
All absolutely special parahorics are conjugate under the $G_\ad(F)$-action. 
\end{lemma}
This lemma generalizes the well-known fact that all hyperspecial parahorics in an unramified group $G$ are conjugate under $G_\ad(F)$. 
To prove the lemma, we may assume $G=G_\ad$ and is quasi-split absolutely simple. Then it easily follows from the classification. Note that, however, the lemma fails for very special parahorics. In fact, for odd ramified unitary group $\on{U}_{2m+1}$ (say $\cha\kappa_F\neq 2$), there are two conjugacy classes of very special parahorics, one with reductive quotient $\SO_{2m+1}$ and the other with reductive quotient $\Sp_{2m}$ (e.g. see \cite{Z15}). Only the former is absolutely special.  
\end{remark}

Let $\widetilde W=N_G(T)(\breve F)/\ker \kappa_T$ be the Iwahori-Weyl group of $G_{\breve F}$ with respect to $T_{\breve F}$, which fits into the short exact sequence
$1\to \xch(\hat{T}^{I_F})\to \widetilde{W}\to W_0\to 1$,
where as before $W_0$ is the finite Weyl group for $G_{\breve F}$. As the vertex $x$ corresponding to $\Lambda$ remains special for $G_{\breve F}$, it gives a splitting of the above sequence so one can write
\begin{equation}\label{E: dec IW group}
\widetilde{W}=\xch(\hat{T}^{I_F})\rtimes W_0.
\end{equation}
The alcove $\mathbf{a}$ also remains to be an alcove for $G_{\breve F}$ (corresponding an Iwahori subgroup $\breve I\subset G(\breve F)$), and determines the subgroup
\begin{equation}\label{E: length zero}
\Omega\cong N_{G(\breve F)}(\breve I)/\breve I\subset \widetilde{W}
\end{equation}
that fixes this alcove. 
It is well-known that the Kottwitz map \eqref{E: Kotmap} induces an isomorphism $\Omega\simeq \xch(Z_{\hat{G}}^{I_F})$. Therefore, every $\ga\in \xch(Z_{\hat G}^{I_F})$ can be uniquely written as
\begin{equation}\label{E: decomofbeta}
\ga=\la_\ga w_\ga, \quad \mbox{for } \la_\ga\in \xch(\hat{T}^{I_F}), \ w_\ga\in W_0.
\end{equation}

Let $H_I$ be the Iwahori Hecke algebra of $I$. Note that $I\cap T(F)$ is an Iwahori subgroup of $T$ so there is the corresponding Iwahori Hecke algebra $H_{T,I}$. Similarly we have the pro-$p$ Iwahori Hecke algebras $H_{I(1)}$ and $H_{T,I(1)}$. It is known\footnote{We thank Vigneras for pointing out this.} that as $G(F)$-representations,
$$\delta_{I}\cong \Ind_{B(F)}^{G(F)} \delta_{T, I}, \quad \delta_{I(1)}\cong \Ind_{B(F)}^{G(F)} \delta_{T, I(1)}$$ 
where $\delta_{T,I}$ and $\delta_{T,I(1)}$ are the representations of $T(F)$ compactly induced from its Iwahori and pro-$p$-Iwahori subgroup. (These isomorphisms are probably well-known if $\La=\bC$, and are implicitly contained in \cite[3.6, 6.2, 6.3]{Da} for general $\Lambda$ in which $p$ is invertible.) It follows that there are canonical maps of algebras
\begin{equation}\label{E: Bernstein subalg}
H_{T,I}\to H_I, \quad H_{T,I(1)}\to H_{I(1)},
\end{equation}
which (after taking $H^0$) are injective maps. They are nothing but the commutative subalgebras of the (pro-$p$) Iwahori Hecke algebra constructed by Bernstein. On the other hand, by writing 
\[
\delta_{I}=c\on{-ind}_K^{G(F)}\on{c-ind}_I^K\mathbf{1}, \quad \delta_{I(1)}=c\on{-ind}_K^{G(F)}\on{c-ind}_{I(1)}^K\mathbf{1},
\]
we obtain canonical maps 
\begin{equation}\label{E: fin Hecke}
H_{f}:=\End_{K} (\on{c-ind}_{I}^K\mathbf{1})^{\on{op}}\to H_{I}, \quad H_{f,(1)}:=\End_{K} (\on{c-ind}_{I(1)}^K\mathbf{1})^{\on{op}}\to H_{I(1)}.
\end{equation}

\begin{remark}\label{R: Unicty}
We note that, the Iwahori-Weyl group and the decomposition \eqref{E: dec IW group} \eqref{E: decomofbeta}, and the (pro-$p$) Iwahori Hecke algebra and \eqref{E: Bernstein subalg} \eqref{E: fin Hecke}, are canonically attached to an element in $W_G$. Indeed, if $(B_1,T_1,e_1)$ to $(B_2,T_2,e_1)$ are two pinnings in the same $G(F)$-conjugacy class, then a choice of $g\in G(F)$ that conjugates the first to the second induces isomorphisms between these data, and the isomorphisms are in fact independent of the choice of $g$.  
\end{remark}

\begin{remark}
It is interesting to know whether the map
$H_{T, I}\otimes_\Lambda H_f\to H_I$ of $\Lambda$-modules induced by  \eqref{E: Bernstein subalg} and \eqref{E: fin Hecke} is an isomorphism. This is well-known to be the case after taking $H^0$.
\end{remark}

Let us also mention the following result.
\begin{prop}\label{P: Whit to pairs}
Fix an additive character $\psi_0: F\to \La^\times$ with conductor $\mO_F$.
The assignment $(B,T,e)\mapsto (U,\psi)$ and $(B,T,e)\mapsto (I\subset K)$ induces a well-defined $E_G$-equivariant map $(U,\psi)\mapsto (I\subset K)$ from the set of $G(F)$-conjugacy classes of Whittaker data
 to the set of $G(F)$-conjugacy classes of pairs consisting of an absolutely special parahoric $K$ and an Iwahori $I\subset K$. This assignment is independent of the choice of $\psi_0$. If $(U,\psi)$ maps to $(I\subset K)$, 
then $\on{Whit}_{U,\psi}^K$ is a free $H^0H_K$-module of rank one (known as the Casselman-Shalika formula \cite{CaSh}), and $\on{Whit}_{U,\psi}^I\simeq M_{\on{asp}}$ is the antispherical module of $H^0H_I$ (i.e. the representation induced from the sign representation of $H_f\subset H^0H_I$).
\end{prop}

This finishes our discussion of quasi-split groups.  To discuss not necessarily split reductive groups, let us first notice that from a geometric point of view, it is more natural to consider the groupoid $(\Pin_G/G)(F)$ classifying liftings of $\Pin_G$ to $G$-torsors, which contains $\Pin_G(F)/G(F)$ as a subgroupoid. Note that $(\Pin_G/G)(F)$ is a neutral gerbe, or more precisely is a torsor under the Picard groupoid $\Tor_{Z_G}$ of $Z_G$-torsors over $F$ (and in particular is acted by $\Tor_{Z_G}^0\subset \Tor_{Z_G}$). Even if $G$ is not quasi-split so $\Pin_G(F)=\emptyset$, one can still consider the groupoid $(\Pin_G/G)(F)$, which might be non-empty. More precisely,
it is non-empty if and only if the class $\al\in H^1(F,G_\ad)$ can be lifted to a class to $H^1(F,G)$, in which case it is still a $\Tor_{Z_G}$-torsor. For many applications, however, this groupoid is still not large enough as often $\al$ cannot be lifted to a class in $H^1(F,G)$. So we will introduce a larger groupoid $\TS_G$, which is sufficient for most applications.

First, similar to the groupoid $\Tor_{Z, \iso_{F}}$ introduced in \S \ref{S: tori}, we let $\Tor_{G, \iso_{F}}$ be the groupoid of pairs $(\mE,\varphi)$ consisting of a $G$-torsor $\mE$ over $\breve{F}$ and an isomorphism $\varphi: \mE\simeq \sigma^*\mE$ of $G$-torsors. The set of its isomorphism classes is just the Kottwitz set $B(G)$ (\cite{Ko,Ko2}).
Given $b=(\mE,\varphi)$ in $\Tor_{G, \iso_{F}}$, one can define an $F$-algebraic group $G_b$ whose $A$-points (for classical $F$-algebra $A$) form the group of automorphisms of $(\mE\otimes_FA,\varphi\otimes 1)$ over $\breve F\otimes_FA$.
Kottwitz showed that over $\breve F$, $G_b$ is naturally isomorphic to a Levi subgroup of $G$. If it is isomorphic to $G$ over $\breve F$, in which case $G_b$ is naturally an inner form of $G$,  then $b$ is called basic.
The set of isomorphism classes of basic $b$ is denoted by $B(G)_{\on{bsc}}$. There is a fully faithful embedding from the category $\Tor_G$ of $G$-torsors over $F$ to $\Tor_{G,\iso_F}$ by sending $\mE\mapsto (\mE\otimes_F\breve{F}, \varphi=1\otimes\sigma)$.  This induces an embedding $H^1(F,G)\subset B(G)_{\on{bsc}}$.
Recall the following cohomological results of Kottwitz.
\begin{itemize}
\item For every $G$, the map \eqref{E: Kotmap} induces a map $\kappa_G: B(G)\to \xch(Z_{\hat{G}}^{\Ga_F})$ (still called the Kottwitz map), which restricts to a bijection $B(G)_{\on{bsc}}\cong \xch(Z_{\hat{G}}^{\Ga_F})$. 
\item The natural map $H^1(F,G)\to B(G)_{\on{bsc}}$ is a bijection if $G=G_\ad$ is of adjoint type.
\end{itemize}

Now we may regard $\Pin_G$ as an object in $\Tor_{G_\ad,\iso_{F}}$ via the embedding $\Tor_{G_\ad}\subset\Tor_{G_\ad,\iso_F}$, and consider the groupoid $\TS_G$ of liftings of $\Pin_G$ to an object in $\Tor_{G,\iso_F}$. Explicitly, an object of $\TS_G$ 
consists of $t=(b, B,T,e)$, where $b=(\mE,\varphi)\in \Tor_{G, \iso_{F}}$ is basic, and $(B,T,e)$ is a pinning of $G_b$.
A morphism between $t$ and $t'$ is an isomorphism between $b$ and $b'$ in $\Tor_{G,\iso_F}$ that induces an isomorphism $(G_b,B,T,e)\simeq (G_{b'},B',T',e')$. 
This groupoid is non-empty if and only if $\al$ can be lifted to an element in $\xch(Z_{\hat{G}}^{\Ga_F})$. This still might not always be possible.
For example, if $G=D^{\Nm=1}$ is the group of reduced norm $1$ elements in a quaternion algebra $D$ over $F$, then such extension does not exist. 
However, such lifting always exists if $G$ is quasi-split or if the center of $G$ is connected, in which case $\xch(Z_{\hat G}^{\Ga_F})\to \xch(Z_{\hat G_\s}^{\Ga_F})$ is surjective (where we recall $\hat{G}_\s$ denotes the dual group of $G_\ad$ so is the simply-connected cover of the derived group of $\hat{G}$).
If $\TS_G$ is non-empty, then it is a torsor under $\Tor_{Z_G,\iso_F}$ (so the set of its isomorphism classes $\pi_0\TS_G$ is a torsor under $B(Z_G)$). 
Note that if $G$ is quasi-split, then $\Wh_G\subset\TS_G$ and $\TS_G=\Wh_G\times_{\Tor_{Z_G}^0}\Tor_{Z_G,\iso_F}$.

Now we fix $t\in \TS_G$, and write $(G^*,B^*,T^*,e^*)$ for $(G_b,B,T,e)$. We can canonically identify the dual group $\hat{G}$ with the dual group $\hat{G}^*$. We have various objects attached to $(G^*,B^*,T^*,e^*)$  such as the Iwahori-Weyl group $\widetilde W^*=\xch(\hat{T}^{I_F})\rtimes W^*_0$ and the Iwahori-Hecke algebra $H_{I^*}$. 
The class of $b$ is an element $\beta\in B(G)_{\on{bsc}}\cong \xch(Z_{\hat{G}}^{\Ga_F})$ lifting the class $\al\in \xch(Z_{\hat{G}_\s}^{\Ga_F})$.
In addition, for every lifting $\ga$ of $-\beta$ along $\xch(Z_{\hat{G}}^{I_F})\twoheadrightarrow \xch(Z_{\hat{G}}^{\Ga_F})$, we obtain a canonically defined Iwahori-Hecke algebra $H_{I_\ga}$ of $G(F)$. Namely, if we further lift $\ga$ along
$N_{G^*(\breve F)}(\breve I^*)\twoheadrightarrow\Omega^*\cong \xch(Z_{\hat{G}}^{I_F})$ to an element $\tilde\ga$, we obtain an Iwahori subgroup $I_{\tilde\ga}$ of $G(F)$. In fact, using $\mE$ one may identify $G(F)\cong \{g\in G^*(\breve F)\mid \tilde\ga \sigma(g)\tilde\ga^{-1}=g\}$. Then $I_{\tilde\ga}=\{g\in I^*_{\breve F}\mid \tilde\ga\sigma(g)\tilde\ga^{-1}=g\}$. The corresponding Iwahori-Hecke algebra only depends on $\ga$, and therefore can be denoted by $H_{I_\ga}$.  

\subsection{Derived Satake isomorphism}\label{SS: DerSat}
We fix $\iota:\Ga_q\to \Ga_F^t$ so we have the stack $\Loc_{{}^cG,F,\iota}$ over $\bZ[1/p]$. 
In this subsection, we assume that $G$ is unramified. Then we have $\Loc_{{}^cG,F}^{\ur}\subset \Loc_{{}^cG,F,\iota}^{\ta}$. Let $\Lambda$ be a regular noetherian $\bZ[1/p]$-algebra. We use the same notations to denote the base change of these stacks to $\Lambda$. 
Our first conjecture can be regarded as the derived Satake isomorphism.\footnote{The author came up with this conjecture during conference on ``Modularity and Moduli Spaces" in Oaxaca, inspired by Emerton's hope to ``see" the action of derived Hecke algebra on the cohomology of modular curves (and general Shimura varieties), and encouraged by Feng's result on spectral Hecke algebra \cite{F}. See Remark \ref{R: derived global} for a discussion.} 

\begin{conjecture}\label{C: derSat}
For every hyperspecial subgroup $K$, there is a natural isomorphism of $\Lambda$-algebras
\[
   H_K\cong \bigl(\End_{\mO_{\Loc_{{}^cG,F,\iota}}}(\mO_{\Loc_{{}^cG,F}^{\ur}})\bigr)^{\on{op}},
\]
which reduces to the classical Satake isomorphism after taking $H^0$:
$$C_c(K\backslash G(F)/K,\Lambda)\cong H^0H_K\cong H^0\End_{\mO_{\Loc_{{}^cG,F,\iota}}}(\mO_{\Loc_{{}^cG,F}^{\ur}})\cong H^0\Gamma(\Loc_{{}^cG,F}^{\ur},\mO_{\Loc_{{}^cG,F}^{\ur}}).$$
In addition, this isomorphism is compatible with the isomorphism from Proposition \ref{P:comp iota12} for different choices of $\iota$.
\end{conjecture}

As $\Loc_{{}^cG,F,\iota}^{\ta}$ is an open and closed substack in $\Loc_{{}^cG,F,\iota}$, we may replace $\mO_{\Loc_{{}^cG,F,\iota}}$ by $\mO_{\Loc^{\ta}_{{}^cG,F,\iota}}$ in the above conjecture. 

\begin{remark}\label{R: derSat}
\begin{enumerate}
\item Note that this conjecture is non-trivial even if $\La=\bC$. It amounts to saying that $\End_{\mO_{\Loc_{{}^cG}}}(\mO_{\Loc_{{}^cG}^{\ur}})=\End_{\mO_{\Loc^{\widehat\un}_{{}^cG}}}(\mO_{\Loc_{{}^cG}^{\ur}})$ concentrates in degree zero. This can be deduced from Theorem \ref{T: REnd of Spr} below. But we invite readers to check it directly for $G=\GL_2$ to see its content. 

\item Geometric Langlands suggests that both $H_K$ and $\bigl(\End_{\mO_{\Loc_{{}^cG,F,\iota}}}(\mO_{\Loc_{{}^cG,F}^{\ur}})\bigr)^{\on{op}}$ admit natural commutative structures (making them $E_3$-algebras)\footnote{One possible way to see this (in equal characteristic) is taking the trace of the corresponding $E_3$-categories in the geometric Langlands.}, although we do not see how to construct such structures directly. If this is indeed this case, one might further expect that the isomorphism in the above conjecture respects the commutative structures.  Note that the existence of $E_3$-structure on $H_K$ would imply the cohomology ring $\oplus_i H^iH_K$ is graded commutative, which currently is only know under some assumption of the base ring $\Lambda$ (\cite{Ve}).

\item It would be interesting to formulate a mod $p$ derived Satake isomorphism (or even an integral derived Satake isomorphism) in this style. The non-derived version with integral coefficients appears in \cite{Z20}, whose formulation involves the Vinberg monoid of $\hat{G}$.
\end{enumerate}
\end{remark}

One can check this conjecture by hands when $G=T$ is an unramified torus.
\begin{prop}\label{P: derSat tori}
Conjecture \ref{C: derSat} holds for unramified tori.
\end{prop}
\begin{proof}
By \eqref{E: unramified}, we have
\[
\End_{\mO_{\Loc^{\ta}_{{}^cT,F,\iota}}}(\mO_{\Loc_{{}^cT,F}^{\ur}})\simeq \End_{ ({}^{cl}\mR_{\kappa_{\widetilde F}^\times, \hat{T}})^\sigma}\mO_{\{1\}} \otimes \Ga(\hat{T}/(\sigma-1)\hat{T},\mO).
\]
On the other hand, there is the canonical isomorphism $H_K\cong C^*(T(\kappa_F),\Lambda)\otimes H^0H_K$. Then the desired isomorphism follows from the classical Satake isomorphism 
$$\Ga(\hat{T}/(\sigma-1)\hat{T},\mO)\cong H^0H_K$$ 
and the canonical isomorphism (constructed below)
\begin{equation}\label{E: LD}
\Lambda[({}^{cl}\mR_{\kappa_{\widetilde F}^\times, \hat{T}})^\sigma]\simeq \Lambda T(\kappa_F),
\end{equation}
where we recall the l.h.s is the ring of regular functions of $({}^{cl}\mR_{\kappa_n^\times, \hat{T}})^\sigma$, and the r.h.s is the group ring of $T(\kappa_F)$. 

To construct \eqref{E: LD}, we first assume that $T$ is split, so $\sigma$ acts trivially on $\hat{T}$ and $\widetilde F=F$. Then
$$\Lambda[({}^{cl}\mR_{\kappa_F^\times, \hat{T}})^\sigma]= \Lambda[\xcoch(T)\otimes \kappa_F^\times],$$ 
and $\xcoch(T)\otimes \kappa_F^\times\cong T(\kappa_F)$, where $\xcoch(T)$ denote the cocharacter lattice of $T$ (defined over $\overline{F}$). Using the norm map $\Res_{\kappa_{\widetilde F}/\kappa_F} T_{\kappa_{\widetilde F}}\to T_{\kappa_F}$, the construction \eqref{E: LD} for general unramified tori reduces to the split case.
\end{proof}

\subsection{Coherent Springer sheaf}\label{SS: CohSpr}
In this subsection, we assume that $\widetilde F/F$ is tamely ramified. We define a (complex of) coherent sheaf on $\Loc_{{}^cG,F,\iota}^{\ta}$ and discuss some of its (conjectural) properties. Because its definition resembles that of the Springer sheaf, we refer to it as the coherent Springer sheaf\footnote{We learned this name from D. Ben-Zvi.}. 
As before, all stacks are base changed to $\Lambda$. 

 Recall the morphism $\pi^{\ta}: \Loc^{\ta}_{{}^cB,F,\iota}\to  \Loc^{\ta}_{{}^cG,F,\iota}$, and we write $\pi^{\un}:\Loc^{\un}_{{}^cB,F,\iota}\to  \Loc^{\ta}_{{}^cG,F,\iota}$. For $?=\ta$ and $\un$, let
\begin{equation*}\label{E: cohspr}
   \on{CohSpr}^?_{{}^cG,F,\iota}:=\pi^?_{*}\mO_{\Loc_{{}^cB,F,\iota}^{?}} \in \Coh(\Loc^{\ta}_{{}^cG,F,\iota}).
\end{equation*}   
Again, we recall all the functors are derived. We first notice the following property of $\on{CohSpr}^?_{{}^cG,F,\iota}$.

\begin{prop}\label{R: Self-duality}
The (complex of) coherent sheaf $\on{CohSpr}^?_{{}^cG,F,\iota}$ is a self-dual with respect to the Grothendieck-Serre duality on $\Loc^{\ta}_{{}^cG,F,\iota}$.
\end{prop}
\begin{proof}
By Proposition \ref{P: tame stack}  and Remark \ref{R: LocB}, $\Loc_{{}^cB,F,\iota}^{\ta}$ is quasi-smooth with trivial dualizing complex. The same is true for $\Loc_{{}^cB,F,\iota}^{\un}$.
Therefore, we may replace $\mO_{\Loc_{{}^cB,F,\iota}^{?}}$ by the dualizing complex $\omega_{\Loc_{{}^cB,F,\iota}^{?}}$ of $\Loc_{{}^cB,F,\iota}^{?}$ in the definition of $\on{CohSpr}^?_{{}^cG,F,\iota}$. The claim then follows as Grothendieck-Serre duality commutes with proper push-forward. 
\end{proof}

Our conjectures in \S \ref{SS: local-global} suggests that coherent Springer sheaves are related to patched modules from automorphic lifting theorems. As explained to us by Emerton, patched modules should always be (ordinary) maximal Cohen-Macaulay module over the (classical) deformation ring. This leads us to make the following conjecture (see also \cite[3.15]{BCHN} when $G$ is split and $\La=\bC$).
\begin{conjecture}\label{C: abel cat}
The complex $\on{CohSpr}^{?}_{{}^cG,F,\iota}$ is in the abelian category of coherent sheaves $\Coh(\Loc^{\ta}_{{}^cG,F,\iota})^\heart$.
\end{conjecture}

\begin{cor}\label{C: CM sheaf}
Assuming Conjecture \ref{C: abel cat}, then $\on{CohSpr}^{?}_{{}^cG,F,\iota}$ is a self-dual maximal Cohen-Macaulay sheaf on $\Loc^{\ta}_{{}^cG,F,\iota}$. In particular, it is finite locally free over the smooth locus of $\Loc^\ta_{{}^cG,F,\iota}$.
\end{cor}
Note that we regard $\on{CohSpr}^{\un}_{{}^cG,F,\iota}$ as a coherent sheaf on $\Loc^{\ta}_{{}^cG,F,\iota}$.
\begin{proof}
This follows from Proposition \ref{R: Self-duality}.
\end{proof}

\begin{example}
Assume that $G=\PGL_2$ so $\hat{G}=\SL_2$ and ${}^cG=\GL_2$. Then over $\La=\bZ[1/2q(q+1)]$, one can show that 
\[
\on{CohSpr}^{\un}_{{}^cG,F,\iota}\simeq \mO_{{}^{red}\Loc_{{}^cG,F,\iota}^{\un}}\oplus \mO_{\Loc_{{}^cG,F,\iota}^{\ur}}.
\] 
We refer to \cite{EZ} for more details.
\end{example}

\begin{remark}
It will be proved in \cite{HZ} that Conjecture \ref{C: abel cat} holds when $G$ is umramified and $\Lambda$ is a field of characteristic zero.
\end{remark}

We have the following conjecture.\footnote{
Let us comment on the history of this conjecture, to the best of our knowledge. Some form of this conjecture was first studied by Ben-Zvi, Helm, and Nadler a few years ago as a natural continuation of their previous work. Hellmann independently proposed a similar conjecture while exploring $p$-adic automorphic forms and $p$-adic Galois representations (see his article \cite{Hel} for more details). We arrived at these ideas while attempting to generalize the work in \cite{XZ} to the Iwahori level structure (see \S \ref{SS: local-global} for discussion). The emphasis on general coefficients in our formulation reflects our hope to understand the arithmetic implications of level raising and lowering within this framework. It is quite remarkable that different considerations have led to the study of the same object.}
\begin{conjecture}\label{C: REnd of Spr}   
   Let $G$ be quasi-split over $F$ with a pinning, and let $H_{I}$ (resp. $H_{I(1)}$) be the associated Iwahori (resp. pro-$p$ Iwahori) Hecke algebra (see Remark \ref{R: Unicty}).
   Then, there are natural isomorphisms of $\Lambda$-algebras
   $$H_{I}\cong (\End_{\mO_{\Loc^{\ta}_{{}^cG,F,\iota}}}\on{CohSpr}^{\un}_{{}^cG,F,\iota})^{\on{op}},\quad H_{I(1)}\cong (\End_{\mO_{\Loc^{\ta}_{{}^cG,F,\iota}}}\on{CohSpr}^{\ta}_{{}^cG,F,\iota})^{\on{op}},$$ 
   compatible with the isomorphism from Proposition \ref{P:comp iota12}, for different choices of $\iota$.
   In particular, there is a fully faithful embedding
   \[
       \Mod_{H_{I(1)}}\to \on{IndCoh}(\Loc^{\ta}_{{}^cG,F,\iota}),\quad M\mapsto \on{CohSpr}^{\ta}_{{}^cG,F,\iota}\otimes_{H_{I(1)}}M.
   \]
   In addition, the following diagrams should be commutative
   \[\xymatrix{
   H_{T,I}\ar^{\eqref{E: Bernstein subalg}}[r]\ar[d]            &H_I \ar[d]                                                          \\      
   (\End\on{CohSpr}^{\un}_{{}^cT,F,\iota})^{\on{op}}\ar[r] &(\End\on{CohSpr}^{\un}_{{}^cG,F,\iota})^{\on{op}}   
   }
   \]
   \[
   \xymatrix{
   H_{T,I(1)}\ar^{\eqref{E: Bernstein subalg}}[r]\ar[d]      &H_{I(1)} \ar[d] \\
     (\End\on{CohSpr}^{\ta}_{{}^cT,F,\iota})^{\on{op}}\ar[r] &(\End\on{CohSpr}^{\ta}_{{}^cG,F,\iota})^{\on{op}},
   }
   \]
   where bottom maps are induced by the morphism $\Loc_{{}^cB,F,\iota}^\ta\to \Loc_{{}^cT,F,\iota}^\ta$.
\end{conjecture}

Note that in the conjecture, when computing the endomorphisms, $ \on{CohSpr}^{\un}_{{}^cG,F,\iota}$ is still considered as a coherent sheaf on  $\Loc^{\ta}_{{}^cG,F,\iota}$, similar to the unramified case as in Conjecture \ref{C: derSat}.

\begin{remark}
The conjecture in particular implies that there should exist a natural morphism
\begin{equation}\label{E: stable center pro-p-Iw}
  Z_{{}^cG,F}^{\ta}:=H^0\Gamma(\Loc_{G,F,\iota}^{\ta},\mO)\to Z(H_{I(1)}),
\end{equation}
where $Z(H_{I(1)})$ is the center of $H_{I(1)}$, which should fit into the following commutative diagram
\begin{equation}\label{E: center tame}
\xymatrix{
Z_{{}^cG,F}^{\ta}\ar[r]\ar[d]& Z(H_{I(1)})\ar^-\cong[d]\\
(Z_{{}^cT,F}^{\ta})^{W_{\on{rel}}}\ar^-\cong[r]& (H_{T, I(1)})^{W_{\on{rel}}}.
}
\end{equation}
Here $T$ denotes the abstract Cartan of $G$ (e.g. see \cite[1.4]{Z20} for the meaning), and $W_{\on{rel}}$ is the relative Weyl group of $G$. The left vertical map is from \eqref{E: st center levi}. (Note that $W_{\on{rel}}\cong W_{{}^cG,{}^cT}$.) The right vertical isomorphism comes from \cite[5.1]{Vi}, and the bottom isomorphism is induced by Conjecture \ref{C: REnd of Spr} for tamely ramified tori (in this case $\on{CohSpr}^\ta_{{}^cT,F,\iota}\cong \mO_{\Loc^{\ta}_{{}^cT,F,\iota}}$). 
\end{remark}

We mention that proof of Proposition \ref{P: derSat tori} already verifies the conjecture for unramified tori. In addition,
in a forthcoming work with Hemo (\cite{HZ}), we will prove the following result. 
\begin{theorem}\label{T: REnd of Spr}
Let $\La=\overline\bQ_\ell$. 
Assume that $G$ is unramified with a pinning $(B,T,e)$ and let $(U,\psi)$ and $I\subset K$ be associated to $(B,T,e)$ as in Proposition \ref{P: Whit to pairs}.
Then there is a natural isomorphism
\begin{equation}\label{E: endo of CohSprunip}
H_{I}\cong \End_{\mO_{\Loc^{\ta}_{{}^cG,F}}}\on{CohSpr}_{{}^cG,F}^{\un}
\end{equation}
inducing a fully faithful embedding
\[
\Mod_{H_I}\to \on{IndCoh}(\Loc_{{}^cG,F}^{\ta}),\quad M\mapsto \on{CohSpr}_{{}^cG,F}^{\un}\otimes_{H_I}M.
\]
This functor sends
\begin{itemize}   
 \item the antispherical module $M_{\on{asp}}$  of $H_I$ (see Proposition \ref{P: Whit to pairs}) to  $\mO_{\Loc_{{}^cG,F}^{\widehat\un}}$.
 \item ${}_IH_K$ to $\mO_{\Loc_{{}^cG,F}^{\ur}}$. In particular, Conjecture \ref{C: derSat} holds when $\La=\overline\bQ_\ell$.
\end{itemize}
\end{theorem}

The theorem in fact follows from Theorem \ref{T: tame local Langlands} stated below.
We remark that Hellmann has obtained partial results in this direction (see \cite{Hel}). In addition,
Ben-Zvi-Chen-Helm-Nadler also proved the isomorphism \eqref{E: endo of CohSprunip} when the group $G$ is split (\cite{BCHN}).

We end up this subsection by discussing the relation between $\on{CohSpr}^{\ta}_{{}^cG,F,\iota}$ and $\on{CohSpr}^{\un}_{{}^cG,F,\iota}$ when $G$ is unramified.
First in this case as we just mentioned, by (the proof of) Proposition \ref{P: derSat tori}, the group algebra $\La T(\kappa_F)\subset H_{T,I(1)}$ acts on $\on{CohSpr}^{\ta}_{{}^cG,F,\iota}$.
\begin{lemma}
There is a natural isomorphism $\on{CohSpr}^{\ta}_{{}^cG,F,\iota}\otimes_{\La T(\kappa_F)} \La\cong \on{CohSpr}^{\un}_{{}^cG,F,\iota}$,
where $\La T(\kappa_F)\to \La$ is the augmentation map.
\end{lemma}
\begin{proof} By (the proof of) Proposition \ref{P: derSat tori}, the right square in the following diagram is Cartesian
\[
\xymatrix{
\Loc_{{}^cB,F,\iota}^{\un}\ar[r]\ar[d] & \Loc_{{}^cT,F,\iota}^{\ur}\ar[r] \ar[d] & \{1\}\ar[d]\\
\Loc_{{}^cB,F,\iota}^{\ta}\ar[r]  & \Loc_{{}^cT,F,\iota}^{\ta}\ar[r] &({}^{cl}\mR_{\kappa_n^\times, \hat{T}})^\sigma.
}\]
The left square is also Cartesian by the definition (see \eqref{E: unip locsys B}). So 
\[
\mO_{\Loc_{{}^cB,F,\iota}^{\un}}=\mO_{\Loc_{{}^cB,F,\iota}^{\ta}}\otimes_{\Lambda[({}^{cl}\mR_{\kappa_n^\times, \hat{T}})^\sigma]}\La=\mO_{\Loc_{{}^cB,F,\iota}^{\ta}}\otimes_{\La T(\kappa_F)}\La.
\]
As the push-forward along $\pi^{\ta}$ commutes with colimits, the lemma follows.
\end{proof}

\subsection{Conjectural coherent sheaves}\label{SS: coh sheaf}
With the conjectures in the previous two subsections in mind, it is natural to go one step further to conjecture that for every open compact subgroup $K\subset G(F)$, there is a coherent sheaf $\frakA_{G,K}$ on $\Loc_{{}^cG,F,\iota}$, whose (opposite) endomorphism algebra $\End \frakA_{G,K}$ in $\Coh(\Loc_{{}^cG,F,\iota})$ in $H_K$. The goal of this subsection is to formulate the conjecture precisely.\footnote{When $G$ is split, a closely related conjecture also appeared in \cite{Hel}.} We fix once for all an additive character $\psi_0: F\to \La^\times$ with conductor $\mO_F$. (See Remark \ref{R: add char} for the discussion of the dependence on this choice.) All stacks are base changed to $\Lambda$.

Recall our convention of the category of coherent sheaves on $\Loc_{{}^cG,F,\iota}$ in Remark \ref{R: Coh cat}. Recall the decomposition of this category \eqref{E: gerbe dec}.
It is acted by $\Tor_{Z_G,\iso_F}$ via \eqref{E: BZ action}, and therefore each direct summand is acted by $\Tor_{Z_G}^0\subset \Tor_{Z_G,\iso_F}$.   
On the other hand, $\Tor_{Z_G}^0$ also acts on $\Rep(G(F),\Lambda)$ as in \eqref{E: sym Rep}. Recall the $\Tor_{Z_G}^0$-torsor $\Wh_G$ if $G$ is quasi-split and the $\Tor_{Z_G,\iso_F}$-torsor $\TS_G$ for general $G$ from \S \ref{S: Aux}. 
In addition, recall that if $\widetilde F/F$ is tame, we have the spectral Deligne-Lusztig stacks \eqref{E: SpDL} and \eqref{E: SpDLunip}. We will also use the following notation. 

\begin{notation}\label{N: line bundle}
Note that every weight $\la\in\xch(\hat{T}^{\bar\tau})$ gives a line bundle on $\hat{T}\bar\tau/\hat{T}$, and therefore a line bundle on
$\widetilde{\Loc}_{{}^cG,F,\iota}^{\ta}$ by pullback along 
$\widetilde{\Loc}_{{}^cG,F,\iota}^{\ta}\xrightarrow{\pr}\hat{B}\bar\tau/\hat{B}\to \hat{T}\bar\tau/\hat{T}$.
We denote this line bundle by $\mO(\la)$. If $\mF$ is a (complex of) coherent sheaf on $\widetilde{\Loc}_{{}^cG,F,\iota}^{\ta}$, we write $\mF(\la)$ for $\mF\otimes\mO(\la)$ for simplicity.
\end{notation}

\begin{conjecture}\label{Conj: coh sheaf}
We fix $t\in \TS_G$, and let $\beta\in \xch(Z_{\hat{G}}^{\Ga_F})$ be the element determined by $t$.
\begin{enumerate}
\item\label{Conj: coh sheaf-1} 
There is a $\Tor_{Z_G}^0$-equivariant fully faithful embedding 
$$\frakA_G: \Rep_{\on{f.g.}}(G(F),\Lambda)\to \Coh^{-\beta}(\Loc_{{}^cG,F,\iota}),$$
compatible with the isomorphism in Proposition \ref{P:comp iota12} for different choices of $\iota$. 
There should be a natural isomorphism of functors
\[
 \frakA_G\circ \bD^{\on{coh}}\cong {}'\bD^{\on{Se}}\circ \frakA_G: \Rep_{\on{f.g.}}(G(F),\Lambda)\to \Coh^{-\beta}(\Loc_{{}^cG,F,\iota}),
\]
where $\bD^{\on{coh}}$ is from \eqref{E: hom duality} and ${}'\bD^{\on{Se}}$ is from \eqref{E: Cartan inv}.

\item\label{Conj: coh sheaf-2} 
The induced colimit preserving functor $\Ind\Rep(G(F),\Lambda)\to \Ind\Coh(\Loc_{{}^cG,F,\iota})$ is still denoted by $\frakA_G$. If $\beta=0$ (so in particular $G$ is quasi-split with a pinning), then 
$$\frakA_{G} (\on{Whit}_{U,\psi})\cong \mO_{\Loc_{{}^cG,F,\iota}},$$ 
where $\on{Whit}_{U,\psi}$ is the Whittaker module determined by the pinning (see \eqref{E: pin to whit}).

For every open compact subgroup $K$ of $G(F)$, let $\frakA_{G,K}:=\frakA_G(\delta_K)$. 
Then $\frakA_{G,K}$ should belong to $ \Coh(\Loc_{{}^cG,F,\iota})^\heart$.
Let
$$\frakA_{G,\{1\}}:=\frakA_G(\delta_{\{1\}})\simeq \frakA_G(\varinjlim_K\delta_K)=\varinjlim_K \frakA_{G,K}.$$
Then it is an ordinary ind-coherent sheaf on $\Loc_{{}^cG,F,\iota}$, equipped with an action of $G(F)$ (as $\delta_{\{1\}}$ is a $G(F)\times G(F)$-representation via the left and right regular representation).  
Then the restriction of $\frakA_{G,\{1\}}$ to each connected component $D$ of $\Loc_{{}^cG,F,\iota}$ should be finitely generated over $\mO_D[G(F)]$.

\item\label{Conj: coh sheaf-3} Assume that $G$ splits over a tamely ramified extension $\widetilde F/F$. 
Let $\gamma$ be a lifting of $-\beta$ to $\xch(Z_{\hat{G}}^{I_F})$, and write $\gamma=w_{\ga}\la_{\ga}$ as in \eqref{E: decomofbeta}. Let $I_\ga$ (resp. $I_\ga(1)$) be the corresponding Iwahori (resp. pro-$p$ Iwahori) subgroup. Then
\[
\frakA_G(\delta_{I_\ga(1)})\simeq \widetilde\pi_*\mO_{\Loc_{{}^cG,F,\iota}^{\ta,w_{\ga}}}(\la_{\ga}), \quad \frakA_G(\delta_{I_\ga})\simeq \widetilde\pi_*\mO_{\Loc_{{}^cG,F,\iota}^{\un,w_{\ga}}}(\la_{\ga}).
\]
If $G=G^*$ is unramified and $K$ is the hyperspecial subgroup determined by $t$, then 
$$\frakA_{G,K}\simeq \mO_{\Loc_{G,F}^{\ur}}.$$

\item\label{Conj: coh sheaf-1/2}  Let $P\subset G$ be a rational parabolic subgroup and $M$ its Levi quotient.
The functor $\frakA_M$ and $\frakA_G$ should also be compatible with parabolic induction in the representation side and spectral parabolic induction from Proposition \ref{P: spec Eis}.
\end{enumerate}
\end{conjecture}
We will discuss how the functor $\frakA_G$ depends on the choice of $t\in \TS_G$ below. 
But let us first make Part \eqref{Conj: coh sheaf-3} of the conjecture more explicit in some cases.
\begin{example}\label{E: cohspr as special}
Assume that $G=G^*$ and is tamely ramified and $\beta=0$.
We take $\ga=0\in \xch(Z_{\hat{G}}^{I_F})$ so $\la_{\ga}=0$ and $w_{\ga}=1$. In this case Part \eqref{Conj: coh sheaf-3} of the conjecture says that
\[
\frakA_{G,I(1)}\simeq \on{CohSpr}^\ta_{{}^cG,F,\iota},\quad \frakA_{G,I}\simeq \on{CohSpr}^\un_{{}^cG,F,\iota},
\]
which is consistent with Conjecture \ref{C: REnd of Spr}. In addition, the expected commutative diagrams in Conjecture \ref{C: REnd of Spr} are also consistent with Part \eqref{Conj: coh sheaf-1/2}. 
\end{example}
\begin{example}
Let $G=D^\times/F^\times$, where $D$ is a degree $n$ central division algebra over $F$ of invariant $1/n$. Then $G$ is an inner form of $\PGL_n$ so $\hat{G}=\SL_n$. Note that 
$$\ga=-\beta=-\alpha=1/n\in \xch(Z_{\hat{G}})\cong \bZ/n.$$
Let $w=(12\cdots n)\in W=S_n$ be the cyclic permutation.
Let $\omega_i: \hat{T}\to \bG_m$ be the $i$th fundamental weight $\hat{T}$. 
Then
\[
\frakA_{G,I(1)}\simeq\widetilde\pi_* \mO_{\widetilde{\Loc}_{{}^cG,F,\iota}^{\ta,w}}(\omega_1),\quad \frakA_{G,I}\simeq\widetilde\pi_* \mO_{\widetilde{\Loc}_{{}^cG,F,\iota}^{\un,w}}(\omega_1).
\]
One can show that when $D$ is a quaternion algebra over $F$ and $\La=\bZ_\ell$ with $\ell>2$ and $\ell\mid q-1$, the completion of $\frakA_{G,I}$ at the point of $\Loc_{{}^cG,F,\iota}^\ta$ given by the trivial representation coincides with a module over the local deformation ring studied by Manning \cite{Man}. We refer to \cite{EZ} for more discussions.
\end{example}

\begin{remark}\label{R: CM self dual}
Part \eqref{Conj: coh sheaf-1} and \eqref{Conj: coh sheaf-2} of the conjecture would imply that $\frakA_{G,K}$ is a maximal Cohen-Macaulay sheaf. If $\beta=0$ (so $G$ is quasi-split), we further conjecture that it is self-dual with respect to the usual (a.k.a. non-modified) Grothendieck-Serre duality.
See Corollary \ref{C: CM sheaf} for the case of coherent Springer sheaves.
\end{remark}

\begin{remark}\label{R: EH sheaf}
We let $\La=W(\overline\bF_\ell)$.
When $G=\GL_n$, the sheaf $\frakA_{\GL_n,\{1\}}$ should be isomorphic to the Emerton-Helm sheaf $\frakA_{\on{EH}}$ interpolating local Langlands correspondence for $\GL_n$ in families (see \cite{EH, He16, HM1, HM, Hel} for the constructions and in particular \cite{Hel} for a discussion of this point).
On the other hand, inspired by a conjecture of Braverman-Finkelberg in the geometric Langlands (\cite{BF}), we have the following conjectural description of $\frakA_{\GL_n,\{1\}}$. Consider the derived stack $\mW_n$ classifying chains $\{V_1\to V_2\to\cdots\to V_n\}$, where $V_i$ is an $i$-dimensional representation of $W_{F}$ (i.e. $V_i\in \Loc_{{}^c\GL_i,F}$). There is a natural morphism $\pi: \mW_n\to \Loc_{{}^c\GL_n,F}$ by only remembering $V_n$. Then the arithmetic analogue of Braverman-Finkelberg's conjecture is
\[
\frakA_{\GL_n,\{1\}}\cong \frakA_{\on{BF}}:=\pi_*\omega_{\mW_n}.
\]
Combining these two conjectural descriptions of $\frakA_{\GL_n,\{1\}}$, we arrive at the following conjecture.
\end{remark}

\begin{conjecture}\label{C: EH=BF}
There is a natural isomorphism between $\frakA_{\on{EH}}$ and $\frakA_{\on{BF}}$ as quasi-coherent sheaves on $\Loc_{{}^c\GL_n,F}$.
\end{conjecture}

\begin{remark}\label{R: dep on t}
To discuss the dependence of $\frakA_G$ on $t$, we write it by $\frakA_G^{t}$ in this remark.
If $\theta\in \Tor_{Z_G,\iso_F}$ that sends $t_1\in \TS_G$ to $t_2\in \TS_G$, then there should exist a canonical isomorphism of functors 
\begin{equation}\label{E: change local normalization}
   \frakA^{t_2}_{G}(-)\simeq \frakA^{t_1}_{G}(-)\otimes\mL_\theta,
\end{equation}
where $\mL_\theta$ is as in Conjecture \ref{C: LL for group of mult type}. 
More precisely,
there should exist a $\Tor_{Z_G,\iso_F}$-equivariant exact fully faithful functor 
\[
\frakA_G: \Rep_{\on{f.g.}}(G(F),\Lambda)\times^{\Tor_{Z_G}^0}\TS_G\to \Coh(\Loc_{{}^cG,F,\iota}).
\]
If $G$ is quasi-split, $\frakA_G$ is induced from a canonical fully faithful functor 
$$\Rep_{\on{f.g}}(G(F),\Lambda)\times^{\Tor_{Z_G}^0}\Wh_G\to \Coh(\Loc_{{}^cG,F,\iota}).$$ 
\end{remark}

Let us record the following consequence of the conjecture. Recall the stable center $Z_{{}^cG,F}$ as in \eqref{E: st center}, and the Hecke algebra $H_G$ of $G$ as in Remark \ref{R: full Hecke}. Let $Z_{G,F}:=Z(H_G)$ denote the center of $H_G$ (the Bernstein center of $G(F)$).
\begin{cor}
Assuming the conjecture, there exists a natural map 
\begin{equation}\label{E: stcen to cen}
Z_{{}^cG,F}\to Z_{G,F},
\end{equation} 
independent of the choice  of $t\in \TS_G$. In addition, this map should be compatible with parabolic induction (which would in particular imply \eqref{E: center tame}).
For a connected component $D$ of $\Loc_{{}^cG,F,\iota}$, let
$Z_{{}^cG,F,D}$ and $Z_{G,F,D}$ be the corresponding idempotent components. Then $Z_{G,F,D}$ is finite over $Z_{{}^cG,F,D}$. 
If $G=G^*$, then \eqref{E: stcen to cen} is split injective.
\end{cor}

\begin{remark}
In the case of $\GL_n$ over a $p$-adic field and $\La=\overline\bQ$, the map in the corollary is constructed earlier by Scholze \cite{Schllc}. Using the local Langlands for $\GL_n$, such map is constructed by Helm and Helm-Moss \cite{He16, HM1, HM} for $\La=\overline\bZ_\ell$. Note that for $\GL_n$, \eqref{E: stcen to cen} is an isomorphism. For general $G$, a map from the excursion algebra (see Remark \ref{R: point coarse moduli}) to $Z_{G,F}$ is constructed by Genestier-Lafforgue \cite{GL} (in equal characteristic and after $\ell$-adic completion). The map \eqref{E: stcen to cen} in general  (for $\La=\bZ_\ell$) appears in the work of Fargues-Scholze \cite{FS}, without the construction of $\frakA_G$. Then finiteness of $Z_{{}^cG,F}\to Z_{G,F}$ (when restricted to each component $D$ of $\Loc_{{}^cG,F,\iota}$) is proved recently in \cite{DHKM2}.
\end{remark}

\begin{remark}\label{E: A for torus}
If $G=T$ is a torus, the existence of \eqref{E: stcen to cen} should follow from Conjecture \ref{C: LL for tori},
which in turn would induce the functor 
$$\Rep(T(F),\Lambda)\cong \Mod_{Z_{{}^cT,F}}\subset \on{Qcoh}(\Loc_{{}^cT,F,\iota}),$$ sending $\Rep_{\on{f.g.}}(T(F),\Lambda)$ to $\Coh(\Loc_{{}^cT,F})$. This should be the desired functor $\frakA_T$. 
\end{remark}

Unfortunately, we do not have explicit conjectural descriptions of $\frakA_{G,K}$ in general at the moment. Here are some expectations and remarks.
\begin{enumerate}
\item  We expect that if $K$ is the pro-unipotent radical of a parahoric subgroup, then $\frakA_{G,K}$ is supported on $\Loc_{{}^cG,F,\iota}^\ta$. In particular, there should exist a map
\begin{equation}\label{E: stable center K1}
 Z_{{}^cG,F}^{\ta}\to Z(H_{G, K}).
\end{equation}
generalizing \eqref{E: stable center pro-p-Iw}.

\item Assume that $G$ is quasi-split. We expect that for a cofinal set of open compact subgroups $K\subset G(F)$, there exist a quasi-smooth derived stack $\widetilde{\Loc}_{{}^cG,F,\iota}^K$ and a proper schematic morphism $\pi^K: \widetilde{\Loc}_{{}^cG,F,\iota}^K\to \Loc_{{}^cG,F,\iota}$ such that 
\[
\frakA_{G,K}\cong\pi_!^K\mO_{\widetilde{\Loc}_{{}^cG,F,\iota}^K}\cong \pi_!^K\omega_{\widetilde{\Loc}_{{}^cG,F,\iota}^K}.
\]
Note that this would in particular imply that $\frakA_{G,K}$ is self-dual with respect to the Grothendieck-Serre duality (see Remark \ref{R: CM self dual}).  

\item Using the fact that some connected component of $\Loc_{{}^cG,F,\iota}$ ``looks like" the tame stack of local Langlands parameters for another group (see the proof of Proposition \ref{P: tame and wild stack}), it might be possible to relate the restriction of $\frakA_G$ to this component with the coherent Springer sheaf of the other group. For $G=\GL_n$, this might give a construction of $\frakA_G$ ``by hand". We refer to \cite{BCHN} for an approach along this line.

\item Even if we understand $\{\frakA_{G,K}\}_K$ for various $\Lambda$ (so knowing that the functor $\frakA_G$ is well-defined), it is still important (and sometimes challenging) to understand the (ind)-coherent sheaves on $\Loc_{{}^cG,F,\iota}$ corresponding to specific $G(F)$-representations. To give an example, let $X$ be a $G$-variety over $F$. Then $C_c(X(F))$ is a natural $G(F)$-representation, and therefore should correspond to an ind-coherent sheaf $\frakA_{X}:=\frakA_G(C_c(X(F)))$ on $\Loc_{{}^cG,F,\iota}$. The recent conjectures of Ben-Zvi-Sakellaridis-Venkatesh in relative Langlands program should have analogue in the current setting, giving conjectural construction of $\frakA_{X}$ (for some $X$) purely from the Galois side (at least for $\Lambda$ being a field of characteristic zero).
\end{enumerate}

\subsection{Categorical arithmetic local Langlands correspondence}\label{SS: cat LLC}
In this subsection, we explain how the conjectural sheaf $\frakA_{G}$ fits into a hypothetical categorical form of the local Langlands conjecture. More detailed discussions will appear in \cite{HZ}. Let
$\Lambda$ be over $\bZ_\ell$ where $\ell\neq p$. For simplicity, we write $\Loc_{{}^cG}$ for $\Loc_{{}^cG,F}\otimes_{\bZ_\ell} k$ in this subsection. We fix a non-trivial character $\psi_0: F\to \La^\times$ with conductor $\mO_F$.

A general wisdom shared among various people is that in local Langlands it is better not to just study representation theory of a single $p$-adic group $G$, but simultaneously to study representation theory of a collection of groups closely related to $G$. There are various ways to formulate the idea precisely by appropriately choosing such collection, such as Vogan's pure inner forms, Kottwitz-Kaletha's extended pure inner forms, or Kaletha's rigid inner forms. It should be clear from previous discussion that the collection $\{G_b, b\in B(G)_{\on{bsc}}\}$, i.e. extended pure inner forms of $G$, is most relevant to us. But
it turns out one can go one step further to consider the representation theory of $G_b$ (for all $b\in B(G)$) altogether. The representation categories of these groups glue nicely together to a category which is conjecturally equivalent to the category of (ind-)coherent sheaves on $\Loc_{{}^cG}$, as we now explain.

The basic idea is that these representation categories glue to the category of sheaves on some stack. Indeed, individual $\Rep(G_b(F),\Lambda)$ can be thought as the category of sheaves with $\Lambda$-coefficient on the classifying stack $[*/G_b(F)]$ of the locally profinite group $G_b(F)$ in appropriate sense. 
Note that $B(G)$ underlies the category $\Tor_{G,\iso_F}$ (as introduced in \S \ref{S: Aux}), and the automorphism group of every $b\in \Tor_{G,\iso_F}$ is $G_b(F)$.
Then it is natural to expect $B(G)$ is the set of $\overline{\kappa}_F$-points of some stack, whose automorphism group $\Aut_b$ at $b$ is $G_b(F)$ (or some closely related group),
so the sought after glued category is the category of sheaves $\Shv(B(G),\Lambda)$ on this stack in appropriate sense. In particular, for each $b\in B(G)$, there should exist a pair of adjoint functors
\begin{equation}\label{E: emb Jb}
  i_{b,!}: \Rep(G_b(F),\Lambda)\cong \Shv([*/\Aut_b],\Lambda)\rightleftharpoons \Shv(B(G),\Lambda): i^!_{b}
\end{equation}
where $i_b: [*/\Aut_b]\to B(G)$ is the corresponding embedding.

As far as we know, there are two ways to make this idea precise.
One is due to Fargues-Scholze. In this approach, $B(G)$ is regarded as the set of points of the $v$-stack $\Bun_G$ of $G$-bundles on the Fargues-Fontaine curve and $\Shv(B(G),\Lambda)$ is defined as category $D(\Bun_G,\Lambda)$ of appropriately defined \'etale sheaves on $\Bun_G$ \cite{FS}. The definition in this way is quite sophisticated, relying on Scholze's work on $\ell$-adic formalism of diamond and condensed mathematics. 

In another approach\footnote{This approach has been the folklore among the geometric Langlands community for a while.}, which might be less sophisticated and stays in the realm of traditional $\ell$-adic formalism of schemes\footnote{But this approach probably is insufficient for some purposes such as the $p$-adic local Langlands program.}, $B(G)$ is regarded as the set of points of the quotient stack 
$$\frakB(G):=LG/\Ad_\sigma LG,$$ where $LG$ denotes the loop group of $G$, which is a (perfect) group ind-scheme over $\kappa_F$, and $\Ad_\sigma$ denotes the Frobenius twisted conjugation given by 
$\Ad_\sigma: LG\times LG\to LG,\quad  (h,g)\mapsto hg\sigma(h)^{-1}$
 (e.g. see \cite[2.1]{Z18} for a review). 
Then $\Shv(B(G),\Lambda)$ is defined as the category of $\Lambda$-sheaves $\Shv(\frakB(G)_{\overline\kappa_F},\Lambda)$ in appropriate sense. 

More precisely, this category can be also realized  (via ``$h$-descent") as the category of sheaves on the moduli $\Sht^\loc$ of local Shtukas (with the leg at the closed point $0\in\on{Spec} \mO_F$) with morphisms given by cohomological correspondences.
A discussion is sketched at the end of \cite{Z18} (see also \cite[4.1]{Ga}), and a detailed study of this category will appear in \cite{HZ}. 
Here we repeat the outline given in \cite{Z18}. All geometric objects below are defined over $\overline\kappa_F$ even some of them can be originally defined over $\kappa_F$. 

First we consider a simpler situation to define an $\infty$-category $\Shv([*/G(F)],\Lambda)$ of sheaves on the classifying stack of $G(F)$, which is equivalent to the category $\Rep(G(F),\Lambda)$ of smooth representations of $G(F)$.
Let $K\subset G(F)$ be an open compact subgroup. As we can write $K=\varprojlim K_i$ with each $K_i$ finite, we can regard $\Lambda$ as an affine group scheme over $\kappa_F$.
We consider the groupoid of stacks $K\backslash G(F)/K\cong [*/K]\times_{[*/G(F)]}[*/K]\rightrightarrows [*/K]$, which 
extends to a simplicial diagram of stacks (with degeneracy maps omitted)
\begin{equation}\label{E: G(F)}
\cdots\ \substack{\longrightarrow\\[-1em] \longrightarrow \\[-1em] \longrightarrow \\[-1em] \longrightarrow}\ K\backslash G(F)/K\times_{[*/K]}K\backslash G(F)/K \ \substack{\longrightarrow\\[-1em] \longrightarrow \\[-1em] \longrightarrow} \  K\backslash G(F)/K \rightrightarrows [*/K],
\end{equation}
Although $[*/K]$ and $K\backslash G(F)/K$ (and each term in the above diagram) are not algebraic, they can be nevertheless approximated by nice (perfect) Deligne-Mumford stacks (perfectly) of finite type over $\kappa_F$, and one can associate the $\infty$-category of $\Lambda$-sheaves $\Shv(-,\Lambda)$ to them. For example, we can define $\Shv([*/K],\Lambda)=\varinjlim \Shv([*/K_i],\Lambda)$, with connecting functors given by pullback of sheaves along the classifying stacks of finite groups $[*/K_i]\to [*/K_j]$. Then $\Shv([*/K],\Lambda)=\Rep(K,\Lambda)$. For $K\backslash G(F)/K$, we may write $G(F)$ as an increasing union of $K\times K$-stable subsets $G(F)=\varinjlim_i G(F)_i$ (so regarding $G(F)$ as an ind-scheme over $\kappa_F$). Then we can first define the category $\Shv(K\backslash G(F)_i/K,\Lambda)$ in a way as above and then define $\Shv(K\backslash G(F)/K,\Lambda)=\varinjlim \Shv(K\backslash G(F)_i/K,\Lambda)$. 

All the morphisms in the above simplicial diagrams are ind-representable (in fact ind-finite). Then we can define $\Shv([*/G(F)],\Lambda)$ as the geometric realization of a simplicial $\infty$-category
\[
\cdots\ \substack{\longrightarrow\\[-1em] \longrightarrow \\[-1em] \longrightarrow \\[-1em] \longrightarrow} \ \Shv(K\backslash G(F)/K\times_{[*/K]}K\backslash G(F)/K,\Lambda)\ \substack{\longrightarrow\\[-1em] \longrightarrow \\[-1em] \longrightarrow} \  \Shv(K\backslash G(F)/K,\Lambda)\rightrightarrows\Shv([*/K],\Lambda),
\]
with connecting functors given by proper push-forward (\cite[Remark 6.2]{Z18}). One then shows that $\Shv([*/G(F)],\Lambda)$ defined in this way is independent of the choice of $\Lambda$ and is indeed equivalent to $\Rep(G(F),\Lambda)$.

To define $\Shv(\frakB(G),\Lambda)$, we following the same strategy, with $\Lambda$ replaced by the positive loop group $L^+\mG$ of an Iwahori model $\mG$ of $G$ over $\mO_F$ (in fact one can use any parahoric model of $G$), and with $[*/K]$ replaced by
\begin{equation}\label{E: Shtloc}
\Sht^\loc:=\frac{LG}{\Ad_\sigma L^+\mG},
\end{equation}
the moduli of local $\mG$-Shtukas (with the leg at $0\in\on{Spec}\mO_F$, see \cite[(4.1.1)]{Z18}). Then let
\begin{equation}\label{E: Hk groupoid}
\Hk(\Sht^\loc):=\Sht^\loc\times_{\frakB(G)}\Sht^\loc
\end{equation}
be the Hecke stack of local Shtukas (see \cite[(4.1.2)]{Z18} with $s=t=1$). We similarly have a simplicial diagram 
\begin{equation}\label{E: mer Sht}
\cdots\ \substack{\longrightarrow\\[-1em] \longrightarrow \\[-1em] \longrightarrow \\[-1em] \longrightarrow} \ \Hk(\Sht^\loc)\times_{\Sht^\loc}\Hk(\Sht^\loc)\ \substack{\longrightarrow\\[-1em] \longrightarrow \\[-1em] \longrightarrow} \  \Hk(\Sht^\loc)\rightrightarrows\Sht^\loc
\end{equation}
with morphisms ind-(perfectly) proper. Again, each term in the above diagram is not algebraic, but can be approximated by nice (perfect) algebraic stacks (perfectly) of finite type over $\kappa_F$ (see \cite{XZ} for a detailed discussion and \cite[4.1]{Z18} for a summary). Then one can associate the $\infty$-category of $\Lambda$-sheaves to each term and define $\Shv(\frakB(G),\Lambda)$ as the geometric realization of the corresponding simplicial $\infty$-category. By definition, there is a natural functor $\Shv(\Sht^\loc,\Lambda)\to \Shv(\frakB(G),\Lambda)$. This is nothing but the proper push-forward along the Newton map $\on{Nt}: \Sht^\loc\to \frakB(G)$.

There is a closed embedding of the simplicial diagram \eqref{E: G(F)} into \eqref{E: mer Sht} induced by the embedding
\begin{equation}\label{E: geom deltaK}
[*/I]\cong \frac{L^+\mG}{\Ad_\sigma L^+\mG}\subset \Sht^\loc,
\end{equation}
where $I=\mG(\mO_F)$.
 This gives a fully faithful embedding
\[
i_!: \Rep(G(F),\Lambda)\cong \Shv([*/G(F)],\Lambda)\to \Shv(\frakB(G),\Lambda).
\]
Then for every open compact subgroup $K'$, the object $\delta_{K'}\in \Rep(G(F),\Lambda)$ gives a corresponding object in $\Shv(\frakB(G),\Lambda)$, denoted by the same notation. If $K'\subset I$, geometrically $\delta_{K'}$ is given by the proper push-forward of the constant sheaf $\Lambda$ along the morphism $[*/K']\to [*/I]\to \Sht^\loc\to \frakB(G)$.

\begin{remark}
As explained in \cite{Z18}, the homotopy category of $\Shv(\frakB(G),\Lambda)$ can be expressed as the category of sheaves on $\Sht^\loc$ with morphisms given by cohomological correspondences supported on $\Hk(\Sht^\loc)$. The latter was constructed in details in \cite{XZ}, and is very useful for global applications. Using this interpretation, there is a more elementary way to show that the  endomorphism algebra of the sheaf
$\delta_{K'}$ (defined as the proper push-forward of $\Lambda$ along $[*/K']\to \frakB(G)$) is the derived Hecke algebra $H_K$ (see \cite[Remark 5.4.5]{XZ}). 
\end{remark}

More generally, for a basic $b$, we lift it to an element $\tilde{b}\in G(\breve F)$ that normalizes $\mG(\mO_{\breve F})$, where as before $\breve F$ denotes the completion of
maximal unramified extension of $F$. There is a closed embedding similar to \eqref{E: geom deltaK}
\begin{equation}\label{E: delta on b}
[*/I_b]\cong \frac{L^+\mG\cdot\tilde{b}}{\Ad_\sigma L^+\mG}\subset \Sht^\loc.
\end{equation} 
Here $I_{b}$ is the twisted centralizer of $\tilde b$ in $\mG(\mO_{\breve F})$, which is an Iwahori subgroup of $G_b(F)$. Then there is a simplicial diagram similar to \eqref{E: G(F)} associated to the groupoid $[*/I_b]\times_{[*/G_b(F)]}[*/I_b]\rightrightarrows[*/I_b]$ with a closed embedding into \eqref{E: mer Sht}. This gives us the embedding $i_{b,!}$ in \eqref{E: emb Jb} as promised. 

\begin{remark}
The optimal guess would be the category $D(\Bun_G,\Lambda)$ defined by Fargues-Scholze and $\Shv(\frakB(G),\Lambda)$ outlined above are equivalent.  
A striking feature is in the above two interpretations of $B(G)$, the partial order on $B(G)$ gets reversed.
\end{remark}

\begin{remark}\label{R: ad quot}
As mentioned in \cite{Z18}, exactly the same construction allows one to define and study the category of sheaves on the adjoint quotient space $LG/\Ad LG$.
\end{remark}

Now we formulate our conjecture.
Let $\hat{\mN}_{{}^cG}$ denote the subset of $\on{Sing}(\Loc_{{}^cG})$ as in \eqref{E: lgnlip}. Recall our convention of the category of coherent sheaves on $\Loc_{{}^cG}$ in Remark \ref{R: Coh cat}.

\begin{conjecture}\label{Conj: cat LLC}
Assume that $(G,B,T,e)$ is a quasi-split reductive group equipped with a pinning over $F$.
Then there is a natural $\Tor_{Z_G,\iso_F}$-equivariant equivalence of $\infty$-categories
\[\bL_G:\Shv(\frakB(G),\Lambda)\to \Ind\Coh_{\hat\mN_{{}^cG}}(\Loc_{{}^cG})\]
sending $\on{Whit}_{(U,\psi)}$ (see \eqref{E: pin to whit}) to the structural sheaf $\mO_{\Loc_{{}^cG}}$. 

In addition, for every basic element $b\in B(G)$, the conjectural functor $\frakA_{G_b}$ in Conjecture \ref{Conj: coh sheaf}, when tensored with $\Lambda$, fits into the following commutative diagram
\[\xymatrix{
  \Rep_{\on{f.g.}}(G_b,\Lambda)\ar_-{i_{b,!}}[d]\ar^{\frakA_{G_b}}[r]& \Coh(\Loc_{{}^cG})\ar[d]\\
  \Shv(\frakB(G),\Lambda)\ar^-{\bL_G}[r]& \Ind (\Coh_{\hat{\mN}_{{}^cG}}(\Loc_{{}^cG})).
}\]
\end{conjecture}
\begin{remark}\label{R: AJb} 
Note that the conjecture 
implies that for every $b$ (not necessarily basic), there should exist an ind-coherent sheaf 
$$\frakA_{G_b,\{1\}}:=\bL_G(i_{b,!}(\delta_{G_b,\{1\}})),\quad \delta_{G_b,\{1\}}:=C_c(G_b(F),\Lambda),$$ 
on $\Loc_{{}^cG}$, where $i_{b,!}$ is the functor from \eqref{E: emb Jb}, and $\nu_b$ is Newton point of $b$ (which is a dominant rational character of $\hat{G}$ \cite[4.2]{Ko2}). 
The sheaf $\frakA_{G_b,\{1\}}$ should not be an ordinary coherent sheaf in general, unlike the basic case considered in Conjecture \ref{Conj: coh sheaf} \eqref{Conj: coh sheaf-2}. However,  
we conjecture that $\frakA_{G_b,\{1\}}$ concentrates in cohomological degrees $[-(2\rho,\nu_b),0]$.
\end{remark}

\begin{remark}
In Fargues-Scholze' approach where $\Shv(B(G),\bZ_\ell)$ is defined as as $D(\Bun_G,\bZ_\ell)$, this conjecture formally looks like the global geometric Langlands conjecture as proposed by Arinkin-Gaitsgory \cite{AG}. Indeed, Fargues-Scholze independently announced the same conjecture using $D(\Bun_G,\bZ_\ell)$ in the formulation.
\end{remark}

\begin{remark}
For $\bZ_\ell$-coefficient and $\ell$ the so-called non banal prime, the existence of $\frakA_{G_b}$ does not follow directly from the existence of $\bL_G$, as $\Rep_{\on{f.g.}}(G_b,\bZ_\ell)$ does not belong to the subcategory of compact objects of $\Shv(\frakB(G),\bZ_\ell)$.
However, there is a renormalized version $\Ind\Shv_{\on{f.g.}}(\frakB(G),\bZ_\ell)$ of $\Shv(\frakB(G),\bZ_\ell)$, which will contain $\Rep_{\on{f.g.}}(G_b,\bZ_\ell)$ inside its subcategory of compact objects (the definition is similar to \cite[12.2.3]{AG} and will be given in \cite{HZ}). We expect that $\bL_G$ extends to an equivalence
$$\bL_G^{\Ind\on{f.g.}}:\Ind\Shv_{\on{f.g.}}(\frakB(G),\bZ_\ell)\cong \Ind\Coh(\Loc_{{}^cG}),$$
which would imply the existence of $\frakA_{G_b}$.  
Note that when $\bZ_\ell$ is replaced by $\bQ_\ell$, we have $\Ind\Shv_{\on{f.g.}}(\frakB(G),\bQ_\ell)=\Shv(\frakB(G),\bQ_\ell)$, and the nilpotent singular support condition is automatic by Lemma \ref{L: nilp=sing}. So $\bL_G^{\Ind\on{f.g.}}$ would coincide with $\bL_G$.
\end{remark}

\begin{remark}
It would be interesting to formulate a ``motivic" (i.e. independent of $\ell$) version of the above equivalence. When the coefficient $\La=\bQ_\ell$, Proposition \ref{P:comp iota12} suggests that in the Galois side instead of considering $\Coh(\Loc_{{}^cG,F}\otimes\bQ_\ell)$, one may consider $\Coh(\Loc_{{}^cG,F}^{\on{WD}}/\bG_m\otimes\bQ_\ell)$. On other other hand, we expect that $\Shv(\frakB(G),\bQ_\ell)$ admits a mixed version $\Shv^m(\frakB(G),\bQ_\ell)$. Then $\bL_G$ might be lifted to an equivalence of mixed categories which might then have a chance to descend to $\bQ$.
\end{remark}

\begin{remark}
The conjectural equivalence is supposed to satisfy a set of compatibility conditions similar to those in the global geometric Langlands correspondence (\cite{AG,Ga3}). 
For example, it should be compatible with parabolic induction on both sides, and should be compatible with cohomological duality on $\Shv(\frakB(G),\Lambda)$ (a generalization of \eqref{E: hom duality}) and the modified Grothendieck-Serre duality \eqref{E: Cartan inv}. 
As discussing these compatibilities would require introducing additional constructions related to $\Shv(\frakB(G),\Lambda)$, we skip them here and refer to \cite{HZ} for more details.

On the other hands, the conjectural equivalence predict that there should exist an action of the category $\on{Perf}(\Loc_{{}^cG})$ of perfect complexes on $\Loc_{{}^cG}$ on $\Shv(B(G),\Lambda)$, usually called the spectral action. Fargues-Scholze have announced a construction of such action in their setting. But the existence of such spectral action on $\Shv(\frakB(G),\Lambda)$  is not known. 
\end{remark}
An evidence that $\Shv(\frakB(G),\bZ_\ell)$ might also be the correct input for the conjecture, we first recall the following result from \cite{XZ,Z18,Yu2}.

\begin{theorem}\label{T: spectral hs}
Assume that $\mG$ is reductive. Then there is a natural functor $\Coh(\Loc_{{}^cG}^{\ur})\to \Shv(\frakB(G),\Lambda)$ making the following diagram commutative
\[
\xymatrix{
\Rep(\hat{G},\Lambda)^{\heart}\ar^-{\Sat}[r]\ar[d] & \Shv(\Sht^\loc,\Lambda)\ar[d]\\
\Coh(\Loc_{{}^cG}^{\ur}) \ar[r] & \Shv(\frakB(G),\Lambda)
}\]
where $\Sat$ is induced by the geometric Satake equivalence (\cite{MV, Z17a,Yu1}), and the left vertical functor is the natural pullback functor along $\Loc^{\ur}_{{}^cG}\to \bB\hat{G}$.
\end{theorem}
More convincingly, we have the following statement which will be established by Hemo and the author in \cite{HZ}. 
\begin{theorem}\label{T: tame local Langlands}
Assume that $(G,B,T,e)$ is a pinned unramified group over a local field $F$, and that $\La=\overline\bQ_\ell$.
Then the functor in Theorem \ref{T: spectral hs} extends to a fully faithful embedding
\[
\Coh(\Loc_{{}^cG}^{\widehat\un})\to \Shv(\frakB(G),\overline\bQ_\ell)
\]
into the subcategory of compact objects of $\Shv(\frakB(G),\overline\bQ_\ell)$. It sends $\on{CohSpr}^{\un}_{{}^cG}$ to $\delta_{I}$. 
More generally, for every element $b\in B(G)$, let $H_{I_b}$ the corresponding Iwahori-Hecke algebra of $G_b$. Then
there is the following commutative diagram
\[
\xymatrix{
\Mod_{H_{I_b}}\ar@{^{(}->}[r] \ar@{^{(}->}[d] & \Rep(G_b(F),\overline\bQ_\ell)\ar^{i_{b,!}}[d]\\
\Ind\Coh(\Loc_{{}^cG}^{\un})\ar[r] &  \Shv(\frakB(G),\overline\bQ_\ell)
}\]
\end{theorem}
Further properties of the embedding in the theorem will be studied in \cite{HZ}.

\subsection{Cohomology of modular varieties and local-global compatibility}\label{SS: local-global}
In this last subsection, we formulate conjectural formulas for the cohomology of moduli of Shtukas and to give some evidences.
We will mainly consider the function field case as the picture is more complete. 
But we will also discuss a conjectural geometric realization of Jacquet-Langlands transfer via cohomology of Shimura varieties, generalization the main construction of \cite{XZ}.

Let $F$ be a global field, and $G$ a connected reductive group over $F$.
Let $\Lambda$ be a noetherian $\bZ_\ell$-algebra, where $\ell\neq \on{char}F$ if $F$ is a function field.
We will use notations from \S \ref{SS: glob par}. 

We first discuss function field case.
Let $F=\bF_q(X)$ be a global function field, where $X$ is a geometrically connected smooth projective curve over $\bF_q$. We denote the Weil group of $F$ by
$W_F$. Let $\eta=\on{Spec} F$ be the generic point of $X$, and $\overline\eta$ a geometric point over $\eta$. Let $\bO=\prod_{v\in|X|}\mO_v$ be the integral ad\`eles, where $\mO_v\subset F_v$ is the ring of integers.
We extend the group $G$ to a Bruhat-Tits integral model $\mG$ over $X$, by which we mean a smooth affine group scheme over $X$ such that $\mG|_{\mO_v}$ is a parahoric group scheme of $G_v$ in the sense of Bruhat-Tits.
We will consider the compactly supported cohomology of the moduli of $\mG$-Shtukas. For basic constructions and facts about the moduli of $\mG$-Shtukas, we refer to \cite{L}.

We fix a level $K\subset\mG(\bO)$. Let $S_K$ be the set of places $v$ such that $K_v\neq \mG(\mO_v)$, and $S\supset S_K$ the set of places where $K_v$ is not hyperspecial.
For a finite set $I$, let $\Sht_{(X-S_K)^I,K}$ denote the moduli of $\mG$-shtukas on $X$ with $I$-legs in $X-S_K$ and with $K$-level structure.
This is an ind-Deligne-Mumford stack over $(X-S_K)^I$. 
The base change of it along the diagonal map $\overline\eta\to(X-S_K)^I$ is denoted by $\Sht_{\Delta(\overline\eta),K}$. 
For every representation $V$ of $({}^cG)^I$ on a finite projective $\Lambda$-module, the geometric Satake correspondence provides a perverse sheaf $\on{Sat}(V)$ on $\Sht_{\Delta(\overline\eta),K}$ (in fact, it is defined on $\Sht_{(X-S)^I,K}$). 
Let 
$$C_c\bigl(\Sht_{\Delta(\overline\eta),K},\Sat(V)\bigr)\in\Mod_{H_K}$$ denote the (cochain complex of the) total compactly supported cohomology of $\Sht_{\Delta(\overline\eta),K}$ with coefficient in $\Sat(V)$, on which the corresponding global (derived) Hecke algebra (with coefficients in $\Lambda$) $H_K=C_c(K\backslash G(\bA)/K,\Lambda)$ acts. 
When $V=\mathbf{1}$ is the trivial representation, we have
$$C_c\bigl(\Sht_{\Delta(\overline\eta),K}, \on{Sat}(\mathbf{1})\bigr)=\bigsqcup_{\xi\in \ker^1(F,G)} C_c(G^\xi(F)\backslash G(\bA)/K,\Lambda).$$
Here $ \ker^1(F,G)\subset H^1(F,G)$ consisting of those classes that are locally trivial, and for $\xi \in\ker^1(F,G)$, $G^\xi$ denotes the corresponding pure inner form of $G$; $G^\xi(F)\backslash G(\bA)/K$ is regarded as a discrete DM stack over $\overline\eta$, and $C_c(G^\xi(F)\backslash G(\bA)/K,\Lambda)$ denotes its compactly supported cohomology.
When $\La=\bQ_\ell$ and $G$ satisfies the Hasse principle (e.g. $G$ is quasi-split), this is the space of compactly supported functions on $G(F)\backslash G(\bA)/K$.

Let $H^i_{I,V}=H^iC_c\bigl(\Sht_{\Delta(\overline\eta),K},\Sat(V)\bigr)$.
By  \cite{Xu,Xu2,Xu3}, the natural Galois action and the partial Frobenii action together induce a canonical $W_{F,S}^I$-action on $H^i_{I,V}$.  
The following statement can be regarded as a generalization of the main construction of \cite{LZ}.
\begin{theorem}\label{T: gen LZ}
Assume that $\La=\bQ_\ell$ and regard $\Loc_{{}^cG,F,S}$ as an algebraic stack over $\bQ_\ell$.
Then for each $i$, there is a quasi-coherent sheaf $\frakA^i_\Lambda$ on ${}^{cl}\Loc_{{}^cG,F,S}$, equipped with an action of $H_K$, such that for every finite dimensional representation $V$ of $({}^cG)^I$, there is a natural $(H_K\times W_{F,S}^I)$-equivariant isomorphism
\begin{equation}\label{E: coh Sht}
H^i_{I,V}\cong \Gamma\big({}^{cl}\Loc_{{}^cG,F,S},({}_{W_{F,S}}V)\otimes \frakA^i_\Lambda\big),
\end{equation}
where ${}_{W_{F,S}}V$ is the vector bundle on $\Loc_{{}^cG,F,S}$ equipped with an action by $W_{F,S}^I$ as in  Remark \ref{R: fiber action}.
\end{theorem}
\begin{proof}
As explained in \cite[\S 5]{LZ}, for a representation $V$ of $\hat{G}\times ({}^cG)^I$, we can define $H_{\{0\}\cup I,V}^i$, which admits an action of $H_K\times W_{F}^I$, such that if the restriction of $V$ to the $\hat{G}$-factor is trivial then $H_{\{0\}\cup I,V}^i=H_{I,V}^i$.
In particular, we have the $H_K$-module $H_{\{0\},\on{Reg}}^i$, where $\on{Reg}$ denotes the regular representation of $\hat{G}$. 

We regard $W_{F,S}$ as an abstract group and consider ${}^{cl}\mR_{W_{F,S},{}^cG}$. The construction of \cite[\S 6]{LZ} gives a homomorphism $\bQ_\ell[{}^{cl}\mR_{W_{F,S},{}^cG}]\to \End(H_{\{0\},\on{Reg}}^i)$. Let $A^i$ be the image of the map.
For $f\in \bQ_\ell[{}^cG]$ and $\ga\in W_{F,S}$, we have the regular function $F_{f,\ga}$ on ${}^{cl}\mR_{W_{F,S},{}^cG}$ given by
$F_{f,\ga}(\rho)=f(\rho(\ga))$. Let $\bar{F}^i_{f,\ga}$ be the image of $F_{f,\ga}$ in $A^i$. Note that when it is regarded as a representation of $\pi_1(\overline Y)^I$, $H^i_{I,V}$ is a union of finite dimensional continuous subrepresentations.
Then the argument as in \cite[6.2]{LZ} and in Lemma \ref{L: cont criterion} shows that the map $\pi_1(\overline Y)\to A^i,\ \ga\mapsto \bar{F}_{f,\ga}$ is continuous, if $A^i$ is equipped with the ind-$\ell$-adic topology. Therefore, we have the factorization
\[
\on{Spec} A^i \to {}^{cl}\mR^{sc}_{W_{F,S},{}^cG}\to {}^{cl}\mR_{W_{F,S},{}^cG}.
\]
So $H_{\{0\},\on{Reg}}^i$ can be regarded as a quasi-coherent sheaf on  ${}^{cl}\mR^{sc}_{W_{F,S},{}^cG}$. As explained in \cite{LZ},
there is also $\hat{G}$-action on $H_{\{0\},\on{Reg}}^i$ compatible with the action of $A^i$, so $H_{\{0\},\on{Reg}}^i$  descends to a quasi-coherent sheaf $\frakA_\Lambda^i$ on ${}^{cl}\mR^{sc}_{W_{F,S}, {}^cG/\hat{G}}$. 
It follows from construction that $\frakA_\Lambda^i$ is supported on ${}^{cl}\Loc_{{}^cG,F,S}$ and the argument as in \cite{LZ} shows that \eqref{E: coh Sht} holds.
\end{proof}

\begin{remark}As explained in \cite{LZ}, the sheaf $\frakA^i_\Lambda$ is in fact the pullback of a quasi-coherent sheaf on $\bigl({}^{cl}\Loc_{{}^cG,F,S}^\Box/ (\hat{G}/Z_{\hat{G}}^{\Ga_F})\bigr)\otimes \bQ_\ell$. We expect that each $\frakA^i_\Lambda$ is coherent. 
\end{remark}

\begin{example}\label{E: coh elliptic}
Assume that $G$ is semisimple (for simplicity), and recall elliptic Langlands parameters from Example \ref{E: ell part}. It follows that the localization of $\frakA_\Lambda^i$ at an elliptic $\rho$, denoted by $\frakA_{K,\rho}^i$, is an $\Ql$-vector space equipped with an action of $H_K\times S_\rho$. Then the localization of $H^i_{I,W}$ at $\rho$ is isomorphic to $(\frakA_{K,\rho}^i\otimes W_\rho)^{S_\rho}$. Therefore, Theorem \ref{T: gen LZ} recovers the main result of \cite{LZ} (except the finite dimensionality of $\frakA_{K,\rho}^i$). We refer to \emph{loc. cit.} for the relation between this formula and the Arthur-Kottwitz multiplicity formula.
\end{example}

\begin{remark}
\begin{enumerate}
\item The idea that something like \eqref{E: coh Sht} should exist is due to Drinfeld, as an interpretation of certain construction of \cite{L}. As explained in \cite{Ga,GKRV,AGK}, (the derived version of) the isomorphism \eqref{E: coh Sht} should follow by taking categorical trace of a categorical geometric Langlands correspondence. 

\item We do not expect Theorem \ref{T: gen LZ} holds in general when $\La=\bZ_\ell$. The problem is that neither the functor $V\mapsto H^i_{I,V}$ nor the functor $\Ga({}^{cl}\Loc_{{}^cG,F,S},-)$ is $t$-exact for integral coefficients. However, we do expect a derived version of \eqref{E: coh Sht} holds when individual cohomology groups in the formula are put together as the total cochain complex $C_c\bigl(\Sht_{\Delta(\overline{\eta}),K},\Sat(V)\bigr)$, and individual $\frakA_\Lambda^i$s are put together as a quasi-coherent complex on $\Loc_{{}^cG,F,S}$. A precise conjecture is given below.
\end{enumerate}
\end{remark}

In \cite{LZ}, in light of the Arthur-Kottwitz conjecture, we conjecture that $\frakA^i_\Lambda$ factorizes as a tensor product of local factors. 
Now we further conjecture that these local factors should exactly be the coherent sheaves appearing in Conjecture \ref{Conj: coh sheaf}.  For simplicity, we will assume from now until the end of this subsection that the center $Z_G$ of $G$ is connected.

To formula the precisely conjecture, first note that 
we can define analogous $\Wh_G$, $\Tor_{G,\iso_F}$ and $\TS_G$ (as introduced in \S \ref{S: Aux}) in the global setting, by the same construction with the completion of a maximal unramified extension of a local field there replaced by the maximal unramified extension of $F$ in the global case.
The set of isomorphism classes of $\Tor_{G,\iso_F}$ is still denoted by $B(G)$. The subset of basic elements $B(G)_{\on{bsc}}$ is defined analogously. A global basic element of $G$ gives a local basic element for $G_v$ at every place (whose image in $\xch(Z_{\hat{G}}^{\Ga_v})$ is zero for almost all $v$) and there is following exact sequence of pointed sets
\begin{equation*}\label{E: local-global B(G)}
B(G)_{\on{bsc}}\to \oplus_v B(G_v)_{\on{bsc}}\to \xch(Z_{\hat{G}}^{\Ga_F}).
\end{equation*}

Now we fix a non-trivial character $\psi_0: F\backslash \bA\to \La^\times$, and
 fix a global element $t\in \TS_G$. These data induce the corresponding data at every local place. 
Then we have the functor $\frakA_{G_v}$ at every place $v$ as in Conjecture \ref{Conj: coh sheaf}. If $K_v\subset G_v$ is an open compact subgroup, we sometimes write $\frakA_{K_v}$ instead of $\frakA_{G_v,K_v}$ for simplicity.

Recall that we fix a level structure $\Lambda$. By enlarging the set $S$ if necessary, we may assume that for every $v\not\in S$, $t_v\in \Wh_{G_v}$,
$K_v$ is hyperspecial determined by the pinning (up to $G(F_v)$-conjugacy). 
We denote by $\boxtimes_{v\in S} \frakA_{K_v}$ the external tensor product of those coherent sheaves on $\prod_{v\in S}\Loc_{v}$, and by $\res^!(\boxtimes_{v\in S} \frakA_{K_v})$ its  $!$-pullback to $\Loc_{{}^cG,F,S}$ via \eqref{E: global-to-local}. 
By our expectation \eqref{E: change local normalization}, $\res^!(\boxtimes_{v\in S} \frakA_{K_v})$ should be independent of the choices of $t\in \TS_G$ (and $\psi_0$) and descends to a quasi-coherent sheaf on $\Loc_{{}^cG,F,S}^\Box/ (\hat{G}/Z_{\hat{G}}^{\Ga_F})$.

\begin{conjecture}\label{Con: local-global}
For every representation $V$ of $({}^cG)^I$ on free $\Lambda$-module,
there is a canonical $(H_{K}\times W_{F,S}^I)$-equivariant isomorphism
\[
C_c\bigl(\Sht_{\Delta(\overline\eta),K}, \on{Sat}(V)\bigr)\cong \Gamma\bigl(\Loc_{{}^cG,F,S}, ({}_{W_{F,S}}V)\otimes \res^!(\boxtimes_{v\in S} \frakA_{K_v})\bigr).
\]
\end{conjecture}

Note that the conjecture is consistent with enlarging $S$, as $\frakA_{K_w}\cong \mO_{\Loc^{\on{unr}}_w}$ when $K_w$ is hyperspecial (and is determined by $t_w$), and we have the Cartesian diagram by Lemma \ref{R: Cartesian local-to-global}. 

\begin{remark}\label{R: mid deg}
Suppose (for simplicity) $G$ is of adjoint type. Let $\rho$ be an elliptic Langlands parameter as in Example \ref{E: ell part}.
As $\rho$ is isolated smooth, the localization of  $({}_{W_{F,S}}V)\otimes \res^!(\boxtimes_{v\in S} \frakA_{K_v})$ at $\rho$ is a complex of vector spaces given by  $V\otimes (\otimes_{v\in S} \frakA_{K,v}^!)$, where $\frakA_{K,v}^!$ denotes the $!$-fiber of $\frakA_{K,v}$ at $\rho_v:=\rho|_{W_v}$. As $\Ad_\rho$ is pure of weight zero, each $\rho_v$ is a smooth point of $\Loc_v$ (Proposition \ref{P: sm pt}).
Note that $\frakA_{K_v}$ should be a maximal Cohen-Macaulay ordinary coherent sheaf (Conjecture \ref{Conj: coh sheaf} \eqref{Conj: coh sheaf-2}). This would imply that $\frakA_{K,v}^!$ sits in cohomological degree zero. It follows that $\frakA_{K,\rho}^i$ from Example \ref{E: coh elliptic} should vanish unless $i=0$. This is consistent with the general expectation.
\end{remark}

\begin{example}
We make this conjecture more explicit in the everywhere unramified case, i.e. $\mG$ is reductive over $X$ and $K=\mG(\bO)$. In this case, we can consider $\Loc_{{}^cG,X}=\Loc_{{}^cG,F,\emptyset}$ as in Remark \ref{R: everywhere unramified loc}. As $\frakA_{K_v}\cong \mO_{\Loc^{\ur}_v}\cong \omega_{\Loc^{\ur}_v}$, Conjecture \ref{Con: local-global} in this case reduces to
\[
C_c\bigl(\Sht_{\Delta(\overline\eta),K}, \on{Sat}(V)\bigr)\cong \Gamma\bigl(\Loc_{{}^cG,X},({}_{W_F}V)\otimes \omega_{\Loc_{{}^cG, X}}\bigr).
\]
We note that when $G$ is split and $\La=\overline\bQ_\ell$, this formula is also independently conjectured in \cite{AGK}\footnote{Except that the definition of $\Loc_{{}^cG,X}$ in \emph{loc. cit.} is a priori different.}. 

We further specialize to the case where $X=\bP^1$, and $V=\mathbf{1}$ is the trivial representation of ${}^cG$. In this case, $G$ necessarily is quasi-split and split over an extension of the field of constant $\bF_{q'}/\bF_q$. Then as mentioned before,
$C_c\bigl(\Sht_{\Delta(\overline\eta),K}, \on{Sat}(\mathbf{1})\bigr)$ is just the compactly supported cohomology of $G(F)\backslash G(\bA)/G(\bO)$, regarded as a discrete DM stack. If $\La=\bQ_\ell$, this is the space of compactly supported functions on $G(F)\backslash G(\bA)/G(\bO)$.

We regard the characteristic function the double coset $G(F)\backslash G(F)G(\bO)/G(\bO)$ as a map $k\to C_c(G(F)\backslash G(\bA)/G(\bO),\Lambda)$. 
The action of the derived Hecke algebra $H_{K_0}=\End \delta_{K_0}$ at $0\in\bP^1$ on $H_{\emptyset}(\mathbf{1})=C_c(G(F)\backslash G(\bA)/G(\bO),\Lambda)$ induces a derived version of the Radon transform
\[
H_{K_0}\cong H_{K_0}\otimes k\to H_{K_0}\otimes C_c(G(F)\backslash G(\bA)/G(\bO),\Lambda)\to C_c(G(F)\backslash G(\bA)/G(\bO),\Lambda),
\]
which is an isomorphism by an argument similar to the underived version (see \cite{HZ} for details). Then we have the following commutative diagram
\[
\xymatrix{
H_{K_0}\ar^-\cong[r]\ar_{\mathrm{Conj.}\ \ref{C: derSat}}^\cong[d]& C_c(G(F)\backslash G(\bA)/G(\bO),\Lambda)\ar^{\mathrm{Conj.}\ \ref{Con: local-global}}_\cong[d]\\
\End_{\Loc_0^\ta}\mO_{\Loc_0^\ur}\ar^-\cong[r]&  \Gamma\bigl(\Loc_{{}^cG,\bP^1},\omega_{\Loc_{{}^cG, \bP^1}}\bigr),
}\]
where the bottom isomorphism follows from \eqref{E: case of P1}.
Therefore, Conjecture \ref{C: derSat} implies Conjecture \ref{Con: local-global} in this special case. As Conjecture \ref{C: derSat} holds when $\La=\bQ_\ell$ (see Remark \ref{R: derSat}), so is Conjecture \ref{Con: local-global} in this special case.
As also mentioned in Remark \ref{R: derSat}, this in particular implies that over $\bQ_\ell$, $\Gamma\bigl(\Loc_{{}^cG,\bP^1},\omega_{\Loc_{{}^cG, \bP^1}}\bigr)$ concentrates in degree zero (however one can show that the cohomological amplitude of the sheaf $\omega_{\Loc_{{}^cG,\bP^1}}$ is unbounded from above.) 
\end{example}

\begin{example}
We still assume $\mG$ is reductive but with $K_v$ Iwahori subgroup of $\mG(\mO_v)$ for $v\in S$. Then $\frakA_{K_v}\cong \pi^{\un}_*\mO_{\Loc_{{}^cB,F_v}^{\un}}\cong \pi^{\un}_*\omega_{\Loc_{{}^cB,F_v}^{\un}}$ when $v\in S$.
We consider 
\[
\widetilde{\Loc}^{\un}_{{}^cG,X,S}:=\Loc^{\ta}_{{}^cG,X,S}\times_{\prod_v \Loc^{\ta}_v}\prod \Loc^{\un}_{{}^cB,F_v}.
\]
Then Conjecture \ref{Con: local-global} in this case reduces to
\[
C_c\bigl(\Sht_{\Delta(\overline\eta),K},\on{Sat}(V)\bigr)\cong \Gamma\bigl(\widetilde{\Loc}^{\un}_{{}^cG,X,S},({}_{W_F}V)\otimes\omega_{\widetilde{\Loc}^{\un}_{{}^cG, X,S}}\bigr).
\]
Again, in the special case when $X=\bP^1$, $S=\{0,\infty\}$ and $W=\mathbf{1}$, Conjecture \ref{Con: local-global} follows from Conjecture \ref{C: REnd of Spr}. In particular, it holds when $\La=\overline\bQ_\ell$. We refer to \cite{HZ} for details.
\end{example}

To make analogy between moduli of Shtukas and Shimura varieties, we generalize the above conjecture, using the 
formalism of the conjectural categorical local Langlands correspondence from \S \ref{SS: cat LLC}.
Fix a finite set $T$ of places. 
For a (possibly empty) finite set $I$,
let $\Sht_{(X-T)^I,T}$ be the moduli of $\mG$-shtukas on $X$ with $I$-legs in $X-T$ and extra legs at every $v\in T$. We simply write $\Sht_T$ instead of $\Sht_{(X-T)^\emptyset, T}$.
For each $v\in T$, we choose a uniformizer $\varpi_v\in \mO_v$, and regard $\mG_{\mO_v}$ as a parahoric group scheme over $\bF_q[[\varpi_v]]$, denoted by $\mG_v$. Then we have the moduli of local $\mG_v$-shtukas \eqref{E: Shtloc}.
There is a natural a morphism
\[
\Sht_{(X-T)^I,T}\xrightarrow{\res} \prod_{v\in T} \Sht^{\loc}_{v}
\]
by restricting global Shtukas on $X$ to local Shtukas with legs at $v\in T$. As before, let $\Sht_{\Delta(\overline\eta),T}$ denote the base change of $\Sht_{(X-T)^I,T}$ along $\overline\eta\to X-T\xrightarrow{\Delta}(X-T)^I$.

Now let $T=S$ be a set of places such that if $v\not\in S$ then $\mG(\mO_v)$ is reductive and is determined by $t_v$. At each place $v\in S$ we choose $\mK_{v}\in \Shv(\Sht^{\loc}_{v})$.  This collection of sheaves will serve as the chosen ``generalized level structure" at $v\in S$.
Proper push-forward of $\mK_{v}$ along the Newton map $\on{Nt}_{v}:\Sht^{\loc}_{v}\to \frakB(G_{v})$ should correspond a(n ind-)coherent sheaf $\frakA_{\mK_{v}}$ on $\Loc_{v}$ via Conjecture \ref{Conj: cat LLC}. 

\begin{conjecture}\label{C: gen coh}
For $V\in \Rep({}^cG^I)$, we have
\[
C_c\bigl(\Sht_{\Delta(\overline\eta),S}, \Sat(V)\otimes \res^!(\boxtimes_{v\in S}\mK_{v})\bigr)\cong \Gamma\bigl(\Loc_{{}^cG,F,S},({}_{W_F}V)\otimes  \res^!(\boxtimes_{v\in S} \frakA_{\mK_{v}})\bigr).
\]
\end{conjecture}

\begin{remark}
There is a more conceptual formulation of this conjecture, saying two functors 
$\prod_v\Shv(\frakB(G_v),\Lambda)\to \Ind\Coh(\Loc_{{}^cG,F,S})$, one constructed using cohomology of moduli of Shtukas and one obtained from Conjecture \ref{Conj: cat LLC}), are canonically isomorphic. We refer to \cite{HZ} for details.
\end{remark}

We discuss this conjecture in some special cases. 
\begin{example}\label{Ex: gen coh}
Let $K\subset \mG(\bO)$ be a level structure as in Conjecture \ref{Con: local-global}. Assume that $S\supset S_K$.
If at each $v\in S$, we take $\mK_{v}$ to be the push-forward of the constant sheaf along $[*/K_{v}]\to [*/\mG(\mO_{v})]\hookrightarrow \Sht^{\loc}_{v}$ (see \eqref{E: geom deltaK}), then Conjecture \ref{C: gen coh} gives back to Conjecture \ref{Con: local-global}, as $\Sat(V)\otimes \res^!(\boxtimes\mK_{v})$ is just the push-forward of $\Sat(V)$ along $\Sht_{\Delta(\overline\eta),K}\to \Sht_{\Delta(\overline\eta),S}$ and  $\frakA_{\mK_{v}}$ should exactly be $\frakA_{K_{v}}$ as predicted in Conjecture \ref{Conj: cat LLC}.
\end{example}

\begin{example}\label{Ex: nearby cycle}
Keep the above situation and specialize to $I=\{1\}$ so $V\in\Rep({}^cG)$. In addition, fix $v_0\in S$. Consider the following diagram
\[
\Sht_{X-S,S}\hookrightarrow\Sht_{X-(S-\{v_0\}),S-\{v_0\}}\hookleftarrow \Sht_S.
\]
Taking the nearby cycles of the sheaf $\Sat(V)\otimes \res^!(\boxtimes_{v\in S}\mK_{v})$ on $\Sht_{X-S,S}$ with respect to the above diagram gives a sheaf 
$R\Psi(\Sat(V)\otimes \res^!(\boxtimes\mK_{v}))$ on $\Sht_S\otimes\overline\bF_q$. It is known that there is a sheaf on $\Sht^\loc_{v_0}\otimes\overline\bF_q$, denoted by $\Sat(V)\star \mK_{v_0}$ such that 
$$R\Psi(\Sat(V)\otimes \res^!(\boxtimes\mK_{v}))\simeq \res^!(\boxtimes_{v\neq v_0}\mK_{v}\boxtimes(\Sat(V)\star \mK_{v_0})).$$
In addition, under Conjecture \ref{Conj: cat LLC}, $\Sat(V)\star \mK_{v_0}$ should correspond to $({}_{W_{v_0}}V)\otimes \frakA_{K_{v_0}}$. Now Conjecture \ref{C: gen coh} predicts a canonical isomorphism
\begin{multline*}
C_c\bigl(\Sht_{S}\otimes\overline\bF_q,  \res^!(\boxtimes_{v\neq v_0}\mK_{v}\boxtimes(\Sat(V)\star \mK_{v_0}))\bigr)\\
\cong \Gamma\bigl(\Loc_{{}^cG,F,S},  \res^!(\boxtimes_{v\neq v_0} \frakA_{K_{v}}\boxtimes(({}_{W_{v_0}}V)\otimes\frakA_{K_{v_0}}))\bigr).
\end{multline*}
In particular, the conjecture would imply that 
\[
C_c\bigl(\Sht_{\Delta(\overline\eta),S}, \Sat(V)\otimes \res^!(\boxtimes\mK_{v})\bigr)\cong C_c\bigl(\Sht_{S}\otimes\overline\bF_q, R\Psi(\Sat(V)\otimes \res^!(\boxtimes\mK_{v}))\bigr).
\]
\end{example}

\begin{example}\label{E: nearby cycle}
Suppose $G$ is quasi-split with a pinning.
Suppose $T\subset S$ is a collection of finite places with $\mG_v$ Iwahori given by the pinning for $v\in T$. For each $v$, choose $w_v\in\Omega_v$ (see \eqref{E: length zero}) in the Iwahori-Weyl group $\widetilde W_v$ of $G(\breve F_v)$, such that the sum of their images in $\xch(Z_{\hat{G}}^{\Ga_F})$ under the Kottwitz map is zero. Then the collection $\{w_v\}$ gives an inner form $G'$ of $G$ with an integral model $\mG'$ such that $\mG'_{\mO_v}=\mG_{\mO_v}$ for $v\not\in T$. 
We have the moduli of $\mG$-Shtukas $\Sht_S$ with legs at $S$ and the moduli of $\mG'$-Shtukas $\Sht'_S$ with legs at $S$.
Choose $\mK_v$ at $v\in T$ to be the push-forward of the constant sheaf along the closed embedding $L^+\mG_{v}\cdot w_v/\Ad_\sigma L^+\mG_{v}\to \Sht^\loc_{v}\otimes\overline\bF_q$ (see \eqref{E: delta on b}), and $\mK_v$ at $v\in S-T$ to be the sheaf associated to the level $\mG(\mO_v)$ as in Example \ref{Ex: gen coh}. 
Then 
$$C_c(\Sht_S\otimes\overline\bF_q, \res^!\boxtimes\mK_v)=C_c(\Sht'_S\otimes\overline\bF_q, \Lambda).$$
In this way, we see that the space of automorphic forms of $G'$ appears in the cohomology of Shtukas of $G$. One can use this to realize Jacquet-Langlands transfer via the cohomology of moduli of Shtukas, generalizing \cite{XZ}. We will not discuss details here as we shall formulate a conjecture in the Shimura variety setting. 
\end{example}

\begin{example}
Let us consider the Drinfeld modular varieties associated to $G$, which would be the analogue of Shimura varieties over function fields. We fix a place of $X$ degree one called $\infty$.
For simplicity, we assume that $G$ is split (with a pinning), and 
suppose $\mG$ is the group scheme over $X$ such that $\mG|_{X-\{\infty\}}=G\times (X-\{\infty\})$ and that
$\mG_\infty$ is the Iwahori group scheme (determined by the pinning). 

Let $V_\mu$ be a minuscule representation of $\hat{G}$ of highest weight $\mu$. The central character of $V_\mu$ is denoted by $[\mu]\in \xch(Z_{\hat{G}})$.
Let $w_\mu\in\Omega_\infty$ (see \eqref{E: length zero}) be the unique element in the Iwahori-Weyl group of $G(\breve F_\infty)$ such that its image in $\xch(Z_{\hat{G}})$ under the Kottwitz map is $-[\mu]$. 
We choose a level structure $K\subset G(\bO)$ for a finite set $S_K$ away from $\infty$. Then we define the Drinfeld modular variety $\on{Dr}_{K}(G,\mu)$ associated to $(G,\mu, K)$ as the moduli of $\mG$-Shtukas on $X$ with a leg at $\overline\eta$ of singularity bounded by $V_\mu$, a leg at $\infty$ with singularity bounded by $w_\mu$, and level structure $\Lambda$. For example, when $G=\GL_2$, $V_\mu$ is the dual standard representation of $\hat{G}=\GL_2$ (in which case we can take a representative of $w_\mu$ in $\GL_2(\breve F_\infty)$ as $\begin{pmatrix}  & 1\\ \varpi_\infty & \end{pmatrix}$ where $\varpi_\infty$ is a uniformizer of $F_\infty$),
 this gives back to the original Drinfeld modular curve.

The compactly supported cohomology $C_c(\on{Dr}_\Lambda(G,\mu),\Lambda)$ is a special case of the cohomology considered in Conjecture \ref{C: gen coh}. Namely, let $I=\{1\}$, $S=\{\infty\}\cup S_K$.
Let $\mK_\infty$ be the push-forward of the constant sheaf $\Lambda$ along
$[*/I_b]\cong L^+\mG_{\infty} \cdot w_\mu/\Ad_\sigma L^+\mG_{\infty}\subset \Sht^{\loc}_{\infty}$ (see \eqref{E: delta on b}), and let $\mK_{v}$ at other places $v\neq \infty$ in $S$ as in Example \ref{Ex: gen coh}. Then
\[
C_c\bigl(\on{Dr}_\Lambda(G,\mu),k\bigr)\cong C_c\bigl(\Sht_{\Delta(\overline\eta), S}, \Sat(V)\otimes \res^!(\boxtimes_{v\in S_K}\mK_{v} \boxtimes \mK_\infty)\bigr)
\]
On the other hand, we should have $\frakA_{\mK_\infty}\simeq \frakA_{G_b,I_b}$ by Conjecture \ref{Conj: cat LLC}. 
Then Conjecture \ref{C: gen coh} predicts
\[
C_c\bigl(\on{Dr}_\Lambda(G,\mu),k\bigr)\cong \Gamma\bigl(\Loc_{{}^cG,F,S}, {}_{W_{F,S}}V\otimes \res^!(\boxtimes_{v\in S_K}\frakA_{K_{v}}\otimes \frakA_{G_b,I_b})\bigr).
\]
\end{example}

\begin{example}\label{E: Igusa}
We can also consider the compactly supported cohomology of the so-called Igusa varieties. For simplicity, we assume that $G$ is split and $\mG=G\times X$. 
We fix a place $v_0$. 
Let $\Sht_{v_0,K}$ be the moduli of $\mG$-Shtukas on $X$ with a leg at $v_0$ and $K$-level structure at a set of finite places $S_K$ disjoint with $v_0$. We have
$\res: \Sht_{v_0,K}\to \Sht^{\loc}_{v_0}$.
Let $x$ be an $\overline\bF_q$-point of $\Sht^{\loc}_{v_0}$, i.e. a local Shtuka with leg at $v_0$. Let $b$ be the associated element in $B(G_{v_0})$. Then the automorphism $\Aut_x$ is an affine group scheme over $\overline\bF_q$, and we have  
$[*/\Aut_x]\to \Sht^{\loc}_{v_0}$. The central leaf $C_{v_0,K,x}$ in $\Sht_{v_0,K}$ is defined as the fiber product 
$$C_{v_0,K,x}:=\Sht_{v_0,K}\times_{\Sht^\loc_{v_0}}[*/\Aut_x],$$ 
while the Igusa variety is defined as the fiber product
\[
\on{Ig}_{v_0,K,x}:=\Sht_{v_0,K}\times_{\Sht^\loc_{v_0}} \{x\},
\]
which is an $\Aut_x$-torsor over $C_{v_0,K,x}$. 
The dimension of both are $d=(2\rho,\nu_b)$, where $\nu_b$ is the Newton point of $b$ (as in Remark \ref{R: AJb}).
Its compactly supported cohomology also appears in Conjecture \ref{C: gen coh}. Namely, let $I=\emptyset$ and $S=\{v_0\}\cup S_K$. Let $\mK_{v_0}=\varinjlim_m x_{m,!}\Lambda[d]$, where $x_{m}: [*/\Aut_{x,m}]\to [*/\Aut_x]\to\Sht^{\loc}_{v_0}$ and $\Aut_{x,m}\subset \Aut_x$ is a system of normal subgroups such that $\Aut_x/\Aut_{x,m}$ is (perfectly) of finite type.
Let $\mK_{v}$ ($v\in S_K$) be the sheaf associated to the level structure $K_v$ as in Example \ref{Ex: gen coh}. 
Then
\[
C_c\bigl(\on{Ig}_{v,K,x},\Lambda[d]\bigr)\cong C_c\bigl(\Sht_{S}, \res^!((\boxtimes_{v\in S_K}\mK_{v})\otimes \mK_{v_0})\bigr).
\]
Let $\frakA_{G_b,\{1\}}$ be the ind-coherent sheaf from Remark \ref{R: AJb}. Then Conjecture \ref{C: gen coh} predicts that
\[
C_c\bigl(\on{Ig}_{v,K,x},\Lambda[d]\bigr)\cong  \Gamma\bigl(\Loc_{{}^cG,F,S}, \res^!((\boxtimes_{v\in S_K}\frakA_{K_{v}})\boxtimes \frakA_{G_b,\{1\}})\bigr).
\]
\end{example}

Now we turn to the number field case. In fact, the work \cite{XZ} on the Jacquet-Langlands transfer via the cohomology of Shimura varieties motivated all the conjectures discussed here. Therefore, we should expect analogous conjectural formulas for the cohomology of Shimura varieties\footnote{It would be quite interesting to explore whether the cohomology of locally symmetric spaces admits similar descriptions.}, although we currently lack a description of $\frakA_{K_v}$ at places above $\ell$ and $\infty$. (In particular, the sheaf at $\ell$ or $\infty$ is expected to encode information about the "weights".) Additionally, we do not yet have a stack of global Langlands parameters in the number field case. Consequently, we defer a precise formulation of the analogues of Conjecture \ref{Con: local-global} and \ref{C: gen coh} for number fields to \cite{EZ}.

Here we formulate a conjecture, which would be a generalization of one of the main results of \cite{XZ}, and would imply the geometric realization of the Jacquet-Langlands correspondence between inner forms that are different at $\{p,\infty\}$ (the work \cite{XZ} only gives JL transfers between inner forms that are different at $\infty$). 
Let $(G,X)$ be a Shimura datum. Let $V_\mu$ denote the irreducible representation of $\hat{G}$ of highest weight $\mu$ associated to the Shimura cocharacter of $G$ in the usual way.
Let $p$ be a prime, and $\mG_p$ a parahoric model of $G_{\bQ_p}$. Let $K=K_pK^p$ be a level with $K_p=\mG_p(\bZ_p)$. Recall that we (for simplicity) assume that the center $Z_G$ of $G$ is connected. In addition, we make the following assumptions:
\begin{itemize}
\item The maximal anisotropic torus in $Z_G$ is anisotropic over $\bR$;
\item The group $G$ satisfies the Hasse principle;
\item The $G(\bR)$-conjugacy class $X$ of $h:\bS\to G_\bR$ is in fact a $G_\ad(\bR)$-conjugacy class.
\end{itemize}
The first assumption is essential in order to relate Shimura varieties with moduli of local shtukas.
The last two assumptions are imposed to simplify the exposition. They can be dropped if one considers certain union of Shimura varieties in the sequel. 

Let $\Sh_K(G,X)$ be the corresponding Shimura variety (defined over the reflex field $E$), and we assume that it has a canonical reduction mod $p$. 
Let $\Sh_{G,\mu,K}$ denote the perfection of the mod $p$ fiber base changed to $\overline\bF_p$. Let $\Sht^\loc_p$ denote the corresponding moduli of local $\mG_p$-shtukas with leg at $p$, also base changed to $\overline\bF_p$. We assume that there is a perfectly smooth morphism
\[
\res: \Sh_{G,\mu,K}\to \Sht^\loc_{p,\mu},
\]
where $\Sht^\loc_{p,\mu}\subset\Sht^{\loc}_p$ is the closed substack consisting of those local $\mG_p$-shtukas with singularities bounded by $\mu$ in appropriate sense.
We note that when $(G,X)$ is of abelian type, such mod $p$ fiber $\Sh_{G,\mu,K}$ is constructed in \cite{KP} and the morphism $\res$ is constructed in \cite{SYZ} under some mild restrictions.

Now for $\mK_p\in\Shv(\Sht^\loc_{p,\mu})$, we obtain a sheaf $\res^!\mK_p$ on $\Sh_{G,\mu,K}\otimes\overline\bF_p$. 
As in Conjecture \ref{C: gen coh}, we may consider the compactly supported cohomology $C_c(\Sh_{G,\mu,K}, \res^!\mK_p)$.
One can keep the following two examples in mind.
\begin{itemize}
\item If $\res^!\mK_p=R\Psi$ is the nearby cycles of the shifted constant sheaf $\Lambda[d]$ on the generic fiber $\Sh_K(G,X)$, where $d=\dim \Sh_{G,\mu,K}$,
then $C_c(\Sh_{G,\mu,K}, \res^!\mK_p)$ is isomorphic to the (shifted) compactly supported cohomology of $\Sh_{K}(G,X)$ by \cite[5.20]{LS}, and $\frakA_{\mK_p}$ should be $({}_{W_p}V)\otimes \frakA_{K_p}$ as in Example \ref{E: nearby cycle}. 
\item If $\res^!\mK_p$ is the push-forward to $\Sh_{G,\mu, K}$ of the shifted constant sheaf $\Lambda[d]$ on an Igusa variety $\on{Ig}_{p,x,K}$, where $d$ is the dimension of  $\on{Ig}_{p,x,K}$, then $C_c(\Sh_{G,\mu,K}, \res^!\mK_p)$ is isomorphic to $C_c(\on{Ig}_{p,x,K},\Lambda[d])$ and $\frakA_{\mK_p}$ should be $\frakA_{G_b,\{1\}}$ as in Example \ref{E: Igusa}.
\end{itemize}

Now $(G,X)$ and $(G',X')$ be two Shimura data satisfying the above conditions, and we fix auxiliary choices for each of them. Let $p$ be a prime.
We assume that there is an inner twist $\Psi: G\to G'$ (which identifies the dual group of $G$ and $G'$ via $\Psi$) such that $\beta_v=\beta'_v$ for all $v\neq p$. 
This in particular implies there is a well-defined isomorphism $\theta: G(\bA^p_f)\cong G'(\bA^p_f)$ up to $G(\bA_f^p)$-conjugacy. We fix such an isomorphism.
Let $\mu$ and $\mu'$ denote the corresponding Shimura cocharacters, giving irreducible representation $V_\mu$ and $V_{\mu'}$ of $\hat{G}$.

We choose a prime-to-$p$ level $K^p\subset G(\bA_f^p)$, and let ${K'}^p=\theta(K^p)$.
Let $K_p\subset G(\bQ_p)$ and $K'_p \subset G'(\bQ_p)$ be parahoric subgroups. Write $H_{K^p}=H_{{K'}^p}$ for the corresponding prime-to-$p$ Hecke algebra. Choose $\mK_p\in\Shv(\Sht^{\loc}_{p,\mu})$ and $\mK'_p\in\Shv(\Sht^{\loc}_{p,\mu'})$.
Conjecture \ref{C: gen coh} suggests the following.

\begin{conjecture}
There is a natural map
\[
\Hom_{\Coh(\Loc_{p})}\bigl(({}_{W_p}V)\otimes \frakA_{\mK_p}, ( {}_{W_p}V')\otimes \frakA_{\mK'_p}\bigr)\to \Hom_{H_{K^p}}\bigl(C_c(\Sh_{G,\mu,K}, \res^!\mK_p), C_c(\Sh_{G',\mu',K'},\res^!\mK'_p)\bigr),
\]
compatible with compositions. 
In the particular case when $G=G'$ and $\Psi, \theta$ are the identity map, and $\res^!\mK_p=\res^!\mK'_p=R\Psi$ as above, we obtain an action
$$S:\End_{\Coh(\Loc_{p})}\bigl( ({}_{W_p}V)\otimes \frakA_{K_p}\bigr)\to \End_{Z_{p}^{\ta}\otimes H_{K^p}}\bigl(C_c(\Sh_K(G,X),\Lambda)\bigr),$$
where $Z_{p}^{\ta}=H^0\Gamma(\Loc_p^{\ta},\mO)$ be the tame stable center \eqref{E: stable center pro-p-Iw}, which should act on $C_c(\Sh_K(G,X),\La)$  
through the map $Z_{p}^{\ta}\to Z(H_{K_p})$ (see \eqref{E: stable center K1}).  
The composition 
$$H_{K_p}\cong \End(\frakA_{K_p})\to \End\bigl( ({}_{W_p}V)\otimes \frakA_{K_p}\bigr)\xrightarrow{S} \End_{Z_{p}^{\ta}}\bigl(C_c(\Sh_K(G,X),\La)\bigr)$$
should coincide with the natural Hecke action of $H_{K_p}$ on $C_c(\Sh_K(G,X),\La)$.
\end{conjecture}

\begin{remark}\label{R: derived global}
The works of \cite{XZ,Yu2} confirm a weak form of this conjecture in the case
$G\otimes\bA_f\cong G'\otimes\bA_f$ and $K_p$ is hyperspecial.
But we note that even in this case, the conjecture is stronger. Namely,
the derived Hecke algebra $H_{K_p}$ acts on $C_c(\Sh_K(G,X),\La)$, when $C_c(\Sh_K(G,X),\La)$ is regarded as a $Z_{p}^{\ta}$-module\footnote{Unlike the cohomology of general locally symmetric space as considered in \cite{Ve,F}, 
the derived Hecke action is invisible when $C_c(\Sh_V, k)$ is merely regarded as a $\Lambda$-module.}. So the conjecture includes a derived $S=T$ statement.
\end{remark}

Finally, let us briefly discuss the local analogue of the above conjectures, which is a conjectural formula of cohomology of (generalized)
Rapoport-Zink spaces. In fact, such conjectural formula is more or less built into the conjectural properties of the equivalence $\bL_G$ from Conjecture \ref{Conj: cat LLC}.

We assume that $G$ is over a local field $F$ and let $\mG$ be a parahoric model of $G$ over $\mO$.
Let $(G,b,\mu)$ be a local Shimura datum in the sense of \cite[5.1]{RV}. I.e. $b\in B(G)$ and $\mu$ is a minuscule dominant weight of $\hat{G}$ such that $\kappa_G(b)=\mu|_{Z_{\hat{G}}^{\Ga_F}}\in\xch(Z_{\hat{G}}^{\Ga_F})$. In this case, Rapoport and Viehmann expect that there is a tower of rigid analytic varieties $\{\on{RZ}_{G,b,\mu,K}\}_\Lambda$ (denoted by $\{\bM^K\}$ in \cite[\S 5]{RV}) over $\breve E$ indexed by open compact subgroups $K\subset\mG(\mO_F)$, as the local analogue of Shimura varieties. Here $\breve E$ is the completion of a maximal unramified extension of the reflex field $E$ of $\mu$. For certain $(G,b,\mu)$ and $K=\mG(\mO_F)$, $\on{RZ}_{G,b,\mu,K}$ can be realized as the rigid generic fiber of the corresponding Rapoport-Zink space. (This tower in general has been constructed in \cite[\S 24]{SW}.)
We refer to \cite{RV} for some expected properties of this tower, except mentioning that the compactly supported cohomology $C_c(\on{RZ}_{G,b,\mu,K}\otimes \overline{\breve E}, k)$ should afford the action of $H_K\times  W_{E}\times G_b(F)$, and as a $G_b(F)$-representation, it should belong to $\Rep_{\on{f.g.}}(G_b(F),\La)$. Let $\frakA_{G_b,\{1\}}$ be the ind-coherent sheaf from Remark \ref{R: AJb}.

\begin{conjecture}
We have an $H_K\times  W_{E}\times G_b(F)$-equivariant isomorphism
\[
C_c(\on{RZ}_{G,\mu,b,K}\otimes \overline{\breve E}, \Lambda[(2\rho,\mu)])\cong \Hom_{\Loc_{{}^cG,F}}\bigl(\frakA_{G_b,\{1\}},({}_{W_F}V_{\mu})\otimes\frakA_{G,K}\bigr).
\]
\end{conjecture}

One easily check that this formula holds when $b=1$ and $\mu=0$. We end with a few remarks.

\begin{remark}
\begin{enumerate}
\item First, similar to the global case, 
this conjecture can be regarded as a refinement of Kottwitz' and Harris-Viehmann's conjecture on the cohomology of Rapoport-Zink spaces (\cite{RV}). 

\item Assume that $b$ is basic. One can apply ${}'\bD^{\on{Se}}$ to the right hand side of the formula and see that the that the cohomology of RZ spaces for $(G,\mu,b)$ and $(G_b,-\mu,-b)$ should become isomorphic at the infinity level.
This is consistent with the fact that the two towers for $(G,\mu,b)$ and $(G_b,-\mu,-b)$ become isomorphic at infinite level (\cite[5.8]{RV} and \cite[23.3.2]{SW}). Also note that we conjecture that $\frakA_{J_b,\{1\}}$ is a connective (ind-)coherent sheaf (Remark \ref{R: AJb}) and $\frakA_{G,K}$ is an ordinary coherent sheaves (Conjecture \ref{Conj: coh sheaf} \eqref{Conj: coh sheaf-2}), so r.h.s. only concentrates in non-negative degrees. This means that the compactly supported cohomology of (basic) Rapoport-Zink spaces should vanish below the middle degree, which is consistent with the general expectation.  In addition, similar to Remark \ref{R: mid deg}, we expect that over isolated smooth points of $\Loc_{{}^cG,F}$, the right hand side should only  concentrate in degree zero.

\item Finally, the generalization of this conjectural formula to non-minuscule and multiple leg situation (i.e. the generalized Rapoport-Zink spaces as introduced in \cite[\S 23]{SW}) is immediately.
\end{enumerate}
\end{remark}

\bibliographystyle{amsplain}

\begin{thebibliography}{999999}
\bibitem{An}J. Ant\'onio, Moduli of $p$-adic representations of a profinite group, arXiv:1709.04275.
\bibitem{AG}D. Arinkin, D. Gaitsgory, Singular support of coherent sheaves, and the geometric Langlands conjecture, {\it Selecta Math.} {\bf 21}, (2015), 1--199.
\bibitem{AGK}D. Arinkin, D. Gaitsgory, D. Kazhdan, S. Raskin, N. Rozenblyum, Y. Varshavsky, The stack of local systems with restricted variation and geometric Langlands theory with nilpotent singular support, {\it arXiv:2010.01906}.
\bibitem{BG}R. Bellovin, T. Gee, $G$-valued local deformation rings and global lifts, {\it Algebra and Number Theory} {\bf 13} no. 2, (2019), 333--378.
\bibitem{BCHN}D. Ben-Zvi, H. Chen, D. Helm, D. Nadler, Coherent Springer theory and the categorical Deligne-Langlands correspondence,  {\it Invent. math.} {\bf 235}, (2024), 255--344.
\bibitem{Be}R. Bezrukavnikov, On two geometric realizations of an affine Hecke algebra, {\it Publ. Math. IH\'ES} {\bf 123}, (2016), 1--67.
\bibitem{BHKT} G. B\"ockle, M. Harris, C. Khare, and J. Thorne, $\hat{G}$-local systems on smooth projective curves are potentially automorphic, {\it Acta Math.} {\bf 223} no. 1, (2019), 1-111.
\bibitem{BK}G. B\"ockle,  C. Khare, Mod $\ell$ representations of arithmetic fundamental groups II: A conjecture of A. J. de Jong, {\it Comp. Math.} {\bf 142}, (2006), 271--294.
\bibitem{Bo}A. Borel, Automorphic $L$-functions, in {\it Automorphic forms, representations and $L$-functions}, {\it Proc. Symp. Pure Math. Am. Math. Soc.}, Corvallis/Oregon 1977, {\it Proc. Symp. Pure Math.} {\bf 33} Part {\bf2}, (1979), 27--61.
\bibitem{BS}A. Borel, J.P. Serre, Sur certains sous-groupes des groupes de Lie compacts, {\it Comment. Math. Helv.}, {\bf 27}, (1953), 128--139.
\bibitem{BF}A. Braverman, M. Finkelberg, A quasi-coherent description of the category $\mathrm{D}$-$\on{mod}(\Gr_{\GL(n)})$, in {\it Representation Theory and Algebraic Geometry}, Trends in Mathematics, Birkh\"auser, (2022), 133--149.
\bibitem{BG}K. Buzzard, T. Gee, The conjectural connections between automorphic representations and Galois representations, {\it Proceedings of the LMS Durham Symposium}, 2011.
\bibitem{CaSh}B. Casselman, J. Shalika, The unramified principal series of $p$-adic groups. II. The Whittaker function, {\it Comp. Math.} {\bf 41}, (1980), 207--231.
\bibitem{Sch}K. \v{C}esnavi\v{c}ius, P. Scholze, Purity for flat cohomology, {\it Ann. of Math.} {\bf 199}, (2024), 51--180.
\bibitem{CZ}T.-H. Chen and X. Zhu, Geometric Langlands in prime characteristic, {\it Comp. Math.} {\bf 153} (2017), 395--452.
\bibitem{Ch}G. Chenevier, The $p$-adic analytic space of pseudocharacters of a profinite group, and pseudorepresentations over arbitrary rings, in {\it Automorphic forms and Galois representations : Volume 1}, {\it London Mathematical Society Lecture Note Series} {\bf 414}, (2014), 221-285.
\bibitem{Cotner}S. Cortner, Morphisms of Character Varieties, {\it Int. Math. Res. Not.} {\bf 2024} Issue 16, (2024), 11540--11548.
\bibitem{Da}J.F. Dat, Finitude pour les repr\'esentations lisses de groupes $p$-adiques, {\it  J. Inst. Math. Jussieu} {\bf 8} (2009), 261--333. 
\bibitem{DHKM}J.F. Dat, D. Helm, R.  Kurinczuk, G. Moss, Moduli of Langlands parameters, Moduli of Langlands parameters, {\it J. Eur. Math. Soc.} {\bf 27} no. 5, (2024),1827--1927.
\bibitem{DHKM2}J.F. Dat, D. Helm, R.  Kurinczuk, G. Moss, Finiteness for Hecke algebras of $p$-adic groups, {\it J. Amer. Math. Soc.} {\bf 37}, (2024), 929--949.
\bibitem{Do}S. Donkin, Invariants of several matrices, {\it Invent. Math.} {\bf 110} (1992) no. 2, 389--401.
\bibitem{de}A J. de Jong, A conjecture on arithmetic fundamental groups, {\it Isr. J. Math.} {\bf 121} (2001), 61--64.
\bibitem{EG}M. Emerton, T. Gee, Moduli stack of \'etale $(\varphi,\Ga)$-modules and the existence of crystalline lifts, {\it arXiv: 1908.07185}.
\bibitem{EGH}M. Emerton, T. Gee, E. Hellman, An introduction to the categorical $p$-adic Langlands program, {\it arXiv: 2210.01404}.
\bibitem{EH}M. Emerton, D. Helm, The local Langlands correspondence for $\GL_n$ in families, {\it Ann. Sci. \'Ec. Norm. Sup\'er.} {\bf 47} no. 4,  (2014), 655--722.
\bibitem{EZ}M. Emerton, X. Zhu, in preparation.
\bibitem{FS}L. Fargues, P. Scholze, Geometrization of the local Langlands correspondence, arXiv:2102.13459.
\bibitem{F}T. Feng, The spectral Hecke algebra, {\it arXiv:1912.04413}.
\bibitem{FVDK}V. Franjou and W. van der Kallen, Power reductivity over an arbitrary base, {\it Doc. Math.}, Extra Volume Suslin, (2010), 171--195.
\bibitem{Ga2}D. Gaitsgory, On De Jong's conjecture, {\it Isr. J. Math.} {\bf 157}, (2007), 155--191. 
\bibitem{Ga3}D. Gaitsgory, Outline of the proof of the geometric Langlands conjecture for $\GL(2)$, {\it Ast\'erisque} {\bf 370}, (2015), 1--112.
\bibitem{Ga}D. Gaitsgory, From geometric to function-theoretic Langlands (or how to invent shtukas), {\it arXiv:1606.09608}.
\bibitem{GKRV}D. Gaitsgory, D. Kazhdan, N. Rozenblyum, Y. Varshavsky,  A toy model for the Drinfeld-Lafforgue Shtuka construction, {\it Indagationes Mathematicae} {\bf 33} Issue 1, (2022), 39--189.
\bibitem{GR}D. Gaitsgory, N. Rozenblyum, DG indschemes, {\it Perspectives in representation theory}, Contemp. Math. {\bf 610}, Amer. Math. Soc., Providence, RI, (2014), 139--251.
\bibitem{SV}S. Galatius, A. Venkatesh, Derived Galois deformation rings, {\it Adv. Math.} {\bf 327}, (2018), 470--623.
\bibitem{GL}A. Genestier, V. L\'afforgue, Chtoucas restreints pour les groupes r\'eductifs et param\'etrisation de Langlands locale, {\it arXiv:1709.00978}.
\bibitem{Ha}T. Haines, The stable Bernstein center and test functions for Shimura varieties. In {\it Automorphic forms and Galois representations}, {\it Proc. LMS} {\bf 415}, (2014), 118-186. 
\bibitem{Hel}E. Hellmann, On the derived category of the Iwahori-Hecke algebra, {\it Compos. Math.} {\bf 159}, (2023), 1042--1110.
\bibitem{He16}D. Helm, Whittaker models and the integral Bernstein center for $\GL_n$, {\it Duke Math. Journal} {\bf165} no. 9, (2016), 1597--1628.
\bibitem{HM1}D. Helm, Curtis Homomorphisms and the integral Bernstein center for $\GL_n$, {\it Algebra Number Theory} {\bf 14} no. 10, (2020), 2607--2645.
\bibitem{HM}D. Helm, G. Moss, Converse theorems and the local Langlands correspondence in families, {\it Invent. Math.} {\bf 214} (2018) 999--1022.
\bibitem{Iw}K. Iwasawa, On Galois groups of local fields, {\it Trans. of the AMS}, {\bf 80} (1955) no. 2, 448--469.
\bibitem{K} V. V. Kashin, Orbits of adjoint and coadjoint actions of Borel subgroups of semisimple algebraic groups, {\it Questions of Group Theory and Homological Algebra (Russian)}, Yaroslavl, (1990) 141--159. 
\bibitem{Ki}M. Kisin, Moduli of finite flat group schemes, and modularity, {\it Ann. of Math.} {\bf 170} (2009), 1085--1180.
\bibitem{KP}M. Kisin, G. Pappas, Integral models for Shimura varieties with parahoric level structure, {\it Publ. Math. IH\'ES} {\bf 128} (2018), 121--218. 
\bibitem{Ko}R. Kottwitz, Isocrystals with additional structure, {\it Compos. Math.} {\bf 56} (1985), 201--220. 
\bibitem{Ko2}R. Kottwitz, Isocrystals with additional structure II, {\it Compos. Math.} {\bf 109} (1997) 255--339.
\bibitem{LL}L. Lafforgue, Chtoucas de Drinfeld et correspondance de Langlands, {\it Invent. math.} {\bf 147} (2002), 1--241.
\bibitem{L}V. Lafforgue, Chtoucas pour les groupes r\'eductifs et param\`etrisation de Langlands globale, {\it J. Amer. Math. Soc.} {\bf 31} (2018), 719--891.
\bibitem{LZ}V. Lafforgue, X. Zhu, D\'ecomposition au-dessus des param\`etres de Langlands elliptiques, {\it arXiv: 1811.07976}.
\bibitem{LS}K.-W. Lan, B. Stroh, Nearby cycles of automorphic \'etale sheaves, {\it Compos. Math.} {\bf 154}, (2018), 80--119.
\bibitem{La}R. Langlands, Representations of abelian algebraic groups, {\it Pacific J. Math.} {\bf 181}, (1997), 231--250.
\bibitem{LTXZZ}Y. Liu, Y. Tian, L. Xiao, W. Zhang, X. Zhu, Deformation of rigid conjugate self-dual Galois representations, {\it  Acta Mathematica Sinica, English Series} {\bf 40}, (2024), 1599--1644.
\bibitem{Lu1}J. Lurie, Higher Topos Theory, {\it Ann. Math. Study} {\bf 170}, 2009.
\bibitem{Lu2}J. Lurie, Higher Algebra, available at \url{https://www.math.ias.edu/~lurie/}.
\bibitem{Lu3}J. Lurie, Spectral algebraic geometry, available at\url{https://www.math.ias.edu/~lurie/}.
\bibitem{Lu4}J. Lurie, Derived algebraic geometry, available at \url{https://www.math.ias.edu/~lurie/papers/DAG.pdf}.
\bibitem{Man}J. Manning, Patching and Multiplicity $2^k$ for Shimura Curve, arXiv:1902.06878.
\bibitem{Ma}B. Martin, Reductive subgroups of reductive groups in nonzero characteristic, {\it Journal of Algebra} {\bf 262}, (2003) 265--286.
\bibitem{MV}I. Mirkovi\'c, K. Vilonen, Geometric Langlands duality and representations of algebraic groups over commutative rings, {\it Ann.of Math.} {\bf 166}, (2007), 95--143.
\bibitem{PY} G. Prasad, J.-K. Yu, On finite group actions on reductive groups and buildings, {\it Invent. Math.} {\bf 147} no. 3, (2002), 545--560.
\bibitem{RV} M. Rapoport, E. Viehmann, Towards a theory of local Shimura varieties, {\it M\"unster J. Math.} {\bf 7} no. 1, (2014), 273--326.
\bibitem{Ro}F. Rodier, Mod\`ele de Whittaker et caract\`eres de repr\'esentations, in {\it Non-commutative harmonic analysis}, {\it Lecture Notes in Math.} {\bf 466}, (1975), 151--171. 
\bibitem{Sc}P. Schneider, Smooth representations and Hecke modules in characteristic $p$, {\it Pacific J. Math.} {\bf 279}, (2015), 447--464.
\bibitem{Schllc}P. Scholze, The Local Langlands Correspondence for $\GL_n$ over $p$-adic fields, {\it Invent. Math.} {\bf 192} no. 3, (2013), 663--715.
\bibitem{SW}P. Scholze, J. Weinstein, Moduli of $p$-divisible groups, {\it Cambridge Journal of Mathematics} {\bf 1}, (2013), 145--237.
\bibitem{SW}P. Scholze, J. Weinstein, Berkeley lectures on $p$-adic geometry, {\it Ann. Math. Studies} {\bf 141}, Princeton University Press, 2020.
\bibitem{SYZ}X. Shen, C.-F. Yu, C. Zhang, EKOR strata for Shimura varieties with parahoric level structure, {\it Duke Math. J.} {\bf 170} no. 14, (2021), 3111--3236.
\bibitem{To}B. To\"en, Derived algebraic geometry, {\it EMS Surv. Math. Sci.} {\bf 1}, (2014), 153--240.
\bibitem{TVDK}Van Der Kallen, Good Grosshans filtration in a family, {\it Panoramas et synth\`eses} {\bf 47}, (2015), 111--129.
\bibitem{TVDK2}Van Der Kallen, Reductivity properties over an affine base, {\it Indagationes Mathematicae} {\bf 32} Issue 5, (2021), 961--967.
\bibitem{Ve}A. Venkatesh, Derived Hecke algebra and cohomology of arithmetic groups, {\it Forum of Math., Pi}, 7, E7.
\bibitem{Vi}M. Vigneras, The pro-$p$-Iwahori Hecke algebra of a reductive $p$-adic group III, {\it  J. Inst. Math. Jussieu}, (2015), 1--38.
\bibitem{Vin} E.B. Vinberg, On invariants of a set of matrices, {\it J. Lie Theory} {\bf 6}, (1996), 249--269.
\bibitem{WE}C. Wang-Erickson, Algebraic families of Galois representations and potentially semi-stable pseudodeformation rings, {\it Math. Ann.} {\bf 371}, (2018), 1615--1681. 
\bibitem{Wei}M. Weidner, Pseudocharacters of homomorphisms into classical groups, {\it Transformation Groups} {\bf 25}, (2020), 1345--1370.
\bibitem{XZ}L. Xiao, X. Zhu, Cycles on Shimura varieties via geometric Satake, {\it arXiv:1707.05700}.
\bibitem{XZ2}L. Xiao, X. Zhu, On vector-valued twisted conjugate invariant functions on a group, in {\it Representations of Reductive Groups}, Proceedings of Symposia in Pure Mathematics {\bf 101}, (2019), 361--426.
\bibitem{Xu}C. Xue, Finiteness of cohomology groups of stacks of shtukas as modules over Hecke algebras, and applications, {\it \'Epijournal de G\'eom\'etrie Alg\'ebrique}, {\bf 4} (2020).
\bibitem{Xu2}C. Xue, Cohomology with integral coefficients of stacks of shtukas, {\it arXiv:2001.05805}.
\bibitem{Xu3}C. Xue, Smoothness of cohomology sheaves of stacks of shtukas, {\it arXiv:2012.12833}.
\bibitem{Yu1}J. Yu, The integral geometric Satake equivalence in mixed characteristic, {\it Representation Theory} {\bf 26}, (2022), 874--905.
\bibitem{Yu2}J. Yu, A Geometric Jacquet-Langlands Transfer for automorphic forms of higher weights, {\it Trans. Amer. Math. Soc.} {\bf 375}, (2022), 6843--6873.
\bibitem{Z15}X. Zhu, The Geometric Satake Correspondence for Ramified Groups, {\it Ann. Sci. \'Ec. Norm. Sup\'er.} {\bf 48}, (2015), 409--451.
\bibitem{Z17a}X. Zhu, Affine Grassmannians and the geometric Satake in mixed characteristic, {\it Ann. Math.} {\bf 185}, (2017), 403--492.
\bibitem{Z18}X. Zhu, Geometric Satake, categorical traces, and arithmetic of Shimura varieties, {\it Current Developments in Mathematics}, (2016), 145--206.
\bibitem{Z20}X. Zhu, A note on integral Satake isomorphisms, {\it Arithmetic geometry}, Tata Inst. Fundam. Res. Stud.
Math., vol. 41, Tata Inst. Fund. Res., Mumbai, (2024), 469--489.
\bibitem{HZ}X. Zhu, Tame categorical local Langlands correspondence, {\it arXiv:2504.07482}.
\end{thebibliography}

\end{document}